\documentclass[11pt, twoside]{article}
\usepackage{amscd,amssymb}
\usepackage{amsthm,amsmath,amssymb}
\usepackage[matrix,arrow]{xy}
\usepackage{fancyhdr}
\usepackage{supertabular}
\usepackage{setspace}
\usepackage{color}
\usepackage{longtable}

\newenvironment{Note}
{\begin{list}{$\bullet$ }
{\setlength{\leftmargin}{7pt}%
\setlength{\labelsep}{3pt}%
\setlength{\itemindent}{5pt}%
\setlength{\labelwidth}{2pt}}}
{\end{list}}

\newtheorem{theorem}[equation]{Theorem}

\newtheorem*{theorem*}{Main Theorem}
\newtheorem*{corollary*}{Main Corollary}

\newtheorem{proposition}[equation]{Proposition}
\newtheorem{lemma}[equation]{Lemma}
\newtheorem{corollary}[equation]{Corollary}

\newtheorem{conjec}{Conjecture}
\newtheorem{corol}[conjec]{Corollary}
\newtheorem{theor}[conjec]{Theorem}

\theoremstyle{definition}
\newtheorem{example}[equation]{Example}
\newtheorem{definition}[equation]{Definition}
\newtheorem{exam}[conjec]{Example}

\theoremstyle{remark}
\newtheorem{remark}[equation]{Remark}

\usepackage{mathtools}
\newsavebox\FBox
\def\Boxed#1{\sbox\FBox{
   $\boxed{#1}$}\rule[-1.2\dp\FBox]{0pt}{1.0\ht\FBox}\usebox\FBox}

\newcommand{\mult}{\operatorname{mult}}

\newcommand{\boundary}{{\small \textcircled{b}}}
\newcommand{\nef}{{\small \textcircled{n}}}
\newcommand{\surface}{{\small \textcircled{s}}}
\newcommand{\family}{{\small \textcircled{f}}}
\newcommand{\positive}{{\small \textcircled{p}}}

\newcommand{\quadratic}{{\small \Boxed{\tau}}}
\newcommand{\quadraticone}{{\small \Boxed{\tau \text{\tiny 1}}}}

\newcommand{\elliptic}{{\small \Boxed{\epsilon}}}
\newcommand{\ellipticone}{{\small \Boxed{\epsilon\text{\tiny 1}}}}
\newcommand{\elliptictwo}{{\small \Boxed{\epsilon\text{\tiny 2}}}}

\newcommand{\ellipticfour}{{\small \Boxed{\iota}}}
\newcommand{\ellipticfive}{{\small \Boxed{\iota\text{\tiny 1}}}}

\newcommand{\vorder}{Vanishing order}

\makeatletter\@addtoreset{equation}{subsection} \makeatother

\makeatletter\@addtoreset{conjec}{section} \makeatother

\begin{document}

\thispagestyle{empty}
 \Large
\begin{center}
\textbf{Birationally rigid Fano threefold hypersurfaces}
\end{center}
\vspace{5mm}

\normalsize
\begin{center}
\textbf{Ivan Cheltsov and Jihun Park}
\end{center}

\vspace{5mm}

 \small {\noindent\textbf{Abstract.} We prove that every quasi-smooth weighted Fano threefold hypersurface  in the 95 families of Fletcher and Reid is birationally rigid.\vspace{3mm}

\vspace{3mm}

 }

\vspace{2mm}

\normalsize

\tableofcontents

\newpage

\section{Introduction}
\subsection{Birational rigidity and Main Theorem}

Let $V$ be a smooth projective variety. If its canonical class
$K_{V}$ is pseudo-effective, then Minimal Model Program produces a
birational model  $W$ of the variety $V$, so-called \emph{minimal
model}, which has mild singularities (terminal and
$\mathbb{Q}$-factorial) and the canonical class $K_W$ of which is nef. This
has been verified in dimension $3$ and in any dimension for
varieties of general type (see \cite[Theorem~1.1]{BCHM}).
Meanwhile, if the canonical class $K_{V}$ is not pseudo-effective,
then Minimal Model Program yields a birational model $U$ of $V$,
so-called \emph{Mori fibred space}. It also has  terminal and
$\mathbb{Q}$-factorial singularities and it admits a fiber
structure $\pi\colon U\to Z$ of relative Picard rank $1$ such that
the divisor $-K_{U}$ is ample on fibers. This has been proved in
all dimensions (see \cite[Corollary~1.3.3]{BCHM}).

Mori fibred spaces, alongside the minimal models, represent the
terminal objects in Minimal Model Program. If the canonical class
is pseudo-effective and its minimal models  exist, then they are
unique up to flops. However, this is not the case when the
canonical class is not pseudo-effective, since Mori fibred spaces
are usually not unique terminal objects in Minimal Model Program.
Nevertheless, some Mori fibred spaces behave very much the same as
minimal models. To  distinguish them, Corti introduced

\begin{definition}[{\cite[Definition~1.3]{Co00}}]
\label{definition:rigidity} Let $\pi\colon U\to Z$ be a Mori
fibred space. It is called \emph{birationally rigid} if for a
birational map $\xi\colon U\dasharrow U^{\prime}$ to  a Mori
fibred space $\pi^\prime\colon U^\prime\to Z^\prime$ there exist
a birational automorphism $\tau : U\dasharrow U$ and a birational
map $\sigma : Z\dasharrow Z'$ such that the birational map
$\xi\circ\tau$ induces an isomorphism between the~generic fibers
of the~Mori fibrations $\pi\colon U\to Z$ and $\pi^{\prime}\colon
U^\prime\to Z^\prime$ and the diagram
$$
\xymatrix{
&U\ar@{->}[d]_{\pi}\ar@{-->}[rr]^{\tau}&&U\ar@{-->}[rr]^{\xi}&&U^{\prime}\ar@{->}[d]^{\pi^{\prime}}&\\
&Z\ar@{-->}[rrrr]^{\sigma}&&&&Z^{\prime}&}
$$
commutes.
\end{definition}

Fano varieties of Picard rank one with at most terminal
$\mathbb{Q}$-factorial singularities are the basic examples of
Mori fibred spaces. For them, Definition~\ref{definition:rigidity}
can be simplified as follows:

\begin{definition}
\label{definition:Fano-rigidity} Let $V$ be a Fano variety of
Picard rank $1$ with at most terminal $\mathbb{Q}$-factorial
singularities. Then the Fano variety $V$ is called
\emph{birationally rigid} if the following property holds.
\begin{itemize}
\item If there is a birational map $\xi\colon V\dasharrow U$ to a
Mori fibred space $ U\to Z$,  then the Fano variety $V$ is
biregular to $U$ (and hence $Z$ must be a point).
\end{itemize} If, in
addition, the birational automorphism group of $V$ coincides with
its biregular automorphism group, then $V$ is called
\emph{birationally super-rigid}.
\end{definition}

Birationally rigid Fano varieties behave very much like
\emph{canonical models}. Their birational geometry is very simple.
In particular, they are non-rational. The first example of a
birationally rigid Fano variety is due to Iskovskikh and Manin. In
1971, they proved

\begin{theorem}[{\cite{IsM71}}]\label{theorem:IsM71}
\label{theorem:IM} A smooth quartic hypersurface in $\mathbb{P}^4$
is birationally super-rigid.
\end{theorem}

In fact, Iskovskikh and Manin only proved that smooth quartic
hypersurfaces in $\mathbb{P}^4$ do not admit any non-biregular
birational automorphisms and, therefore, they are non-rational. In late
nineties, Corti observed in \cite{Co95} that their proof implies
Theorem~\ref{theorem:IM}. Inspired by this observation, Pukhlikov
generalized Theorem~\ref{theorem:IM} as

\begin{theorem}[{\cite{Pu98}}]
\label{theorem:Pukhlikov} A general hypersurface of degree $n\geq
4$  in $\mathbb{P}^n$ is birationally super-rigid.
\end{theorem}

Shortly after Theorem~\ref{theorem:Pukhlikov} was proved, Reid
suggested to Corti and Pukhlikov that they should generalize
Theorem~\ref{theorem:IM} for singular threefolds. Together they
proved

\begin{theorem}[{\cite{CPR}}]
\label{theorem:CPR} Let $X$ be a quasi-smooth hypersurface of
degree $d$ with only terminal singularities in weighted
projective space $\mathbb{P}(1, a_1, a_2, a_3, a_4)$, where
$d=\sum  a_i$. Suppose that $X$ is a general hypersurface in this
family. Then $X$ is birationally rigid.
\end{theorem}

The singular threefolds in Theorem~\ref{theorem:CPR} have a long
history.
 In 1979 Reid discovered the 95
families of $K3$ surfaces in three dimensional weighted projective
spaces (see \cite{Rei79}). After this, Fletcher, who was a Ph.D.
student of Ried, announced the 95 families of weighted Fano
threefold hypersurfaces in his Ph.D. dissertation in 1988. These
are quasi-smooth hypersurfaces of degrees $d$ with only terminal
singularities in weighted projective spaces $\mathbb{P}(1, a_1,
a_2, a_3, a_4)$, where $d=\sum a_i$. The 95 families are
determined by the quadruples of non-decreasing positive integers
$(a_1, a_2, a_3, a_4)$. All Reid's 95 families of $K3$ surfaces
arise as anticanonical divisors of the Fano threefolds  in Fletcher's 95 families . Because of this, the latter  95 families are often
called the 95 families of Fletcher and Reid.

It is quite often that we need to know the non-rationality of an explicitly
given Fano variety (which does not follow from the non-rationality
of a general member in its family).

\begin{example}
\label{example:Prokhorov} Recently Prokhorov classified all finite
simple subgroups in the birational automorphism group
$\mathrm{Bir}(\mathbb{P}^3)$ of the three-dimensional projective
space. Up to isomorphism, $\mathrm{A}_5$,
$\mathrm{PSL}_{2}(\mathbb{F}_7)$, $\mathrm{A}_6$, $\mathrm{A}_7$,
$\mathrm{PSL}_{2}(\mathbb{F}_8)$ and
$\mathrm{PSU}_{4}(\mathbb{F}_2)$ are all non-abelian finite simple
subgroups in $\mathrm{Bir}(\mathbb{P}^3)$
(\cite[Theorem~1.3]{Prokhorov}). Prokhorov's proof implies more.
Up to conjugation, the group $\mathrm{Bir}(\mathbb{P}^3)$ contains a unique
subgroup isomorphic to $\mathrm{PSL}_{2}(\mathbb{F}_8)$ and exactly two
subgroups isomorphic to $\mathrm{PSU}_{4}(\mathbb{F}_2)$. For the alternating
group $\mathrm{A}_7$, he proved that $\mathrm{Bir}(\mathbb{P}^3)$
contains exactly one such subgroup provided that the threefold
\begin{equation}
\label{equation:A7-threefold}
\sum_{i=0}^{6}x_i=\sum_{i=0}^{6}x_i^2=\sum_{i=0}^{6}x_i^3=0\subset\mathrm{Proj}\Big(\mathbb{C}[x_0,\ldots,x_6]\Big)\cong\mathbb{P}^6
\end{equation}
is not rational. This threefold is the unique complete
intersection of a quadric and a cubic hypersurfaces in
$\mathbb{P}^5$ that admits a faithful action of $\mathrm{A}_7$.
Back in nineties Iskovskikh and Pukhlikov proved that a general
threefold in this family is birationally rigid (see
\cite{IskovskikhPukhlikov}). The
threefold~\eqref{equation:A7-threefold} is smooth. However, it
does not satisfy the generality assumptions imposed in
\cite{IskovskikhPukhlikov}. It is in 2012 that Beauville proved that the threefold~\eqref{equation:A7-threefold}
is not rational (see
\cite{Beauville}). It is still unknown whether it is birationally
rigid or not.
\end{example}

It took more than ten years to prove
Theorem~\ref{theorem:Pukhlikov} for \emph{every}  smooth
hypersurface in $\mathbb{P}^n$ of degree $n\geq 4$, which was
conjectured in \cite{Pu98}. This was done by de Fernex who proved

\begin{theorem}[{\cite{dF}}]
\label{theorem:deFernex} Every smooth hypersurface of degree
$n\geq 4$ in $\mathbb{P}^n$  is birationally super-rigid.
\end{theorem}

The goal of this paper is to prove Theorem~\ref{theorem:CPR} for
\emph{all} quasi-smooth hypersurfaces in each of the 95 families
of Fletcher and Reid, which was conjectured in \cite{CPR}. To be
precise, we prove

\begin{theorem*}
Let $X$ be a quasi-smooth hypersurface of degree $d$ with only
terminal singularities in the weighted projective space $\mathbb{P}(1,
a_1, a_2, a_3, a_4)$, where $d=\sum  a_i$. Then $X$ is
birationally rigid.
\end{theorem*}

Since birational rigidity implies non-rationality, we immediately  obtain
\begin{corollary}
\label{corollary:main} Let $X$ be a quasi-smooth hypersurface of
degree $d$ with only terminal singularities in the weighted
projective space $\mathbb{P}(1, a_1, a_2, a_3, a_4)$, where
$d=\sum  a_i$. Then $X$ is not rational.
\end{corollary}

In addition, the proof of Main Theorem shows

\begin{theorem}\label{theorem:auxiliary}
Every quasi-smooth hypersurface in the families of the 95 families
of Fletcher and Reid whose general members are birationally
super-rigid is birationally super-rigid.
\end{theorem}

The families corresponding to Theorem~\ref{theorem:auxiliary} are
those in the list of Fletcher and Reid with entry numbers No.~1,
3, 10, 11, 14, 19, 21, 22, 28, 29, 34, 35, 37,  39, 49, 50, 51,
52, 53, 55, 57, 59, 62, 63, 64, 66, 67, 70, 71, 72, 73, 75, 77,
78, 80, 81, 82, 83, 84, 85, 86, 87, 88,  89, 90, 91, 92, 93, 94
and  95 (see Section~\ref{section:table}).

The 95 families of Fletcher and Reid  contain the family (No.~1)
of quartic hypersurfaces in $\mathbb{P}^4$ and the family (No.~3)
of hypersurfaces of degree $6$ in $\mathbb{P}(1,1,1,1,3)$, i.e.,
double covers of $\mathbb{P}^3$ ramified along sextic surfaces.
 However, we do not consider these two families in the present
paper since every smooth quartic threefold and every smooth double
covers of $\mathbb{P}^3$ ramified along sextic surfaces (see
\cite{Is79}) are already proved to be birationally super-rigid.

\vspace{10mm} \textbf{Acknowledgements.} This work was motivated by
the authors' conversations  over several dinners at a kitchen of
Max Planck Institut f\"ur Mathematik (Bonn, Germany) and initiated
during the second author's participating in Mini-Workshop: Higher
Dimensional Elliptic Fibrations held at Mathematisches
Forschungsinstitut Oberwolfach (Germany) from October 3rd to
October 9th in 2010. It was carried out during the authors' stays at
Max Planck Institut f\"ur Mathematik, Institut des Hautes \'Etudes
Scientifiques (Bures-sur-Yvette, France), International Centre for
Mathematical Sciences (Edinburgh, UK), and the Centro
Internazionale per la Ricerca Matematica (Trento, Italy) in the
period 2010-2011. Both the authors appreciate their excellent
environments and supports for their research. The second author would
like to thank Professor Shigefumi Mori for his hospitality at
Research Institute for Mathematical Sciences (Kyoto, Japan) during
the second author's visit from August  15th to
August  30th in 2011. This work was finalized during the authors'  second
stay at the Centro Internazionale per la Ricerca Matematica from July 2nd to July 31st in 2013. They thank
Professor Marco Andreatta for his hospitality in Trento. They also thank the anonymous referee for his/her report that enables them
to improve Sections~\ref{subsection:hard-involutions} and \ref{section:invisible involution}.

The first author is partially supported by AG Laboratory GU-HSE,
RF government grant, ag. 11 11.G34.31.0023. The second author has
been supported by Institute for Basic Science (Grant No. CA1305-02) in Korea.

\subsection{How to prove Main Theorem}

In this section we present the synopsis of our proof of Main Theorem.
Before we proceed, we introduce a terminology that is frequently
used in  birational geometry as well as in the present paper.

\begin{definition}\label{defintion-center}
Let $U$ be a normal $\mathbb{Q}$-factorial variety and
$\mathcal{M}_U$ a mobile linear system (a linear system without a fixed component) on $U$. Let $a$ be a
non-negative rational number. An irreducible subvariety $Z$ of $U$
is called a center of non-canonical singularities (or simply
non-canonical center) of the log pair $(U,a\mathcal{M}_U)$ if there is
a birational morphism $h\colon W\to U$ and an $h$-exceptional
divisor $E_1\subset W$ such that
$$
K_W+ah^{-1}_*(\mathcal{M}_U)=h^*(K_U+a\mathcal{M}_U)+\sum_{i=1}^{m}c_iE_i,
$$
where each $E_i$ is an $h$-exceptional divisor, $c_1< 0$ and
$h(E_1)=Z$.
\end{definition}

The following result is known as the classical N\"other--Fano
inequality.

\begin{theorem}[{\cite[Theorem~4.2]{Co95}}]
\label{theorem:Nother-Fano} Let $X$ be a terminal
$\mathbb{Q}$-factorial Fano variety with $\mathrm{Pic}(X)\cong
\mathbb{Z}$.
\begin{itemize}
\item If the log pair $(X, \frac{1}{n}\mathcal{M})$  has canonical
singularities for every positive integer $n$ and every mobile
linear subsystem $\mathcal{M}$ in $|-nK_X|$, then $X$ is
birationally super-rigid.

\item If for every positive integer $n$ and every mobile linear
system $\mathcal{M}$ in $|-nK_X|$ there exists  a birational automorphism $\tau$ of $X$
such that the log pair $(X,
\frac{1}{n_\tau}\tau(\mathcal{M}))$ has canonical singularities,
where $n_\tau$ is the positive integer such that
$\tau(\mathcal{M})$ is contained in $|-n_\tau K_X|$, then $X$ is birationally
rigid.
\end{itemize}
\end{theorem}

The N\"other-Fano inequality will be the master key to the proof of
Main Theorem.

To prove Main Theorem, we take the following steps in order.
\medskip

 \textbf{Step 1.}
 We suppose that a given hypersurface
$X$ from the 95 families has a mobile linear system
$\mathcal{M}$ in $|-nK_X|$ for some positive integer $n$ such
that the log pair $(X, \frac{1}{n}\mathcal{M})$ is not canonical.
Then we must have a center of non-canonical singularities of the
pair $(X, \frac{1}{n}\mathcal{M})$.
A  center of non-canonical singularities of the log pair $(X,
\frac{1}{n}\mathcal{M})$ can be, \emph{a priori}, one of the following:
\[\left\{\aligned
& \mbox{a smooth point,}\\
& \mbox{an irreducible curve, }\\
& \mbox{a singular point} 
\endaligned \right.\]
on the Fano threefold $X$.

\medskip

 \textbf{Step 2.} We prove that a smooth point
of $X$ cannot be a center of non-canonical singularities of the
pair $(X, \frac{1}{n}\mathcal{M})$. This will be done in Section~\ref{section:smooth point} (Theorem~\ref{theorem:smooth point excluding}).

\medskip

 \textbf{Step 3.}  In Section~\ref{section:curves} we show that a curve contained in
the smooth locus of $X$ cannot be a center (Theorem~\ref{theorem:excluding-curve}). Then
Theorem~\ref{theorem:Kawamata} implies that a singular point of
$X$ must be a center.

\medskip
\textbf{Step 4.}
For a given singular point of the
hypersurface $X$  we prove that either
\begin{itemize}
\item it cannot be a center of non-canonical singularities of the log
pair $\left(X,\frac{1}{n}\mathcal{M}\right)$ (the job proving this part will be called \emph{excluding}) or 
\item there exists a birational
automorphism $\tau$ of $X$ such that $\tau(\mathcal{M})$ is
contained in $ |-n_\tau K_X|$ for some positive integer
$n_{\tau}<n$  (the job proving this part will be called \emph{untwisting}). 
\end{itemize}
With using induction on $n$, it then follows from
Theorem~\ref{theorem:Nother-Fano} that the given hypersurface $X$
is birationally rigid.
Step 4 will be done mainly in Section~\ref{section:super-rigid}. However, to exclude or untwist singular points, we will need several pieces of machinery, some of which are light and some of which are heavy. 
These machines will be assembled from Section~\ref{section:Methods}
 to Section~\ref{section:invisible involution}. In fact, the machines for excluding are relatively simple to use, so that they could be introduced in Section~\ref{section:Methods}. Meanwhile, the machines for untwisting are complicated to assemble. It will be carried out one by one from Section~\ref{section:untwisting} to Section~\ref{section:invisible involution}. 
Before using these machines in practical situation, i.e., before reading the tables in Section~~\ref{section:super-rigid}, we require the reader to be acquainted with the manual for the machinery provided in Section~\ref{section:manuals}.

Theorem~\ref{theorem:auxiliary} can be proved by excluding all
the singular points of $X$ as a center. Fifty  families out of the
95 families are those considered in
Theorem~\ref{theorem:auxiliary}. In
Section~\ref{section:super-rigid}, we are immediately able to notice that a singular point
of $X$ cannot be a center if the hypersurface $X$ belongs to one
of the families considered in Theorem~\ref{theorem:auxiliary}. Such families have the underlined entry numbers in their tables
in Section~~\ref{section:super-rigid}.

\subsection{Notations}\label{subsection:notation}
Let us describe the notations we will use in the rest of the
present paper. Unless otherwise mentioned, these notations are
fixed from now until the end of the paper.

\begin{itemize}
\item In the weighted projective space $\mathbb{P}(1,a_1,
a_2,a_3,a_4)$, we assume that $a_1\leq a_2\leq a_3\leq a_4$. For
weighted homogeneous coordinates, we always use $x$, $y$, $z$, $t$
and $w$ with weights $\mathrm{wt}(x)=1$, $\mathrm{wt}(y)=a_1$,
$\mathrm{wt}(z)=a_2$, $\mathrm{wt}(t)=a_3$ and
$\mathrm{wt}(w)=a_4$.

\item $f_m(x_{i_1}, \ldots,  x_{i_k})$,  $g_m(x_{i_1}, \ldots,
x_{i_k})$ and $h_m(x_{i_1}, \ldots,  x_{i_k})$ are
quasi-homogeneous polynomials of degree $m$ in variables $x_{i_1},
\ldots,  x_{i_k}$ in the given weighted projective space
$\mathbb{P}(1, a_1, a_2, a_3, a_4)$.

\item If a monomial appears individually in a quasi-homogeneous polynomial,
then the monomial is assumed  not to be contained in any other
terms. For example, in the  polynomial
$w^2+t^3+wf_{6}(x,y,z,t)+f_{12}(x,y,z,t)$, the polynomial $f_{12}$
does not contain the monomial $t^3$.

\item In each family, we always let $X$ be a quasi-smooth
hypersurface of degree $d$ in the weighted projective space
$\mathbb{P}(1,a_1, a_2,a_3,a_4)$ with only terminal singularities,
where $d=\sum_{i=1}^4 a_i$. We also use $X_d$, instead of $X$, in order to indicate the degree $d$ of $X$.

\item On the threefold $X$, a given mobile linear system is
denoted by $\mathcal{M}$.

\item For a given  mobile linear system $\mathcal{M}$, we always
assume that $\mathcal{M}\sim_{\mathbb{Q}} -nK_X$.

\item $S_x$ is the surface on the hypersurface $X$ cut by the
equation $x=0$.

\item $S_y$ is the surface on the hypersurface $X$ cut by the
equation $y=0$.

\item $S_z$ is the surface on the hypersurface $X$ cut by the
equation $z=0$.

\item $S_t$ is the surface on the hypersurface $X$ cut by the
equation $t=0$.

\item $S_w$ is the surface on the hypersurface $X$ cut by the
equation $w=0$.

\item $L_{tw}$ is the one-dimensional stratum on
$\mathbb{P}(1,a_1, a_2,a_3,a_4)$  defined by $x=y=z= 0$, and the
other one-dimensional strata are labelled similarly.

\item $O_y:=[0:1:0:0:0]$.

\item $O_z:=[0:0:1:0:0]$.

\item $O_t:=[0:0:0:1:0]$. \item $O_w:=[0:0:0:0:1]$.

\item When we consider  a singular point of type
$\frac{1}{r}(1,a,r-a)$ on $X$, the weighted blow up of $X$ at the
singular point  with weights $(1,a,r-a)$ will be denoted by
$f\colon Y\to X$ unless otherwise stated. 

\item $A$ is the
pull-back of $-K_X$ by $f$.

\item $B$ is the anticanonical class of $Y$.

\item $E$ is the  exceptional
divisor  of $f$. 

\item $S$ is the
proper transform of $S_x$ by $f$.

\item  $\mathcal{M}_{Y}$ is the proper
transform of the linear system $\mathcal{M}$ by $f$.

\item When we have a curve $C$
on $X$, its proper transform on $Y$ will be always denoted by
$\tilde{C}$. For instance,
 $\tilde{L}_{tw}$ is the proper transform of the curve $L_{tw}$ on $X$ (if it is contained in $X$)
 by the weighted blow up $f$.

\end{itemize}

\subsection{The $95$ families of Fletcher and Ried}
\label{section:table}

We list the $95$ families of Fletcher and Ried
for the convenience of the reader (see  \cite[Table 5]{IF00}).
Here, $X_d\subset \mathbb{P}(1,a_1,a_2,a_3,a_4)$ is a quasi-smooth hypersurface of degree $d$ in the projective space
$\mathbb{P}(1,a_1,a_2,a_3,a_4)$. 
The entry numbers of the list are originally 
 given in the
lexicographic order of $(d,a_1,a_2,a_3,a_4)$.
\bigskip

\small
\begin{longtable}{lll}
\textbf{No. 01.} $X_{4}\subset \mathbb{P}(1,1,1,1,1)$&
\textbf{No. 02.} $X_{5}\subset \mathbb{P}(1,1,1,1,2)$&
\textbf{No. 03.} $X_{6}\subset \mathbb{P}(1,1,1,1,3)$\\
\textbf{No. 04.} $X_{6}\subset \mathbb{P}(1,1,1,2,2)$&
\textbf{No. 05.} $X_{7}\subset \mathbb{P}(1,1,1,2,3)$&
\textbf{No. 06.} $X_{8}\subset \mathbb{P}(1,1,1,2,4)$\\
\textbf{No. 07.} $X_{8}\subset \mathbb{P}(1,1,2,2,3)$&
\textbf{No. 08.} $X_{9}\subset \mathbb{P}(1,1,1,3,4)$&
\textbf{No. 09.} $X_{9}\subset \mathbb{P}(1,1,2,3,3)$\\
\textbf{No. 10.} $X_{10}\subset \mathbb{P}(1,1,1,3,5)$&
\textbf{No. 11.} $X_{10}\subset \mathbb{P}(1,1,2,2,5)$&
\textbf{No. 12.} $X_{10}\subset \mathbb{P}(1,1,2,3,4)$\\
\textbf{No. 13.} $X_{11}\subset \mathbb{P}(1,1,2,3,5)$&
\textbf{No. 14.} $X_{12}\subset \mathbb{P}(1,1,1,4,6)$&
\textbf{No. 15.} $X_{12}\subset \mathbb{P}(1,1,2,3,6)$\\
\textbf{No. 16.} $X_{12}\subset \mathbb{P}(1,1,2,4,5)$&
\textbf{No. 17.} $X_{12}\subset \mathbb{P}(1,1,3,4,4)$&
\textbf{No. 18.} $X_{12}\subset \mathbb{P}(1,2,2,3,5)$\\
\textbf{No. 19.} $X_{12}\subset \mathbb{P}(1,2,3,3,4)$&
\textbf{No. 20.} $X_{13}\subset \mathbb{P}(1,1,3,4,5)$&
\textbf{No. 21.} $X_{14}\subset \mathbb{P}(1,1,2,4,7)$\\
\textbf{No. 22.} $X_{14}\subset \mathbb{P}(1,2,2,3,7)$&
\textbf{No. 23.} $X_{14}\subset \mathbb{P}(1,2,3,4,5)$&
\textbf{No. 24.} $X_{15}\subset \mathbb{P}(1,1,2,5,7)$\\
\textbf{No. 25.} $X_{15}\subset \mathbb{P}(1,1,3,4,7)$&
\textbf{No. 26.} $X_{15}\subset \mathbb{P}(1,1,3,5,6)$&
\textbf{No. 27.} $X_{15}\subset \mathbb{P}(1,2,3,5,5)$\\
\textbf{No. 28.} $X_{15}\subset \mathbb{P}(1,3,3,4,5)$&
\textbf{No. 29.} $X_{16}\subset \mathbb{P}(1,1,2,5,8)$&
\textbf{No. 30.} $X_{16}\subset \mathbb{P}(1,1,3,4,8)$\\
\textbf{No. 31.} $X_{16}\subset \mathbb{P}(1,1,4,5,6)$&
\textbf{No. 32.} $X_{16}\subset \mathbb{P}(1,2,3,4,7)$&
\textbf{No. 33.} $X_{17}\subset \mathbb{P}(1,2,3,5,7)$\\
\textbf{No. 34.} $X_{18}\subset \mathbb{P}(1,1,2,6,9)$&
\textbf{No. 35.} $X_{18}\subset \mathbb{P}(1,1,3,5,9)$&
\textbf{No. 36.} $X_{18}\subset \mathbb{P}(1,1,4,6,7)$\\
\textbf{No. 37.} $X_{18}\subset \mathbb{P}(1,2,3,4,9)$&
\textbf{No. 38.} $X_{18}\subset \mathbb{P}(1,2,3,5,8)$&
\textbf{No. 39.} $X_{18}\subset \mathbb{P}(1,3,4,5,6)$\\
\textbf{No. 40.} $X_{19}\subset \mathbb{P}(1,3,4,5,7)$&
\textbf{No. 41.} $X_{20}\subset \mathbb{P}(1,1,4,5,10)$&
\textbf{No. 42.} $X_{20}\subset \mathbb{P}(1,2,3,5,10)$\\
\textbf{No. 43.} $X_{20}\subset \mathbb{P}(1,2,4,5,9)$&
\textbf{No. 44.} $X_{20}\subset \mathbb{P}(1,2,5,6,7)$&
\textbf{No. 45.} $X_{20}\subset \mathbb{P}(1,3,4,5,89)$\\
\textbf{No. 46.} $X_{21}\subset \mathbb{P}(1,1,3,7,10)$&
\textbf{No. 47.} $X_{21}\subset \mathbb{P}(1,1,5,7,8)$&
\textbf{No. 48.} $X_{21}\subset \mathbb{P}(1,2,3,7,9)$\\
\textbf{No. 49.} $X_{21}\subset \mathbb{P}(1,3,5,6,7)$&
\textbf{No. 50.} $X_{22}\subset \mathbb{P}(1,1,3,7,11)$&
\textbf{No. 51.} $X_{22}\subset \mathbb{P}(1,1,4,6,11)$\\
\textbf{No. 52.} $X_{22}\subset \mathbb{P}(1,2,4,5,11)$&
\textbf{No. 53.} $X_{24}\subset \mathbb{P}(1,1,3,8,12)$&
\textbf{No. 54.} $X_{24}\subset \mathbb{P}(1,1,6,8,9)$\\
\textbf{No. 55.} $X_{24}\subset \mathbb{P}(1,2,3,7,12)$&
\textbf{No. 56.} $X_{24}\subset \mathbb{P}(1,2,3,8,11)$&
\textbf{No. 57.} $X_{24}\subset \mathbb{P}(1,3,4,5,12)$\\
\textbf{No. 58.} $X_{24}\subset \mathbb{P}(1,3,4,7,10)$&
\textbf{No. 59.} $X_{24}\subset \mathbb{P}(1,3,6,7,8)$&
\textbf{No. 60.} $X_{24}\subset \mathbb{P}(1,4,5,6,9)$\\
\textbf{No. 61.} $X_{25}\subset \mathbb{P}(1,4,5,7,9)$&
\textbf{No. 62.} $X_{26}\subset \mathbb{P}(1,1,5,7,13)$&
\textbf{No. 63.} $X_{26}\subset \mathbb{P}(1,2,3,8,13)$\\
\textbf{No. 64.} $X_{26}\subset \mathbb{P}(1,2,5,6,13)$&
\textbf{No. 65.} $X_{27}\subset \mathbb{P}(1,2,5,9,11)$&
\textbf{No. 66.} $X_{27}\subset \mathbb{P}(1,5,6,7,9)$\\
\textbf{No. 67.} $X_{28}\subset \mathbb{P}(1,1,4,9,14)$&
\textbf{No. 68.} $X_{28}\subset \mathbb{P}(1,3,4,7,14)$&
\textbf{No. 69.} $X_{28}\subset \mathbb{P}(1,4,6,7,11)$\\
\textbf{No. 70.} $X_{30}\subset \mathbb{P}(1,1,4,10,15)$&
\textbf{No. 71.} $X_{30}\subset \mathbb{P}(1,1,6,8,15)$&
\textbf{No. 72.} $X_{30}\subset \mathbb{P}(1,2,3,10,15)$\\
\textbf{No. 73.} $X_{30}\subset \mathbb{P}(1,2,6,7,15)$&
\textbf{No. 74.} $X_{30}\subset \mathbb{P}(1,3,4,10,13)$&
\textbf{No. 75.} $X_{30}\subset \mathbb{P}(1,4,5,6,15)$\\
\textbf{No. 76.} $X_{30}\subset \mathbb{P}(1,5,6,8,11)$&
\textbf{No. 77.} $X_{32}\subset \mathbb{P}(1,2,5,9,16)$&
\textbf{No. 78.} $X_{32}\subset \mathbb{P}(1,4,5,7,16)$\\
\textbf{No. 79.} $X_{33}\subset \mathbb{P}(1,3,5,11,14)$&
\textbf{No. 80.} $X_{34}\subset \mathbb{P}(1,3,4,10,17)$&
\textbf{No. 81.} $X_{34}\subset \mathbb{P}(1,4,6,7,17)$\\
\textbf{No. 82.} $X_{36}\subset \mathbb{P}(1,1,5,12,18)$&
\textbf{No. 83.} $X_{36}\subset \mathbb{P}(1,3,4,11,18)$&
\textbf{No. 84.} $X_{36}\subset \mathbb{P}(1,7,8,9,12)$\\
\textbf{No. 85.} $X_{38}\subset \mathbb{P}(1,3,5,11,19)$&
\textbf{No. 86.} $X_{38}\subset \mathbb{P}(1,5,6,8,19)$&
\textbf{No. 87.} $X_{40}\subset \mathbb{P}(1,5,7,8,20)$\\
\textbf{No. 88.} $X_{42}\subset \mathbb{P}(1,1,6,14,21)$&
\textbf{No. 89.} $X_{42}\subset \mathbb{P}(1,2,5,14,21)$&
\textbf{No. 90.} $X_{42}\subset \mathbb{P}(1,3,4,14,21)$\\
\textbf{No. 91.} $X_{44}\subset \mathbb{P}(1,4,5,13,22)$&
\textbf{No. 92.} $X_{48}\subset \mathbb{P}(1,3,5,16,24)$&
\textbf{No. 93.} $X_{50}\subset \mathbb{P}(1,3,5,16,24)$\\
\textbf{No. 94.} $X_{54}\subset \mathbb{P}(1,4,5,18,27)$&
\textbf{No. 95.} $X_{66}\subset \mathbb{P}(1,5,6,22,33)$.&
\\
\end{longtable}
\normalsize
\newpage

\section{Smooth points and curves}
\subsection{Excluding smooth points}\label{section:smooth point}

In this section we show that smooth points of $X$ cannot be
non-canonical centers  of the  log pair
$\left(X,\frac{1}{n}\mathcal{M}\right)$.

Let $X\subset\mathbb{P}(1,a_1,a_2,a_3,a_4)$ be a quasi-smooth
weighted hypersurface of degree $d=\sum a_i$ with terminal
singularities. Suppose that a smooth point $p$ on $X$ is a center
of non-canonical singularities of the log pair
$\left(X,\frac{1}{n}\mathcal{M}\right)$. Then we obtain
$$
\mult_p(\mathcal{M}^2)>4n^2
$$
by \cite[Corollary~3.4]{Co00}.

 Let $s$ be an integer not greater than $\frac{4}{-K_X^3}$.  Suppose
that we have a divisor $H$ in $|-sK_X|$ such that \begin{itemize}
\item it passes through the point $p$, \item it contains no
$1$-dimensional component of the base locus of the linear system
$\mathcal{M}$ that passes through the point $p$.
\end{itemize}

Then we can obtain the following contradictory inequality:
\[-sn^2K_X^3=H\cdot\mathcal{M}^2\geq \mult_p(H)\cdot
\mult_p(\mathcal{M}^2)>4n^2.\]

In order to show that a center of non-canonical singularities of the log pair $\left(X,\frac{1}{n}\mathcal{M}\right)$ cannot be  a smooth point, we mainly try to find such a divisor.

Before we proceed, set
$\widehat{a}_2=\operatorname{lcm}\{a_1,a_3,a_4\}$,
$\widehat{a}_3=\operatorname{lcm}\{a_1,a_2,a_4\}$ and
$\widehat{a}_4=\operatorname{lcm}\{a_1,a_2,a_3\}$.
\begin{lemma}
Suppose that the hypersurface $X$ satisfies one of the following:
\begin{itemize}
\item $X$ does not pass through the point $O_w$ and
$d\cdot\widehat{a}_4\leq4a_1a_2a_3a_4$; \item $X$ does not pass
through the point $O_t$ and $d\cdot\widehat{a}_3\leq
4a_1a_2a_3a_4$; \item $X$ does not pass through the point $O_z$
and $d\cdot\widehat{a}_2\leq 4a_1a_2a_3a_4$.
\end{itemize}
Then a smooth point of $X$ cannot be a non-canonical  center of
 the log pair $\left(X,\frac{1}{n}\mathcal{M}\right)$.
\end{lemma}
\begin{proof}
For simplicity we suppose that the hypersurface $X$ satisfies the
first condition. The proofs for the other cases are the same.

 Let $\pi_4 : X\to \mathbb{P}(1,a_1,
a_2, a_3)$ be the regular projection centered at the point $O_w$.
The linear system
$|\mathcal{O}_{\mathbb{P}(1,a_1,a_2,a_3)}(\widehat{a}_4)|$ is base
point free. Choose a general member  in the linear system
$|\mathcal{O}_{\mathbb{P}(1,a_1,a_2,a_3)}(\widehat{a}_4)|$ that
passes through the point $\pi_{4}(p)$. Then its pull-back by the
finite morphism $\pi_4$ can play the role of the divisor $H$ in
the explanation at the beginning.
\end{proof}

The condition above is satisfied by all the families except the
families
$$\mbox{No.}\ 2,5,12,13,20, 23, 25, 33,40, 58, 61,76.$$

No quasi-smooth hypersurface in the families No.~$23$, $40$, $61$, $76$ contains the curve $L_{tw}$. 
Using suitable coordinate changes, we may write
its defining equation
as
\[tw^2+w(tg_{a_4}(x,y,z)+g_{d}(x,y,z))+x_it^3+t^2g_{d-2a_3}(x,y,z)+tg_{d-a_3}(x,y,z)+g_{d}(x,y,z)=0,\]
where $x_i=y$ for the families No.~$23$, $61$ and   $x_i=z$ for the families  No.~$40$, $76$.

There are two kinds of quasi-smooth hypersurfaces in each family of No.~$5$, $12$, $13$, $20$, $25$, $33$, $58$. The first kind are those that  do not contain the curve $L_{tw}$. The second are those that contain the curve $L_{tw}$. 
After appropriate coordinate changes, every quasi-smooth hypersurface of the first kind in each family can be defined by 
\[wt^2+t(wg_{a_3}(x,y,z)+g_{d}(x,y,z))+x_iw^2+wg_{d-a_4}(x,y,z)+g_{d}(x,y,z)=0,\]
where $x_i=y$ for the families No.~$13$, $25$ and   $x_i=z$ for the families  No.~$5$, $12$, $20$, $33$, $58$.

\begin{lemma}\label{lemma:Train from Wolfach to Bonn}
Suppose that the hypersurface $X$ satisfies the following conditions:
\begin{itemize}
\item $d\cdot\widehat{a}_4\leq4a_1a_2a_3a_4$; 
\item  $d\cdot\widehat{a}_3\leq 4a_1a_2a_3a_4$.
\end{itemize}
In addition, we suppose that the curve $L_{tw}$  is not contained
in the hypersurface $X$. Then  a smooth point of $X$ cannot be a
center of non-canonical singularities of the log pair
$\left(X,\frac{1}{n}\mathcal{M}\right)$.
\end{lemma}
\begin{proof}
Let $\mathcal{H}_4$ be the linear system consisting of the divisors in $|\mathcal{}O_X(\hat{a}_4)|$ 
that pass through the point $p$
and $\mathcal{H}_3$ be the linear system consisting of the divisors in $|\mathcal{}O_X(\hat{a}_3)|$ 
that pass through the point $p$.

Let $\pi_4 : X\dasharrow \mathbb{P}(1,a_1, a_2, a_3)$ be the
projection centered at the point $O_w$ and $\pi_3 : X\dasharrow
\mathbb{P}(1,a_1, a_2, a_4)$  the projection centered at the point
$O_t$. The linear systems
$|\mathcal{O}_{\mathbb{P}(1,a_1,a_2,a_3)}(\widehat{a}_4)|$ and
$|\mathcal{O}_{\mathbb{P}(1,a_1,a_2,a_4)}(\widehat{a}_3)|$ are
base point free. 
The pull-backs of  a general member in 
$|\mathcal{O}_{\mathbb{P}(1,a_1,a_2,a_3)}(\widehat{a}_4)|$ that
passes through the point $\pi_{4}(p)$ and a general member  in the
linear system $|\mathcal{O}_{\mathbb{P}(1,a_1,a_2,a_4)}(\widehat{a}_3)|$
that passes through the point $\pi_{3}(p)$ show that  a general member  either in $\mathcal{H}_4$ or in $\mathcal{H}_3$ can serve as the divisor $H$
 in the explanation at the beginning unless the point $p$ belongs to both a curve contracted by $\pi_4$ and a curve contracted by $\pi_3$.

Suppose that the point $p$ belongs to both a curve contracted by $\pi_4$ and a curve  contracted by $\pi_3$.
Let   $A_1$, $A_2,\cdots, A_k$ (resp. $B_1$, $B_2,\cdots, B_m$)  be quasi-homogeneous polynomials
of degree $\hat{a}_4$ (resp. $\hat{a}_3$) that generate the linear system $\mathcal{H}_4$ (resp. $\mathcal{H}_3$).

Any irreducible curve except $L_{tw}$ cannot be contracted both by $\pi_4$ and by $\pi_3$.
Therefore, the base locus of $\mathcal{H}_4$ has no common $1$-dimensional component with the base locus of $\mathcal{H}_3$ around the point $p$ since we do not have the curve $L_{tw}$ on $X$.
This shows that the base locus of the linear system $\mathcal{H}$ generated by
quasi-homogeneous  polynomials
$A_1^{\hat{a}_3}$, $A_2^{\hat{a}_3},\cdots, A_k^{\hat{a}_3}$, $B_1^{\hat{a}_4}$, 
$B_2^{\hat{a}_4},\cdots, B_m^{\hat{a}_4}$  of degree $\hat{a}_4\hat{a}_3$ has no $1$-dimensional component 
passing through the point $p$. Therefore, for a general member $H'$ of the linear system $\mathcal{H}$, we have
\[-\hat{a}_3\hat{a}_4n^2K_X^3=H'\cdot\mathcal{M}^2\geq \mult_p(H')\cdot
\mult_p(\mathcal{M}^2)\geq \min\{\hat{a}_3, \hat{a}_4\}\cdot
\mult_p(\mathcal{M}^2)>4n^2\min\{\hat{a}_3, \hat{a}_4\},\]
which implies $d\hat{a}_3\hat{a}_4> \min\{\hat{a}_3, \hat{a}_4\}a_1a_2a_3a_4$. This contradicts our condition.
\end{proof}
The conditions above are satisfied by  the families
$$\mbox{No.}\ 23,  40,  61, 76.$$
Also, the members of the first kind in the families
No.~$5$, $12$, $13$, $20$, $25$, $33$, $58$, i.e., those that do not contain $L_{tw}$, 
meet these conditions.

The members of the second kind in the families
No.~$5$, $12$, $13$, $20$, $25$, $33$, $58$,   i.e., those that  contain $L_{tw}$, and the family No.~$2$   remain.

We deal with the family No.~$2$  in the end of this section.
Instead, we first consider the members  of the second kind in the families
No.~$5$, $12$, $13$, $20$, $25$, $33$, $58$,   i.e., those that  contain $L_{tw}$.
These members are not covered by Lemma~\ref{lemma:Train from Wolfach to Bonn}. Since these members are the ones that contain the curve $L_{tw}$, the defining polynomials of $X$ do not contain the
monomial $t^2w$. Therefore, using coordinate changes, we may assume that the polynomial is given by
\[w^2z+w(tg_{a_3}(x,y,z)+g_{2a_3}(x,y,z))+t^3y+t^2h_{d-2a_3}(x,y,z)+th_{d-a_3}(x,y,z)+h_{d}(x,y,z)=0 \]
for the families No.~$12$, $20$,
\[w^2y+w(tg_{a_3}(x,y,z)+g_{2a_3}(x,y,z))+t^3z+t^2h_{d-2a_3}(x,y,z)+th_{d-a_3}(x,y,z)+h_{d}(x,y,z)=0 \]
for the families  No.~$5$, $13$, $25$, $33$, $58$.
Note that for the family No.~5 the coefficients of $w^2$ and $t^3$ cannot coincide, i.e., we cannot assume that 
the hypersurface $X$ is defined by 
\[w^2y+w(tg_{2}+g_{4})+t^3y+t^2h_{3}+th_{5}+h_{7}=0.\]
In such a case, the hypersurface is not quasi-smooth at the point defined by $x=y=z=w^2+t^3=0$.

\begin{lemma}\label{lemma-MPIM}
Suppose that the hypersurface $X$ satisfies the following
conditions:
\begin{itemize}
\item $d\cdot\widehat{a}_4\leq4a_1a_2a_3a_4$; \item  $d\cdot\widehat{a}_3\leq 4a_1a_2a_3a_4$.
\end{itemize}
Suppose that the curve $L_{tw}$ is contained in the hypersurface
$X$. If   a smooth point of $X$ is a  non-canonical center of
the log pair $\left(X,\frac{1}{n}\mathcal{M}\right)$, then the
point lies on the curve $L_{tw}$.
\end{lemma}
\begin{proof}
The proof of Lemma~\ref{lemma:Train from Wolfach to Bonn}
immediately shows the statement.
\end{proof}

\begin{lemma}\label{lemma-MPIM2}
Suppose that the curve $L_{tw}$ is contained in the hypersurface
$X$. In addition, we suppose that $a_3>1$, $(a_3,a_4)=1$, $a_3a_4>d$, and
there are non-negative integers $m_1$ and $m_2$ such that
$m_1a_1+m_2a_2=a_3a_4$. Then any smooth point on $L_{tw}$ 
cannot be a center of non-canonical singularities of the log pair
$\left(X,\frac{1}{n}\mathcal{M}\right)$ if $-a_3a_4K_{X}^3\leq 4$.
\end{lemma}

\begin{proof}
Suppose that the hypersurface $X$ is defined by $F(x,y,z,t,w)=0$.
Let $p$ be a smooth point on the curve $L_{tw}$. Then there are
non-zero constants $\lambda$ and $\mu$ such that the surface cut
by $\lambda t^{a_4}+\mu w^{a_3}=0$ contains the point $p$. We then
consider the linear system $\mathcal{H}$ on $X$ generated by
$x^{a_3a_4}$, $y^{m_1}z^{m_2}$, and $\lambda t^{a_4}+\mu w^{a_3}$.
The base locus of this linear system consists of the locus cut by
$$x=y^{m_1}z^{m_2}=\lambda t^{a_4}+\mu w^{a_3}=0.$$ 
The degree $d$ of $F$ is smaller than $a_3a_4$ by the condition and the polynomial
$\lambda t^{a_4}+\mu w^{a_3}$ is irreducible since $(a_3,a_4)=1$. Therefore, 
 neither $F(0,0,z,t,w)$ nor $F(0,y,0,t,w)$ can divide $\lambda
t^{a_4}+\mu w^{a_3}$ and vice versa. Therefore, the base locus of the linear
system $\mathcal{H}$ is of dimension at most $0$. Then a general
member of this linear system is able to play the role of the
divisor $H$ in the explanation at the beginning.
\end{proof}

Combining Lemmas~\ref{lemma-MPIM} and \ref{lemma-MPIM2}, we can conclude that any smooth point cannot be a center of non-canonical singularities of the log pair
$\left(X,\frac{1}{n}\mathcal{M}\right)$
for the families $\mbox{No.}\ 33$ and $58$.

\begin{lemma}\label{lemma-smooth-point-13-25}
For the families No.~$13, 25$, a smooth point of $X$ cannot
be a center of non-canonical singularities of the log pair
$\left(X,\frac{1}{n}\mathcal{M}\right)$.
\end{lemma}
\begin{proof}
The following method works for both the families exactly in the
same way. For this reason, we demonstrate the method only for the
family $\mbox{No.}\ 25$.

 Suppose that the log
pair $\left(X,\frac{1}{n}\mathcal{M}\right)$ is not canonical at some smooth
point $p$. Then Lemma~\ref{lemma-MPIM} shows that the point $p$
must lie on the curve $L_{tw}$.

Consider the pencil $|-K_{X}|$. Its base locus consists of two
reduced and irreducible curves. One is the curve $L_{tw}$ and the
other is the curve $C$ defined by the equations $$x=y=t^3-cz^4=0,$$
where $c$ is a non-zero constant. Note that the curve $L_{tw}$ is
quasi-smooth everywhere and $C$ is quasi-smooth outside the point
$O_w$. They intersect only at the point $O_w$.  Choose a general
member $H$ in the pencil $|-K_{X}|$. Then the log pair $\left(X, H+\frac{1}{n}\mathcal{M}\right)$ is not
log canonical at the point $p$. By Inversion of Adjunction (\cite[Theorem~5.50]{KoMo98}), we see
that the log pair $\left(H,\frac{1}{n}\mathcal{M}|_H\right)$ is not log canonical
at the point $p$.

Let $D_y$ be the divisor on $H$ defined by the
equation $y=0$. Then $D_y=L_{tw}+C$.
We have the following intersection numbers on the surface $H$:
\[L_{tw}^2=-\frac{11}{28}, \ \ \ C^2=-\frac{2}{7}, \ \ \
C\cdot L_{tw}=\frac{3}{7}, \ \ \ D_y\cdot L_{tw}=\frac{1}{28}, \ \ \
D_y\cdot C=\frac{1}{7}.\]
Indeed, we can obtain these intersection numbers directly from the polynomials defining the curves. On the other hand, we are also able to obtain them from the singularity types of the K3 surface $H$.
Note that $H$ is a $K3$
surface with $A_3$ and $A_6$ singularities at the points $O_t$ and $O_w$,
respectively. For instance, the curve $L_{tw}$ is a smooth rational curve on the K3 surface $H$ passing through one $A_3$-singular point and one $A_6$-singular point, and hence the self-intersection number $L_{tw}^2$ is obtained by $-2+\frac{3}{4}+\frac{6}{7}$. The $A_3$-singular point contributes to the self-intersection number by $\frac{3}{4}$ and the $A_6$-singular point by $\frac{6}{7}$ (see Remark~\ref{remark:plt} below for more detail). 

Let $M$ be a general member in the mobile linear system $\mathcal{M}$ and then put
\[M_H:=\frac{1}{n}M\big|_H=aL_{tw}+bC+\Delta,\]
where $a$ and $b$ are non-negative rational numbers and $\Delta$
is an effective divisor whose support contains neither $L_{tw}$
nor $C$. We then obtain
\[\frac{1}{7}=C\cdot M_H=aL_{tw}\cdot C+bC^2+\Delta\cdot C\geq
\frac{3a}{7}-\frac{2b}{7}.\] On the other hand, we obtain
\[\frac{5}{28}=M_H^2=aL_{tw}\cdot D_y +bC\cdot D_y+\Delta\cdot D_y\geq
\frac{a}{28}+\frac{b}{7}.\] Combining these two inequalities we
see that $a\leq 1$. Therefore, the log pair $(H, L_{tw}+bC+\Delta)$ is
not log canonical at the point $p$, and hence the log pair $\left(L_{tw},
(bC+\Delta)|_{L_{tw}}\right)$ is not log canonical at the point $p$.
Consequently, we see that
\[\mult_p\left((bC+\Delta)|_{L_{tw}}\right)>1.\]
However,
\[(bC+\Delta)\cdot L_{tw}=(M_H-aL_{tw})\cdot
L_{tw}=\frac{1}{28}+\frac{11a}{28}\leq \frac{3}{7}.\] This
completes the proof.
\end{proof}
\begin{remark}\label{remark:plt}
Let $p$ be an $A_n$-singular point on a normal surface $\Sigma$. Suppose that a smooth curve $C$ on $\Sigma$ passes through the point $p$. Let $\phi:\bar{\Sigma}\to\Sigma$ be the minimal resolution of the point $p$. Then we have $(-2)$-curves
$E_1,\cdots, E_n$ over the point $p$ whose intersection matrix is 
\[(E_i\cdot E_j)=\left(\begin{array}{cccccc}
        -2&1&0&\cdots&0&0 \\
        1&-2&1&0&\cdots&0 \\
        0&1&-2&1&0&\cdots \\
        \multicolumn{6}{c}\dotfill\\
        0&\cdots&0&1&-2&1 \\
        0&0&\cdots&0&1&-2
\end{array}\right).
\]
The log pair $(\Sigma, C)$ is purely log terminal by Inversion of Adjunction  (\cite[Theorem~5.50]{KoMo98}). Therefore, the proper transform $\bar{C}$ by $\phi$ intersects transversally only one of $E_i$'s and it should be either $E_1$ or $E_n$ (\cite[Theorem~9.6]{Ka88}).  We may assume that it is $E_n$. We then obtain
$$\phi^*(C)=\bar{C}+\frac{1}{n+1}\left(E_1+2E_2+\cdots+(n-1)E_{n-1}+nE_n\right).$$
Therefore, $$C^2=\bar{C}^2+\frac{n}{n+1}.$$
\end{remark}

\begin{lemma}
For the families $\mbox{No.}\ 12, 20$, any smooth point of $X$
cannot be a center of non-canonical singularities of the log pair
$\left(X,\frac{1}{n}\mathcal{M}\right)$.
\end{lemma}
\begin{proof}
The method for its proof is the same as that of
Lemma~\ref{lemma-smooth-point-13-25} with some slight difference. The only difference is that two base locus curves of the pencil $|-K_X|$ intersect at the point~$O_t$.
For this reason, we omit the proof.
\end{proof}

\begin{lemma}\label{lemma-smooth-point-5}
For the family $\mbox{No.}\ 5$, a smooth point of $X_7$ cannot be a
center of non-canonical singularities of the log pair
$\left(X_7,\frac{1}{n}\mathcal{M}\right)$.
\end{lemma}
\begin{proof}
Suppose that the log pair $\left(X_7,\frac{1}{n}\mathcal{M}\right)$ is not
canonical at a smooth point $p$. Then Lemma~\ref{lemma-MPIM}
shows that the point $p$ lies on the curve $L_{tw}$.

Consider the 2-dimensional linear system $|-K_{X_7}|$. Its base
locus consists of the reduced and irreducible curve $L_{tw}$. The
curve $L_{tw}$ is a quasi-smooth curve passing through the
singular points $O_t$ and $O_w$.

Let $H$ be the surface cut by the equation $z=\lambda x+\mu y$ with general complex numbers $\lambda$ and $\mu$.
 It is a $K3$
surface with $A_1$ and $A_2$ singularities at the points $O_t$ and $O_w$,
respectively. It also contains the rational curve $L_{tw}$.
The self-intersection number of $L_{tw}$ on $H$ is
$-\frac{5}{6}$ ($=-2+\frac{1}{2}+\frac{2}{3}$). Let $D_y$ be the divisor on $H$ defined by the equation
$y=0$. Then we can easily see that $D_y=L_{tw}+R$, where $R$ is
the curve defined by the equation $$y=z-\lambda x=\lambda t^3+xh_5(x,t,w)=0.$$ The two curves
$L_{tw}$ and $R$ meet only at the point $O_w$.

Let $M$ be a general member of the linear system $\mathcal{M}$ and then write
\[M_H:=\frac{1}{n}M\big|_H=aL_{tw}+\Delta,\]
where $a$ is a non-negative rational number and $\Delta$ is an
effective divisor whose support does not contain the curve
$L_{tw}$. The log pair $\left(X_7, H+\frac{1}{n}\mathcal{M}\right)$ is not log
canonical at the point $p$. By Inversion of Adjunction (\cite[Theorem~5.50]{KoMo98}), we see
that the log pair $\left(H,\frac{1}{n}\mathcal{M}|_H\right)$ is not log canonical
at the point $p$. We then obtain
\[1=R\cdot M_H=aL_{tw}\cdot R+\Delta\cdot R\geq
a.\]  Therefore, the log pair $(H, L_{tw}+\Delta)$ is not log
canonical at the point $p$, and hence the log pair $\left(L_{tw},
\Delta|_{L_{tw}}\right)$ is not log canonical at the point $p$.
Consequently, we see that
\[\mult_p\left(\Delta|_{L_{tw}}\right)>1.\]
However,
\[\Delta\cdot L_{tw}=\left(M_H-aL_{tw}\right)\cdot
L_{tw}=\frac{1}{6}+\frac{5a}{6}\leq 1.\] This completes the proof.
\end{proof}

Finally, we deal with smooth points on quasi-smooth hypersurfaces in the family No.~$2$.
\begin{lemma}\label{lemma-smooth-point-2}
For the family $\mbox{No.}\ 2$, a smooth point of $X_5$ cannot be a
center of non-canonical singularities of the log pair
$\left(X_5,\frac{1}{n}\mathcal{M}\right)$.
\end{lemma}
\begin{proof}
This case has been resolved completely in \cite{CPR}. For the convenience of the reader we reproduce the proof from p.211 in
\cite{CPR}.

 By suitable coordinate change we may
assume that the hypersurface $X_5$ in $\mathbb{P}(1,1,1,1,2)$ is
given by
\[w^2x+wf_3+f_5=0,\]
where $f_m$ is a quasi-homogeneous polynomial of degree $m$ in
variables $x$, $y$,  $z$ and $t$.

Suppose that the log pair $\left(X_5,\frac{1}{n}\mathcal{M}\right)$ is not
canonical at some smooth point $p$. Then the point $p$ must lie on
the curve $L$ contracted by the projection $\pi_4:X_5\dashrightarrow
\mathbb{P}^3$ centered at the point $O_w$. By an additional coordinate change, we may assume that
the curve $L$ is defined by the equations $x=y=z=0$, i.e.,
$L=L_{tw}$.

Let $H$ be a general element in $|-K_{X_5}|$ containing the curve
$L_{tw}$. Then the surface $H$ is a $K3$ surface with an $A_1$
singularity at the point $O_w$.  The self-intersection number of
$L_{tw}$ on $H$ is $-\frac{3}{2}$.

We write
\[\mathcal{M}_H:=\frac{1}{n}\mathcal{M}\big|_H=aL_{tw}+\mathcal{L},\]
where $a$ is a non-negative rational number and $\mathcal{L}$ is a
mobile linear system on $H$ whose base locus does not contain the curve $L_{tw}$.

Choose another curve $R$ that is contacted by the projection
$\pi_4$. Note that such a curve  is given by a point on the zero set
in $\mathbb{P}^3$ defined by $x=f_3=f_6=0$. Then we see that the
intersection number $L_{tw}$ and $R$ is $\frac{1}{2}$. We then obtain
\[\frac{1}{2}=R\cdot \mathcal{M}_H=aL_{tw}\cdot R+\mathcal{L}\cdot R\geq
\frac{a}{2},\] and hence $a\leq 1$.

The log pair $\left(X_5, H+\frac{1}{n}\mathcal{M}\right)$ is not log canonical at
the point $p$. By Inversion of Adjunction (\cite[Theorem~5.50]{KoMo98}), we see that the log pair
$(H,\mathcal{M}_H)$ is not log canonical at the point $p$. We then
obtain from \cite[Theorem~3.1]{Co00}
\[4(1-a)<\mathcal{L}^2=(\mathcal{M}_H-aL_{tw})^2=\mathcal{M}_H^2-2a\mathcal{M}_H\cdot L_{tw}+a^2L_{tw}^2=\frac{5}{2}-a-\frac{3a^2}{2}.\]
However, this inequality cannot be satisfied with any value of
$a$. This completes the proof.
\end{proof}
In summary, we have verified
\begin{theorem}\label{theorem:smooth point excluding}
A smooth point on $X$ cannot be a center of non-canonical singularities of the log pair $(X, \frac{1}{n}\mathcal{M})$.
\end{theorem}

\subsection{Excluding curves}\label{section:curves}

We now show that an irreducible curve on $X$ can not
be a center of non-canonical singularities of the log pair
$\left(X,\frac{1}{n}\mathcal{M}\right)$ provided that no point on this curve
is  a center of non-canonical singularities of the log pair
$\left(X,\frac{1}{n}\mathcal{M}\right)$. Indeed, the proof comes from
\cite[pp. 206-207]{CPR} and it is based on the following
local result of Kawamata:

\begin{theorem}[{\cite[Lemma~7]{Ka96}}]
\label{theorem:Kawamata} Let $(U,p)$ be a germ of a threefold
terminal quotient singularity of type $\frac{1}{r}(1,a, r-a)$,
where $r\geq 2$ and $a$ is coprime to $r$ and let
$\mathcal{M}_{U}$ be a mobile linear system on $U$.
 Suppose that $(U,\lambda\mathcal{M}_{U})$ is not
canonical at $p$ for a positive rational number $\lambda$.  Let
$f\colon W\to U$ be the weighted blowup at the point $p$ with
weights $(1,a,r-a)$. Then
$$
\mathcal{M}_{W}=f^*(\mathcal{M}_{U})-mE
$$
for some positive rational number $m>\frac{1}{r\lambda}$, where
$E$ is the exceptional divisor of $f$ and  $\mathcal{M}_{W}$ is
the proper transform of $\mathcal{M}_{U}$.
 In
particular,
$$
K_{W}+\lambda
\mathcal{M}_{W}=f^*(K_{U}+\lambda\mathcal{M}_{U})+\Big(\frac{1}{r}-\lambda
m\Big)E,
$$
where $\frac{1}{r}-\lambda m<0$, and hence the point $p$ is a
center of non-canonical singularities of the log pair
$(U,\lambda\mathcal{M}_{U})$.
\end{theorem}
Note that, in this theorem, we do not assume that the point $p$ is  a
center of non-canonical singularities of the log pair
$(U,\lambda\mathcal{M}_{U})$.  A log pair may  not be canonical at a point
that is not a center of non-canonical singularities of the log pair. For example, consider the linear system
$\mathcal{M}_{\mathbb{C}^3}$ generated by $z_1^2$ and $z_2^2$ on $\mathbb{C}^3$, where $(z_1, z_2, z_3)$ is the standard coordinate system for $\mathbb{C}^3$. Then the log pair $(\mathbb{C}^3, \mathcal{M}_{\mathbb{C}^3})$ is not canonical at the origin. The line $z_1=z_2=0$ is a center of non-canonical center of the log pair $(\mathbb{C}^3, \mathcal{M}_{\mathbb{C}^3})$.
However, the origin is not a  center of non-canonical center of the log pair $(\mathbb{C}^3, \mathcal{M}_{\mathbb{C}^3})$.

Theorem~\ref{theorem:Kawamata} and the mobility of the linear
system $\mathcal{M}$ imply the following global properties.

\begin{corollary}[{\cite[Lemma~5.2.1]{CPR}}]\label{corollary:Mori-cone}
Let $\Lambda$ be a center of non-canonical singularities of the
log pair $\left(X,\frac{1}{n}\mathcal{M}\right)$. In case when $\Lambda$ is a
singular point of type $\frac{1}{r}(1,a, r-a)$, let $f\colon Y\to
X$ be the weighted blow up at $\Lambda$ with weights $(1,a,r-a)$.
 In case when $\Lambda$ is  a smooth curve  contained in the smooth locus of $X$,  let $f\colon Y\to X$ be the blow up along $\Lambda$.   Then the $1$-cycle $(-K_Y)^2\in
N_1(Y)$ lies in the interior of the Mori cone of $Y$:
$$
(-K_Y)^2\in \mathrm{Int}(\overline{\mathrm{NE}(Y)}).
$$
\end{corollary}

\begin{corollary}[{\cite[Corollary~5.2.3]{CPR}}]
\label{corollary:test-class} Under the same notations as in
Corollary~\ref{corollary:Mori-cone}, we have $H\cdot (-K_Y)^2>0$
for a non-zero nef divisor $H$ on $Y$.
\end{corollary}

Let $L$ be an irreducible curve on $X$. Suppose that $L$ is a
center of non-canonical singularities of the log pair
$\left(X,\frac{1}{n}\mathcal{M}\right)$. Then it follows from
Theorem~\ref{theorem:Kawamata} that every singular point of $X$
contained in $L$ (if any) must be a center of non-canonical
singularities of the log pair $\left(X,\frac{1}{n}\mathcal{M}\right)$.
Later we will show that for a  given singular point of $X$ either  it cannot be a
center of non-canonical singularities of the log pair
$\left(X,\frac{1}{n}\mathcal{M}\right)$ or it can be untwisted by a birational involution (see Definition~\ref{definition:untwisting}). Moreover, it will be done regardless
of the fact that the log pair $\left(X,\frac{1}{n}\mathcal{M}\right)$ is
canonical outside of this singular point. Therefore it is enough
to exclude only  irreducible curves contained in the  smooth locus
of $X$.

Suppose that $L$ is contained in the smooth locus of $X$. Pick two
general members $H_1$ and $H_2$ in the mobile linear system
$\mathcal{M}$. Then we obtain
$$
-n^2K_X^3=-K_X\cdot H_1\cdot H_2\geq
(\mult_L(\mathcal{M}))^2(-K_X\cdot L)> -n^2K_X\cdot L.
$$
since we have $\mult_L(\mathcal{M})>n$. Therefore,
$-K_X\cdot L<-K_X^3$.

Since the curve $L$ is contained in the smooth locus of $X$, we
have $-K_X\cdot L\geq 1$. Therefore the curve $L$  can exists only
on the hypersurface $X$  with $-K^3_X >1$ as a curve of degree
less than $-K_X^3$. Such conditions can be satisfied only in the
following cases:
\begin{itemize}
\item quasi-smooth hypersurface of degree $5$ in
$\mathbb{P}(1,1,1,1,2)$ with a curve $L$ of degree $1$ or $2$;

\item quasi-smooth hypersurface of degree $6$ in
$\mathbb{P}(1,1,1,2,2)$ with a curve $L$ of degree $1$;

\item quasi-smooth hypersurface of degree $7$ in
$\mathbb{P}(1,1,1,2,3)$ with a curve $L$ of degree $1$.
\end{itemize}

Let $f\colon Y\to X$ be the blow up of the ideal sheaf of the
curve $L$. Then $Y$ is smooth whenever the curve $L$ is smooth. As
explained in \cite[page~207]{CPR} (it is independent of
generality), in each of the three cases listed above, there exists
a non-zero nef divisor $M$ on $Y$ such that $M\cdot
K_{Y}^2\leqslant 0$. Corollary~\ref{corollary:test-class}
therefore shows that the curve $L$ must be singular. Consequently,
the curve $L$ must be an irreducible curve of degree $2$ in
 a quasi-smooth hypersurface  of degree $5$
in $\mathbb{P}(1,1,1,1,2)$.  More precisely,  the curve $L$  has
either an ordinary double point (which implies that $Y$ has an
ordinary double point on the exceptional divisor $E$) or $L$ has a cusp  (which implies
that $Y$ has an isolated double point that is locally given by
$x^2+y^2+z^2+t^3=0$ in $\mathbb{C}^4$). In both the cases, we can
proceed exactly as explained in \cite[page~207]{CPR} (the very end
of the proof of \cite[Theorem 5.1.1]{CPR}) to obtain a
contradiction.

In summary, so far we have proved 
\begin{theorem}\label{theorem:excluding-curve}
An irreducible curve contained in the smooth locus of $X$ cannot be a center of  non-canonical singularities of the log pair
$\left(X,\frac{1}{n}\mathcal{M}\right)$.
\end{theorem}

At this stage, we are therefore able to draw a conclusion that  if  the log pair
$\left(X,\frac{1}{n}\mathcal{M}\right)$ is not canonical,
then at least one singular point of $X$ must be a non-canonical center of $\left(X,\frac{1}{n}\mathcal{M}\right)$.

\newpage

\section{Singular points}

\subsection{Cyclic quotient singular points}\label{section:quotient singular point}

Let $(U, p)$ be a germ of a cyclic quotient singular point of type $\frac{1}{r}(1,a,r-a)$,
where $r$ and $a$ are relatively prime positive integers with $a<r$.  
We have an orbifold chart $\pi:(\hat{U}, 0)\to (U, p)$, where $\hat{U}$ is an open neighbourhood of the origin in $\mathbb{C}^3$ and the morphism $\pi$ is the quotient map by the group action of $\mathbb{Z}/r\mathbb{Z}$. 
We say that  functions $z_1$, $z_2$, $z_3$  on $U$ induce local parameters at the point $p$ 
if their pull-backs $\pi^*(z_1)$, $\pi^*(z_2)$, $\pi^*(z_3)$  are eigen-coordinate functions around the origin in $\hat{U}\subset\mathbb{C}^3$, corresponding to the weights
 $1$, $a$, $r-a$.
 
Let $\tilde{f}:(\tilde{U}, E)\to (U, p)$ be the local weighted blow up at the point $p$ with weights $(1,a, r-a)$, where $E$ is the exceptional divisor. The multiplicity of an effective Weil
divisor $D$ on $U$ at the point $p$ is defined by the number $\frac{m}{r}$ such that
\[\tilde{f}^*(D)=\tilde{D}+\frac{m}{r}E,\] 
where $\tilde{D}$ is the proper transform of $D$ by $\tilde{f}$. 
An analytic function $g(z_1, z_2, z_3)$ on $U$ defines a divisor $D$ on $U$.
The vanishing order (or the multiplicity) of $g$ at the point $p$ is defined by the multiplicity of the divisor $D$ at the point $p$.  
The multiplicity can be also obtained in the following way.
The functions $z_1$, $z_2$, $z_3$ induce local parameters at the point $p$, so that we could assume that
 their pull-backs $\pi^*(z_1)$, $\pi^*(z_2)$, $\pi^*(z_3)$ are eigen-coordinate functions on $\mathbb{C}^3$ locally around the origin, corresponding to the weights
 $1$, $a$, $r-a$, respectively. 
Counting the multiplicities of $\pi^*(z_1)$, $\pi^*(z_2)$, $\pi^*(z_3)$ as $1$, $a$, $r-a$, respectively, we see that
 the multiplicity of $g$ at the point $p$ coincides with the number $$\frac{1}{r}\mult_0(\pi^*(g(z_1, z_2, z_3)))$$
(see \cite[Lemma~3.2.1]{Pro}).

In the present paper, it is crucial to obtain the multiplicities of various quasi-homogeneous polynomials $G(x,y,z,t,w)$ at a singular point $p$ on a given quasi-smooth hypersurface $X$.  At the point $p$, we can always see that three, say $z_1$, $z_2$ and $z_3$,  of the homogenous coordinates $x$, $y$, $z$, $t$, $w$ induce local parameters 
at the point $p$. Locally around the point $p$, the quasi-homogeneous polynomial $G(x,y,z,t,w)$ induces a function $g(z_1, z_2, z_3)$ as a formal power series in variables
 $z_1$, $z_2$, $z_3$. The vanishing order (or  the multiplicity) of $G$ at the point $p$ is defined by the multiplicity of $g(z_1,z_2, z_3)$ at the point $p$  which is equal to the number $\frac{1}{r}\mult_0(\pi^*(g(z_1, z_2, z_3)))$ with  counting the multiplicities of $\pi^*(z_1)$, $\pi^*(z_2)$, $\pi^*(z_3)$ as $1$, $a$, $r-a$, respectively, as before.

\subsection{Excluding singular points}\label{section:Methods}
This section  provides the methods we apply when we exclude  the singular points as centers of non-canonical singularities of the log pair $(X, \frac{1}{n}\mathcal{M})$ .

Let $p$ be a singular point of type $\frac{1}{r}(1,a,r-a)$ on $X$,
where $r$ and $a$ are relatively prime positive integers with $a<r$.  
As mentioned in Section~\ref{subsection:notation}, the weighted blow up of $X$ at the point $p$ with weights
$(1,a,r-a)$ will be denoted by $f\colon Y\to X$. Its exceptional
divisor and the anticanonical divisor of $Y$ will be denoted by $E$ and
$B$, respectively. We denote by $\mathcal{M}_{Y}$ the proper
transform of the linear system $\mathcal{M}$ by the weighted blow up $f$. The
pull-back of $-K_X$ will be denoted by $A$. The surface $S$ is the
proper transform of the surface on $X$ cut by the equation
$x=0$.

Since the Picard group of $X$ is generated by $-K_X$, the
surface $S$ is always irreducible. The surface $S$ can be assumed
to be $\mathbb{Q}$-linearly equivalent to $B$ if one of the
following conditions holds:
\begin{itemize}
\item $a_1=1$; \item $d-1$ is not divisible by $r$.
\end{itemize}
If $a_1>1$ and $d-1$ is divisible by $r$, then it is easy to check that
$S$ is always $\mathbb{Q}$-linearly equivalent to either $B$ or
$B-E$.

Before we explain how to show that $p$ is not a center of
non-canonical singularities of the log pair
$\left(X,\frac{1}{n}\mathcal{M}\right)$, let us prove the following  statement slightly modified from Corollary~\ref{corollary:Mori-cone}.

\begin{lemma}\label{lemma:Mori-cone2}
Suppose that $p$ is a center of non-canonical singularities of the log
pair $\left(X,\frac{1}{n}\mathcal{M}\right)$. Then the $1$-cycle $B\cdot S
\in N_1(Y)$ lies in the interior of the Mori cone of $Y$:
$$
B\cdot S \in \mathrm{Int}(\overline{\mathrm{NE}(Y)}).
$$
\end{lemma}
\begin{proof}
It follows from Theorem~\ref{theorem:Kawamata} that
$$
\mathcal{M}_{Y}\sim_{\mathbb{Q}}nB-\epsilon E
$$
for some positive rational number $\epsilon$. Since
$$
S\cdot\mathcal{M}_{Y}=S\cdot(nB-\epsilon E)
$$
is an effective $1$-cycle, the $1$-cycle $B\cdot S$ must lie in
the interior of the Mori cone of $Y$ because the $1$-cycles
$S\cdot E$ and $S\cdot B$ are not proportional in $N_1(Y)$ and
the $1$-cycle $S\cdot E$ generates the extremal ray contracted by
$f$.
\end{proof}

We have two kinds of singular points on $X$. The singular points with $B^3\leq 0$ are one kind and the singular points with $B^3>0$ are the other kind.
Those with $B^3\leq 0$ will be excluded  as centers of non-canonical singularities of the log pair 
$\left(X, \frac{1}{n}\mathcal{M}\right)$. Meanwhile,
those with $B^3>0$ will be either excluded or untwisted (see Definition~\ref{definition:untwisting}).

To exclude singular points with
$B^3\leq 0$, we mainly apply the following lemma. It is a slightly
modified version of \cite[Lemma~5.4.3]{CPR}.

\begin{lemma}\label{lemma:boundary}
Suppose that $B^3\leq 0$ and there is an index $i$ such that
\begin{itemize}
\item there is a surface $T$ on $Y$ such that $T\sim_{\mathbb{Q}} a_iA-\frac{m}{r}E$ with $a_i\geq m>0$;%
\item the intersection $\Gamma=S\cap T$ consists of irreducible curves that are numerically proportional to each other; %
\item $T\cdot\Gamma\leq 0$.%
\end{itemize}
Then the point $p$ is not a center of non-canonical singularities
of the log pair $\left(X,\frac{1}{n}\mathcal{M}\right)$.
\end{lemma}

\begin{proof}
Let $\Gamma=\sum e_i\tilde{C}_i$, where $e_i>0$ and $\tilde{C}_i$'s are distinct
irreducible and reduced curves. Let $\tilde{R}$ be the extremal ray of the Mori cone
$\overline{\mathrm{NE}(Y)}$  of $Y$ contracted by $f\colon Y\to X$.

Since the curves $\tilde{C}_i$ are numerically proportional to each other,
each irreducible curve $\tilde{C}_i$ defines the same ray in the Mori cone
 of $Y$.  None of the curves $\tilde{C}_i$ is contained in $E$ because $T\cdot \tilde{C}_i\leq 0$ and $T\cdot E^2<0$. Therefore, the ray $\tilde{Q}$ defined by $\Gamma$ cannot be $\tilde{R}$.

 We first claim that the ray $\tilde{Q}$
 is an extremal ray of
$\overline{\mathrm{NE}(Y)}$, so that the Mori cone
$\overline{\mathrm{NE}(Y)}$ could be spanned by $\tilde{R}$ and $\tilde{Q}$ .

Since $\tilde{C}_i\not\subset E$ for each
$i$, we have $E\cdot \tilde{C}_i\geq 0$. Therefore, 
$$
a_iB\cdot\Gamma\leq T\cdot \Gamma \leq 0,
$$
where the first inequality follows from $a_i\geq m$.

If the surface $T$ is
nef, then $T\cdot \Gamma=0$ and hence $\Gamma$ is in the boundary
of $\overline{\mathrm{NE}(Y)}$. Therefore,  the ray $\tilde{Q}$
 is an extremal ray of
$\overline{\mathrm{NE}(Y)}$.

Suppose that the surface $T$ is not nef and that  the ray $\tilde{Q}$
 is not an extremal ray.
Then there is a curve $\tilde{C}$ with $T\cdot \tilde{C}<0$ that generates a ray between $\tilde{Q}$ and the extremal ray other than $\tilde{R}$ since
$T\cdot \tilde{R}=-\frac{m}{r}E\cdot \tilde{R}>0$.
It follows from $\tilde{C}\not\subset E$ that
\[S\cdot \tilde{C}\leq B\cdot \tilde{C}=\frac{1}{a_i}\left(T-\frac{a_i-m}{r}E\right)\cdot \tilde{C}<0,\]
and hence $\tilde{C}\subset S\cap T$. Therefore, the curve $\tilde{C}$ must be one
of the component of $\Gamma$, and hence it generates the ray $\tilde{Q}$. This is a contradiction. Therefore,
$\tilde{Q}$ must be the extremal ray of $\overline{\mathrm{NE}(Y)}$  other than $\tilde{R}$.

If $B\cdot S\in\mathrm{Int}(\overline{\mathrm{NE}(Y)})$, then the
ray
\[\tilde{Q}=\mathbb{R}_+\left[S\cdot\left(a_iB+\frac{a_i-m}{r}E\right)\right]\]
cannot be a boundary of $\overline{\mathrm{NE}(Y)}$. Therefore,
Lemma~\ref{lemma:Mori-cone2} implies that the point $p$ cannot be
a center of non-canonical singularities of the log pair
$\left(X,\frac{1}{n}\mathcal{M}\right)$.
\end{proof}

\begin{remark}\label{remark:boundary}
The condition $T\cdot\Gamma \leq 0$ is equivalent to  the
inequality $$ra(r-a)a_i^2A^3\leq km^2,$$ where $k=1$ if
$S\sim_{\mathbb{Q}}  B$; $k=r+1$ otherwise.
\end{remark}

We have  singular points with $B^3\leq 0$ to which we cannot apply
Lemma~\ref{lemma:boundary} in a simple way.
Such singular points are dealt with in a special way in \cite[Subsections 5.7.2 and 5.7.3]{CPR}.  However, we
are dealing with every quasi-smooth hypersurface, not only a
general one and the method of \cite{CPR} is too complicated for us
to analyze  the irreducible components of the intersections
$\Gamma$, which is inevitable for our purpose. We here present
another method that enables us to avoid such difficulty.

\begin{lemma}\label{lemma:boundary-2} Suppose that there is a nef
divisor $T$ on $Y$  with $T\cdot S\cdot B\leq 0$ and $T\cdot S\cdot A>0$. Then the point
$p$ cannot be a center of non-canonical singularities of the log
pair $\left(X,\frac{1}{n}\mathcal{M}\right)$.
\end{lemma}

\begin{proof}
Suppose that $p$ is a center of non-canonical singularities of the
log pair $\left(X,\frac{1}{n}\mathcal{M}\right)$. Then it follows from
Theorem~\ref{theorem:Kawamata} that
$$\frac{1}{n}\mathcal{M}_Y=f^*\left(\frac{1}{n}\mathcal{M}\right)-m E$$ with some rational number
$m>\frac{1}{r}$. The
intersection of the surface $S$ and a general surface $M_Y$ in the
mobile linear system $\mathcal{M}_Y$ gives us an effective
$1$-cycle. However,
$$
T\cdot S\cdot M_Y=nT\cdot S\cdot (A-m E)<nT\cdot S\cdot \left(A-\frac{1}{r} E\right)=nT\cdot S\cdot B\leq 0,%
$$
where the first inequality follows from $ 0<T\cdot S\cdot A\leq \frac{1}{r}T\cdot S\cdot E$.
This contradicts the condition that $T$ is nef.
\end{proof}

\begin{remark}
\label{remark:boundary-2} For the divisor $T$ equivalent to
$cA-\frac{m}{r}E=cB+\frac{c-m}{r}E$ with some positive integers $c$
and $m$, the condition $T\cdot S\cdot B\leq 0$ is equivalent to
the inequality $$ra(r-a)cA^3\leq km,$$ where $k=1$ if
$S\sim_{\mathbb{Q}} B$; $k=r+1$ otherwise. The condition $T\cdot S\cdot A>0$ is always satisfied  by any divisor $T$ equivalent to
$cA-\frac{m}{r}E$ with positive integers $c$.
\end{remark}

To apply Lemma~\ref{lemma:boundary-2}, we construct a nef divisor
$T$ in $|cB+bE|$ for some integers $c\geq 0$ and
$b\leq\frac{c}{r}$.
 To construct  a nef  divisor $T$  the following will be useful.

\begin{lemma}
\label{lemma:nefness} Let $\mathcal{L}_{X}$ be a mobile linear
subsystem in $|-cK_X|$ for some positive integer $c$. Denote the
proper transforms of the base curves of the linear system
$\mathcal{L}_X$ on $Y$ by $\tilde{C}_1,\ldots ,\tilde{C}_s$ (if
any). Let $T$ be the proper transform of a general surface in
$\mathcal{L}_X$. Then the following hold.
\begin{itemize}
\item  The divisor $T$ belongs to $|cB+bE|$ for some  integer  $b$
not greater than $\frac{c}{r}$. \item  The divisor $T$ is nef if
$T\cdot\tilde{C}_i\geqslant 0$ for every $i$. In particular, it is
nef if the base locus of $\mathcal{L}_{X}$ contains no curves.
\end{itemize}
\end{lemma}

\begin{proof} Since $T\sim_{\mathbb{Q}}cA-\frac{m}{r}E$ for some non-negative integer $m$ and  $B\sim_{\mathbb{Q}} A-\frac{1}{r}E$, we obtain $T\sim_{\mathbb{Q}} cB+\frac{c-m}{r}E$. The number $b:=\frac{c-m}{r}$ must be an integer because the divisor class group of $Y$ is generated by $B$ and $E$.

Suppose that $T$ is not nef. Then there exists a curve
$\tilde{C}\subset Y$ such that $T\cdot\tilde{C}<0$, which implies
that the curve $\tilde{C}$ is contained in the base locus of the
proper transform of the linear system $\mathcal{L}_X$. Since $E\cong \mathbb{P}(1,a,r-a)$, $\mathcal{O}_E(E)=\mathcal{O}_E(-r)$ and  $b\le
\frac{c}{r}$, the divisor $T\vert_{E}$ is nef, and hence
$\tilde{C}\not\subset E$. We then draw an absurd conclusion that
$\tilde{C}$ is one of the curves $\tilde{C}_1,\ldots,
\tilde{C}_s$.
\end{proof}

With Lemma~\ref{lemma:boundary-2}  we can easily exclude the singular points that are taken special cares in
\cite[5.7.2 and 5.7.3]{CPR}. 
However, in spite
of our new methods, we encounter special cases that cannot  be
excluded by the methods proposed so far. To deal with these
special cases, we apply the following two lemmas.

\begin{lemma}\label{lemma:negative-definite}
Suppose that the surface $S$ is $\mathbb{Q}$-linearly equivalent
to $B$ and there is a normal surface $T$ on $Y$ such that the
support of the $1$-cycle $S\vert_{T}$ consists of curves on $T$
whose intersection form is negative-definite. Then the singular
point $p$ cannot be a center of non-canonical singularities of the
pair $\left(X, \frac{1}{n}\mathcal{M}\right)$.
\end{lemma}

\begin{proof}
Put $S\vert_{T}=\sum c_i \tilde{C}_i$, where $c_i$'s are positive numbers
and $\tilde{C}_i$'s are  distinct irreducible and reduced curves on the normal
surface $T$. Suppose that the point $p$ is a center of
non-canonical singularities of the log pair $(X,
\frac{1}{n}\mathcal{M})$. Then we have
$$
K_Y+\frac{1}{n}\mathcal{M}_Y+cE=
f^*\left(K_X+\frac{1}{n}\mathcal{M}\right)\sim_{\mathbb{Q}} 0,
$$
where $c$ is a positive constant. Therefore, we
obtain $\mathcal{M}_Y+ncE\sim_{\mathbb{Q}} nS$, and hence
$$
\left(\mathcal{M}_Y+ncE\right)\big |_{T} \sim_{\mathbb{Q}} n\sum
c_i \tilde{C}_i.
$$
We may write the left-hand side as
$$
\left(\mathcal{M}_Y+ncE\right)\big |_{T} =\sum a_j\tilde{D}_j+\sum
b_i \tilde{C}_i,
$$
where each $\tilde{D}_j$ is an irreducible curves on $T$ different from $\tilde{C}_i$ and $a_j$, $b_i$ are positive rational numbers.
Note that $\sum a_j\tilde{D}_j$ cannot be a zero divisor because $\mathcal{M}_Y$ is a mobile linear system.
We then obtain
$$
\sum a_j\tilde{D}_j+\sum_{nc_i-b_i<0}
-(nc_i-b_i )\tilde{C}_i\sim_{\mathbb{Q}} \sum_{nc_i-b_i>0}
(nc_i-b_i )
 \tilde{C}_i.
$$
Therefore,
$$
\left(\sum a_j\tilde{D}_j+\sum_{nc_i-b_i<0}
-(nc_i-b_i )\tilde{C}_i\right)\cdot\left( \sum_{nc_i-b_i>0}
(nc_i-b_i )
 \tilde{C}_i\right)=\left( \sum_{nc_i-b_i>0}
(nc_i-b_i )
 \tilde{C}_i\right)^2.
$$
However, since the divisor $\sum  \tilde{C}_i$ is
negative-definite and
$\sum_{nc_i-b_i>0}
(nc_i-b_i )
 \tilde{C}_i$ cannot be a zero divisor on $T$, the equality is absurd.
\end{proof}

\begin{lemma}\label{lemma:bad-link}
Suppose that there is a one-dimensional family of irreducible
curves $\tilde{C}_\lambda$ on $Y$ with $E\cdot \tilde{C}_\lambda>0$ and $-K_Y\cdot
\tilde{C}_\lambda\leq 0$. Then the singular point $p$ cannot be a center
of non-canonical singularities of the log pair
$\left(X,\frac{1}{n}\mathcal{M}\right)$.
\end{lemma}

\begin{proof}
We have
$$
K_Y+\frac{1}{n}\mathcal{M}_Y=f^*\left(K_X+\frac{1}{n}\mathcal{M}\right)+cE
$$
with a negative number $c$. Suppose that
there is a one-dimensional family of curves $\tilde{C}_\lambda$ on $Y$
with $E\cdot \tilde{C}_\lambda>0$ and $-K_Y\cdot \tilde{C}_\lambda\leq 0$. Then
for each member $\tilde{C}_\lambda$, we have
$$
\mathcal{M}_Y\cdot \tilde{C}_\lambda=-nK_Y\cdot \tilde{C}_\lambda+cnE\cdot \tilde{C}_\lambda\leq cnE\cdot \tilde{C}_\lambda<0,%
$$
and hence the curve $\tilde{C}_\lambda$ is contained in the base locus of
the linear system $\mathcal{M}_Y$. This is a contradiction since
the linear system $\mathcal{M}_Y$ is mobile.
\end{proof}

Notice that Lemmas~\ref{lemma:boundary-2},
\ref{lemma:negative-definite} and \ref{lemma:bad-link} do not
require $B^3$ to be non-positive. Therefore, these lemmas can be
applied  to exclude the singular points with $B^3>0$.

For example, 
the lemma below, which follows from
Lemma~\ref{lemma:bad-link},
 excludes all the singular points with $B^3>0$,
except $O_z$ in the family No.~62, that appear in
Theorem~\ref{theorem:auxiliary}. The exception, the singular point
$O_z$ in the family No.~62, can be also treated in the same way as
Lemma~\ref{lemma:2-ray game}. The only difference is that the
variable $z$ plays the role of $t$ in Lemma~\ref{lemma:2-ray
game}.

\begin{lemma}\label{lemma:2-ray game}
Suppose that the hypersurface $X$ is given by a quasi-homogeneous
equation
\[w^2+x_it^k+wf_{d-a_4}(x,x_i, x_j, t)+f_{d}(x,x_i, x_j, t)=0\]
of degree $d$, where one of the variables $x_i$ and $x_j$ is
$y$ and the other is $z$.  Let $a_i$ and $a_j$ be the weights of
the variables $x_i$ and $x_j$, respectively. If $2a_4=3a_3+a_i$,
then the singular point $O_t$ cannot be a center of non-canonical
singularities of the log pair $\left(X,\frac{1}{n}\mathcal{M}\right)$.
\end{lemma}

\begin{proof}

The singular point $O_t$ is of type
$\frac{1}{a_{3}}(1,a_{j},a_{4}-a_3)$. Local parameters at $O_t$ are induced by $x$, $x_j$, $w$ with multiplicities 
$\frac{1}{a_3}$, $\frac{a_j}{a_3}$, $\frac{a_4-a_3}{a_3}$. 

Let $T$ be the proper
transform of the surface $S_{x_{i}}$ on $X$ cut by the equation
$x_{i}=0$.  Due to the monomial $w^2$, we see that the surface
$S_{x_{i}}$  has multiplicity $\frac{2(a_4-a_3)}{a_{3}}$  at the
point $O_t$. Therefore, the surface $T$ belongs to $|a_{i}B-E|$
since $2a_4=3a_3+a_i$.

Let $C_{\lambda}$ be the curve on the surface $S_{x_{i}}$ defined
by
$$
\left\{%
\aligned
&x_{i}=0,\\%
&x_{j}=\lambda x^{a_{j}}\\%
\endaligned\right.%
$$
for a sufficiently general complex number  $\lambda$. Then the
curve $C_{\lambda}$ is a curve of degree $d$ in $\mathbb{P}(1,
a_{3}, a_{4})$ defined by the equation
\[w^2+wf_{d-a_4}(x,0, \lambda x^{a_{j}}, t)+f_{d}(x,0, \lambda x^{a_{j}}, t)=0.\]
Then
$$
-K_{Y}\cdot \tilde{C}_{\lambda}=a_{j}B^2\cdot
(a_{i}B-E)=a_1a_2A^3-\frac{2a_j(a_4-a_3)}{a_3^3}E^3
=\frac{2}{a_3}-\frac{2}{a_3}=0.
$$
If the curve $\tilde{C}_{\lambda}$ is reducible, it consists of
two irreducible components that are numerically equivalent since
the two components of the curve $C_\lambda$ are symmetric with
respect to the biregular quadratic involution of $X$ defined by
$$[x:y:z:t:w]\mapsto [x:y:z:t:-f_{d-a_4}(x,y, z, t)-w].$$ Then each
component of $\tilde{C}_{\lambda}$ intersects $-K_Y$ trivially.
Consequently, Lemma~\ref{lemma:bad-link} implies the statement.
\end{proof}



\subsection{Untwisting singular points}\label{section:untwisting}

Excluding methods are introduced in the previous section. Now we explain how to deal with singular points of $X$ that require some treatments by birational automorphisms of $X$. 
For us to prove Main Theorem, for a given singular point  either it should be excluded as a center of non-canonical singularities of the log pair $(X, \frac{1}{n}\mathcal{M})$ or it should be untwisted as a center of non-canonical singularities of the log pair $(X, \frac{1}{n}\mathcal{M})$.
Untwisting  is defined as follows:

\begin{definition}
\label{definition:untwisting} Let $\tau$ be a birational automorphism of
$X$. Suppose that a singular point $p$ of $X$ is a center of non-canonical singularities of the log
pair $\left(X,\frac{1}{n}\mathcal{M}\right)$. We say that the birational automorphism $\tau$
\emph{untwists} the point $p$ (as a center of non-canonical
singularities of the log pair $\left(X,\frac{1}{n}\mathcal{M}\right)$) if
\begin{itemize}
\item the birational automorphism $\tau$ is not biregular;
\item there exists
a biregular
in codimension one birational automorphism $\tau_{Y}$ of $Y$   that fits the commutative diagram
$$
\xymatrix{
Y\ar@{->}[d]_{f}\ar@{-->}[rr]^{\tau_Y}&&Y\ar@{->}[d]^{f}\\%
X\ar@{-->}[rr]_{\tau}&&X.\\}
$$ %

\end{itemize}
\end{definition}
In fact, this is a special case of a Sarkisov link of Type~II (cf. \cite[Definition~3.1.4]{CPR}). The reason why such a birational automorphism is said to untwist a singular point is that   it improves the singularities of the mobile linear system $\mathcal{M}$. This improvement results from the following property of such a birational automorphism.
\begin{lemma}\label{lemma:untwisting-invoultion}
Suppose that a singular point $p$ of $X$ is a center of non-canonical singularities of
the log pair $(X\frac{1}{n}\mathcal{M})$ and that there exists a birational automorphism $\tau$ of $X$
that untwists the point $p$ as a center
of non-canonical singularities of the log pair
$\left(X,\frac{1}{n}\mathcal{M}\right)$. Then $\tau(\mathcal{M})\subset
|-n_{\tau}K_X|$ for some positive integer $n_{\tau}<n$.
\end{lemma}

\begin{proof}
Put $\tau_Y=f^{-1}\circ \tau\circ f$. Then $\tau_Y$  is biregular
in codimension one and $\tau_Y$ is not biregular. In particular,
$\tau_Y$ acts on the Picard group $\mathrm{Pic}(Y)$. Then $\tau_Y(B)=B$ since
$B=-K_{Y}$.  However, $\tau_Y(E)\ne E$ since $\tau_Y$ is biregular in codimension one. Indeed, if $\tau_Y(E)= E$, then $\tau$ is also biregular in codimension one. Then \cite[Proposition~3.5]{Co95} implies that $\tau$ is biregular since $\mathrm{Pic}(X)\cong \mathbb{Z}$.

 On the other
hand, we have
$$
f^*(\mathcal{M})=\mathcal{M}_Y+m E,
$$
for some positive rational number $m$. Furthermore,
$m>\frac{n}{r}$ by Theorem~\ref{theorem:Kawamata}.
Since $\tau_Y$ acts on $\mathrm{Pic}(Y)$, there are rational
numbers $a, b, c, d$ such that $a, c>0$ and
$$
\left\{%
\aligned
&\tau_Y(A)=aA-bE,\\%
&\tau_Y(E)=cA-dE.\\%
\endaligned\right.%
$$
Since $\tau_Y(B)=B$, we obtain
$$
A-\frac{1}{r}E=\tau_Y\left(A-\frac{1}{r}E\right)=\tau_Y(A)-\frac{1}{r}\tau_Y(E)=\left(a-\frac{c}{r}\right)A-\left(b-\frac{d}{r}\right)E,
$$
and hence $a-\frac{c}{r}=1$. We then see
$$
\tau_Y(\mathcal{M}_Y)=\tau_Y(nA-mE)=n\tau_Y(A)-m\tau_Y(E)=(na-mc)A-(nb-md)E.
$$
Since $$na-mc=na-m(ar-r)=na-mr(a-1)<na-n(a-1)=n,$$ we obtain
$\tau(\mathcal{M})\subset |-n_{\tau}K_X|$ with $n_{\tau}<n$. This
proves the statement.
\end{proof}

Thus, to complete the proof of Main Theorem after Theorems~\ref{theorem:smooth point excluding} and~\ref{theorem:excluding-curve}, it is enough to show
that every singular point of $X$ either is  not a center of
non-canonical singularities of the log pair
$(X\frac{1}{n}\mathcal{M})$ or can be untwisted by some
appropriate birational automorphism of $X$. This follows from
Theorem~\ref{theorem:Nother-Fano} and
Lemma~\ref{lemma:untwisting-invoultion} with induction on $n$.
The appropriate birational automorphisms to untwist singular points are introduce in the following chapter.

\begin{remark}\label{remark:untwisting-involution}
The proof of Lemma~\ref{lemma:untwisting-invoultion} shows that in order  to find a birational automorphism of $X$ untwisting the center $p$, it is enough to find  a biregular
in codimension one birational automorphism $\tau_{Y}$ of $Y$ such that $\tau_{Y}(E)\ne E$. Indeed, this untwisting birational automorphism  is defined by $\tau=f\circ\tau_Y\circ f^{-1}$.
\end{remark}

As in \cite{CPR}, in the case when a singular point of $X$ is untwisted by
some birational automorphism of $X$, it can be untwisted by
a very explicit birational involution. Since $X$ has only finitely
many singular points, there are finitely many such involutions
for a given hypersurface $X$. These birational automorphisms generate a subgroup, denoted by $\Gamma_X$,  in the birational automorphism  group
 $\mathrm{Bir}(X)$. Using
\cite[Theorem~4.2]{Co95} instead of
Theorem~\ref{theorem:Nother-Fano}, we prove

\begin{theorem}
\label{theorem:Bir} Let $X$ be a quasi-smooth hypersurface of
degrees $d$ with only terminal singularities in the weighted
projective space $\mathbb{P}(1, a_1, a_2, a_3, a_4)$, where
$d=\sum  a_i$. Then the  birational automorphism group of  $X$ is generated by the subgroup $\Gamma_X$
and the biregular automorphism group  of $X$.
\end{theorem}
In the case when $X$ is a general hypersurface in its family,
Theorem~\ref{theorem:Bir} is proved in \cite{CPR} (see
\cite[Remark~1.4]{CPR}).

\newpage

\section{Birational involutions}

\subsection{Quadratic involution}
In many cases, explicit birational automorphisms arise from generically 2-to-1 rational maps of $X$ onto appropriate $3$-dimensional weighted projective spaces. The birational automorphism constructed by interchanging the two points on a generic fiber of 
the generically 2-to-1 rational map  is called a quadratic involution.

\begin{lemma}[{\cite[Theorem~4.9]{CPR}}]
\label{lemma:Quadratic involution} Suppose that the hypersurface
$X$ is given by
\begin{equation}\label{equation:quadratic-involution-1}
x_{i_3}x_{i_4}^2+f_ex_{i_4}+g_d=0,
\end{equation}
where $x_{i_4}$, $x_{i_3}$ are two of the coordinates and $f_e$,
$g_d$ are quasi-homogeneous polynomials of degrees $e$ and $d$ not
involving $x_{i_4}$. In addition, suppose that the polynomial
$f_e$ is not divisible by $x_{i_3}$.
Then interchanging the roots of the equation with respect to $x_{i_4}$
defines a birational involution $\tau_{O_{x_{i_4}}}$ of $X$.  The involution
$\tau_{O_{x_{i_4}}}$  untwists the point  $O_{x_{i_4}}$ as a center of
non-canonical singularities of the log pair $\left(X,\frac{1}{n}\mathcal{M}\right)$.
\end{lemma}
\begin{proof}
If the the polynomial $f_e$ is not divisible by $x_{i_3}$, the equations
$x_{i_3}=f_e=g_d=0$ define a finitely many lines passing through the point $O_{x_{i_4}}$. The statement then follows from 
the proof of \cite[Theorem~4.9]{CPR}.
\end{proof}
Now we suppose that $f_e$ in \eqref{equation:quadratic-involution-1} is divisible by $x_{i_3}$.
Then we are able to write $f_e=2x_{i_3}g$ for some polynomial $g$ not
involving $x_{i_4}$.  Therefore, we obtain
\[x_{i_3}x_{i_4}^2+f_ex_{i_4}+g_d=x_{i_3}\left(x_{i_4}^2+2gx_{i_4}\right)+g_d=x_{i_3}\left(x_{i_4}+g\right)^2-x_{i_3}g^2+g_d.\]
Using the change of coordinate $x_{i_4}+g\mapsto x_{i_4}$, we see that  the singular point $O_{x_{i_4}}$ on the hypersurface of $X$ defined by \eqref{equation:quadratic-involution-1} with $f_e$ divisible by $x_{i_3}$
can be excluded by the following lemma.

\begin{lemma}\label{lemma:Quadratic involution-biregular-1}
Suppose that the hypersurface $X$ is given by
\[x_{i_3}x_{i_4}^2+x_{i_3}g_e(x,x_{i_1},x_{i_2}, x_{i_3})+h_d(x,x_{i_1},x_{i_2})=0,\]
where $x_{i_k}$'s are the coordinates of
$\mathbb{P}(1,a_1,a_2,a_3,a_4)$ different from $x$.
If the weights of $x_{i_1}$,
$x_{i_2}$ are less than the weight of $x_{i_4}$, then the singular
points $O_{x_{i_4}}$ cannot be a center of non-canonical
singularities of  the log pair $\left(X,\frac{1}{n}\mathcal{M}\right)$.
\end{lemma}

\begin{proof}
Note that the quasi-homogeneous polynomial $h_d$ must be
irreducible. Indeed, if it is reducible, then we may write
$h_d(x,x_{i_1},x_{i_2})= g_{d_1}(x,x_{i_1},x_{i_2})g_{d_2}(x,x_{i_1},x_{i_2})$ for some non-constant polynomials $g_{d_1}$ and $g_{d_2}$. Then, the hypersurface $X$ is not quasi-smooth at the  points defined by
$x_{i_3}=x_{i_4}^2+g_e(x,x_{i_1},x_{i_2}, x_{i_3})=g_{d_1}(x,x_{i_1},x_{i_2})=g_{d_2}(x,x_{i_1},x_{i_2})=0$.

Let
$T$ be the proper transform on $Y$ of the surface $S_{x_{i_3}}$ cut by
$x_{i_3}=0$.
The singular point $O_{x_{i_4}}$ is of type
$\frac{1}{a_{i_4}}(1,a_{i_1},a_{i_2})$.
Since  local parameters at $O_{x_{i_4}}$ are induced by $x$, $x_{i_1}$, $x_{i_2}$ whose multiplicities  are
$\frac{1}{a_{i_4}}$, $\frac{a_{i_1}}{a_{i_4}}$,$\frac{a_{i_2}}{a_{i_4}}$, respectively,  and  the polynomial $h_d(x,x_{i_1},x_{i_2})$ cannot be zero, the
surface cut by $x_{i_3}=0$ has  multiplicity
$\frac{d}{a_{i_4}}$   at $O_{x_{i_4}}$.  Therefore, the surface $T$ belongs to
$|a_{i_3}B-2E|$ since $a_{i_3}+2a_{i_4}=d$.

Let $C_{\lambda}$ be the curve on the surface $S_{x_{i_3}}$ defined
by
$$
\left\{%
\aligned
&x_{i_3}=0,\\%
&x_{i_2}=\lambda x^{a_{i_2}}\\%
\endaligned\right.%
$$
for a sufficiently general complex number  $\lambda$. Then the
curve $C_{\lambda}$ is a curve of degree $d$ in $\mathbb{P}(1,
a_{i_1}, a_{i_4})$ defined by equation
\[h_d(x,x_{i_1},\lambda x^{a_{i_2}})=0.\]

To obtain a one-dimensional family of irreducible curves on $Y$ that is required for Lemma~\ref{lemma:bad-link}, we claim
that every curve  on $T$ intersects $B$ non-negatively. To this
end, we consider the linear system $\mathcal{L}$ on $X$ given by
the monomials $x^{a_{i_1}+a_{i_2}}$, $x_{i_1}x_{i_2}$,
$x^{a_{i_1}}x_{i_2}$, $x^{a_{i_2}}x_{i_1}$. The proper transform
of a  surface in $\mathcal{L}$ is equivalent to
$(a_{i_1}+a_{i_2})B$. The base locus of the proper transform $\mathcal{L}_Y$ of
the linear system $\mathcal{L}$ consists of the proper transform
of the curve cut by $x=x_{i_1}=0$  and the proper transform of the
curve cut by $x=x_{i_2}=0$.

Suppose that we have a curve $R$ on $T$ such that $B\cdot R<0$. Since the linear system $\mathcal{L}_Y$ is free outside
the proper transforms
of the curve cut by $x=x_{i_1}=0$  and the
curve  by $x=x_{i_2}=0$, one of the proper transforms must contain the curve $R$.
Therefore, the
curve $R$ on the surface $T$ should be either the proper transform $\tilde{L}_{24}$
of the curve $L_{24}$ defined by $x=x_{i_1}=0$ and $x_{i_3}=0$ or the
proper transform $\tilde{L}_{14}$ of the curve $L_{14}$ defined by
$x=x_{i_2}=0$ and $x_{i_3}=0$. However, since $E\cdot
\tilde{L}_{24}=\frac{1}{a_{i_2}}$  and $E\cdot
\tilde{L}_{14}=\frac{1}{a_{i_1}}$, we obtain
\[\begin{split} & B\cdot \tilde{L}_{24}=\left(A-\frac{1}{a_{i_4}}E\right)\cdot \tilde{L}_{24}
=\frac{1}{a_{i_2}a_{i_4}}-\frac{1}{a_{i_4}a_{i_2}}=0;\\
& B\cdot \tilde{L}_{14}=\left(A-\frac{1}{a_{i_4}}E\right)\cdot \tilde{L}_{14}
=\frac{1}{a_{i_1}a_{i_4}}-\frac{1}{a_{i_4}a_{i_1}}=0.\\ \end{split}\] This
verifies the claim.
Then  from the equation
$$
-K_{Y}\cdot \tilde{C}_{\lambda}=a_{i_2}B^2\cdot (a_{i_3}B-2E)=0
$$
we obtain a one-dimensional family of irreducible curves on $Y$ that is required for Lemma~\ref{lemma:bad-link}.
It then follows from Lemma~\ref{lemma:bad-link} that 
$O_{x_{i_4}}$ cannot be  a
center of non-canonical singularities of the log pair
$\left(X,\frac{1}{n}\mathcal{M}\right)$.
\end{proof}

\begin{theorem}\label{theorem:quadratic-biregular}
Suppose that the weights $a_3$, $a_4$ are relatively prime and
$2a_3+a_4=d$. In addition, the equation of the hypersurface $X$
does not involves the monomial $wt^2$. Then the singular point
$O_t$ cannot be a center of non-canonical singularities of the log
pair $(X, \frac{1}{n}\mathcal{M})$.
\end{theorem}

\begin{proof}
We first note that the singular point $O_t$ of the hypersurface
$X$ is of type $\frac{1}{a_3}(1, a_1, a_2)$. The hypersurface $X$
may be assumed to be defined by the equation
\[x_it^3+t^2g_{a_4}(x,y,z)+twg_{a_3}(x,y,z)+tg_{a_3+a_4}(x,y,z)+\]\[+w^2g_{d-2a_4}(x,y,z)+wg_{2a_3}(x,y,z)+g_d(x,y,z)=0,\]
where $x_i$ is either $y$ or $z$.  We let $x_j$ be $z$ if $x_i$ is $y$ and vice
versa.
By a suitable coordinate change (if necessary), we may assume that the
polynomial $g_{d-2a_4}$ contains the monomial $x_j$.

 Suppose that the singular point $O_t$ is a center of non-canonical singularities of the log
pair $(X, \frac{1}{n}\mathcal{M})$.
 Consider the linear
system $\mathcal{L}$ on $X$ generated by $x^{e}$ and $x_j$, where
$e$ is the weight of $x_j$. The proper transform of each member of
$\mathcal{L}$ is $\mathbb{Q}$-linearly equivalent to $eB$. The base locus of
the linear system $\mathcal{L}$ consists of the curve cut by
$x=x_j=0$. It consists of the curve $L_{tw}$  and its residual curve $R$. Note
that the residual curve $R$ cannot pass through the point $O_t$
since we have the monomial $x_it^3$.
Therefore,
\[B\cdot \tilde{L}_{tw}=eB^3+K_X\cdot
R=\frac{2ea_3+ea_4}{a_1a_2a_3a_4}-\frac{e}{a_1a_2a_3}-\frac{3ea_3}{a_1a_2a_3a_4}=-\frac{e}{a_1a_2a_4}.\]

Let $T$ be the proper transform of the surface on $X$ cut by the
equation $x_i=0$. In addition, let $\tilde{S}_\lambda$ be the proper
transform of the surface on $X$ cut by the equation $x_j-\lambda
x^e=0$ for a general constant $\lambda$.
The intersection $1$-cycle of the surface on $X$ cut by the
equation $x_i=0$ and  the surface on $X$ cut by the equation $x_j-\lambda
x^e=0$  is defined in $\mathbb{P}(1,a_3,a_4)$ by the equation
\[t^2g_{a_4}(x,0,\lambda
x^e)+twg_{a_3}(x,0,\lambda
x^e)+tg_{a_3+a_4}(x,0,\lambda
x^e)+\]\[+w^2g_{d-2a_4}(x,0,\lambda
x^e)+wg_{2a_3}(x,0,\lambda
x^e)+g_d(x,0,\lambda
x^e)=0\] if $x_i=y$,
\[t^2g_{a_4}(x,\lambda
x^e,0)+twg_{a_3}(x,\lambda
x^e,0)+tg_{a_3+a_4}(x,\lambda
x^e,0)+\]\[+w^2g_{d-2a_4}(x,\lambda
x^e,0)+wg_{2a_3}(x,\lambda
x^e,0)+g_d(x,\lambda
x^e,0)=0\] if $x_i=z$. Since $g_{d-2a_4}$ contains the monomial $x_j$, the equation in both the cases is divisible by $x^e$ but not by $x^{e+1}$. This implies
 that
\[T\cdot \tilde{S}_\lambda= e\tilde{L}_{tw}+\tilde{R}_\lambda= (2a_3-a_4)\tilde{L}_{tw}+\tilde{R}_\lambda,\]
where $\tilde{R}_\lambda$ is the residual curve.
The multiplicity of the
surface cut by $x_{i}=0$ along $E$  is determined by the monomial $w^2x_j$. It is $\frac{e+2e'}{a_3}$, where $e'$ is the weight of $x_i$, since the multiplicity of $w$ is $\frac{e'}{a_3}$ and that of $x_j$ is $\frac{e}{a_3}$ .
 Therefore, the surface $T$
is equivalent to $e'B-E$ since $e+e'=a_1+a_2=a_3$. Then
\[B\cdot \tilde{R}_\lambda= eB^2\cdot T-(2a_3-a_4)B\cdot \tilde{L}_{tw}=\frac{a_1a_2+ea_3-ea_4}{a_1a_2a_4}=0\]
since $a_4=e'+a_3$ and $ee'=a_1a_2$. We then obtain a
contradiction from Lemma~\ref{lemma:bad-link}.
\end{proof}

\subsection{Elliptic involution}
\label{subsection:hard-involutions}
Another way to obtain an involution is from an elliptic fibration with a section and the group structure on its generic fiber. 
We can, roughly speaking, construct the involution by sending every point to its inverse point with respect to the group structure. The involution 
constructed in this way is called an elliptic involution. 

\begin{proposition}\label{proposition:elliptic involution}
Let $\pi: W\to \Sigma$ be an elliptic fibration over a normal surface $\Sigma$ with a section $F$. Then there is a birational involution $\tau_W$ of $W$ such that it induces the elliptic involution with respect to the point $C\cap F$ on a general fiber $C$.
\end{proposition}
\begin{proof}
Let $W_{\zeta}$ be the (scheme) fiber of $\pi$ over a generic
point $\zeta$ of $\Sigma$.  Then $W_{\zeta}$ is a smooth geometrically
irreducible curve over the rational function field $\mathbb{K}$ of $\Sigma$ over  $\mathbb{C}$,  which is
birational to a cubic curve on $\mathbb{P}^2_{\mathbb{K}}$. Since $F$ is a section of
$\pi$, it defines a
$\mathbb{K}$-rational point
of the curve $W_{\zeta}$. We  denote this point by $F_\zeta$. Thus,
$W_{\zeta}$ is an
elliptic curve  defined over $\mathbb{K}$.  To be precise, $W_{\zeta}$ has a group structure such that  the $\mathbb{K}$-rational point $F_\zeta$ is its identity and
all the group operations are morphisms defined over $\mathbb{K}$  (see, for example, \cite[Theorem~3.6 in Chapter~III]{Silverman}).
 This group structure gives an
involution $\tau_{W_{\zeta}}$ of $W_{\zeta}$  that sends every
$\mathbb{K}$-rational point to its inverse.
By construction, the involution $\tau_{W_{\zeta}}$ is a biregular
automorphism of the curve $W_{\zeta}$ defined over $\mathbb{K}$ that leaves the point $F_\zeta$ fixed.  Since the
rational function field of $W$ over $\mathbb{C}$  and the rational function filed of $W_{\zeta}$ over $\mathbb{K}$ are naturally isomorphic
as $\mathbb{C}$-algebras, the involution
$\tau_{W_{\zeta}}$ defines a $\mathbb{C}$-algebra  involution of the
rational function field of $W$ that leaves the subfield $\mathbb{K}$ fixed. Therefore, it induces
a birational involution
$\tau_{W}\in\mathrm{Bir}(W)$ such that the diagram
$$
\xymatrix{
W\ar@{->}[drr]_{\pi}\ar@{-->}[rrrr]^{\tau_W}&&&&W\ar@{->}[dll]^{\pi}\\%
&&\Sigma&&}
$$ %
commutes.
\end{proof}

Taken the Weierstrass equation of an elliptic curve into consideration,  an elliptic involution can be also regarded as a quadratic involution.
Because its expression in polynomials becomes extremely complicated after weighted blow ups and log flips (see \eqref{equation:CPR-inv2}),  it is difficult to see the virtue of an elliptic involution from the point of view of a quadratic involution.

In this section, we deal with the singular point $O_t$ on each quasi-smooth hypersurface in  the families No.~23, 40, 44, 61, 76,  and the singular point $O_z$ on each quasi-smooth hypersurface in  the families No.~20, 36. Also, the singular points of type $\frac{1}{2}(1,1,1)$ on each quasi-smooth hypersurface in  the family No.~7 are treated. These singular points on general hypersurfaces in such families are untwisted by birational involutions induced by the elliptic fibration models in \cite[4.10]{CPR}.  This section deals with these singular points on \emph{every}
quasi-smooth hypersurface in the families mentioned above with the more \emph{geometric} point of view.

Before we proceed, we divide  the family No.~7, quasi-smooth hypersurfaces $X_8$ of
degree $8$ in $\mathbb{P}(1,1,2,2,3)$, into two types.
\begin{proposition} Let $X_8$ be a quasi-smooth hypersurface of
degree $8$ in $\mathbb{P}(1,1,2,2,3)$. Then it  may be
assumed to be defined by an equation of one of the following forms
\begin{equation}\label{equation:typeII}
\aligned
&\mbox{Type \ I:} \\ &tw^2+wg_5(x,y,z)-zt^3-t^2g_4(x,y,z)-tg_6(x,y,z)+g_8(x,y,z)=0;\\%
&\mbox{Type II:} \\ & (z+f_2(x,y))w^2+wf_5(x,y,z,t)-zt^3-t^2f_4(x,y,z)-tf_6(x,y,z)+f_8(x,y,z)=0.\\%
\endaligned%
\end{equation}
 In the latter equation, the quasi-homogeneous
polynomial $f_5$ must contain either $xt^2$ or $yt^2$.
\end{proposition}
\begin{proof}
Let $F(x,y,z,t,w)$ be a quasi-homogeneous polynomial of degree $8$ that defines the hypersurface $X_8$.
The hypersurface $X_8$ has exactly four singular points
of type $\frac{1}{2}(1,1,1)$. They correspond to the four solutions to the equation $F(0,0,z,t,0)=0$ and they are located along the curve  $L_{zt}$. Let $p$ be
one of the singular points. By a coordinate change, we may assume
that $p$ is the point $O_t$.  Then the polynomial $F$ does not contain the monomial $t^4$. Therefore, we may write
\[\begin{split}F(x,y,z,t,w)&=w^2A_2(x,y,z,t)+w\left(2t^2B_1(x,y,z)+2tB_3(x,y,z)+B_5(x,y,z)\right)+\\ &+t^3B_2(x,y,z)-t^2B_4(x,y,z)-tB_6(x,y,z)+B_8(x,y,z),\\ \end{split}\]
where $A_i(x,y,z,t)$ is a quasi-homogeneous polynomial of degree $i$ in $x$, $y$, $z$, $t$ and
$B_j(x,y,z)$ is a quasi-homogeneous polynomial of degree $j$ in variables $x$, $y$, $z$.

Now we have two kinds of possibility for $A_2(x,y,z,t)$. The first possibility is that $A_2(x,y,z,t)$ contains the monomial $t$ (this is a general case). In this case, we may assume that $A_2(x,y,z,t)=t$ by the coordinate change $A_2(x,y,z,t)\mapsto t$. Note that
\[t\left(w^2+2wtB_1(x,y,z)+2wB_3(x,y,z)\right)\]\[=t\left(w+tB_1(x,y,z)+B_3(x,y,z)\right)^2-t\left(tB_1(x,y,z)+B_3(x,y,z)\right)^2.\]
By the coordinate change $w+tB_1(x,y,z)+B_3(x,y,z)\mapsto w$, we may assume that
\[\begin{split}F(x,y,z,t,w)&=tw^2+wB_5(x,y,z)+\\ &+t^3B_2(x,y,z)-t^2B_4(x,y,z)-tB_6(x,y,z)+B_8(x,y,z).\\ \end{split}\]

The second possibility is that $A_2(x,y,z,t)$ does not contain the monomial $t$ (this is a special case). In this case,  it must contain the monomial $z$ since  $X_8$ is quasi-smooth at $O_w$. We may then write $A_2(x,y,z,t)=z+f_2(x,y)$.

Since $X_8$ is quasi-smooth at $O_t$, $B_2$ must contain the monomial  $z$. Therefore, by the coordinate change $B_2(x,y,z)\mapsto -z$, we see that the quasi-homogeneous polynomial $F$ can be written in either Type~I or Type~II.

In the equation of Type~II, the quasi-homogeneous
polynomial $f_5$ must contain either $xt^2$ or $yt^2$. If not, then the hypersurface $X_8$ is  not quasi-smooth at the point $[0:0:0:1:1]$.
\end{proof}
The hypersurface  in the family No.~7 defined by the equation of Type II may have an
involution that untwists the singular point $O_t$. Since its
construction is quite complicated, we explain the method in a
separate section.

First, we consider the following six families and their singular
point $O_t$.
\begin{itemize}
\item  No.~\phantom{0}$7$ (Type~I), \  $X_{8} \ \subset \mathbb{P}(1,1,2,2,3)$;
\item  No.~$23$, \phantom{Type~I,} \ \  $X_{14}\subset \mathbb{P}(1,2,3,4,5)$;
\item
No.~$40$, \phantom{Type~I,} \ \ $X_{19}\subset \mathbb{P}(1,3,4,5,7)$;
\item No.~$44$, \phantom{Type~I,}
\ \  $X_{20}\subset \mathbb{P}(1,2,5,6,7)$;
\item No.~$61$, \phantom{Type~I,}\ \
$X_{25}\subset \mathbb{P}(1,4,5,7,9)$;
\item No.~$76$, \phantom{Type~I,}\ \
$X_{30}\subset \mathbb{P}(1,5,6,8,11)$.
\end{itemize}
For these six families, we may assume  that the hypersurface $X$
is defined by the equation
\begin{equation}\label{equation:defining equation-elliptic}
tw^2+wg_{d-a_4}(x,y,z)-x_it^3-t^2g_{d-2a_3}(x,y,z)-tg_{d-a_3}(x,y,z)+g_d(x,y,z)=0,
\end{equation}
where $x_i$ is either $y$ or $z$.

Put $y=\lambda_1 x^{a_1}$ and $z=\lambda_2 x^{a_2}$. We then consider the curve $C_{\lambda_1, \lambda_2}$ defined by
\begin{equation}\label{equation:irreducible-fiber}
\begin{split}
& tw^2+wg_{d-a_4}(x,\lambda_1 x^{a_1},\lambda_2 x^{a_2})-\lambda_ix^{a_i}t^3\\ &-t^2g_{d-2a_3}(x,\lambda_1 x^{a_1},\lambda_2 x^{a_2})-tg_{d-a_3}(x,\lambda_1 x^{a_1},\lambda_2 x^{a_2})+g_d(x,\lambda_1 x^{a_1},\lambda_2 x^{a_2})=0,\\
\end{split}
\end{equation}
where $i=1$ if $x_i=y$; $i=2$ if $x_i=z$,
in $\mathbb{P}(1, a_3, a_4)$.
From now
let $x_j$ be the variable such that $\{x_i, x_j\}=\{y,z\}$. If $x_i=y$, then put $a_i=a_1$ and $\lambda_i=\lambda_1$.
If  $x_i=z$, then put $a_i=a_2$ and $\lambda_i=\lambda_2$. Also we define $a_j$ and $\lambda_j$ in the same manner.

\begin{theorem}\label{theorem:geometric-elliptic} Let $X$ be a quasi-smooth hypersurface in the families No.~7~(Type I), 23, 40, 44, 61, 76. If the singular point $O_t$ is 
a center of non-canonical singularities of
the log pair $(X, \frac{1}{n}\mathcal{M})$, then
it is untwisted by a birational involution.
\end{theorem}

\begin{proof}
Let $\pi\colon X\dasharrow\mathbb{P}(1,a_1,a_2)$ be the rational map
induced by $$[x:y:z:t:w]\mapsto [x:y:z].$$ It is a morphism outside of the point $O_t$ and the point $O_w$. Moreover, the map is dominant. Its general fiber is an irreducible curve birational to an elliptic curve. To see this, on the hypersurface $X$, consider the surface cut by $y=\lambda_1 x^{a_1}$ and the surface cut by $z=\lambda_2 x^{a_2}$, where $\lambda_1$ and $\lambda_2$ are sufficiently general complex numbers. Then the intersection of these two surfaces is the curve  $C_{\lambda_1, \lambda_2}$ defined by \eqref{equation:irreducible-fiber}.
 From the equation we can easily see
that the curve $C_{\lambda_1, \lambda_2}$ is irreducible and reduced.
Furthermore, plugging $x=1$ into \eqref{equation:irreducible-fiber},
we see that the curve is birational to an elliptic curve. The
curve $C_{\lambda_1, \lambda_2}$ is a general fiber of the map $\pi$.

Let $\mathcal{H}$ be the linear subsystem of $|-a_2K_X|$ generated by the monomials of degree $a_2$ in the variables $x$, $y$, $z$. Its proper transform $\mathcal{H}_Y$ on $Y$ coincides with $|-a_2K_Y|$.

  Let $g\colon W\to Y$ be the weighted blow up at the point over
$O_w$ with weight $(1, a_1,a_2)$ and let $F$ be its exceptional
divisor. Let $\hat{E}$ be the proper transform of the exceptional divisor $E$ by the morphism $g$.
Let $\mathcal{H}_W$ be the proper transform of the linear system $\mathcal{H}$ by  the morphism $f\circ g$. We then see that $\mathcal{H}_W=|-a_2K_W|$.  We also see that $-K_W^3=0$.

We first claim that the divisor class $-K_W$ is nef. Indeed, the base curve of the linear system $|-a_2K_W|$ is given
by the proper transform of the curve $C$  cut by the equation $x=z=0$ on $X$. If the curve is irreducible then its proper transform $\hat{C}$ on $W$ intersects $-K_W$ trivially since $-K_W^3=0$. Suppose that the curve $C$ is reducible. It then consists of two irreducible components. Moreover, one of the components must be $L_{yw}$. Note that it passes through the point $O_w$.
Its proper transform $\hat{L}_{yw}$ on $W$ passes though the singular point of index $a_1$ on the exceptional divisor $F$. We then obtain
\[-K_W\cdot\hat{L}_{yw}=-K_X\cdot L_{yw}-\frac{1}{a_4}F\cdot \hat{L}_{yw}=\frac{1}{a_1a_4}-\frac{1}{a_4a_1}=0.\]
Since $-K_W\cdot\hat{C}=0$, the proper transform of the other component of $C$ intersects $-K_W$ trivially.
Therefore, the divisor class $-K_W$ is nef.

The linear system $|-mK_W|$ is free for sufficiently large
 $m$ by  Log Abundance (\cite{KMM94}). Hence, it induces an elliptic fibration $\eta\colon W\to
 \mathbb{P}(1,a_1,a_2)$.
Moreover, we have proved the existence of a commutative diagram
$$
\xymatrix{
Y\ar@{->}[d]_{f}&&W\ar@{->}[ll]_{g}\ar@{->}[d]^{\eta}\\%
X\ar@{-->}[rr]_{\pi}&&\mathbb{P}(1,a_1,a_2).\\}
$$ %

 We immediately  see from \eqref{equation:irreducible-fiber} that the divisor
$F$ is a section of the
 elliptic fibration
$\eta$ and the divisor $\hat{E}$ is a multi-section of the elliptic fibration $\eta$.
Therefore, by Proposition~\ref{proposition:elliptic involution}, we can construct 
a birational involution
$\tau_{W}\in\mathrm{Bir}(W)$  from   the
reflection of the generic fiber of $\eta$ with respect to the
section $F$. The involution $\tau_W$ is
biregular in codimension one because  $K_{W}$ is $\eta$-nef (\cite[Corollary~3.54]{KoMo98}). In particular,
$\tau_W$ acts on $\mathrm{Pic}(W)$.

Put $\tau_{Y}=g\circ\tau_W\circ g^{-1}$ and
$\tau=f\circ\tau_Y\circ f^{-1}$.

We have  $\tau_W(F)=F$  by our  construction. Therefore, $\tau_{Y}$ is also biregular
in codimension one.
In order to show that the point $O_t$ is untwisted by $\tau$, it is enough to verify $\tau_Y(E)\ne E$
by Remark~\ref{remark:untwisting-involution}.
For this verification, we suppose that $\tau_Y(E)=E$ and then we  look for a contradiction.

First, note that $\tau_Y(E)=E$ immediately implies $\tau_W(\hat{E})=\hat{E}$. It also implies that the involution $\tau$ is biregular in codimension one.
Furthermore, the involutions $\tau$, $\tau_Y$, $\tau_W$ induce the identity maps on the Picard groups of $X$, $Y$, $W$, respectively, since $\tau(-K_X)=-K_X$, $\tau_Y(E)=E$ and $\tau_W(F)=F$. Therefore, it follows from \cite[Proposition~2.7]{Co95} that they are all biregular.

Let $S_{\lambda_i}$ be the  surface on the hypersurface $X$
cut by the equation $x_i=\lambda_i x^{a_i}$ with a general complex number $\lambda_i$. It is a normal
surface (see Remark~\ref{remark:normal surface} below). However, it is not quasi-smooth possibly at the point
$O_t$ and the point $O_{x_j}$. The surface $S_{\lambda_i}$ is
$\tau$-invariant by our construction. Moreover, the projection
$\pi\colon X\dasharrow\mathbb{P}(1,a_1,a_2)$ induces a rational
map $\pi_{\lambda_i}\colon
S_{\lambda_i}\dasharrow\mathbb{P}(1,a_j)\cong\mathbb{P}^1$. The rational
map $\pi_{\lambda_i}\colon S_{\lambda_i}\dasharrow\mathbb{P}^1$ is
given by the pencil  of the curves on the surface
$S_{\lambda_i}\subset\mathbb{P}(1,a_j,a_3,a_4)$ cut by the
equations
\[\delta x^{a_j}=\epsilon x_j,\]
where $[\delta:\epsilon]\in\mathbb{P}^1$.  Its base locus is cut
out on $S_{\lambda_i}$ by $x=x_j=0$, which implies that the base
locus of the pencil consists of two points $O_t$ and
$O_w$. The map $\pi_{\lambda_i}$ is defined outside of the points
$O_w$ and $O_t$.

Denote by $\hat{S}_{\lambda_i}$ the proper transform of the
surface $S_{\lambda_i}$ by the birational morphism $f\circ g$. Then
$\hat{S}_{\lambda_i}$ is a normal surface that belongs to
$ |-a_iK_{W}|$. Moreover, the morphism
$f\circ g$ induces a birational morphism $\gamma\colon
\hat{S}_{\lambda_i}\to S_{\lambda_i}$. Furthermore, we have a
commutative diagram
$$
\xymatrix{
&\hat{S}_{\lambda_i} \ar@{->}[ld]_{\gamma}\ar@{->}[dr]^{\hat{\pi}_{\lambda_i}}\\%
S_{\lambda_i}\ar@{-->}[rr]_{\pi_{\lambda_i}} &&\mathbb{P}^1,}
$$ %
where $\hat{\pi}_{\lambda_i}$ is the morphism induced by the
elliptic fibration $\eta\colon W\to\mathbb{P}(1,a_1,a_2)$. In
particular, a general fiber of $\hat{\pi}_{\lambda_i}$ is a smooth
elliptic curve.

Let $\sigma\colon\bar{S}_{\lambda_i}\to\hat{S}_{\lambda_i}$
be the minimal resolution of singularities of the normal surface
$\hat{S}_{\lambda_i}$. Then  we have a commutative diagram
$$
\xymatrix{
\hat{S}_{\lambda_i} \ar@{->}[d]_{\gamma}\ar@{->}[drr]^{\hat{\pi}_{\lambda_i}}&&\bar{S}_{\lambda_i} \ar@{->}[ll]_{\sigma}\ar@{->}[d]^{\bar{\pi}_{\lambda_i}}\\%
S_{\lambda_i}\ar@{-->}[rr]_{\pi_{\lambda_i}} &&\mathbb{P}^1,}
$$ %
where $\bar{\pi}_{\lambda_i}=\hat{\pi}_{\lambda_i}\circ\sigma$.
Then $\bar{\pi}_{\lambda_i}$ is also an elliptic fibration.

The surface $\hat{S}_{\lambda_i}$ is $\tau_W$-invariant by our
construction. Let $\hat{\tau}_{\lambda_i}$ be the restriction of
the involution $\tau_W$ to the surface $\hat{S}_{\lambda_i}$. Then
it is a biregular involution of the surface $\hat{S}_{\lambda_i}$
since $\tau_W$ is biregular. Put
$\hat{E}_{\lambda_i}=\hat{E}\vert_{\hat{S}_{\lambda_i}}$ and
$F_{\lambda_i}=F\vert_{\hat{S}_{\lambda_i}}$. Then
$\hat{E}_{\lambda_i}$ and $F_{\lambda_i}$ are reduced
$\hat{\tau}_{\lambda_i}$-invariant curves. Moreover, the curve
$F_{\lambda_i}$ is irreducible and is a section of the elliptic
fibration $\hat{\pi}_{\lambda_i}$. The curve
$\hat{E}_{\lambda_i}$ is a multi-section of the elliptic fibration
$\hat{\pi}_{\lambda_i}$. 

Put
$\bar{\tau}_{\lambda_i}=\sigma^{-1}\circ\hat{\tau}_{\lambda_i}\circ\sigma$.
Then $\bar{\tau}_{\lambda_i}$ is biregular because $\hat{\tau}_{\lambda_i}$ is biregular and $\sigma$ is
the minimal resolution of singularities, i.e., $\bar{S}_{\lambda_i}$ is a minimal model
over $\hat{S}_{\lambda_i}$ (\cite[Corollary~3.54]{KoMo98}).
Let $\bar{E}_{\lambda_i}$ and $\bar{F}_{\lambda_i}$ be the proper transforms of $\hat{E}_{\lambda_i}$ and $\hat{F}_{\lambda_i}$ by the birational morphism
$\sigma$, respectively.
These are $(\gamma\circ\sigma)$-exceptional.  Denote the other
$(\gamma\circ\sigma)$-exceptional  curves (if any) by $G_1,\ldots, G_r$.
Again,
$\bar{F}_{\lambda_i}$ is a section of the elliptic
fibration $\bar{\pi}_{\lambda_i}$ and
$\bar{E}_{\lambda_i}$ is a multi-section of the elliptic fibration
$\bar{\pi}_{\lambda_i}$. 

Let $\bar{C}_{\lambda_i}$ be a general fiber of the map
$\bar{\pi}_{\lambda_i}$. Then $\bar{C}_{\lambda_i}$ is
$\bar{\tau}_{\lambda_i}$-invariant. Furthermore,
$\bar{\tau}_{\lambda_i}\vert_{\bar{C}_{\lambda_i}}$ is given by
the reflection with respect to the point
$\bar{F}_{\lambda_i}\cap\bar{C}_{\lambda_i}$. On the other hand,
the curve $\bar{E}_{\lambda_i}$ is
$\bar{\tau}_{\lambda_i}$-invariant. Then Lemma~\ref{lemma:Tschinkel} below implies that the divisor
$\bar{E}_{\lambda_i}-a_i\bar{F}_{\lambda_i}$ must be numerically
equivalent to a $\mathbb{Q}$-linear combination of curves on
$\bar{S}_{\lambda_i}$ that lie in the fibers of
$\bar{\pi}_{\lambda_i}$.

Let $C_x$ be the curve on $S_{\lambda_i}$ cut by the
equation $x=0$. It is defined by the equation
$$tw^2+wh_{d-a_4}(x_j)+h_d(x_j, t)=0$$ in $\mathbb{P}(a_j, a_3,
a_4)$. It can be reducible.  We write $C_x=\sum_{k=1}^{\ell} m_kC_k$,
where $C_k$'s are the irreducible components of $C_x$. Denote by
$\hat{C}_{k}$ be the proper transform of $C_{k}$ by $\gamma$. Put
$\hat{C}_x=\sum_{k=1}^{\ell} m_k\hat{C}_k$. Then all the curves
$\hat{C}_k$ lie in the same fiber of the elliptic fibration
$\hat{\pi}_{\lambda_i}$

Let $\bar{C}_{k}$ be the proper transform of $\hat{C}_{k}$
by $\sigma$. Then all the curves $\bar{C}_k$ lie in the same fiber
of the elliptic fibration $\bar{\pi}_{\lambda_i}$.  In addition, the fiber containing $\bar{C}_{k}$'s does not carry any other
non-$(\gamma\circ\sigma)$-exceptional curve.

We also claim
that every other fiber of $\bar{\pi}_{\lambda_i}$ contains
exactly one irreducible and reduced curve that is not
$(\gamma\circ\sigma)$-exceptional. For this claim, it is enough to show that for
a general complex number $\lambda_i$, the curve $C_{\lambda_1,
\lambda_2}$ is always irreducible and reduced for every value of
$\lambda_j$. 
Suppose that this is not true. Then,
for a general complex number $\lambda_i$ there is a complex number $\lambda_j$ such that the curve $C_{\lambda_1, \lambda_2}$ is reducible.
Therefore there is a
one-dimensional family of reducible curves $C_{\lambda_1, \lambda_2}$ with general $\lambda_i$ and some
$\lambda_j$ depending on $\lambda_i$. Denote the general curve in this
one-dimensional family by $C$.  
Since the defining equation \eqref{equation:irreducible-fiber}  contains $tw^2$, it cam splits into at most three irreducible components. Furthermore, one of them must be the curve $C_1$ defined by  either 
$$y-\lambda_1 x^{a_1}=z-\lambda_2 x^{a_2}=w+f_{a_4}(x,t)=0$$ or 
$$y-\lambda_1 x^{a_1}=z-\lambda_2 x^{a_2}=w^2+wg_{a_4}(x,t)+g_{2a_4}(x,t)=0$$
for some quasi-homogeneous polynomials $f_{a_4}(x,t)$,  $g_{a_4}(x,t)$ and $g_{2a_4}(x,t)$.
We then obtain
\[\begin{split} B\cdot\tilde{C}_1&=\left(A-\frac{1}{a_3}E\right)\cdot\tilde{C}_1\\
&=-K_X\cdot C_1-\frac{1}{a_3}E\cdot C_1\\
&=\frac{ka_1a_2a_4}{a_1a_2a_3a_4}-\frac{k}{a_4}=0\\
\end{split},\]
where $k=1$ for $w+f_{a_4}(x,t)=0$ and $k=2$ for $w^2+wg_{a_4}(x,t)+g_{2a_4}(x,t)=0$.
 By
Lemma~\ref{lemma:bad-link}, the point $O_t$ cannot be a center of
non-canonical singularities of the log pair $(X,
\frac{1}{n}\mathcal{M})$. Therefore, since $O_t$ is a center, every other fiber of $\bar{\pi}_{\lambda_i}$ contains
exactly one irreducible and reduced curve that is not
$(\gamma\circ\sigma)$-exceptional.

Since every
fiber of $\bar{\pi}_{\lambda_i}$ (with scheme structure) is
numerically equivalent   to each other and the divisor
$\bar{E}_{\lambda_i}-a_i\bar{F}_{\lambda_i}$ is numerically
equivalent to a $\mathbb{Q}$-linear combination of curves that lie in the
fibers of $\bar{\pi}_{\lambda_i}$, we obtain
$$
\bar{E}_{\lambda_i}-a_i\bar{F}_{\lambda_i}\sim_{\mathbb{Q}}  \sum_{k=1}^{\ell} \bar{c}_k \bar{C}_k+\sum_{k=1}^{r}g_k G_k%
$$
for some rational numbers $\bar{c}_{1},\ldots,  \bar{c}_{\ell}$,
$g_1,\ldots,g_r$.  On the other hand, the intersection form of the
curves $\bar{E}_{\lambda_i}$, $\bar{F}_{\lambda_i}$, $G_1, \ldots,
G_r$ is negative-definite since these curves are
$\gamma\circ\sigma$-exceptional. This implies
$$
0> \left(\bar{E}_{\lambda_i}-a_i\bar{F}_{\lambda_i} - \sum_{k=1}^{r}g_k G_k\right)^2= \left( \sum_{k=1}^{\ell} \bar{c}_k \bar{C}_k\right)^2.
$$
Therefore,
$\bar{c}_k\ne 0$ for some $k$. On the other hand, we have
$$
\sum_{k=1}^{\ell} \bar{c}_k C_k\sim_{\mathbb{Q}}  0
$$
on the
surface $S_{\lambda_i}$. In particular, the intersection form of
the curve(s) $C_k$'s is degenerate on the surface $S_{\lambda_i}$.
This however contradicts
Lemma~\ref{lemma:non-degenerate-curves} below.

The obtained contradiction shows that $\tau_Y(E)\ne E$. In particular, the involution $\tau$ is not biregular.
Since the involution $\tau_Y$ is biregular in codimension one, the
involution $\tau$ meets the conditions in
Definition~\ref{definition:untwisting}.
  Therefore, the birational involution
$\tau$ untwists the singular point $O_t$.
\end{proof}

\begin{remark}\label{remark:normal surface}
Each affine piece of the surface $S_{\lambda_i}$ is the quotient of a hypersurface  in $\mathbb{C}^3$ by a finite group action.
Since the surface $S_{\lambda_i}$ has only isolated singularities, so does the hypersurface. Therefore, the hypersurface is normal, and hence its quotient by a finite group action is  also normal. Consequently, the original surface $S_{\lambda_i}$ is normal. For the same reason, a hypersurface in a weighted projective space is normal if it is smooth in codimension $1$.
\end{remark}

The lemma below originates from Bogomolov and Tschinkel
(\cite{BoTsch}, \cite{BoTsch2}).
\begin{lemma}\label{lemma:Tschinkel}
Let $\Sigma$ be a smooth surface with
an elliptic fibration $\pi:\Sigma\to B$ over a smooth curve $B$. Let $N$ be a section of $\pi$ and $M$ be a multi-section of degree $m\geq 1$. Suppose that the surface $\Sigma$ has an involution $\tau$ satisfying the following:
\begin{itemize}
\item[(1)] a general fiber $E$ is $\tau$-invariant:
\item[(2)] $M$ is $\tau$-invariant:
\item[(3)] $\tau |_E$ is given by the reflection  on the elliptic curve $E$ with respect to the point $N\cap E$.
\end{itemize}
Then the divisor $M-mN$ is numerically equivalent to a $\mathbb{Q}$-linear combination of curves that lie in the fibres of $\pi$.
\end{lemma}
\begin{proof}
The divisor $(M-mN)\vert_{E}$ on the elliptic curve $E$ belongs to $\mathrm{Pic}^0(E)$. The conditions $(2)$ and $(3)$ imply that
$\tau\left((M-mN)\vert_{E}\right)=(M-mN)\vert_{E}$. On the other hand, the  condition $(3)$ shows $\tau\left((M-mN)\vert_{E}\right)=-(M-mN)\vert_{E}$.
Consequently, the divisor
$$
(M-mN)\vert_{E}\in\mathrm{Pic}^0(E).
$$
is $2$-torsion. Then \cite[Theorem~1.1]{Shioda}  verifies the statement.
\end{proof}
\begin{lemma}\label{lemma:non-degenerate-curves}
Let $S_{\lambda_i}$ be the surface on $X$ cut by the equation $x_i=\lambda_i x^{a_i}$ for a general complex number $\lambda_i$.
Let $C_x=\sum_{k=1}^\ell m_k C_k$ be the divisor on $S_{\lambda_i}$ cut by the equation $x=0$. Then the intersection form of the curves $C_k$'s on the surface $S_{\lambda_i}$ is non-degenerate.
\end{lemma}
\begin{proof}
Suppose that it is not a case.
This immediately implies that $\ell\geq 2$.  It cannot happen in the families No.~44, 61 and 76 since the polynomial $g_d$ must contain a power of $x_j$, i.e., the curve $C_x$ is irreducible.

The curve $C_x$ is defined by
\begin{itemize}
\item $tw^2 +ay^5w+y^4(bt^2+cy^2t+dy^4)=0$ in $\mathbb{P}(1,2,3)$  for the family No. 7 (Type I);
\item $tw^2+az^3w+bz^2t^2=0$ in $\mathbb{P}(3,4,5)$  for the family No. 23;
\item $tw^2+ay^4w+by^3t^2=0$ in $\mathbb{P}(3,5,7)$ for  the family No. 40.
\end{itemize}
The curve $C_x$ must consist of two irreducible components $C_1$ and $C_2$, i.e., $\ell=2$, except the case when
$a=b=d=0$ and $c\ne 0$ in the family No. 7 (Type I). This exceptional case will be considered separately at the end.

By our assumption, the intersection matrix of $C_1$ and $C_2$
on the surface  $S_{\lambda_i} $ is singular.

Suppose that  the curve $C_x$ is reduced. Then
\[\left(\begin{array}{cc}
       C_1^2&C_{1}\cdot C_{2}\\
      C_{1}\cdot C_{2}& C_{2}^2\\
\end{array}\right)= \left(\begin{array}{cc}
C_x\cdot C_1- C_{1}\cdot C_{2}& C_{1}\cdot C_{2}\\
    C_{1}\cdot C_{2}& C_x\cdot C_2-C_{1}\cdot C_{2}\\
        \end{array}\right),
\]
and hence we have
\[C_1\cdot C_2=\frac{(C_x\cdot C_1)(C_x\cdot C_2)}{C_x^2}= \frac{2}{a_jd} \ \ \ \mbox{ (resp. } \frac{a_3+a_4}{a_ja_3d}  \mbox{) } \]
 if $a=0$, $b\ne0$
 (resp. $a\ne0$, $b=0$ ).
 Note that the intersection numbers by the curve $C_x$ can be obtained easily because it is in $|\mathcal{O}_{S_{\lambda_i}}(1)|$.

Meanwhile, since the surface $S_{\lambda_i} $ is not quasi-smooth at the point $O_t$ and possibly at the point $O_{x_j}$, we have some difficulty to find the numbers $C_1\cdot C_2$ without assuming that the matrix is singular. In order to compute the intersection number $C_1\cdot C_2$ on the surface $S_{\lambda_i}$ directly, we consider the divisor $C_t$  (resp. $C_w$) cut by the equation $t=0$  (resp. $w=0$) on the surface $S_{\lambda_i}$ in case when $a=0$  (resp. $a\ne0$).

Consider the case when $a=0$, $b\ne0$. We may assume  that the curve $C_1$ is defined by the equation $x=t=0$ in $\mathbb{P}(1,a_j, a_3, a_4)$. Since the divisor $C_t$ contains the curve $C_1$, we can write $C_t=mC_1+R$, where $R$ is a curve whose support does not contain the curve $C_1$. From the intersection numbers
\[(C_1+C_2)\cdot C_1=C_x\cdot C_1=\frac{1}{a_ja_4}, \ \ \ (mC_1+R)\cdot C_1= C_t\cdot C_1=\frac{a_3}{a_ja_4}\]
we obtain
\[C_1\cdot C_2= \frac{1}{a_ja_4}-C_1^2=\frac{m-a_3}{ma_ja_4}+\frac{1}{m}R\cdot C_1\geq \frac{m-a_3}{ma_ja_4}+\frac{1}{m}(R\cdot C_1)_{O_w},\]
where $(R\cdot C_1)_{O_w}$ is the local intersection number of the curves $C_1$ and $R$  at the point $O_w$.
Note that the curves $C_1$ and $R$ always meet at the point $O_w$ at which the surface $S_{\lambda_i} $ is quasi-smooth. They may also intersect at the point $O_{x_j}$. However, we do not care about the intersection at the point $O_{x_j}$. The local intersection at the point $O_w$ will be enough for our purpose.

For the family No.~7 (Type I), we are considering the case when $a=d=0$ and $b\ne 0$.
In such a case, if $c\ne 0$, then the curves $C_1$ and $C_2$ intersect at a smooth point of $S_{\lambda_i}$ and hence $C_1\cdot C_2\geq 1$. If $c=0$, then the conditions imply that the defining equation of $X_8$ must contain either $xy^7$ or $zy^6$. Therefore, we can conclude that $m=1$ or $2$, depending on the existence of the monomials $xy^7$, $xy^4w$ in the defining equation of $X_8$, and that  the local intersection number  $(R\cdot C_1)_{O_w}$ is at least $\frac{4}{3}$.
For the family No.~23, we see that $m=2$ and $C_1\cdot R=\frac{3}{5}$.
For the family No.~40, we can easily  see that $m$ can be $1$, $3$, or $4$ , depending on the existence of the monomials $xy^6$ and $wy^3x^3$ in the defining equation, and that  the local intersection number  $(R\cdot C_1)_{O_w}$ is at least $\frac{3}{7}$.
In all the cases, we see $C_1\cdot C_2>\frac{2}{a_jd}$. It is a contradiction.

Consider the case when $a\ne 0$, $b=0$. We may assume  that the curve $C_1$ is defined by the equation $x=w=0$ in $\mathbb{P}(1,a_j, a_3, a_4)$. Since we have the monomial  of the form $x_j^sw$ in each defining equation, the surface $S_{\lambda_i}$ is quasi-smooth at the point $O_{x_j}$. Furthermore,  by changing the coordinate $w$ in suitable ways for the hypersurface $X$, we may assume that we have neither $xy^6$  nor $x^2y^4t$ for the family No.~40 and that we have neither $x^2z^4$  nor $xz^3t$ for the family No.~23 by changing the coordinate function $w$.
For the family No.~7 (Type I), we may assume that none of the monomials $xy^7$, $ty^6$, $xy^5t$  appear in the defining equation of $X_8$.

Since the divisor $C_w$ contains the curve $C_1$, we can write $C_w=mC_1+R$, where $R$ is a curve whose support does not contain the curve $C_1$. From the intersection numbers
\[(C_1+C_2)\cdot C_1=C_x\cdot C_1=\frac{1}{a_ja_3}, \ \ \ (mC_1+R)\cdot C_1= C_w\cdot C_1=\frac{a_4}{a_ja_3}\]
we obtain
\[C_1\cdot C_2= \frac{1}{a_ja_3}-C_1^2=\frac{m-a_4}{ma_ja_3}+\frac{1}{m}R\cdot C_1\geq \frac{m-a_4}{ma_ja_3}+\frac{1}{m}(R\cdot C_1)_{O_{x_j}},\]
where $(R\cdot C_1)_{O_{x_j}}$ is the local intersection number of the curves $C_1$ and $R$  at the point $O_{x_j}$.
Similarly as in the previous case, they may also intersect at the point $O_{t}$. We do not care about the intersection at the point $O_{t}$. As before, the local intersection at the point $O_{x_j}$ will be big enough.

For the family No.~7 (Type I), we have $b=c=d=0$ and $a\ne 0$. Note that the point $O_y$ is a smooth point of the surface $S_{\lambda_i}$. We see that $m$ can be $1$ or $2$ , depending on the existence of the monomial $xy^3t^2$ in the defining equation, and that  the local intersection number  $(R\cdot C_1)_{O_{y}}$ is at least $2$.
For the family No.~23, we see that $m=2$ and $C_1\cdot R=1$.
For the family No.~40, we see that $m$ can be $3$ or $4$ , depending on the existence of the monomial $x^3y^2t^2$ in the defining equation, and that  the local intersection number  $(R\cdot C_1)_{O_{y}}$ is at least $\frac{2}{3}$.
In all the cases, we see $C_1\cdot C_2>\frac{a_3+a_4}{a_ja_3d}$.  Again we have obtained a contradiction.

Suppose that the curve $C_x$ is not reduced. Then $C_x=C_1+2C_2$, where $C_1$ is defined by $x=t=0$ and $C_2$ is defined by $x=w=0$.
We then have
\[\left(\begin{array}{cc}
       C_1^2&C_{1}\cdot C_{2}\\
      C_{1}\cdot C_{2}& C_{2}^2\\
\end{array}\right)= \left(\begin{array}{cc}
C_x\cdot C_1- 2C_{1}\cdot C_{2}& C_{1}\cdot C_{2}\\
    C_{1}\cdot C_{2}& C_x\cdot C_2-\frac{1}{2}C_{1}\cdot C_{2}\\
        \end{array}\right),
\]
and hence we  have
\begin{equation*}C_1\cdot C_2=\frac{2(C_x\cdot C_1)(C_x\cdot C_2)}{C_x\cdot(C_1+4C_2)}=
\frac{2}{a_j(a_3+4a_4)}.\end{equation*}
In this case, the curves $C_1$ and $C_2$ intersect at the point $O_{x_j}$. 
The surface $S_{\lambda_i}$ is not quasi-smooth at the point  $O_{x_j}$, i.e., the defining equation of $X$ contains the monomial of the form $x_j^sx_i$.
If it is quasi-smooth there, then we obtain an absurd identity $C_1\cdot C_2=\frac{1}{a_j}$ from a direct computation. Note that we do not have the monomial of the form $x_j^sw$. Furthermore, we may assume that we  do not have  $xy^6$ (resp. $xy^7$) for the family No.~40 (resp. No.~7) by changing the coordinate function $z$.

Since the divisor $C_t$ contains the curve $C_1$, we can write $C_t=mC_1+R$, where $R$ is a curve whose support does not contain the curve $C_1$. From the intersection numbers
\[(C_1+2C_2)\cdot C_1=C_x\cdot C_1=\frac{1}{a_ja_4}, \ \ \ (mC_1+R)\cdot C_1= C_t\cdot C_1=\frac{a_3}{a_ja_4}\]
we obtain
\[C_1\cdot C_2= \frac{1}{2}\left(\frac{1}{a_ja_4}-C_1^2\right)=\frac{1}{2}\left(\frac{m-a_3}{ma_ja_4}+\frac{1}{m}R\cdot C_1\right)\geq \frac{1}{2}\left(\frac{m-a_3}{ma_ja_4}+\frac{1}{m}(R\cdot C_1)_{O_w}\right),\]
where $(R\cdot C_1)_{O_w}$ is the local intersection number of the curves $C_1$ and $R$  at the point $O_w$.

As in the first case, $m=1$ or $2$, depending on the existence of the monomials $xy^7$, $xy^4w$ in the defining equation of $X_8$, and $(R\cdot C_1)_{O_w}\geq \frac{4}{3}$ for the family No.~7 (Type~I).
We also obtain $m=2$ and $C_1\cdot R=\frac{3}{5}$ for the family No.~23.
For the family No.~40, we  obtain $m=3$ or $4$, depending on the existence of the monomial $wy^3x^3$ in the defining equation, and  $(R\cdot C_1)_{O_w}\geq \frac{3}{7}$.
In all the cases, we see $C_1\cdot C_2>\frac{2}{a_j(a_3+4a_4)}$. It is a contradiction again.

We now consider the exceptional case $a=b=d=0$ and $c\ne 0$ in the family No.~7 (Type~I).
The curve $C_x$ is defined by
 $$t(w-\alpha_1y^3)(w-\alpha_2y^3)=0$$
 in $\mathbb{P}(1,2,3)$.
It consists of three irreducible components $L$, $C_1$ and $C_2$. The curve $L$ is defined by $x=t=0$ in $\mathbb{P}(1,1,2,3)$ and the curve $C_k$ by  $$x=w-\alpha_ky^3=0$$ in $\mathbb{P}(1,1,2,3)$. The curves $L$ and $C_k$ intersect at the point defined
by $x=t=w-\alpha_ky^3=0$. At this point the surface $S_{\lambda_i}$ is smooth.
We then have
\[(L+C_1+C_2)\cdot L=\frac{1}{3}, \ \ \ (L+C_1+C_2)\cdot C_1=(L+C_1+C_2)\cdot C_2=\frac{1}{2}, \ \ \ L\cdot C_1=L\cdot C_2=1.\]
The intersection matrix of the
curves $L$,  $C_{1 }$ and $C_2$ on the surface $S_{\lambda_i}$
\[ \left(\begin{array}{ccc}
-\frac{5}{3}&1& 1\\
       1&-\frac{1}{2}-C_1\cdot C_2& C_1\cdot C_2\\
       1& C_1\cdot C_2&-\frac{1}{2}-C_1\cdot C_2\\
        \end{array}\right)
\]
is non-singular regardless of the value of $C_1\cdot C_2$. This completes the proof.
\end{proof}

Now, we consider the following two families and their singular
point $O_z$.
\begin{itemize}
 \item No.~$20$, \ \ $X_{13}\subset \mathbb{P}(1,1,3,4,5)$; \item
 No.~$36$, \ \
$X_{18}\subset \mathbb{P}(1,1,4,6,7)$.
\end{itemize}
Before we proceed, we put a remark here.  The proof of Theorem~\ref{theorem:geometric-elliptic} works verbatim to treat these two cases. Indeed, we are able to obtain elliptic fibrations right after taking weighted blow ups at the point $O_z$ and at the point $O_w$ with the corresponding weights.  We however follow another way that has evolved from \cite[Section~4.10]{CPR}, instead of applying the same method. This can enhance our understanding of the involutions described in this section with various points of view.

We have two types of hypersurfaces in the family No.~$20$.
One is the hypersurfaces whose defining equations contain the monomial $tz^3$ (Type I) and the other is the hypersurfaces not containing the monomial $tz^3$ (Type II).

 We first consider both $X_{13}$ of Type~I  in the family No.~20 and $X_{18}$ in the family No.~36 at the same time.
 Note that the defining equation of $X_{18}$
always contains the monomial $tz^3$.

We may then assume that these hypersurfaces $X$ are defined by the
equation
\begin{equation}\label{equation:20-36}zw^2+wf_{d-a_4}(x,y,t)-tz^3-z^2f_{d-2a_2}(x,y,t)-zf_{d-a_2}(x,y,t)+f_d(x,y,t)=0.\end{equation}
 We
can define an involution $\tau_{z}$ of $X$ as follows:
\begin{equation}\label{equation:CPR-inv2}\begin{split} & [x:y:z:t:w]\mapsto \\ &\ \ \ \ \ \left[x:y:
\frac{f_{d-a_4}^2(u+f_{d-a_2})-f_d^2}{f_{d-a_4}uw+f_{d-a_4}^2zt+f_du}:t:
\frac{-f_{d-a_4}u(u+f_{d-a_2})-f_{d}(uw+f_{d-a_4}zt)}{f_{d-a_4}uw+f_{d-a_4}^2zt+f_du}\right],\\
\end{split}\end{equation}
where $u=w^2-tz^2-zf_{d-2a_2}-f_{d-a_2}$.
Indeed, the involution is obtained by the following way.
We have a birational map $\phi$ from $X$ to a hypersurface $Z$ of degree $6a_4$ in $\mathbb{P}(1, a_1, 2a_4, a_3, 3a_4)$ defined by
$$[x:y:z:t:w]\mapsto [x:y:u:t: v],$$ where $v=uw+f_{d-a_4}zt+f_{d-a_4}f_{d-2a_2}$.
Note that we have

\[\left(\begin{array}{cc}
       f_{d-a_4} & u\\
      f_d & v\\
\end{array}\right)\left(\begin{array}{c}
       w\\
      z\\
\end{array}\right)= -\left(\begin{array}{c}
f_d\\
  f_{d-a_4}(u+f_{d-a_2})\\
        \end{array}\right).
\]
The hypersurface $Z$  is defined by the equation
\[v^2-f_{d-a_4}f_{d-2a_2}v=u^3+u^2f_{d-a_2}-(f_{d-2a_2}f_{d}+f_{d-a_4}^2t)u+(-f_{d-a_4}^2f_{d-a_2}+f_d^2)t.\] Therefore the hypersurface $Z$ has a biregular involution $\iota$ defined by $$[x:y:u:t:v]\mapsto [x:y:u:t:f_{d-a_4}f_{d-2a_2}-v].$$
The birational involution of $X$ is obtained by
\[\tau_z=\phi^{-1}\circ\iota\circ \phi.\]
To see that it is a
birational involution in detail, refer to \cite[Section~4.10]{CPR}.
However, it can be a biregular automorphism under a certain
condition. For example, if the polynomial $f_{d-a_4}$ is
identically zero, then the involution becomes  biregular. Indeed, it is the biregular
involution $$[x:y:z:t:w]\mapsto [x:y:z:t:-w].$$
Moreover, the converse is true.

\begin{lemma}\label{lemma:non-biregular involution}
The involution $\tau_z$ is biregular if and only if the polynomial
$f_{d-a_4}$ is identically zero.
\end{lemma}
\begin{proof}
 Suppose that $f_{d-a_4}$ is not a zero
polynomial. Consider the surface cut by the equation $u=0$. It is
easy to check that on this surface the involution becomes the map
\[[x:y:z:t:w]\mapsto \left[x:y:-z-\frac{f_{d-2a_2}}{t}:t:w \right].\]
Therefore, unless the polynomial $f_{d-2a_2}$ is either
identically zero or divisible by $t$, the involution $\tau_t$
cannot be biregular since it contracts the curve defined by $u=t=0$ to a point.

If the polynomial $f_{d-2a_2}$ is identically zero, then on the
surface cut by $z=0$, the involution becomes the map
\[[x:y:z:t:w]\mapsto \left[x:y:-\frac{2f_d}{u}:t:-w \right],\]
and hence the involution $\tau_z$ cannot be biregular. It contracts the curve defined by $u=z=0$ to a point.

Finally, suppose that the polynomial $f_{d-2a_2}(x,y,t)$ is divisible by $t$. In this case, we consider the surface cut
by the equation $t=0$. On this surface the involution $\tau_z$
becomes
\[[x:y:z:t:w]\mapsto \left[x:y:-z-\frac{2f_d}{u}:t:-w \right].\]
It  shows that the
involution $\tau_z$ cannot be biregular because it contracts the curve defined by $t=u=0$.
\end{proof}



\begin{theorem}\label{theorem:elliptic-biregular-untwist2}
 Let $X$ be a quasi-smooth hypersurface in the families No.~20~(Type I) and 36.
If the singular point $O_z$ is   a center of non-canonical singularities of
the log pair $(X, \frac{1}{n}\mathcal{M})$, then it is untwisted by the  birational involution $\tau_z$.
\end{theorem}
\begin{proof}
We suppose that $f_{d-a_4}$ is identically zero. Then the polynomial $f_d$ must be a non-zero irreducible polynomial
since
$X$ is quasi-smooth.

 Set
$$u=w^2-tz^2-zf_{d-2a_2}-f_{d-a_2}$$
and then let $T$ be the
proper transform of the surface given by the equation $u=0$. We can immediately check that
the surface $T$ belongs to the linear system $|2a_4B|$.

Choose a general point $[1: \mu_1:\mu_2]$ on the curve defined by
the equation $f_d=0$ in $\mathbb{P}(1,a_1,a_3)$. Then let
$C_{\mu_1, \mu_2}$ be the curve  defined by
the equations $$u=y-\mu_1x^{a_1}=t-\mu_2x^{a_3}=0$$ in $\mathbb{P}(1, a_1, a_2, a_3, a_4)$. This curve lies on the hypersurface $X$ by our construction.
If the curve is irreducible, then we
have
\[B\cdot \tilde{C}_{\mu_1, \mu_2}=(A-\frac{1}{a_2}E)\cdot \tilde{C}_{\mu_1,
\mu_2}=\frac{2}{a_2}-\frac{1}{a_2}E\cdot \tilde{C}_{\mu_1, \mu_2}= 0\]
since $E\cdot \tilde{C}_{\mu_1, \mu_2}=2$.
If $C_{\mu_1, \mu_2}$ is reducible, then it can have at most two irreducible components. Furthermore, each component $C_{\mu_1, \mu_2, i}$ is defined
by $$w-h(x,z)=y-\mu_1x^{a_1}=t-\mu_2x^{a_3}=0$$ in $\mathbb{P}(1, a_1, a_2, a_3, a_4)$ for some polynomial $h$.
This shows \[B\cdot \tilde{C}_{\mu_1, \mu_2, i}=\left(A-\frac{1}{a_2}E\right)\cdot \tilde{C}_{\mu_1,
\mu_2, i}=\frac{1}{a_2}-\frac{1}{a_2}E\cdot \tilde{C}_{\mu_1, \mu_2, i}= 0\]
since $E\cdot \tilde{C}_{\mu_1, \mu_2, i}= 1$. 

Since $O_z$ is a center, this is a contradiction by Lemma~\ref{lemma:bad-link}.
Therefore,  $f_{d-a_4}$ is not identically zero, and hence $\tau_z$ is a non-biregular involution by Lemma~\ref{lemma:non-biregular involution}. 

Note that $\tau_z$ leaves the point $O_w$ fixed. On the threefold  $W$ obtained by the weighted blow ups at $O_z$ and $O_w$ as in the proof of Theorem~\ref{theorem:geometric-elliptic}, the lift $\tau_W$ of the involution $\tau_z$ leaves the exceptional divisor over $O_w$ fixed. 
For the same reason as  in the proof of Theorem~\ref{theorem:geometric-elliptic}, the involution $\tau_W$ is biregular in codimension one,  so is the lift  $\tau_Y$ of the involution $\tau_z$ to $Y$. 

Consequently,   the involution $\tau_z$ untwists the singular point $O_z$.
\end{proof}

Now we consider  $X_{13}$ of Type~II, i.e., its defining equation does not
contain the monomial $tz^3$.
\begin{theorem}\label{theorem:elliptic-biregular3}
Let $X_{13}$ be a quasi-smooth hypersurface of degree $13$ in
$\mathbb{P}(1,1,3,4,5)$ in the family No.~20 (Type~II). Then
the singular point $O_z$ cannot be a center of non-canonical
singularities of the log pair $(X_{13}, \frac{1}{n}\mathcal{M})$.
\end{theorem}
\begin{proof} Since $X_{13}$ is of Type~II, we may assume that the hypersurface $X_{13}$ is
defined by the equation
\[zw^2+w(f_{8}(x,y,t)+at^2)-yz^4-z^3f_4(x,y)-z^2f_{7}(x,y,t)-zf_{10}(x,y,t)+f_{13}(x,y,t)=0,\]
where $a$ is a constant.
Note that the polynomial $f_{13}$ must contain the monomial $xt^3$; otherwise $X_{13}$ would not be quasi-smooth.

Let $\tilde{S}_y$ be the proper transform of the surface $S_y$.
Let $\mathcal{L}$ be the linear system on $X_{13}$ generated
by $x^5$, $xt$ and $w$.

First we consider the case where $a=0$. The base locus of the
linear system $|-K_{X_{13}}|$ consists of the curve cut by $x=y=0$. The
curve has two irreducible components. One is the curve $L_{zt}$
and the other is the curve $L_{tw}$.  We see that
\[S\cdot \tilde{S}_y=\tilde{L}_{tw}+2\tilde{L}_{zt}.\]
Note that the curve $L_{tw}$ does not pass through the
point $O_z$. We obtain
\[B\cdot \tilde{L}_{zt}=\frac{1}{2}B\cdot S\cdot \tilde{S}_y-\frac{1}{2}A\cdot \tilde{L}_{tw}=\frac{1}{2}A^3-\frac{4}{54}E^3-\frac{1}{40}=-\frac{1}{4}.\]

For the proper transform $\tilde{S}_{\lambda, \mu}$ of a general
member in $\mathcal{L}$, we have
\[\tilde{S}_y\cdot \tilde{S}_{\lambda,\mu}=\tilde{L}_{zt}+\tilde{R}_{\lambda, \mu},\]
where $\tilde{R}_{\lambda, \mu}$ is the residual curve and it sweeps the
surface $\tilde{S}_y$.  We then obtain
\[B\cdot \tilde{R}_{\lambda, \mu}=B\cdot \tilde{S}_y\cdot \tilde{S}_{\lambda,\mu}-B\cdot \tilde{L}_{zt}= 5A^3-\frac{8}{27}E^3+\frac{1}{4}=0.\]
It then follows from Lemma~\ref{lemma:bad-link}
that the singular point $O_z$ cannot be a center of non-canonical
singularities of the log pair $(X_{13}, \frac{1}{n}\mathcal{M})$.

Now we consider the case where $a\ne 0$. By a coordinate change we
may assume that $a=1$. The base locus of the linear system
$|-K_{X_{13}}|$ consists of the curve cut by $x=y=0$. The curve has two
irreducible components. One is $L_{zt}$ and the other is the curve
$L$ defined by $$x=y=zw+t^2=0.$$ The curves $\tilde{L}$ and
$\tilde{L}_{zt}$ intersect the exceptional divisor $E$ at a smooth
point.  We have $S\cdot \tilde{S}_y=\tilde{L}_{zt}+\tilde{L}$ and
$$
B\cdot \tilde{L}_{zt}=A\cdot \tilde{L}_{zt}-\frac{1}{3}E\cdot
\tilde{L}_{zt}=-\frac{1}{4}, \ \ \  B\cdot
\tilde{L}=A\cdot\tilde{L}-\frac{1}{3}E\cdot
\tilde{L}=\frac{2}{15}-\frac{1}{3}=-\frac{1}{5}.
$$

For the proper transform $\tilde{S}_{\lambda, \mu}$ of a general
member in $\mathcal{L}$, we have
$$
\tilde{S}_y\cdot
\tilde{S}_{\lambda,\mu}=\tilde{L}_{zt}+\tilde{R}_{\lambda, \mu},
$$
where $\tilde{R}_{\lambda, \mu}$ is the residual curve and it sweeps the
surface $\tilde{S}_y$. Note that the curve $\tilde{R}_{\lambda, \mu}$ does
not contain the curve $\tilde{L}_{zt}$ since the defining
polynomial of $X_{13}$ contains either $xt^3$ or $wt^2$.
Therefore,
\[B\cdot \tilde{R}_{\lambda, \mu}=B\cdot \tilde{S}_y\cdot \tilde{S}_{\lambda,\mu}-B\cdot\tilde{L}_{zt}= 5A^3-\frac{8}{27}E^3+\frac{1}{4}=0.\]
Then the statement immediately follows from Lemma~\ref{lemma:bad-link}.
\end{proof}

\begin{remark}
Note that Theorem~\ref{theorem:geometric-elliptic} can be proved in the same way that we apply to
Theorems~\ref{theorem:elliptic-biregular-untwist2}.
The involution of $X$ for the singular point $O_t$  is defined as follows:
\begin{equation*}\label{equation:CPR-inv1}\begin{split} &[x:y:z:t:w]\mapsto \\ &\ \ \ \ \ \ \left[x:y:z:
\frac{g_{d-a_4}^2(v+g_{d-a_3})-g_d^2}{g_{d-a_4}vw+g_{d-a_4}^2x_it+g_dv}:
\frac{-g_{d-a_4}v(v+g_{d-a_3})-g_{d}(vw+g_{d-a_4}x_it)}{g_{d-a_4}vw+g_{d-a_4}^2x_it+g_dv}\right],
\\ \end{split}\end{equation*}
where
$v=w^2-x_it^2-tg_{d-2a_3}-g_{d-a_3}$. This birational involution is also extracted from \cite[Section~4.10]{CPR}.
We are immediately able to check that it is biregular if and only if the polynomial $g_{d-a_4}$ is identically zero.
\end{remark}

\subsection{Invisible  elliptic involution}\label{section:invisible involution}

In this section we consider the singular point $O_z$ on the hypersurfaces of a special type in  the family No.~23
and the singular points of type $\frac{1}{2}(1,1,1)$ on the hypersurfaces of Type~II in the family No.~7.
The method we use here is almost the same as the one for Theorem~\ref{theorem:geometric-elliptic}.
In the proof of Theorem~\ref{theorem:geometric-elliptic}, only with the weighted blow ups at the point $O_t$ (or $O_z$) and the point $O_w$ we can obtain an elliptic fibration with a section. However, in this special cases of the families No.~7 and ~23,  after these two weighted blow-ups, our elliptic fibrations  still remain invisible.   When we reach a threefold $W$ with $-K_W^3=0$, instead of elliptic fibrations, we see several curves that intersect $-K_W$ negatively.  Eventually, log-flips along these curves reveal elliptic fibrations with sections.

We first consider  the singular point $O_z$ on the hypersurface of the special type in  the family No.~23.  In general, every quasi-smooth  hypersurface  of degree $14$ in
$\mathbb{P}(1,2,3,4,5)$  can be defined by the equation
 \[(t+by^2)w^2+
y(t-\alpha_1y^2)(t-\alpha_2y^2)(t-\alpha_3y^2)+
z^3(a_1w+a_2yz)+cz^2t^2+\]\[+wf_9(x,y,z,t)+f_{14}(x,y,z,t)=0\] for
suitable constants $b$, $\alpha_1$, $\alpha_2$, $\alpha_3$, and
suitable polynomials $f_{9}(x,y,z,t)$,  $f_{14}(x,y,z,t)$. Here,
we will deal with the singular point $O_z$ on this hypersurface.
However, in the cases when at least one of the constants $c$,
$a_1$ is non-zero, the singular point $O_z$ can be easily excluded
(see the table for the family No.~23 in
Section~\ref{section:super-rigid}). For this reason, we consider
only the case when $a_1=c=0$. In this case, the defining equation
must possess the monomial $xtz^3$. If not, then the hypersurface
is not quasi-smooth at the point defined by $x=y=w=t^3+a_2z^4=0$.
Consequently, it is the singular points $O_z$ on the hypersurface
$X_{14}$ of degree $14$ in $\mathbb{P}(1,2,3,4,5)$ defined by the
equation
\[(t+by^2)w^2+y(t-\alpha_1y^2)(t-\alpha_2y^2)(t-\alpha_3y^2) + z^4y+xtz^3+wf_9(x,y,z,t)+f_{14}(x,y,z,t)=0,\]
where $f_9$ does not contain $z^3$ and $f_{14}$ does not contain $z^2t^2$, that we should deal with here.
By replacing $t-\alpha_3y^2$ by $t$, we may assume that $X_{14}$ has a singular point at $O_y$ without loss of generality. Note that by a suitable coordinate change with respect to $t$, we may assume that neither $x^3w^2$ nor $xyw^2$ appears in the defining equation. However we cannot change the coefficient term $(t+by^2)$ of $w^2$ into $t$ by a coordinate change since we have already assumed that $O_y$ is a singular point.

\begin{theorem}\label{theorem:amazing-23}
Suppose that the hypersurface $X_{14}$  of degree $14$ in
$\mathbb{P}(1,2,3,4,5)$ is defined by the equation
\[(t+by^2)w^2+yt(t-\alpha_1y^2)(t-\alpha_2y^2) + z^4y+xtz^3+wf_9(x,y,z,t)+f_{14}(x,y,z,t)=0\]
as explained just before. If the
singular point $O_z$ is a center of non-canonical singularities of
the log pair $\left(X_{14}, \frac{1}{n}\mathcal{M}\right)$, then there is a birational
involution that untwists  $O_z$.
\end{theorem}
\begin{proof}
Let $\mathcal{H}$ be the linear
subsystem of $|-5K_{X_{14}}|$ generated by $x^5$, $xy^2$, $x^3y$ and
$yz+xt$.  Note that the polynomial $yz+xt$ vanishes at the point $O_z$ with
multiplicity $\frac{5}{3}$ (see Remark~\ref{remark:plurianticanonical-23} below).
Let $\pi : X_{14}\dasharrow \mathbb{P}(1,2,5)$ be the rational
map induced by $$[x:y:z:t:w]\mapsto[x:y:yz+xt].$$ Then $\pi$ is a
morphism outside of the curves $L_{zt}$ and $L_{zw}$. Moreover,
the map $\pi$ is dominant, which implies, in particular, that
$\mathcal{H}$ is not composed from a pencil. Furthermore, its
general fiber is an irreducible curve that is birational to an
elliptic curve. To see this, we put $y=\lambda x^2$ and $yz+xt=\mu
x^5$ with sufficiently general complex numbers $\lambda$ and
$\mu$. On the hypersurface $X_{14}$, we take the intersection of the
surface defined by $y=\lambda x^2$ and the surface defined by
$yz+xt=\mu x^5$. This intersection is the same as the intersection
of the surface defined by $y=\lambda x^2$ and the reducible
surface defined by $ x(\lambda xz+t-\mu x^4)=0$. Therefore, the
intersection is the $1$-cycle
\[(L_{zw}+2L_{zt})+(L_{zw}+C_{\lambda, \mu})=2L_{zw}+2L_{zt}+C_{\lambda, \mu},\]
where the curve $C_{\lambda, \mu}$ is defined by the equation
\begin{equation}\label{equation:amazing-23}
\begin{split}
&(\mu x^3-\lambda z+b\lambda^2x^3)w^2+\lambda x^4(\mu x^3-\lambda z)(\mu x^3-\alpha_1\lambda^2 x^3-\lambda z)(\mu x^3-\alpha_2 \lambda^2 x^3-\lambda z) + \mu x^4z^3+\\ & + \frac{wf_9(x,\lambda x^2,z,\mu x^4-\lambda xz)+f_{14}(x,\lambda x^2,z,\mu x^4-\lambda xz)}{x}=0\\
\end{split}
\end{equation}
in $\mathbb{P}(1,3,5)$. The curve $C_{\lambda, \mu}$ is a general
fiber of the map $\pi$. Setting $x=1$ in
\eqref{equation:amazing-23}, we consider the curve
defined by
\begin{equation}\label{equation:amazing-23-irr}
\begin{split}
&(\mu+b\lambda^2 -\lambda z)w^2+\lambda (\mu -\lambda z)(\mu -\alpha_1\lambda^2 -\lambda z)(\mu -\alpha_2 \lambda^2 -\lambda z) +\\
& +\mu z^3+wf_9(1,\lambda ,z,\mu -\lambda z)+f_{14}(1,\lambda ,z,\mu -\lambda z)=0\\
\end{split}
\end{equation}
in $\mathbb{C}^2$. It is a smooth affine plane cubic curve.
Moreover, for a general complex number $\lambda$, the
curve~ \eqref{equation:amazing-23-irr} is always irreducible and
reduced for every value of $\mu$ (see Lemma~\ref{lemma:irreducible-23} below).

Let $\mathcal{H}_Y$ be the proper transform of the linear system
$\mathcal{H}$ by the weighted blow up $f$. It is the linear system
$|-5K_Y|$ because the linear system $\mathcal{H}$ consists exactly of the members of $|-5K_X|$ with multiplicity at least $\frac{5}{3}$ at $O_z$  (see Remark~\ref{remark:plurianticanonical-23} below). Let $g\colon W\to Y$ be the weighted blow up at the
point over $O_w$ with weight $(1,2,3)$ and $\mathcal{H}_W$ the
proper transform of $\mathcal{H}_Y$ by the morphism $g$. Let
$\hat{E}$ be the proper transform of $E$ by the weighted blow up
$g$ and $G$ be the exceptional divisor of $g$.

The linear system $\mathcal{H}_W$ coincides with the linear system
$|-5K_W|$ since every member in $|-5K_Y|$ has multiplicity at least $1$ at the point corresponding to $O_w$ (see Remark~\ref{remark:plurianticanonical-23} below).

The base locus of the linear system $\mathcal{H}$ is given by the
equation $x=yz+xt=0$. Therefore, it consists of $L_{zw}$, $L_{zt}$
and the curve $C$ cut by the equation $x=z=0$. The curve $C$ may
not be irreducible. Indeed, $C$ is irreducible if and only if
$b\ne 0$. If $b=0$, then $C$ consists of two irreducible curves
$L_{yw}$ and $R$, where $R$ is an irreducible curve  passing
through neither the point $O_z$ nor the point $O_w$. Let
$\hat{L}_{zw}$, $\hat{L}_{zt}$, $\hat{L}_{yw}$,  $\hat{R}$ and
$\hat{C}$ be the proper transforms of the curves $L_{zw}$,
$L_{zt}$, $L_{yw}$, $R$  and $C$, respectively, by the morphism
$f\circ g$. We have
\[\begin{split} &-K_W\cdot \hat{L}_{zw}=-K_{X_{14}}\cdot L_{zw}-\frac{1}{3}\hat{E}\cdot \hat{L}_{zw} -\frac{1}{5}G\cdot  \hat{L}_{zw}=-\frac{1}{6};\\
& -K_W\cdot \hat{L}_{zt}=-K_{X_{14}}\cdot L_{zt}-\frac{1}{3}\hat{E}\cdot \hat{L}_{zt} =-\frac{1}{4};\\
&-K_W\cdot \hat{L}_{yw}=-K_{X_{14}}\cdot L_{yw}-\frac{1}{5}G\cdot  \hat{L}_{yw}=\frac{1}{30};\\
& -K_W\cdot \hat{R}=-K_{X_{14}}\cdot R>0; \ \ \ -K_W\cdot \hat{C}=-K_{X_{14}}\cdot C>0. \\ \end{split}\]
Therefore, the curves $\hat{L}_{zw}$ and $\hat{L}_{zt}$ are the
only curves that intersect $-K_W$ negatively. The log pair $\left(W,
\frac{1}{5}\mathcal{H}_W\right)$ is canonical, and hence the log pair
$\left(W, \left(\frac{1}{5}+\epsilon\right)\mathcal{H}_W\right)$ is Kawamata log
terminal for sufficiently small $\epsilon>0$. 
Since $$K_W+\left(\frac{1}{5}+\epsilon\right)\mathcal{H}_W\sim_{\mathbb{Q}}
-\epsilon K_W,$$ the curves
$\hat{L}_{zw}$ and $\hat{L}_{zt}$ are the only curves that
intersect
$K_W+\left(\frac{1}{5}+\epsilon\right)\mathcal{H}_W$ negatively. Therefore, there is
 a log flip  $\chi : W\dasharrow U$ along the curves $\hat{L}_{zw}$ and $\hat{L}_{zt}$ (\cite{Sho93}).
 Let $\check{E}$ and $\check{G}$ be the proper transforms  of the divisors $\hat{E}$ and $G$, respectively, by $\chi$. The anticanonical divisor $K_U+\left(\frac{1}{5}+\epsilon\right)\mathcal{H}_U$ is nef, where $\mathcal{H}_U$ is the proper transform  of $\mathcal{H}_W$ by the birational map  $\chi$  that is an isomorphism in codimension one.

 By Log Abundance (\cite{KMM94}), the linear system $|-mK_U|$ is free for sufficiently large
 $m$. Hence, it induces a dominant morphism $\eta\colon U\to
 \Sigma$ with connected fibers, where $\Sigma$ is a normal variety. We claim that
 $\Sigma$ is a surface and $\eta$ is an elliptic fibration.
For this claim, let $\hat{C}_{\lambda, \mu}$ be the proper
transform of a general fiber $C_{\lambda, \mu}$ of the map $\pi$
on the threefold $W$ and let $\check{C}_{\lambda, \mu}$ be its
proper transform on $U$. Then
$$-K_{W}\cdot \hat{C}_{\lambda, \mu}=-10K_W^3-2(-K_W)\cdot(\hat{L}_{zw}+\hat{L}_{zt})=0.$$
In particular, the curve $\hat{C}_{\lambda, \mu}$  is disjoint
from the curves $\hat{L}_{zt}$ and $\hat{L}_{zw}$ because the base
locus of the linear system $|-5K_{W}|$ contains  the curves
$\hat{L}_{zt}$ and $\hat{L}_{zw}$. Therefore,
$$
-K_{U}\cdot \check{C}_{\lambda, \mu}=0.
$$
It implies that $\eta$ contracts $\check{C}_{\lambda, \mu}$. Since
we already proved that $C_{\lambda, \mu}$  is birational to an
elliptic curve and $\mathcal{H}$ is not composed from a pencil, we
can see that $\eta$ is an elliptic fibration. Moreover, we have
proved the existence of a commutative diagram
$$
\xymatrix{
&W\ar@{->}[ld]_{g}\ar@{-->}[r]^{\chi}&U\ar@{->}[dddr]^{\eta}\\%
Y\ar@{->}[d]_{f}&&\\%
X_{14}\ar@{-->}[rd]_{\pi}&&\\
&\mathbb{P}(1,2,5)&&\Sigma\ar@{-->}[ll]^{\ \ \ \ \ \ \theta}}
$$ %
where $\theta$ is a birational map.

We see from \eqref{equation:amazing-23}  that the divisor
$\check{G}$ is a section of the elliptic fibration $\eta$ and
$\check{E}$ is a $2$-section of  $\eta$. Let $\tau_{U}$ be the
birational involution of the threefold $U$  obtained from the elliptic fibration $\eta:U\to\Sigma$ with the section $\check{G}$ 
by Proposition~\ref{proposition:elliptic involution}.
Then $\tau_{U}$ is biregular in codimension
one because $K_{U}$ is $\eta$-nef by our construction (\cite[Corollary~3.54]{KoMo98}).

Put
$\tau_{W}=\chi^{-1}\circ\tau_U\circ\chi$,
$\tau_{Y}=g\circ\tau_W\circ g^{-1}$ and $\tau=f\circ\tau_Y\circ
f^{-1}$. 

Since $\tau_{U}$ and $\chi$ are  biregular in codimension
one, so is the involution $\tau_{W}$. Moreover, we have $\tau_{W}(G)=G$
since $\tau_{U}(\check{G})=\check{G}$ by our construction. This
implies that $\tau_{Y}$ is also biregular in codimension one.

In order to see that the point $O_z$ is untwisted by $\tau$, we have only to show that the involution $\tau$ is not biregular.
To prove this, we suppose that $\tau$ is biregular and then look for a contradiction.
Note that the proof of Lemma~\ref{lemma:untwisting-invoultion} shows that $\tau_Y(E)=E$ if $\tau$ is biregular.
This is a key point from which we are able to derive a contradiction.

Let $S_\lambda$ be the  surface on the hypersurface $X_{14}$ cut by the
equation $y=\lambda x^2$ with a general complex number $\lambda$. It follows from the defining equation of
the surface $S_\lambda$ that the surface has only isolated singularities. Therefore, it is normal (see Remark~\ref{remark:normal surface}). Moreover, the
surface $S_{\lambda}$ is $\tau$-invariant by our construction. Let
$\tau_\lambda$ be the restriction of $\tau$ to the surface
$S_\lambda$. It is a birational involution of the surface
$S_{\lambda}$ since the surface is $\tau$-invariant.

We have a rational map $\pi_\lambda\colon S_\lambda\dasharrow
\mathbb{P}(1,5)\cong\mathbb{P}^1$ induced by the rational map
$\pi\colon X_{14}\dasharrow \mathbb{P}(1,2,5)$.
Note that the curves $L_{zt}$ and $L_{zw}$ are contained in
$S_{\lambda}$. The rational map $\pi_\lambda\colon
S_\lambda\dasharrow\mathbb{P}^1$ is given by the pencil
$\mathcal{P}$ of the curves on the surface
$S_\lambda\subset\mathbb{P}(1,3,4,5)$ cut by the equations
\[\delta x^4=\epsilon (\lambda xz+t),\]
where $[\delta:\epsilon]\in\mathbb{P}^1$.  Its base locus is cut
out on $S_{\lambda}$ by $x=t=0$, which implies that the base locus
of the pencil $\mathcal{P}$ is the curve $L_{zw}$.

The map $\pi_\lambda$ is not defined only at the points $O_w$ and
$O_z$. To see this, plug in $t=\frac{\delta}{\epsilon}x^4-\lambda xz$
into the defining equation of the surface $S_{\lambda}$ (with
general $[\delta:\epsilon]\in\mathbb{P}^1$), divide the resulting
equation by $x$ (removing the base curve $L_{zw}$), and put $x=0$
into the resulting equation in $x$, $z$, and $w$ (we know that the
base locus of $\mathcal{P}$ is $L_{zw}$). This gives the system of
equations $zw^2=x=t=0$, which means that the map $\pi_\lambda$ is
not defined only at the points $O_w$ and $O_z$.

Let $C_{\lambda}$ be a general fiber of the map $\pi_\lambda$.
Then $C_{\lambda}$ is given by \eqref{equation:amazing-23}
 with a general complex number  $\mu$. As
shown in the beginning, the fiber $C_{\lambda}$ is an irreducible
curve  birational to a smooth elliptic curve. Let
$\nu\colon\breve{C}_{\lambda}\to C_{\lambda}$ be the normalization
of the curve $C_{\lambda}$. It follows from
\eqref{equation:amazing-23} that $\nu^{-1}(O_{w})$ consists of
a single point and  $\nu^{-1}(O_{z})$ consists of two distinct
points. Note that we can consider the curves $C_{\lambda}$ and
$\breve{C}_{\lambda}$ (and the map $\nu)$ to be defined over the
function field $\mathbb{C}(\mu)$. In this case, $\nu^{-1}(O_{z})$
consists of a single point of degree $2$, i.e., a point splitting
into two points over the algebraic closure of $\mathbb{C}(\mu)$.

Let $\hat{S}_\lambda$ be the proper transform of $S_\lambda$ via $f\circ g$. Put $\hat{E}_\lambda=\hat{E}|_{\hat{S}_\lambda}$
and $\hat{G}_\lambda=G|_{\hat{S}_\lambda}$.
Resolving the indeterminacy of the rational map $\pi_\lambda$ through $\hat{S}_\lambda$, we
obtain an elliptic fibration
$\bar{\pi}_\lambda\colon\bar{S}_\lambda\to \mathbb{P}^1$. Thus, we
have a commutative diagram
$$
\xymatrix{
&\bar{S}_\lambda \ar@{->}[ld]_{\sigma}\ar@{->}[dr]^{\bar{\pi}_\lambda}\\%
S_\lambda\ar@{-->}[rr]_{\pi_\lambda} &&\mathbb{P}^1,}
$$ %
where $\sigma$ is a birational map. Note that there exist exactly
two $\sigma$-exceptional prime divisors that do not lie in the
fibers of $\bar{\pi}_\lambda$. One is the proper transform of $\hat{E}_\lambda$ and the other is
the proper transform of $\hat{G}_\lambda$.  Let
$\bar{E}_\lambda$ and $\bar{G}_\lambda$ be these two exceptional
divisors, respectively. Then $\bar{G}_\lambda$ is a section of
$\bar{\pi}_\lambda$ and $\bar{E}_\lambda$  is a $2$-section of
$\bar{\pi}_\lambda$. Denote the other $\sigma$-exceptional curves (if
any) by $F_{1},\ldots,F_{r}$.

Put $\bar{\tau}_\lambda=\sigma^{-1}\circ\tau_\lambda\circ\sigma$.
Due to \cite[Theorem~3.2]{dFE}, we may assume that $\bar{\tau}_\lambda$ is
biregular and $\bar{S}_\lambda$ is smooth.

Let $\bar{C}_{\lambda}$ be the proper transform
of the curve $C_{\lambda}$ on $\bar{S}_{\lambda}$. Then
$\bar{C}_{\lambda}\cong\breve{C}_{\lambda}$, since
$\bar{C}_{\lambda}$ is smooth. Moreover, the curve
$\bar{C}_{\lambda}$ is $\bar{\tau}_{\lambda}$-invariant.
Furthermore, $\bar{\tau}_{\lambda}\vert_{\bar{C}_{\lambda}}$ is
given by the reflection with respect to the point
$\bar{G}_\lambda\cap\bar{C}_{\lambda}$. On the other hand, the
divisor $\bar{E}_{\lambda}$ must be
$\bar{\tau}_{\lambda}$-invariant since $\tau_Y(E)=E$.
Therefore, the divisor $\bar{E}_{\lambda}-2\bar{G}_{\lambda}$ must
be numerically equivalent to a $\mathbb{Q}$-linear combination of curves on
$\bar{S}_{\lambda}$ that lie in the fibers of
$\bar{\pi}_{\lambda}$ by Lemma~\ref{lemma:Tschinkel}.

Let $\bar{L}_{zt}$ and $\bar{L}_{zw}$ be the proper transforms of
the curves $L_{zt}$ and $L_{zw}$ by $\sigma$, respectively. Then
$\bar{L}_{zt}$ and $\bar{L}_{zw}$ lies in the same fiber of the
elliptic fibration $\bar{\pi}_{\lambda}$.
In the fiber containing $\bar{L}_{zt}$ and $\bar{L}_{zw}$, the other components are, if any,
$\sigma$-exceptional since the fiber of $\pi_\lambda$ over the point $[0:1]$ consists only of $L_{zt}$ and $L_{zw}$. 
In addition, we see that
every other fiber of $\bar{\pi}_{\lambda}$ contains exactly one
irreducible reduced curve that is not $\sigma$-exceptional.
Indeed, this immediately follows from Lemma~\ref{lemma:irreducible-23} below. Since all fibers of
$\bar{\pi}_{\lambda}$ (with scheme structure) are numerically
equivalent  and the divisor $\bar{E}_{\lambda}-2\bar{G}_{\lambda}$
is numerically equivalent to a $\mathbb{Q}$-linear combination of curves that
lie in the fibers of $\bar{\pi}_{\lambda}$, we obtain
$$
\bar{E}_{\lambda}-2\bar{G}_{\lambda}\sim_{\mathbb{Q}}  c_{zt}\bar{L}_{zt}+c_{zw}\bar{L}_{zw}+\sum_{i=1}^{r}c_i F_i%
$$
for some rational numbers $c_{zt}$, $c_{zw}$, $c_1,\ldots,c _r$.
The intersection form of the curves $\bar{E}_{\lambda}$,
$\bar{G}_{\lambda}$, $F_1, \ldots, F_r$ is negative-definite since
these curves are $\sigma$-exceptional. Therefore,
$(c_{zt},c_{zw})\ne (0,0)$. On the other hand, we have
$$
0\sim_{\mathbb{Q}}  c_{zt}L_{zt}+c_{zw}L_{zw}
$$
on the surface $S_{\lambda}$. In particular, the intersection form
of the curves $L_{zw}$ and $L_{zt}$ is degenerate on the surface
$S_{\lambda}$.

Meanwhile, from the intersection numbers
\[(2L_{zt}+L_{zw})\cdot L_{zw}=\frac{1}{15}, \ \ \ (2L_{zt}+L_{zw})\cdot L_{zt}=\frac{1}{12}\]
on the surface $S_\lambda $, we obtain
\[\left(\begin{array}{cc}
       L_{zw}^2&L_{zw}\cdot L_{zt}\\
      L_{zw}\cdot L_{zt}& L_{zt}^2\\
\end{array}\right)= \left(\begin{array}{cc}
\frac{1}{15}-2L_{zw}\cdot L_{zt} & L_{zw}\cdot L_{zt}\\
    L_{zw}\cdot L_{zt} & \frac{1}{24}-\frac{1}{2}L_{zw}\cdot L_{zt}\\
        \end{array}\right).
\]
The curves $L_{zw}$ and $L_{zt}$ intersect only at the point $O_z$. However, the surface $S_\lambda$ is not quasi-smooth at the point $O_z$.
To get  the intersection number $L_{zw}\cdot L_{zt}$, we consider the divisor $D_t$ on the surface $S_{\lambda}$ cut by the equation $t=0$. We can immediately see that $D_t=2L_{zw}+R$, where $R$ is the residual  curve. The curves $L_{zw}$ and $R$ intersect only at the point $O_w$ at which the surface $S_{\lambda}$ is quasi-smooth.  Then we obtain   $L_{zw}^2 =-\frac{4}{15}$ from the intersection numbers
\[(2L_{zw}+R)\cdot L_{zw}=\frac{4}{15}, \ \ \ R\cdot L_{zw}=\frac{4}{5}.\]
Therefore,  $L_{zw}\cdot L_{zt}=\frac{1}{6}$ and hence the intersection matrix is non-singular.
This is a contradiction.
It shows that
$\tau_{Y}(E)\ne E$.
In particular, the involution
$\tau$ is not biregular. Since the involution $\tau_Y$ is biregular in codimension one, the involution $\tau$ meets the conditions
in Definition~\ref{definition:untwisting}.
 Consequently, the birational involution
$\tau$ untwists the singular point $O_z$.
\end{proof}

\begin{remark}\label{remark:plurianticanonical-23}
Local parameters at $O_z$ are induced by $x$, $t$, $w$ whose multiplicities are $\frac{1}{3}$, $\frac{1}{3}$, and $\frac{2}{3}$. The monomial $z^4y$ shows that $y$ vanishes at the point $O_z$ with
multiplicity at least $\frac{2}{3}$. Furthermore, since
\[-y=xt+(t+by^2)w^2+yt(t-\alpha_1y^2)(t-\alpha_2 y^2) +wf_9(x,y,1,t)+f_{14}(x,y,1,t) \]
around $O_z$ and $xt$ vanishes at the point $O_z$ with
multiplicity $\frac{2}{3}$, the monomial  $y$ vanishes at the point $O_z$ with
multiplicity exactly $\frac{2}{3}$. Then the relation
\[-(y+xt)=(t+by^2)w^2+yt(t-\alpha_1y^2)(t-\alpha_2 y^2) +wf_9(x,y,1,t)+f_{14}(x,y,1,t)\]
around $O_z$ shows that
$yz+xt$ vanishes at the point $O_z$ with
multiplicity $\frac{5}{3}$.

The linear system $|-5K_X|$ is generated by $w$, $xt$, $yz$, $x^2z$, $xy^2$, $x^3y$ and $x^5$. First of all, the last three monomials vanish at $O_z$ with multiplicity $\frac{5}{3}$.
In terms of the local parameters $x$, $t$, $w$, we have
\[yz=-xt+\mbox{higher degree terms}\]
locally around the point $O_z$.
Furthermore, for any complex numbers $\alpha$, $\beta$, $\delta$, $\epsilon$, we have
\[\alpha w+\beta xt+\delta yz+\epsilon x^2z=(\alpha w+\beta xt-\delta xt+\epsilon x^2)+\mbox{higher degree terms}\]
locally around the point $O_z$.
For the monomial $\alpha w+\beta xt+\delta yz+\epsilon x^2z$ to have multiplicity bigger than $\frac{2}{3}$ at $O_z$, we must have $\alpha=\epsilon=0$ and $\beta=\delta$. Since $yz+xt$ vanishes at the point $O_z$ with
multiplicity $\frac{5}{3}$, we see that  the linear system $\mathcal{H}$
consists exactly of  the members of $|-5K_X|$ vanishing at $O_z$ with multiplicity at least $\frac{5}{3}$.

Meanwhile, the variables $x$, $y$, $z$ induce local parameters at the point $O_w$ with multiplicities  $\frac{1}{5}$,
$\frac{2}{5}$, $\frac{3}{5}$, respectively. Also, since $t$ vanishes at the point $O_w$ with multiplicity at least $\frac{4}{5}$, the polynomial $yz+xt$ vanishes at the point $O_w$ with multiplicity $1$. Therefore, every member in $\mathcal{H}$ vanishes at  $O_w$ with multiplicity at least $1$.
\end{remark}

\begin{lemma}\label{lemma:irreducible-23}
Under the conditions of Theorem~\ref{theorem:amazing-23}, for a general complex number $\lambda$,
the curve  \eqref{equation:amazing-23} is irreducible and reduced for every complex number $\mu$.
\end{lemma}
\begin{proof}
Suppose that for a general complex number $\lambda$ there is always $\mu$ such that
the curve $C_{\lambda, \mu}$ is reducible. There is then a
one-dimensional family of reducible curves $C_{\lambda, \mu}$
given by  \eqref{equation:amazing-23} with  a general  complex number $\lambda$ and a complex number
$\mu$ depending on $\lambda$. Denote a general curve in this
one-dimensional family by $C$. 

Since \eqref{equation:amazing-23} always contains the monomial $zw^2$, the curve $C$ must have an irreducible  component $C_1$ that is defined by either $$y-\lambda x=t-\mu x^4+\lambda xz=w+h_{5}(x,z)=0$$  or  $$y-\lambda x=t-\mu x^4+\lambda xz=w^2+wg_{5}(x,z)+g_{10}(x,z)=0.$$
Then
\[-K_{Y}\cdot\tilde{C}_1=-K_{X_{14}}\cdot C_1-\frac{1}{3}E\cdot\tilde{C}_1= \left\{%
\aligned
&\frac{1\cdot 2\cdot 4\cdot 5}{1\cdot 2\cdot 3\cdot 4\cdot 5}-\frac{1}{3}=0 \mbox{ for the former case},\\%
&\\
&\frac{1\cdot 2\cdot 4\cdot 2\cdot 5}{1\cdot 2\cdot 3\cdot 4\cdot 5}-\frac{2}{3}=0 \mbox{ for the latter case}\\%
\endaligned\right.\]
 and $\tilde{C}_1\cdot E>0$. By
Lemma~\ref{lemma:bad-link}, the point $O_z$ cannot be a center of
non-canonical singularities of the log pair $(X_{14},
\frac{1}{n}\mathcal{M})$. This contradiction proves the statement.
\end{proof}

Now we go back to the hypersurfaces in the family No.~7
described in the previous section. The hypersurface $X_8$  of Type~II is defined in $\mathbb{P}(1,1,2,2,3)$ by the equation of the type
\begin{equation}\label{equation:special-7-original}(z+f_2(x,y))w^2+wf_5(x,y,z,t)-zt^3-t^2f_4(x,y,z)-tf_6(x,y,z)+f_8(x,y,z)=0.\end{equation}
Since  $f_5$ must contain either
$xt^2$ or $yt^2$, we write $f_5(x,y,z,t)=g_5(x,y,z,t)+a_1 xt^2+a_2
yt^2$. Furthermore, we may assume that $a_1=1$ and $a_2=0$ by a suitable coordinate change.

By coordinate change $z+f_2(x,y)\mapsto z$, we may assume that our  hypersurface $X_8$  is defined by
\begin{equation}\label{equation:special-7}zw^2+wf_5(x,y,z,t)-(z-f_2(x,y))t^3-t^2f_4(x,y,z)-tf_6(x,y,z)+f_8(x,y,z)=0.\end{equation}
This assumption will help us understand, without any loss of generality,  the intersection of the surface cut by $y=\lambda x$ and the surface cut by $z=\mu x^2$, where $\lambda$ and $\mu$ are constants.

On the hypersurface $X_8$, consider the surface cut by $y=\lambda x$ and the surface cut by $z=\mu x^2$. Then the intersection of these two surfaces is the $1$-cycle
$L_{tw}+C_{\lambda, \mu}$, where the curve $C_{\lambda, \mu}$ is defined by the equation
  \begin{equation}\label{equation:irr-7}\begin{split}&\mu xw^2+wt^2-(\mu -f_2(1,\lambda )) xt^3+ \\ &+\frac{wg_5(x,\lambda x,\mu x^2,t)-t^2f_4(x,\lambda x,\mu x^2)-tf_6(x,\lambda x,\mu x^2)+f_8(x,\lambda x,\mu x^2)}{x}=0 \\ \end{split}\end{equation}
in $\mathbb{P}(1,2,3)$.  For sufficiently general complex numbers $\lambda$ and $\mu$ the curve  $C_{\lambda, \mu}$ is birational to an elliptic curve. To figure  this out, we plug in $x=1$ into  \eqref{equation:irr-7}  so that  we could see that the curve is  birational to a double cover of $\mathbb{C}$ ramified at four distinct points.

 Let $\mathcal{H}$ be the linear subsystem of $|-2K_{X_8}|$ generated by $x^2$, $xy$, $y^2$ and $z$.
Let $\pi\colon X_8\dasharrow\mathbb{P}(1,1,2)$ be the rational map
 induced by $$[x:y:z:t:w]\mapsto [x:y:z].$$ It is a morphism outside of the curve $L_{tw}$. Moreover, the map is dominant. The curve $C_{\lambda, \mu}$
 is a fiber of the map $\pi$. Its general fiber is an irreducible curve birational to an elliptic curve since the
curve $C_{\lambda, \mu}$ with sufficiently general complex numbers $\lambda$ and $\mu$ is birational to an elliptic curve.

\begin{lemma}\label{lemma:irreducible-7}
Suppose that the hypersurface $X_8$ in the family No.~$7$ is defined
by \eqref{equation:special-7}. If the
singular point $O_t$ is a center of non-canonical singularities of
the log pair $(X_8, \frac{1}{n}\mathcal{M})$, then for a general complex number $\lambda$, the
curve $C_{\lambda, \mu}$ is always irreducible for every value of
$\mu$.
\end{lemma}
\begin{proof}
Suppose that for a general complex number $\lambda$ there is always $\mu$ such that
the curve $C_{\lambda, \mu}$ is reducible. Since the base
locus of $\mathcal{H}$ consists of the curve  $L_{tw}$, there is a
one-dimensional family of reducible curves $C_{\lambda, \mu}$
given by \eqref{equation:irr-7} with  a general  complex number $\lambda$ and a complex number
$\mu$ depending on $\lambda$. Denote a general curve in this
one-dimensional family by $C$. 

We claim that the curve $C$ always has an irreducible component $C_1$ defined
by $$y-\lambda x=z-\mu x^2=w-h_{3}(x,t)=0$$  for some polynomial
$h_3$.
To prove the claim, write $g_5(x,y,z,t)=f_3(x,y,z)t+f_5(x,y,z)$, set $x=1$ for  \eqref{equation:irr-7}, and then  obtain
\[\begin{split}&\mu w^2+w\left( t^2+f_3(1,\lambda ,\mu)t+f_5(1,\lambda ,\mu)\right)- \\ & -(\mu -f_2(1,\lambda )) t^3-f_4(1,\lambda,\mu)t^2-f_6(1,\lambda,\mu)t+f_8(1,\lambda,\mu )=0. \\ \end{split}\]
Suppose that the claim is not a case.
Then we must have $\mu=0$ and the polynomial
\[w\left( t^2+f_3(1,\lambda ,0)t+f_5(1,\lambda ,0)\right)+f_2(1,\lambda )) t^3-f_4(1,\lambda,0)t^2-f_6(1,\lambda,0)t+f_8(1,\lambda,0 )\]
must be reducible. Since $\lambda$ is general, this implies that
\[wf_5(x,y,0,t)+f_2(x,y)t^3-t^2f_4(x,y,0)-tf_6(x,y,0)+f_8(x,y,0)=A(x,y,t,w)B(x,y,t,w),\]
for some non-constant polynomials $A(x,y,t,w)$ and  $B(x,y,t,w)$. Since we may write
\[zw^2+wf_5(x,y,z,t)-(z-f_2(x,y))t^3-t^2f_4(x,y,z)-tf_6(x,y,z)+f_8(x,y,z)\]\[=zH(x,y,z,t,w)+A(x,y,t,w)B(x,y,t,w),\]
for some non-constant polynomial $H(x,y,z,t,w)$, the hypersurface $X_8$ is not quasi-smooth at the points defined by
$z=H(x,y,z,t,w)=A(x,y,t,w)=B(x,y,t,w)=0$.
This is a contradiction.
Consequently, the reducible curve $C$ splits into an irreducible curve $C_1$ defined
by $$y-\lambda x=z-\mu x^2=w-h_{3}(x,t)=0$$  for some polynomial
$h_3$ and the curve $C_2$ (possibly reducible)  defined by $$y-\lambda x=z-\mu
x^2=t^2-h_{4}(x,t, w)=0$$ for some polynomial $h_4$.

Note that
$C_1$ passes through the point $O_t$ but $C_2$ does not.  Then
$-K_{Y}\cdot\tilde{C}_1=0$ and $\tilde{C}_1\cdot E>0$. By
Lemma~\ref{lemma:bad-link}, the point $O_t$ cannot be a center of
non-canonical singularities of the log pair $(X_8,
\frac{1}{n}\mathcal{M})$.  This contradicts our condition. Therefore,
 for a general complex
number $\lambda$, the curve $C_{\lambda, \mu}$ is always
irreducible for every value of $\mu$.
\end{proof}
For a general complex number $\lambda$, the
curve $C_{\lambda, \mu}$ is always reduced for every value of
$\mu$. Indeed, if the curve is not reduced, then the proof shows that $\mu\ne 0$. Then the equation for the curve must contain $xw^2$ and $wt^2$. Hence, it must split into the form
$(t^2+xw+\cdots)(w+\cdots)$. The polynomial of the type $(t^2+xw+\cdots)$ cannot be a square.
Therefore, $C_{\lambda, \mu}$ is always reduced. Moreover, for a general complex number $\lambda$, the curve
$L_{tw}$ cannot be an irreducible component of the
curve $C_{\lambda, \mu}$ for every value of $\mu$.

\begin{theorem}\label{theorem:special-7}
Suppose that the hypersurface $X_8$ in the family No.~$7$ is defined
by  \eqref{equation:special-7}.  If the
singular point $O_t$ is a center of non-canonical singularities of
the log pair $(X_8, \frac{1}{n}\mathcal{M})$, then there is a birational
involution that untwists the singular point $O_t$. \end{theorem}
\begin{proof}

 Let $g\colon Z\to Y$ be the weighted blow up at the point over
$O_w$ with weight $(1, 1,2)$ and let $F$ be its exceptional
divisor. The divisor $F$ contains a singular point of $Z$ that is
of type $\frac{1}{2}(1,1,1)$. Let $h:W\to Z$ be the blow up at
this singular point with the exceptional divisor $G$. Let
$\hat{L}_{tw}$ and $\breve{L}_{tw}$ be the proper transforms of
the curve $L_{tw}$ by the morphism $f\circ g\circ h$ and by the
morphism $f\circ g$, respectively. Also, let $\hat{E}$ and
$\hat{F}$ be the proper transforms of the exceptional divisors $E$
and $F$ by the morphism $g\circ h$ and by the morphism $h$,
respectively.

 Let $\mathcal{H}_Y$ and $\mathcal{H}_W$ be the proper transforms of the linear system $\mathcal{H}$ by the morphism $f$ and by  the morphism $f\circ g\circ h$, respectively. We then see that $\mathcal{H}_Y=|-2K_Y|$ and $\mathcal{H}_W=|-2K_W|$. The base locus of the linear system $\mathcal{H}$ consists of the single curve $L_{tw}$. We have
 \[-K_W\cdot\hat{L}_{tw}=-K_{X_{8}}\cdot L_{tw}-\frac{1}{2}\hat{E}\cdot \hat{L}_{tw}-\frac{1}{3}F\cdot\breve{L}_{tw}-\frac{1}{2}G\cdot \hat{L}_{tw}=-1.\]
 Therefore, the curve $\hat{L}_{tw}$ is the only curve that intersects $-K_W$ negatively.

By the same procedure as in the proof of Theorem~\ref{theorem:amazing-23}, we construct a log flip $\chi:W\dasharrow U$ along the curve $\hat{L}_{tw}$
and a dominant morphism $\eta$ of $U$ into a normal variety $\Sigma$ with connected fibers by the base-point-free
linear system $|-mK_U|$ for sufficiently large
 $m$.

 Let $\check{E}$, $\check{F}$ and $\check{G}$ be the proper transforms  of the divisors $\hat{E}$, $\hat{F}$ and $G$, respectively, by $\chi$.  Let $\hat{C}_{\lambda, \mu}$ be the proper
transform of a general fiber $C_{\lambda, \mu}$ of the map $\pi$
on $W$ and let $\check{C}_{\lambda, \mu}$ be its
proper transform on $U$. We then see
$$-K_{W}\cdot \hat{C}_{\lambda, \mu}=-2K_W^3-(-K_W)\cdot\hat{L}_{tw}=0.$$
By the same reason as in the proof of Theorem~\ref{theorem:amazing-23}, we see that $\eta$ is an elliptic fibration and we obtain
 the following  commutative diagram:
$$
\xymatrix{
&W\ar@{->}[ld]_{h}\ar@{-->}[r]^{\chi}&U\ar@{->}[ddddr]^{\eta}\\%
Z\ar@{->}[d]_{g}&&\\%
Y\ar@{->}[d]_{f}&&\\%
X_8\ar@{-->}[rd]_{\pi}&&\\
&\mathbb{P}(1,1,2)&&\Sigma\ar@{-->}[ll]^{\ \ \ \ \ \ \theta}}
$$ %
where $\theta$ is a birational map.

It follows from \eqref{equation:irr-7} that the divisors
$\check{E}$ and $\check{G}$ are sections of the elliptic fibration
$\eta$. Let $\tau_{U}$ be the birational involution of the threefold
$U$ that is induced by the reflection of the general fiber of
$\eta$ with respect to the section $\check{G}$. Then
$\tau_{U}$ is biregular in codimension one because $K_{U}$ is
$\eta$-nef by our construction  (\cite[Corollary~3.54]{KoMo98}).

Put
$\tau_{W}=\chi^{-1}\circ\tau_U\circ\chi$, $\tau_{Y}=(g\circ
h)\circ\tau_W\circ (g\circ h)^{-1}$ and $\tau=f\circ\tau_Y\circ
f^{-1}$ as before. Then $\tau_{W}$ is also biregular in codimension one
since $\chi$ is a log flip. Moreover, we have $\tau_{W}(G)=G$
since $\tau_{U}(\check{G})=\check{G}$ by our construction.  
The image $\tau_U(\check{F})$ is an irreducible surface since $\tau_U$ is biregular in codimension one.  The map $\pi\circ f\circ g$ sends $F$ to the curve in $\mathbb{P}(1,1,2)$ defined by $z=0$ and the log flip $\chi$ changes nothing on the intersection of $G$ and $\hat{F}$. Therefore, the morphism $\eta$ contracts $\check{F}$ to a curve and the image $\tau_U(\check{F})$ lies over this curve.
Since $\check{F}$ intersects with the section $\check{G}$ along a curve and $\tau_U(\check{F})$ intersects with the section 
$\check{G}$ along the curve, 
 $\tau_U(\check{F})=\check{F}$ and $\tau_W(\hat{F})=\hat{F}$.
Consequently, $\tau_{Y}$ is biregular in codimension one.

We claim that the point $O_t$ is untwisted by $\tau$. For us to prove the claim, it is enough to show that
$\tau_Y(E)\ne E$ due to Remark~\ref{remark:untwisting-involution}.
For this end, we suppose that $\tau_Y(E)=E$ and look for a contradiction.

Let $S_\lambda$ be the  surface on the hypersurface $X_8$ cut by the
equation $y=\lambda x$ with a general complex number $\lambda$. It  is a $K3$ surface with only cyclic du
Val singularities. The point $O_t$ is a $A_1$ singular point of
$S_\lambda$ and the point $O_w$ is  a $A_2$ singular point of
$S_\lambda$.
Let
$\tau_\lambda$ be the restriction of $\tau$ to the surface
$S_\lambda$. It is a birational involution of the surface
$S_{\lambda}$ since the surface is $\tau$-invariant by our construction.

 The projection $\pi\colon
X_8\dasharrow\mathbb{P}(1,1,2)$ induces a rational map
$\pi_{\lambda}\colon
S_\lambda\to\mathbb{P}(1,2)\cong\mathbb{P}^1$.  The rational map
$\pi_\lambda\colon S_\lambda\dasharrow\mathbb{P}^1$ is given by
the pencil of the curves on the surface
$S_\lambda\subset\mathbb{P}(1,2,2,3)$ cut by the equations
\[\delta x^2=\epsilon z,\]
where $[\delta:\epsilon]\in\mathbb{P}^1$.  Its base locus is cut
out on $S_{\lambda}$ by $x=z=0$.  Therefore, the base locus
 is the curve $L_{tw}$. We can easily
see from  \eqref{equation:irr-7} that the map
$\pi_\lambda$ is not defined only at the points $O_w$ and $O_t$.

Let $\hat{S}_\lambda$ be the proper transform of $S_\lambda$ via $f\circ g\circ h$ and put $\hat{E}_\lambda=\hat{E}|_{\hat{S}_\lambda}$
and $\hat{G}_\lambda=G|_{\hat{S}_\lambda}$ as in the proof of Theorem~\ref{theorem:amazing-23}.
Resolving the indeterminacy of the rational map $\pi_\lambda$ through $\hat{S}_\lambda$,
 we
obtain an elliptic fibration
$\bar{\pi}_\lambda\colon\bar{S}_\lambda\to \mathbb{P}^1$. Thus, we
have a commutative diagram
$$
\xymatrix{
&\bar{S}_\lambda \ar@{->}[ld]_{\sigma}\ar@{->}[dr]^{\bar{\pi}_\lambda}\\%
S_\lambda\ar@{-->}[rr]_{\pi_\lambda} &&\mathbb{P}^1,}
$$ %
where $\sigma$ is a birational morphism. There exist exactly
two $\sigma$-exceptional prime divisors that do not lie in the
fibers of $\bar{\pi}_\lambda$. One is the proper transform of $\hat{E}_\lambda$ and the other is
the proper transform of $\hat{G}_\lambda$.  Let
$\bar{E}_\lambda$ and $\bar{G}_\lambda$ be these two exceptional
divisors, respectively.
Then $\bar{E}_\lambda$ and
$\bar{G}_\lambda$ are sections of $\bar{\pi}_\lambda$. Denote the
other $\sigma$-exceptional curves (if any) by
$F_{1},\ldots,F_{r}$.

Put $\bar{\tau}_\lambda=\sigma^{-1}\circ\tau_\lambda\circ\sigma$.
We may assume that $\bar{\tau}_\lambda$ is
biregular and $\bar{S}_\lambda$ is smooth by \cite[Theorem~3.2]{dFE}. 

By the same argument as in the proof of Theorem~\ref{theorem:amazing-23},  the divisor
$\bar{E}_{\lambda}-\bar{G}_{\lambda}$ is numerically
equivalent to a $\mathbb{Q}$-linear combination of curves on
$\bar{S}_{\lambda}$ that lie in the fibers of
$\bar{\pi}_{\lambda}$. Observe that we use the assumption $\tau_Y(E)=E$ at this step.

Note that the equation $x=0$ cuts out $S_\lambda$  into a curve
that splits as a union $L_{tw}+C_x$, where $C_x$ is the curve
defined by  $$x=w^2-t^3+azt^2+bz^2t+cz^3=0$$ for some
constants $a$, $b$, $c$ in $\mathbb{P}(1,2,2,3)$. The curve $C_x$
is irreducible and reduced.

Let $\bar{L}_{tw}$ and $\bar{C}_{x}$ be the proper transforms of
the curves $L_{tw}$ and $C_{x}$ by $\sigma$, respectively. Then
$\bar{L}_{tw}$ and $\bar{C}_{x}$ lie in the same fiber of the
elliptic fibration $\bar{\pi}_{\lambda}$ and they are the only non-$\sigma$-exceptional curves in this fiber. Moreover,
every other fiber of $\bar{\pi}_{\lambda}$ contains exactly one
irreducible and reduced curve that is not $\sigma$-exceptional because for a general complex
number $\lambda$, the curve $C_{\lambda, \mu}$ is always
irreducible and reduced for every value of $\mu$ by Lemma~\ref{lemma:irreducible-7}.  Therefore, as before, we are able to obtain
$$
\bar{E}_{\lambda}-\bar{G}_{\lambda}\sim_{\mathbb{Q}}  c_{tw}\bar{L}_{tw}+c_{x}\bar{C}_{x}+\sum_{i=1}^{r}c_i F_i%
$$
for some rational numbers $c_{tw}$, $c_{x}$, $c_1,\ldots,c _r$.
The intersection form of the curves $\bar{E}_{\lambda}$,
$\bar{G}_{\lambda}$, $F_1, \ldots, F_r$ is negative-definite since
these curves are $\sigma$-exceptional. Therefore,
$(c_{tw},c_{x})\ne (0,0)$. On the other hand, we have
$$
0\sim_{\mathbb{Q}}  c_{tw}L_{tw}+c_{x}C_{x}
$$
on the surface $S_{\lambda}$, and hence the intersection form
of the curves $L_{tw}$ and $C_{x}$ is degenerate on the surface
$S_{\lambda}$.

However, from the intersection numbers
\[(L_{tw}+C_{x})\cdot L_{tw}=\frac{1}{6}, \ \ (L_{tw}+C_{x})^2=\frac{2}{3},\ \ L_{tw}\cdot C_{x}=1\]
on the surface $S_\lambda $, we obtain
\[\left(\begin{array}{cc}
       L_{tw}^2&L_{tw}\cdot C_{x}\\
      L_{tw}\cdot C_{x}& C_{x}^2\\
\end{array}\right)= \left(\begin{array}{cc}
-\frac{5}{6} & 1\\
    1& -\frac{1}{2}\\
        \end{array}\right).
\]
This is a contradiction.
The obtained contradiction verifies  that
$\tau_Y(E)\ne E$. This completes the proof.
\end{proof}


\newpage
\section{Proof of Main Theorem}
\label{section:proof}

\subsection{How to read the tables}\label{section:manuals}

The remaining job is to exclude or untwist all the singular points on quasi-smooth hypersurfaces in the 95 families.
To execute this crucial job, we need to know how to read the tables in the next section. They carry all the information for excluding and untwisting the singular points.

 For each
family we present a table that carries  
\begin{itemize}
\item the entry number (the underlined entry number means that  the family corresponds to Theorem~\ref{theorem:auxiliary}, i.e., birationally super-rigid family),
\item the intersection number of the anticanonical divisor, i.e., $-K_X^3=A^3$,
\item a  defining
equation of the hypersurface $X$,  
\item its singularities, 
\item the sign of
$B^3$, 
\item the linear system on $Y$  containing the key  surface $T$ in the
applied method, 

\item a defining equation  for the surface $f(T)$ or generators of a linear system that contains $f(T)$ as a general member,
\item  terms that determine the multiplicity  of the surface $f(T)$ at
the given singular point.
\end{itemize}
When the table carries only a monomial or a binomial, instead of $B^3$, the linear system,  the surface $T$ and the vanishing order for the corresponding singular point(s), we apply the methods below with the squared symbols to the corresponding singular points. The monomial or the binomial plays an essential role in defining the involution untwisting the singular point.

The table shows which method is applied to each of the singular
points by the symbols  \boundary, \nef, \surface, \family, 
\positive, \quadratic, \quadraticone, \elliptic, \ellipticone, \elliptictwo, \ellipticfour \ and \ellipticfive. The following explain the method corresponding to each
of the symbols.

\begin{itemize}

\item[\boundary] : Apply Lemma~\ref{lemma:boundary}.\\
The condition $T\cdot\Gamma\leq 0$ can be easily checked by the
items in the table (see Remark~\ref{remark:boundary}). The
condition on the $1$-cycle $\Gamma$ can be immediately checked.
This can be done on the hypersurface $X$ even though the cycle
lies on the threefold $Y$. Indeed, in the cases where this method
is applied, the surface $T$ is given in such a way that the
$1$-cycle $\Gamma$ has no component on the exceptional divisor
$E$.

\item[\nef] : Apply Lemma~\ref{lemma:boundary-2}.\\
The divisor $T$ is given as the proper transform of a general
member of the linear system generated by the monomial(s) in the
slot for the item $T$ of the table. Using
Lemma~\ref{lemma:nefness} we check that the given divisor $T$ is
nef. The non-positivity of $T\cdot S\cdot B$ can be immediately
verified from the items in the table and the positivity of $T\cdot
S\cdot A$ is always guaranteed (see
Remark~\ref{remark:boundary-2}).

\item[\surface] : Apply Lemma~\ref{lemma:negative-definite}.\\
We take a general member $H$ in the linear system generated by the
polynomials given in the slot for the item $T$ in the table. We
can easily show that the surface $H$ is normal by checking that it
has only isolated singularities. The surface $T$ is given as the
proper transform of the surface $H$ by the morphism $f$. The
divisor on $T$ cut out by the surface $S$ is a reducible curve. We
check that this reducible curve forms a negative-definite divisor
on the normal surface $T$.

\item[\family] : Apply Lemma~\ref{lemma:bad-link}.\\
We find a $1$-dimensional family of irreducible curves
$\tilde{C}_\lambda$ such that $-K_Y\cdot \tilde{C}_\lambda\leq 0$.
We can find this family on  the surface $T$ that is given as the
proper transform of a general member of the  linear system
generated by the polynomial(s) provided  in the slot for the item
$T$ of the table.

\item[\positive] : Apply Lemma~\ref{lemma:2-ray game}.\\
If the singular point $O_t$ satisfies the conditions of Lemma~\ref{lemma:2-ray game},  
we can always find a $1$-dimensional family of irreducible
curves $\tilde{C}_\lambda$ on the given surface $T$ such that
$-K_Y\cdot \tilde{C}_\lambda\leq 0$, so that we could immediately exclude the singular point $O_t$.
\end{itemize}

As we see, the circled methods are applied to exclude singular points on $X$.
The squared symbols  below  are the methods with which we can 
untwist the corresponding singular point if it is a center of non-canonical singularities of the log pair.

 \bigskip

 \begin{itemize}

\item[\quadratic] : Apply Lemma~\ref{lemma:Quadratic involution} and Lemma~\ref{lemma:Quadratic involution-biregular-1}\\
The given monomial in the table is the monomial $x_{i_3}x_{i_4}^2$
in Lemma~\ref{lemma:Quadratic involution} and
Lemma~\ref{lemma:Quadratic involution-biregular-1} that plays a
central role in defining the involution. If the hypersurface $X$
is defined by the equation as in Lemma~\ref{lemma:Quadratic
involution}, the involution given by the quadratic equation is
birational and untwists the given singular point. If the
hypersurface $X$ is defined by the equation as in
Lemma~\ref{lemma:Quadratic involution-biregular-1}, the involution
given by the quadratic equation is biregular. In such a case,
Lemma~\ref{lemma:Quadratic involution-biregular-1} excludes the
corresponding singular point. Note that both the cases can always
happen.

\item[\quadraticone] : Apply Lemma~\ref{lemma:Quadratic
involution},  Lemma~\ref{lemma:Quadratic involution-biregular-1}
and
Theorem~\ref{theorem:quadratic-biregular}\\
This method is basically the same as the method \quadratic. The
difference is that we may have no $x_{i_3}x_{i_4}^2$ in the
defining equation. Such cases occur only when the corresponding
singular point is $O_t$ and $x_{i_3}x_{i_4}^2=wt^2$. In cases,
Theorem~\ref{theorem:quadratic-biregular} excludes the singular
point $O_t$. These three cases can always occur, i.e., the case
when the defining equation has the monomial $wt^2$ with $f_e$ not divisible by $w$,
the case when the defining equation has the monomial $wt^2$ with
$f_e$ divisible by $w$ and the case when the defining equation does not have the
monomial $wt^2$.

\item[\elliptic] : Apply Theorem~\ref{theorem:geometric-elliptic}\\
This is for the singular point $O_t$ of quasi-smooth hypersurfaces in the
families  No. 7 (Type I), 23, 40, 44, 61 and 76. The given
binomial  in the table is the
binomial $tw^2-x_it^3$ in \eqref{equation:defining
equation-elliptic} that plays a central role in defining the
involution.  The singular point $O_t$ may not be a center of non-canonical singularities of the log pair in some situation. However,
if it is a center, then it can be untwisted by an elliptic involution.

\item[\ellipticone] : Apply Theorem~\ref{theorem:elliptic-biregular-untwist2}.\\
This is for the singular point $O_z$ of quasi-smooth hypersurfaces in the
family No. 36.

\item[\elliptictwo] : Apply Theorem~\ref{theorem:elliptic-biregular-untwist2} and Theorem~\ref{theorem:elliptic-biregular3}\\
This is for the singular point $O_z$ of quasi-smooth hypersurfaces in the
family No. 20.


\item[\ellipticfour] :  Apply Theorem~\ref{theorem:special-7}\\
This is for the singular points of type $\frac{1}{2}(1,1,1)$ on
quasi-smooth hypersurfaces of Type~II in the family No. 7.

\item[\ellipticfive] : Apply Theorem~\ref{theorem:amazing-23}\\
This is for the singular point $O_z$ of the special hypersurfaces
in the family No. 23 described in Section~\ref{section:invisible
involution}.

\end{itemize}

In each table, we present a defining equation of the hypersurface
in the family. For this we use the following notations and
conventions.
\begin{itemize}
\item The Roman alphabets $a$, $b$, $c$, $d$, $e$ with numeric
subscripts or without  subscripts are constants.

\item The Greek alphabets $\alpha$, $\beta$ with numeric
subscripts or without subscripts are constants.

\item The same Roman alphabets with distinct numeric subscripts,
e.g., $a_1$, $a_2$, $a_3$, in an equation are constants one of
which is not zero.

\item The same Greek alphabets with distinct numeric subscripts,
e.g., $\alpha_1$, $\alpha_2$, $\alpha_3$, in an equation are
distinct constants.

\item The singularity types are often given as a form
$\frac{1}{r}(w^1_{x_{k_1}}, w^2_{x_{k_2}}, w^3_{x_{k_3}})$, where
the subscript $x_{k_i}$ is the homogeneous coordinate function
which induces a local parameter corresponding to the weight
$w^i_{x_{k_i}}$.

\end{itemize}
For each family, the defining equation of the hypersurface
$X$ must satisfies the following rules in order to be quasi-smooth
(see \cite{IF00} for more detail).
\begin{itemize}
\item If $a_i>1$, it is relatively prime to the other weights and
it divides $d$, then $x_i^{\frac{d}{a_i}}$ must appear in the
defining equation. \item If $a_i>1$, it is relatively prime to
the other weights but it does not divide $d$, then
$x_i^{\frac{d-a_j}{a_i}}x_j$ for some $j$ must appear in the
defining equation. \item If $a_i$ and $a_j$ are not relatively
prime, then a reduce polynomial of degree $d$ in $x_i$ and $x_j$
must appear in the defining equation.
\end{itemize}
In each table,  the defining equation  is written in the form
$$\mbox{key-monomial part}+wf_{d-a_4}(x,y,z,t)+f_{d}(x,y,z,t) \ \ \mbox{ if } d<3a_4; $$ $$\mbox{key-monomial part}+w^2f_{d-2a_4}+wf_{d-a_4}(x,y,z,t)+f_{d}(x,y,z,t) \ \ \mbox{ if } d=3a_4 ,$$
where key-monomial part consists of the monomials that are
required for quasi-smoothness and necessary for our methods of
excluding or untwisting
the singularities. If necessary, we expand $f_d(x,y,z,t)$ with
respect to the variable $t$, i.e.,  instead of $f_d(x,y,z,t)$, we
write
\[ g_{d-a_3m}(x,y,z)t^m+ g_{d-a_3m+a_3}(x,y,z)t^{m-1}+\cdots+g_d(x,y,z).\]
Note that we do not put all the monomials required for
quasi-smoothness in the key-monomial part. We put only some of
them that play roles for our methods of excluding or untwisting
the singularities on the given hypersurface. To simplify the key
monomial part as much as possible without loss of generality, we
apply suitable coordinate changes, if necessary. It will not be too complicated to check that the given quasi-homogeneous polynomial represents every quasi-smooth hypersurface in the family.

\subsection{The tables}\label{section:super-rigid}

To prove Main Theorem, we suppose that a given quasi-smooth hypersurface
$X$ from the 95 families has a mobile linear system
$\mathcal{M}$ in $|-nK_X|$ for some positive integer $n$ such
that the log pair $(X, \frac{1}{n}\mathcal{M})$ is not canonical.
Therefore, we have a center of non-canonical singularities of the
pair $(X, \frac{1}{n}\mathcal{M})$. Theorems~\ref{theorem:smooth point excluding} and~\ref{theorem:excluding-curve} show that
if there is a center on $X$, then it must be a singular point.
 
In this section, we exclude or untwist  every singular point on a given quasi-smooth hypersurface in each of the 95 families.
To be precise, we prove
\begin{theorem}\label{theorem:excluding or untwisting}
If a singular point on  $X$ 
 is a center of non-canonical singularities of the log pair $\left(X,\frac{1}{n}\mathcal{M}\right)$, then it can be untwisted by a birational involution of $X$. 
\end{theorem}
By verifying this theorem, we  obtain a complete proof of Main Theorem from Theorem~\ref{theorem:Nother-Fano}.

\begin{proof} The proof is given mainly by the  tables. Following the instruction in Section~\ref{section:manuals} with the extra explanation (if necessary) provided right after the table, we prove Theorem~\ref{theorem:excluding or untwisting} for each family.

\begin{center}
\begin{longtable}{|l|c|c|c|c|c|}
\hline
\multicolumn{6}{|l|}{\textbf{No. 2}: $X_{5}\subset\mathbb{P}(1,1,1,1,2)$  \hfill $A^3=5/2$}\\
\multicolumn{6}{|l|}{
\begin{minipage}[m]{.86\linewidth}
\vspace*{1.2mm}

$tw^2+wf_3(x,y,z,t)+f_5(x,y,z,t)$

\vspace*{1.2mm}
\end{minipage}
}\\
\hline \hline
\begin{minipage}[m]{.28\linewidth}
\begin{center}
Singularity
\end{center}
\end{minipage}&
\begin{minipage}[m]{.04\linewidth}
\begin{center}
$B^3$
\end{center}
\end{minipage}&
\begin{minipage}[m]{.11\linewidth}
\begin{center}
Linear

system
\end{center}
\end{minipage}&
\begin{minipage}[m]{.11\linewidth}
\begin{center}
Surface $T$
\end{center}
\end{minipage}&
\begin{minipage}[m]{.11\linewidth}
\begin{center}
\vspace*{1mm}
 \vorder
\vspace*{1mm}
\end{center}
\end{minipage}&
\begin{minipage}[m]{.18\linewidth}
\begin{center}
Condition
\end{center}
\end{minipage}\\
\hline
\begin{minipage}[m]{.28\linewidth}

$O_w=\frac{1}{2}(1,1,1)$ \quadratic

\end{minipage}&
\multicolumn{4}{|l|}{\begin{minipage}[m]{.37\linewidth}
\begin{center}
$tw^2$
\end{center}
\end{minipage}}&
\begin{minipage}[m]{.18\linewidth}
\begin{center}

\end{center}
\end{minipage}\\
\hline
\end{longtable}
\end{center}


\begin{center}
\begin{longtable}{|l|c|c|c|c|c|}
\hline
\multicolumn{6}{|l|}{\textbf{No. 4}: $X_{6}\subset\mathbb{P}(1,1,1,2,2)$\hfill $A^3=3/2$}\\
\multicolumn{6}{|l|}{
\begin{minipage}[m]{.86\linewidth}
\vspace*{1.2mm}
$(t-\alpha_1w)(t-\alpha_2w)(t-\alpha_3w)+wf_4(x,y,z,t)+f_6(x,y,z,t)$

\vspace*{1.2mm}
\end{minipage}
}\\
\hline \hline
\begin{minipage}[m]{.28\linewidth}
\begin{center}
Singularity
\end{center}
\end{minipage}&
\begin{minipage}[m]{.04\linewidth}
\begin{center}
$B^3$
\end{center}
\end{minipage}&
\begin{minipage}[m]{.11\linewidth}
\begin{center}
Linear

system
\end{center}
\end{minipage}&
\begin{minipage}[m]{.11\linewidth}
\begin{center}
Surface $T$
\end{center}
\end{minipage}&
\begin{minipage}[m]{.11\linewidth}
\begin{center}
\vspace*{1mm}
 \vorder
\vspace*{1mm}
\end{center}
\end{minipage}&
\begin{minipage}[m]{.18\linewidth}
\begin{center}
Condition
\end{center}
\end{minipage}\\

\hline
\begin{minipage}[m]{.28\linewidth}

$O_tO_w=3\times\frac{1}{2}(1,1,1)$ \quadratic

\end{minipage}&
\multicolumn{4}{|l|}{\begin{minipage}[m]{.37\linewidth}
\begin{center}
$tw^2$
\end{center}
\end{minipage}}&
\begin{minipage}[m]{.18\linewidth}
\begin{center}

\end{center}
\end{minipage}\\\hline
\end{longtable}
\end{center}

\begin{Note}
\item We may assume that $\alpha_1=0$. To see how to treat the singular
points of type $\frac{1}{2}(1,1,1)$, we have only to consider the
singular point $O_w$. The other points can be dealt with in the
same way.
\end{Note}

\begin{center}
\begin{longtable}{|l|c|c|c|c|c|}
\hline
\multicolumn{6}{|l|}{\textbf{No. 5}: $X_{7}\subset\mathbb{P}(1,1,1,2,3)$\hfill $A^3=7/6$}\\
\multicolumn{6}{|l|}{
\begin{minipage}[m]{.86\linewidth}
\vspace*{1.2mm} $zw^2+wf_4(x,y,z,t)+f_7(x,y,z,t)$

\vspace*{1.2mm}
\end{minipage}
}\\
\hline \hline
\begin{minipage}[m]{.28\linewidth}
\begin{center}
Singularity
\end{center}
\end{minipage}&
\begin{minipage}[m]{.04\linewidth}
\begin{center}
$B^3$
\end{center}
\end{minipage}&
\begin{minipage}[m]{.11\linewidth}
\begin{center}
Linear

system
\end{center}
\end{minipage}&
\begin{minipage}[m]{.11\linewidth}
\begin{center}
Surface $T$
\end{center}
\end{minipage}&
\begin{minipage}[m]{.11\linewidth}
\begin{center}
\vspace*{1mm}
 \vorder
\vspace*{1mm}
\end{center}
\end{minipage}&
\begin{minipage}[m]{.18\linewidth}
\begin{center}
Condition
\end{center}
\end{minipage}\\
\hline
\begin{minipage}[m]{.28\linewidth}

$O_w=\frac{1}{3}(1,1,2)$ \quadratic

\end{minipage}&
\multicolumn{4}{|l|}{\begin{minipage}[m]{.37\linewidth}
\begin{center}
$zw^2$
\end{center}
\end{minipage}}&
\begin{minipage}[m]{.18\linewidth}
\begin{center}

\end{center}
\end{minipage}\\
\hline
\begin{minipage}[m]{.28\linewidth}

$O_t=\frac{1}{2}(1,1,1)$ \quadraticone

\end{minipage}&
\multicolumn{4}{|l|}{\begin{minipage}[m]{.37\linewidth}
\begin{center}
$wt^2$
\end{center}
\end{minipage}}&
\begin{minipage}[m]{.18\linewidth}
\begin{center}

\end{center}
\end{minipage}\\
\hline
\end{longtable}
\end{center}


\begin{center}
\begin{longtable}{|l|c|c|c|c|c|}
\hline
\multicolumn{6}{|l|}{\textbf{No. 6}: $X_{8}\subset\mathbb{P}(1,1,1,2,4)$  \hfill $A^3=1$}\\
\multicolumn{6}{|l|}{
\begin{minipage}[m]{.86\linewidth}
\vspace*{1.2mm}
$(w-\alpha_1t^2)(w-\alpha_2t^2)+wf_4(x,y,z,t)+f_8(x,y,z,t)=0$

\vspace*{1.2mm}
\end{minipage}
}\\
\hline \hline
\begin{minipage}[m]{.28\linewidth}
\begin{center}
Singularity
\end{center}
\end{minipage}&
\begin{minipage}[m]{.04\linewidth}
\begin{center}
$B^3$
\end{center}
\end{minipage}&
\begin{minipage}[m]{.11\linewidth}
\begin{center}
Linear

system
\end{center}
\end{minipage}&
\begin{minipage}[m]{.11\linewidth}
\begin{center}
Surface $T$
\end{center}
\end{minipage}&
\begin{minipage}[m]{.11\linewidth}
\begin{center}
\vspace*{1mm}
 \vorder
\vspace*{1mm}
\end{center}
\end{minipage}&
\begin{minipage}[m]{.18\linewidth}
\begin{center}
Condition
\end{center}
\end{minipage}\\
\hline
\begin{minipage}[m]{.28\linewidth}

$O_tO_w=2\times\frac{1}{2}(1,1,1)$ \quadratic

\end{minipage}&
\multicolumn{4}{|l|}{\begin{minipage}[m]{.37\linewidth}
\begin{center}
$wt^2$
\end{center}
\end{minipage}}&
\begin{minipage}[m]{.18\linewidth}
\begin{center}

\end{center}
\end{minipage}\\
\hline
\end{longtable}
\end{center}

\begin{Note}
\item We may assume that $\alpha_1=0$. To see how to treat the singular
points of type $\frac{1}{2}(1,1,1)$, we have only to consider the
singular point $O_t$. The other point can be dealt with in the
same way. After we set $\alpha_1=0$,  by a suitable coordinate
change with respect to $w$, we may assume that the monomials of
types $t^3g_{2}(x,y,z)$, $t^2g_{4}(x,y,z)$ do not appear in the
defining equation.
\end{Note}



\begin{center}
\begin{longtable}{|l|c|c|c|c|c|}
\hline
\multicolumn{6}{|l|}{\textbf{No. 7}: $X_{8}\subset\mathbb{P}(1,1,2,2,3)$\hfill $A^3=2/3$}\\
\multicolumn{6}{|l|}{
\begin{minipage}[m]{.86\linewidth}
\vspace*{1.2mm} Type I :
$tw^2+wg_5(x,y,z)-zt^3-t^2g_4(x,y,z)-tg_6(x,y,z)+g_8(x,y,z)$
\medskip

Type II :
$(z+f_2(x,y))w^2+wf_5(x,y,z,t)-zt^3-t^2f_4(x,y,z)-tf_6(x,y,z)+f_8(x,y,z)$
\vspace*{1.2mm}
\end{minipage}
}\\
\hline \hline
\begin{minipage}[m]{.28\linewidth}
\begin{center}
Singularity
\end{center}
\end{minipage}&
\begin{minipage}[m]{.04\linewidth}
\begin{center}
$B^3$
\end{center}
\end{minipage}&
\begin{minipage}[m]{.11\linewidth}
\begin{center}
Linear

system
\end{center}
\end{minipage}&
\begin{minipage}[m]{.11\linewidth}
\begin{center}
Surface $T$
\end{center}
\end{minipage}&
\begin{minipage}[m]{.11\linewidth}
\begin{center}
\vspace*{1mm}
 \vorder
\vspace*{1mm}
\end{center}
\end{minipage}&
\begin{minipage}[m]{.18\linewidth}
\begin{center}
Condition
\end{center}
\end{minipage}\\
\hline
\begin{minipage}[m]{.28\linewidth}

$O_w=\frac{1}{3}(1,1,2)$ \quadratic

\end{minipage}&
\multicolumn{4}{|l|}{\begin{minipage}[m]{.37\linewidth}
\begin{center}
$tw^2$
\end{center}
\end{minipage}}&
\begin{minipage}[m]{.18\linewidth}
\begin{center}

\end{center}
\end{minipage}\\
\hline
\begin{minipage}[m]{.28\linewidth}

$O_zO_t=4\times\frac{1}{2}(1,1,1)$ \elliptic
\end{minipage}&
\multicolumn{4}{|l|}{\begin{minipage}[m]{.37\linewidth}
\begin{center}
$tw^2-zt^3$
\end{center}
\end{minipage}}&
\begin{minipage}[m]{.18\linewidth}
\begin{center}
Type I
\end{center}
\end{minipage}\\
\hline
\begin{minipage}[m]{.28\linewidth}

$O_zO_t=4\times\frac{1}{2}(1,1,1)$ \ellipticfour
\end{minipage}&
\multicolumn{4}{|l|}{\begin{minipage}[m]{.37\linewidth}
\begin{center}

\end{center}
\end{minipage}}&
\begin{minipage}[m]{.18\linewidth}
\begin{center}
Type II
\end{center}
\end{minipage}\\

\hline

\end{longtable}
\end{center}

\begin{Note}
\item For the singular points of type $\frac{1}{2}(1,1,1)$ we have only
to consider one of them. The others can be untwisted or excluded
in the same way. The singular point to be considered here may be
assumed to be the point $O_t$ by a suitable coordinate change.
\end{Note}


\begin{center}
\begin{longtable}{|l|c|c|c|c|c|}
\hline
\multicolumn{6}{|l|}{\textbf{No. 8}: $X_{9}\subset\mathbb{P}(1,1,1,3,4)$\hfill $A^3=3/4$}\\
\multicolumn{6}{|l|}{
\begin{minipage}[m]{.86\linewidth}
\vspace*{1.2mm}

$zw^2+wf_{5}(x,y,z,t)+f_{9}(x,y,z,t)$ \vspace*{1.2mm}
\end{minipage}
}\\
\hline \hline
\begin{minipage}[m]{.28\linewidth}
\begin{center}
Singularity
\end{center}
\end{minipage}&
\begin{minipage}[m]{.04\linewidth}
\begin{center}
$B^3$
\end{center}
\end{minipage}&
\begin{minipage}[m]{.11\linewidth}
\begin{center}
Linear

system
\end{center}
\end{minipage}&
\begin{minipage}[m]{.11\linewidth}
\begin{center}
Surface $T$
\end{center}
\end{minipage}&
\begin{minipage}[m]{.11\linewidth}
\begin{center}
\vspace*{1mm}
 \vorder
\vspace*{1mm}
\end{center}
\end{minipage}&
\begin{minipage}[m]{.18\linewidth}
\begin{center}
Condition
\end{center}
\end{minipage}\\
\hline
\begin{minipage}[m]{.28\linewidth}

$O_w=\frac{1}{4}(1,1,3)$ \quadratic

\end{minipage}&
\multicolumn{4}{|l|}{\begin{minipage}[m]{.37\linewidth}
\begin{center}
$zw^2$
\end{center}
\end{minipage}}&
\begin{minipage}[m]{.18\linewidth}
\begin{center}
\end{center}
\end{minipage}\\

\hline
\end{longtable}
\end{center}



\begin{center}
\begin{longtable}{|l|c|c|c|c|c|}
\hline
\multicolumn{6}{|l|}{\textbf{No. 9}: $X_{9}\subset\mathbb{P}(1,1,2,3,3)$\hfill $A^3=1/2$}\\
\multicolumn{6}{|l|}{
\begin{minipage}[m]{.86\linewidth}
\vspace*{1.2mm} $(w-\alpha_1 t)(w-\alpha_2 t)(w-\alpha_3
t)+z^3(a_1t+a_2yz)+w^2f_3(x,y,z)+wf_{6}(x,y,z,t)+f_{9}(x,y,z,t)$
\vspace*{1.2mm}

\end{minipage}
}\\
\hline \hline
\begin{minipage}[m]{.28\linewidth}
\begin{center}
Singularity
\end{center}
\end{minipage}&
\begin{minipage}[m]{.04\linewidth}
\begin{center}
$B^3$
\end{center}
\end{minipage}&
\begin{minipage}[m]{.11\linewidth}
\begin{center}
Linear

system
\end{center}
\end{minipage}&
\begin{minipage}[m]{.11\linewidth}
\begin{center}
Surface $T$
\end{center}
\end{minipage}&
\begin{minipage}[m]{.11\linewidth}
\begin{center}
\vspace*{1mm}
 \vorder
\vspace*{1mm}
\end{center}
\end{minipage}&
\begin{minipage}[m]{.18\linewidth}
\begin{center}
Condition
\end{center}
\end{minipage}\\
\hline
\begin{minipage}[m]{.28\linewidth}

$O_z=\frac{1}{2}(1_x,1_y,1_t)$ \boundary

\end{minipage}&
\begin{minipage}[m]{.04\linewidth}
\begin{center}
$0$
\end{center}
\end{minipage}&
\begin{minipage}[m]{.11\linewidth}
\begin{center}
$B$
\end{center}
\end{minipage}&
\begin{minipage}[m]{.11\linewidth}
\begin{center}
$y$
\end{center}
\end{minipage}
&
\begin{minipage}[m]{.11\linewidth}
\begin{center}
$y$
\end{center}
\end{minipage}&
\begin{minipage}[m]{.18\linewidth}
\begin{center}
$a_1\ne0$
\end{center}
\end{minipage}\\
\hline
\begin{minipage}[m]{.28\linewidth}

$O_z=\frac{1}{2}(1_x,1_t,1_w)$ \boundary

\end{minipage}&
\begin{minipage}[m]{.04\linewidth}
\begin{center}
$0$
\end{center}
\end{minipage}&
\begin{minipage}[m]{.11\linewidth}
\begin{center}
$B-E$
\end{center}
\end{minipage}&
\begin{minipage}[m]{.11\linewidth}
\begin{center}
$y$
\end{center}
\end{minipage}
&
\begin{minipage}[m]{.11\linewidth}
\begin{center}
$w^3$
\end{center}
\end{minipage}&
\begin{minipage}[m]{.18\linewidth}
\begin{center}
$a_1=0$
\end{center}
\end{minipage}\\
\hline
\begin{minipage}[m]{.28\linewidth}

$O_tO_w=3\times\frac{1}{3}(1,1,2)$ \quadratic

\end{minipage}&
\multicolumn{4}{|l|}{\begin{minipage}[m]{.37\linewidth}
\begin{center}
$wt^2$
\end{center}
\end{minipage}}&
\begin{minipage}[m]{.18\linewidth}
\begin{center}

\end{center}
\end{minipage}\\

\hline
\end{longtable}
\end{center}

We may assume that neither $z^3w$ nor $xz^4$ appears in the
defining equation of $X_9$.

\begin{Note}

\item If $a_1\ne 0$, then the $1$-cycle $\Gamma$ for the singular point
$O_z$ is irreducible.

\item Suppose that $a_1=0$. Then $a_2\ne 0$. Then the $1$-cycle $\Gamma$
consists of three irreducible curves $\tilde{C}_i$, $i=1,2,3$,
each of which is the proper transform of the curve defined by
$$x=y=w-\alpha_i t=0.$$ One can easily check that
\[B\cdot\tilde{C}_i=-\frac{1}{3}, \ \ \ E\cdot\tilde{C}_i=1\]
for each $i$. Therefore, these three curves are numerically
equivalent to each other.

\item  For the singular points of type $\frac{1}{3}(1,1,2)$ we may assume
that $\alpha_3=0$ and we have only to consider the singular point
$O_t$.  The others can be untwisted or excluded in the same way.
Note that if $\alpha_3=0$ then we may assume that $wt^2$ is the
only monomial in the defining equation of $X_9$ divisible by
$t^2$.
\end{Note}



\begin{center}
\begin{longtable}{|l|c|c|c|c|c|}
\hline
\multicolumn{6}{|l|}{\underline{\textbf{No. 10}}: $X_{10}\subset\mathbb{P}(1,1,1,3,5)$\hfill $A^3=2/3$}\\
\multicolumn{6}{|l|}{
\begin{minipage}[m]{.86\linewidth}
\vspace*{1.2mm} $w^2+zt^3+wf_{5}(x,y,z,t)+ f_{10}(x,y,z,t)$
\vspace*{1.2mm}
\end{minipage}
}\\
\hline \hline
\begin{minipage}[m]{.28\linewidth}
\begin{center}
 Singularity
\end{center}
\end{minipage}&
\begin{minipage}[m]{.04\linewidth}
\begin{center}
$B^3$
\end{center}
\end{minipage}&
\begin{minipage}[m]{.11\linewidth}
\begin{center}
 Linear

system
\end{center}
\end{minipage}&
\begin{minipage}[m]{.11\linewidth}
\begin{center}
Surface $T$
\end{center}
\end{minipage}&
\begin{minipage}[m]{.11\linewidth}
\begin{center}
\vspace*{1mm} \vorder \vspace*{1mm}
\end{center}
\end{minipage}&
\begin{minipage}[m]{.18\linewidth}
\begin{center}
Condition
\end{center}
\end{minipage}\\
\hline
\begin{minipage}[m]{.28\linewidth}

$O_t=\frac{1}{3}(1_x,1_y,2_w)$ $\positive$

\end{minipage}&
\begin{minipage}[m]{.04\linewidth}
\begin{center}
$+$
\end{center}
\end{minipage}&
\begin{minipage}[m]{.11\linewidth}
\begin{center}
$B-E$
\end{center}
\end{minipage}&
\begin{minipage}[m]{.11\linewidth}
\begin{center}
$z$
\end{center}
\end{minipage}
&
\begin{minipage}[m]{.11\linewidth}
\begin{center}
$w^2$
\end{center}
\end{minipage}&
\begin{minipage}[m]{.18\linewidth}
\begin{center}

\end{center}
\end{minipage}\\
\hline

\end{longtable}
\end{center}


\begin{center}
\begin{longtable}{|l|c|c|c|c|c|}
\hline
\multicolumn{6}{|l|}{\underline{\textbf{No. 11}}: $X_{10}\subset\mathbb{P}(1,1,2,2,5)$\hfill $A^3=1/2$}\\
\multicolumn{6}{|l|}{
\begin{minipage}[m]{.86\linewidth}
\vspace*{1.2mm} $w^2+\prod_{i=1}^5(t-\alpha_i
z)+wf_{5}(x,y,z,t)+f_{10}(x,y,z,t)$ \vspace*{1.2mm}
\end{minipage}
}\\
\hline \hline
\begin{minipage}[m]{.28\linewidth}
\begin{center}
Singularity
\end{center}
\end{minipage}&
\begin{minipage}[m]{.04\linewidth}
\begin{center}
$B^3$
\end{center}
\end{minipage}&
\begin{minipage}[m]{.11\linewidth}
\begin{center}
Linear

system
\end{center}
\end{minipage}&
\begin{minipage}[m]{.11\linewidth}
\begin{center}
Surface $T$
\end{center}
\end{minipage}&
\begin{minipage}[m]{.11\linewidth}
\begin{center}
\vspace*{1mm}
 \vorder
\vspace*{1mm}
\end{center}
\end{minipage}&
\begin{minipage}[m]{.18\linewidth}
\begin{center}
Condition
\end{center}
\end{minipage}\\
\hline
\begin{minipage}[m]{.28\linewidth}

$O_zO_t=5\times\frac{1}{2}(1_x,1_y,1_w)$ \boundary

\end{minipage}&
\begin{minipage}[m]{.04\linewidth}
\begin{center}
$0$
\end{center}
\end{minipage}&
\begin{minipage}[m]{.11\linewidth}
\begin{center}
$B$
\end{center}
\end{minipage}&
\begin{minipage}[m]{.11\linewidth}
\begin{center}
$y$
\end{center}
\end{minipage}
&
\begin{minipage}[m]{.11\linewidth}
\begin{center}
$y$
\end{center}
\end{minipage}&
\begin{minipage}[m]{.18\linewidth}
\begin{center}

\end{center}
\end{minipage}\\
\hline

\end{longtable}
\end{center}

\begin{Note}
\item
The curve defined by $x=y=0$ is irreducible since the defining
polynomial of $X_{10}$ contains the monomial $w^2$ and a reduced
polynomial $\prod_{i=1}^5(t-\alpha_i z)$ of degree $10$.
Therefore, the $1$-cycle $\Gamma$ is irreducible.
\end{Note}



\begin{center}
\begin{longtable}{|l|c|c|c|c|c|}
\hline
\multicolumn{6}{|l|}{\textbf{No. 12}: $X_{10}\subset\mathbb{P}(1,1,2,3,4)$\hfill $A^3=5/12$}\\
\multicolumn{6}{|l|}{
\begin{minipage}[m]{.86\linewidth}
\vspace*{1.2mm} $z(w-\alpha_1 z^2)(w-\alpha_2
z^2)+t^2(a_1w+a_2yt)+cz^2t^2+wf_6(x,y,z,t)+f_{10}(x,y,z,t)$

\vspace*{1.2mm}
\end{minipage}
}\\
\hline \hline
\begin{minipage}[m]{.28\linewidth}
\begin{center}
Singularity
\end{center}
\end{minipage}&
\begin{minipage}[m]{.04\linewidth}
\begin{center}
$B^3$
\end{center}
\end{minipage}&
\begin{minipage}[m]{.11\linewidth}
\begin{center}
Linear

system
\end{center}
\end{minipage}&
\begin{minipage}[m]{.11\linewidth}
\begin{center}
Surface $T$
\end{center}
\end{minipage}&
\begin{minipage}[m]{.11\linewidth}
\begin{center}
\vspace*{1mm}
 \vorder
\vspace*{1mm}
\end{center}
\end{minipage}&
\begin{minipage}[m]{.18\linewidth}
\begin{center}
Condition
\end{center}
\end{minipage}\\
\hline
\begin{minipage}[m]{.28\linewidth}

$O_w=\frac{1}{4}(1,1,3)$ \quadratic

\end{minipage}&
\multicolumn{4}{|l|}{\begin{minipage}[m]{.37\linewidth}
\begin{center}
$zw^2$
\end{center}
\end{minipage}}&
\begin{minipage}[m]{.18\linewidth}
\begin{center}
\end{center}
\end{minipage}\\
\hline
\begin{minipage}[m]{.28\linewidth}

$O_t=\frac{1}{3}(1,1,2)$ \quadraticone

\end{minipage}&
\multicolumn{4}{|l|}{\begin{minipage}[m]{.37\linewidth}
\begin{center}
$wt^2$
\end{center}
\end{minipage}}&
\begin{minipage}[m]{.18\linewidth}
\begin{center}

\end{center}
\end{minipage}\\
\hline
\begin{minipage}[m]{.28\linewidth}

$O_zO_w=2\times\frac{1}{2}(1_x,1_y,1_t)$ \boundary

\end{minipage}&
\begin{minipage}[m]{.04\linewidth}
\begin{center}
$-$
\end{center}
\end{minipage}&
\begin{minipage}[m]{.11\linewidth}
\begin{center}
$B$
\end{center}
\end{minipage}&
\begin{minipage}[m]{.11\linewidth}
\begin{center}
$y$
\end{center}
\end{minipage}
&
\begin{minipage}[m]{.11\linewidth}
\begin{center}
$y$
\end{center}
\end{minipage}&
\begin{minipage}[m]{.18\linewidth}
\begin{center}
$c\ne0$, $a_1\ne 0$
\end{center}
\end{minipage}\\
\hline
\begin{minipage}[m]{.28\linewidth}

$O_zO_w=2\times\frac{1}{2}(1_x,1_y,1_t)$ $\surface$

\end{minipage}&
\begin{minipage}[m]{.04\linewidth}
\begin{center}
$-$
\end{center}
\end{minipage}&
\begin{minipage}[m]{.11\linewidth}
\begin{center}
$B$
\end{center}
\end{minipage}&
\begin{minipage}[m]{.11\linewidth}
\begin{center}
$x$, $y$
\end{center}
\end{minipage}
&
\begin{minipage}[m]{.11\linewidth}
\begin{center}
$x,y$
\end{center}
\end{minipage}&
\begin{minipage}[m]{.18\linewidth}
\begin{center}
$c\ne0$, $a_1=0$
\end{center}
\end{minipage}\\
\hline
\begin{minipage}[m]{.28\linewidth}

$O_zO_w=2\times\frac{1}{2}(1_x,1_y,1_t)$ \family

\end{minipage}&
\begin{minipage}[m]{.04\linewidth}
\begin{center}
$-$
\end{center}
\end{minipage}&
\begin{minipage}[m]{.11\linewidth}
\begin{center}
$B$
\end{center}
\end{minipage}&
\begin{minipage}[m]{.11\linewidth}
\begin{center}
$x$, $y$
\end{center}
\end{minipage}
&
\begin{minipage}[m]{.11\linewidth}
\begin{center}
$ x, y$
\end{center}
\end{minipage}&
\begin{minipage}[m]{.18\linewidth}
\begin{center}
$c=0$
\end{center}
\end{minipage}\\
\hline
\end{longtable}
\end{center}

By a coordinate change we assume that $\alpha_1=0$. Furthermore we
may assume that the monomials $z^3xt$, $z^3yt$, $z^4x^2$, $z^4xy$,
$z^4y^2$ do not appear in the defining equation by changing the
coordinate $w$ in an appropriate way. We  may also assume that
$xt^3$ is not contained in $f_{10}$.

\begin{Note}

\item For the singular points of type $\frac{1}{2}(1,1,1)$ with $c\ne 0$
and $a_1\ne 0$ the $1$-cycle $\Gamma$ is irreducible due to the
monomials $zw^2$, $t^2w$ and $z^2t^2$.

\item For the singular points of type $\frac{1}{2}(1,1,1)$ with $c\ne 0$
and $a_1= 0$ choose a general surface $H$ in $|-K_{X_{10}}|$ and
then let $T$ be the proper transform of the surface $H$. The
surface $H$ is a K3 surface only with du Val singularities.  The
intersection of $T$ with the surface $S$ gives us a divisor
consisting of two irreducible curves on the normal surface $T$.
One is the proper transform  of the curve $L_{tw}$. The other is
the proper transform of the curve $C$ defined by $$x=y=w^2-\alpha_2
z^2w+czt^2=0$$ in $\mathbb{P}(1,1,2,3,4)$. Since we have
\[\tilde{L}_{tw}^2=-\frac{7}{12}, \ \ \ \tilde{L}_{tw}\cdot\tilde{C}=\frac{2}{3}, \ \ \ \tilde{C}^2=-\frac{5}{6}\]
the curves $\tilde{L}_{tw}$ and $\tilde{C}$  are
negative-definite.

We remark here that the surface obtained from $T$ by contracting
the two curves $\tilde{L}_{tw}$ and $\tilde{C}$ is a K3 surface
only with one $E_8$ singular point. Indeed, the surface $T$ has
one $A_1$ singular point on $\tilde{C}$, one $A_3$ singular point
on $\tilde{L}_{tw}$ and the curves $\tilde{C}$, $\tilde{L}_{tw}$
intersect at one $A_2$ singular point tangentially on an orbifold
chart. Therefore, on the minimal resolution of the surface $T$,
the proper transforms of the curves $\tilde{C}$, $\tilde{L}_{tw}$
with the exceptional curves over three du Val points form the
configuration of the $-2$-curves for an $E_8$ singular point.

\item For the singular points of type $\frac{1}{2}(1,1,1)$ with $c=0$ we
may assume that $\alpha_1=0$ and we have only to consider the
point $O_z$. The other singular point can be treated in the same
way by a suitable coordinate change.
   The
quasi-smoothness implies that $a_1=0$ and $a_2\ne 0$. Let
$Z_{\lambda, \mu}$ be the curve on $X_{10}$ cut out  by
$$
\left\{%
\aligned
&y=\lambda x,\\%
&w=\mu x^4\\%
\endaligned\right.%
$$
for some sufficiently general complex numbers $\lambda$ and $\mu$.
Then $Z_{\lambda, \mu}=L_{zt}+C_{\lambda, \mu}$, where
$C_{\lambda, \mu}$ is an irreducible and reduced curve whose
normalisation is an elliptic curve. Indeed, the curve $C_{\lambda,
\mu}$ is defined by  $$y-\lambda x=w-\mu x^4=\mu^2x^7z-\alpha_2\mu
x^3z^3+\lambda a_2t^3+\mu x^3 f_6(x,\lambda
x,z,t)+f_{10}(x,\lambda x,z,t)=0.$$ Then
$$
\left\{%
\aligned
&-K_{Y}\cdot (\tilde{L}_{zt}+\tilde{C}_{\lambda, \mu})=4B^3=-\frac{1}{3},\\%
&-K_Y\cdot \tilde{L}_{zt}=-K_X\cdot L_{zt}-\frac{1}{2}E\cdot\tilde{L}_{zt}=-\frac{1}{3},\\%
\endaligned\right.%
$$
and hence  $-K_{Y}\cdot \tilde{C}_{\lambda, \mu}=0$.

\end{Note}


\begin{center}
\begin{longtable}{|l|c|c|c|c|c|}
\hline
\multicolumn{6}{|l|}{\textbf{No. 13}: $X_{11}\subset\mathbb{P}(1,1,2,3,5)$\hfill $A^3=11/30$}\\
\multicolumn{6}{|l|}{
\begin{minipage}[m]{.86\linewidth}
\vspace*{1.2mm}

$yw^2+t^2(a_1w+a_2zt)+z^3(b_1w+b_2zt+b_3xz^2+b_4yz^2)+wf_{6}(x,y,z,t)+f_{11}(x,y,z,t)$
\vspace*{1.2mm}
\end{minipage}
}\\
\hline \hline
\begin{minipage}[m]{.28\linewidth}
\begin{center}
Singularity
\end{center}
\end{minipage}&
\begin{minipage}[m]{.04\linewidth}
\begin{center}
$B^3$
\end{center}
\end{minipage}&
\begin{minipage}[m]{.11\linewidth}
\begin{center}
Linear

system
\end{center}
\end{minipage}&
\begin{minipage}[m]{.11\linewidth}
\begin{center}
Surface $T$
\end{center}
\end{minipage}&
\begin{minipage}[m]{.11\linewidth}
\begin{center}
\vspace*{1mm}
 \vorder
\vspace*{1mm}
\end{center}
\end{minipage}&
\begin{minipage}[m]{.18\linewidth}
\begin{center}
Condition
\end{center}
\end{minipage}\\
\hline
\begin{minipage}[m]{.28\linewidth}

$O_w=\frac{1}{5}(1,2,3)$ \quadratic

\end{minipage}&
\multicolumn{4}{|l|}{\begin{minipage}[m]{.37\linewidth}
\begin{center}
$yw^2$
\end{center}
\end{minipage}}&
\begin{minipage}[m]{.18\linewidth}
\begin{center}

\end{center}
\end{minipage}\\
\hline
\begin{minipage}[m]{.28\linewidth}

$O_t=\frac{1}{3}(1,1,2)$ \quadraticone

\end{minipage}&
\multicolumn{4}{|l|}{\begin{minipage}[m]{.37\linewidth}
\begin{center}
$wt^2$
\end{center}
\end{minipage}}&
\begin{minipage}[m]{.18\linewidth}
\begin{center}

\end{center}
\end{minipage}\\

\hline
\begin{minipage}[m]{.28\linewidth}

$O_z=\frac{1}{2}(1,1,1)$ \boundary

\end{minipage}&
\begin{minipage}[m]{.04\linewidth}
\begin{center}
$-$
\end{center}
\end{minipage}&
\begin{minipage}[m]{.11\linewidth}
\begin{center}
$B$
\end{center}
\end{minipage}&
\begin{minipage}[m]{.11\linewidth}
\begin{center}
$y$
\end{center}
\end{minipage}
&
\begin{minipage}[m]{.11\linewidth}
\begin{center}
$y$
\end{center}
\end{minipage}&
\begin{minipage}[m]{.18\linewidth}
\begin{center}
$a_1\not = 0$, $b_1\not = 0$\\ $a_1b_2-a_2b_1\not = 0$
\end{center}
\end{minipage}\\
\hline
\begin{minipage}[m]{.28\linewidth}

$O_z=\frac{1}{2}(1_x,1_y,1_t)$ $\surface$

\end{minipage}&
\begin{minipage}[m]{.04\linewidth}
\begin{center}
$-$
\end{center}
\end{minipage}&
\begin{minipage}[m]{.11\linewidth}
\begin{center}
$B$
\end{center}
\end{minipage}&
\begin{minipage}[m]{.11\linewidth}
\begin{center}
$x$, $y$
\end{center}
\end{minipage}
&
\begin{minipage}[m]{.11\linewidth}
\begin{center}
$x,y$
\end{center}
\end{minipage}&
\begin{minipage}[m]{.18\linewidth}
\begin{center}
$a_1\not = 0$, $b_1\not = 0$\\ $a_1b_2-a_2b_1 = 0$
\end{center}
\end{minipage}\\
\hline
\begin{minipage}[m]{.28\linewidth}

$O_z=\frac{1}{2}(1_x,1_y,1_w)$  \family

\end{minipage}&
\begin{minipage}[m]{.04\linewidth}
\begin{center}
$-$
\end{center}
\end{minipage}&
\begin{minipage}[m]{.11\linewidth}
\begin{center}
$B$
\end{center}
\end{minipage}&
\begin{minipage}[m]{.11\linewidth}
\begin{center}
$x$, $y$
\end{center}
\end{minipage}
&
\begin{minipage}[m]{.11\linewidth}
\begin{center}
$x, y$
\end{center}
\end{minipage}&
\begin{minipage}[m]{.18\linewidth}
\begin{center}
$a_1\not = 0$ \\ $b_1= 0$, $b_2\not = 0$
\end{center}
\end{minipage}\\
\hline
\begin{minipage}[m]{.28\linewidth}

$O_z=\frac{1}{2}(1,1,1)$ \family

\end{minipage}&
\begin{minipage}[m]{.04\linewidth}
\begin{center}
$-$
\end{center}
\end{minipage}&
\begin{minipage}[m]{.11\linewidth}
\begin{center}
$B$
\end{center}
\end{minipage}&
\begin{minipage}[m]{.11\linewidth}
\begin{center}
$y$
\end{center}
\end{minipage}
&
\begin{minipage}[m]{.11\linewidth}
\begin{center}
$y$
\end{center}
\end{minipage}&
\begin{minipage}[m]{.18\linewidth}
\begin{center}
$a_1\not = 0$ \\ $b_1 = 0$, $b_2= 0$
\end{center}
\end{minipage}\\
\hline
\begin{minipage}[m]{.28\linewidth}

$O_z=\frac{1}{2}(1_x,1_y,1_t)$ $\surface$

\end{minipage}&
\begin{minipage}[m]{.04\linewidth}
\begin{center}
$-$
\end{center}
\end{minipage}&
\begin{minipage}[m]{.11\linewidth}
\begin{center}
$B$
\end{center}
\end{minipage}&
\begin{minipage}[m]{.11\linewidth}
\begin{center}
$x$, $y$
\end{center}
\end{minipage}
&
\begin{minipage}[m]{.11\linewidth}
\begin{center}
$x,y$
\end{center}
\end{minipage}&
\begin{minipage}[m]{.18\linewidth}
\begin{center}
$a_1 = 0$ \\ $b_1\not = 0$
\end{center}
\end{minipage}\\


\hline
\begin{minipage}[m]{.28\linewidth}

$O_z=\frac{1}{2}(1_x,1_y,1_w)$ \family

\end{minipage}&
\begin{minipage}[m]{.04\linewidth}
\begin{center}
$-$
\end{center}
\end{minipage}&
\begin{minipage}[m]{.11\linewidth}
\begin{center}
$B$
\end{center}
\end{minipage}&
\begin{minipage}[m]{.11\linewidth}
\begin{center}
$x$, $y$
\end{center}
\end{minipage}
&
\begin{minipage}[m]{.11\linewidth}
\begin{center}
$x, y$
\end{center}
\end{minipage}&
\begin{minipage}[m]{.18\linewidth}
\begin{center}
$a_1 = 0$ \\ $b_1 = 0$, $b_2\not = 0$
\end{center}
\end{minipage}\\
\hline
\begin{minipage}[m]{.28\linewidth}

$O_z=\frac{1}{2}(1,1,1)$ \family

\end{minipage}&
\begin{minipage}[m]{.04\linewidth}
\begin{center}
$-$
\end{center}
\end{minipage}&
\begin{minipage}[m]{.11\linewidth}
\begin{center}
$B$
\end{center}
\end{minipage}&
\begin{minipage}[m]{.11\linewidth}
\begin{center}
$y$
\end{center}
\end{minipage}
&
\begin{minipage}[m]{.11\linewidth}
\begin{center}
$y$
\end{center}
\end{minipage}&
\begin{minipage}[m]{.18\linewidth}
\begin{center}
$a_1 = 0$ \\ $b_1 = 0$, $b_2 = 0$
\end{center}
\end{minipage}\\

\hline
\end{longtable}
\end{center}

 To exclude the singular point $O_z$ we first suppose that $a_1\ne
0$. We may then assume that $a_1=1$ and $a_2=0$.

\begin{Note}

\item
The conditions $b_1\ne 0$ and $a_1b_2-a_2b_1\ne 0$ imply that both
$b_1$ and $b_2$ are non-zero.  In such a case  the $1$-cycle
$\Gamma$  is irreducible since we have the monomials $t^2w$,
$z^3w$ and $z^4t$.

\item The conditions $b_1\ne 0$ and $a_1b_2-a_2b_1= 0$ imply that
$b_1\ne 0$ and $b_2=0$.   In such a case we take a general surface
$H$ from the pencil $|-K_{X_{11}}|$ and then let $T$ be the proper
transform of the surface.  The intersection of $T$ with the
surface $S$ defines a divisor consisting of two irreducible curves
on the normal surface $T$. One is the proper transform of the
curve $L_{zt}$ on $H$. The other is the proper transform of the
curve $C$ defined by $$x=y=t^2+b_1z^3=0$$ in
$\mathbb{P}(1,1,2,3,5)$. Since
\[\tilde{L}_{zt}^2=-\frac{4}{3}, \ \ \ \tilde{L}_{zt}\cdot\tilde{C}=1, \ \ \ \tilde{C}^2=-\frac{4}{5}\]
the curves $\tilde{L}_{zt}$ and $\tilde{C}$ on the normal surface
$T$  are negative-definite.

\item In the case when $b_1=0$ and $b_2\not =0$ we may assume that
$b_2=1$, $b_3=b_4=0$ by a suitable coordinate change. Let
$Z_{\lambda, \mu}$ be the curve  on $X_{11}$ cut out by
$$
\left\{%
\aligned
&y=\lambda x,\\%
&t=\mu x^3\\%
\endaligned\right.%
$$
for some sufficiently general complex numbers $\lambda$ and $\mu$.
Then $Z_{\lambda, \mu}=L_{zw}+C_{\lambda, \mu}$, where
$C_{\lambda, \mu}$ is an irreducible and reduced curve.  Then
$$
\left\{%
\aligned
&-K_{Y}\cdot (\tilde{L}_{zw}+\tilde{C}_{\lambda, \mu})=3B^3=-\frac{2}{5},\\%
&-K_Y\cdot \tilde{L}_{zw}=-K_X\cdot L_{zw}-\frac{1}{2}E\cdot\tilde{L}_{zw}=-\frac{2}{5},\\%
\endaligned\right.%
$$
and hence  $-K_{Y}\cdot \tilde{C}_{\lambda, \mu}=0$.

\item In the case when $b_1=b_2=0$, we must have $b_3\ne 0$ since
$X_{11}$ is quaisy-smooth. We may assume that $b_3=1 $ and $b_4=0$
by a suitable coordinate change. Let $Z_{\lambda}$ be the curve on
the surface $S_x$ defined by
$$
\left\{%
\aligned
&x=0,\\%
&t=\lambda y^3\\%
\endaligned\right.%
$$
for a sufficiently general complex number  $\lambda$. Then
$Z_{\lambda}=L_{zw}+C_{\lambda }$, where $C_{\lambda}$ is an
irreducible and  reduced curve.  Then
$$
\left\{%
\aligned
&-K_{Y}\cdot (\tilde{L}_{zw}+\tilde{C}_{\lambda})=(B-E)(3B+E)B=-\frac{2}{5},\\%
&-K_Y\cdot \tilde{L}_{zw}=-K_X\cdot L_{zw}-\frac{1}{2}E\cdot\tilde{L}_{zw}=-\frac{2}{5},\\%
\endaligned\right.%
$$
and hence  $-K_{Y}\cdot \tilde{C}_{\lambda}=0$.

\end{Note}

Now we suppose that $a_1=0$. Then $a_2\not =0$, so that we could
assume that $a_2=1$.

\begin{Note}

\item Suppose that $b_1\not =0$. Then by a suitable coordinate change we
may assume that $ b_1=1$ and $b_2=0$. We take a general surface
$H$ from the pencil $|-K_{X_{11}}|$ and then let $T$ be the proper
transform of the surface.   The intersection of $T$ with $S$ gives
us a divisor consisting of two irreducible curves on  $T$. One is
the proper transform of the curve $L_{tw}$ on $H$. The other is
the proper transform
 of the curve $C$ defined by $$x=y=t^3+z^2w=0$$ in
$\mathbb{P}(1,1,2,3,5)$. Since
\[\tilde{L}_{tw}^2=-\frac{8}{15}, \ \ \ \tilde{L}_{tw}\cdot\tilde{C}=\frac{3}{5}, \ \ \ \tilde{C}^2=-\frac{4}{5}\]
the curves $\tilde{L}_{tw}$ and $\tilde{C}$  form a
negative-definite divisor on the normal surface $T$.

\item We suppose that $b_1=0$ and $b_2\not =0$. We then let $Z_{\lambda,
\mu}$ be the curve on $X_{11}$  cut out by
$$
\left\{%
\aligned
&y=\lambda x,\\%
&t=\mu x^3\\%
\endaligned\right.%
$$
for some sufficiently general complex numbers $\lambda$ and $\mu$.
Then $Z_{\lambda, \mu}=L_{zw}+C_{\lambda, \mu}$, where
$C_{\lambda, \mu}$ is an irreducible and reduced curve. We have
$$
\left\{%
\aligned
&-K_{Y}\cdot (\tilde{L}_{zw}+\tilde{C}_{\lambda, \mu})=3B^3=-\frac{2}{5},\\%
&-K_Y\cdot \tilde{L}_{zw}=-K_X\cdot L_{zw}-\frac{1}{2}E\cdot\tilde{L}_{zw}=-\frac{2}{5},\\%
\endaligned\right.%
$$
and hence  $-K_{Y}\cdot \tilde{C}_{\lambda, \mu}=0$.

\item Finally, we suppose that $b_1=0$ and $b_2=0$. Then  $b_3$ must be
non-zero since $X_{11}$ is quasi-smooth. We may assume that $b_3=1
$ and $b_4=0$ by a suitable coordinate change. Let $Z_{\lambda}$
be the curve on the surface $S$ defined by
$$
\left\{%
\aligned
&x=0,\\%
&t=\lambda y^3\\%
\endaligned\right.%
$$
for a sufficiently general complex number  $\lambda$. Then
$Z_{\lambda}=L_{zw}+C_{\lambda }$, where $C_{\lambda}$ is an
irreducible and reduced curve. We have

$$
\left\{%
\aligned
&-K_{Y}\cdot (\tilde{L}_{zw}+\tilde{C}_{\lambda})=(B-E)(3B+E)B=-\frac{2}{5},\\%
&-K_Y\cdot \tilde{L}_{zw}=-K_X\cdot L_{zw}-\frac{1}{2}E\cdot\tilde{L}_{zw}=-\frac{2}{5},\\%
\endaligned\right.%
$$
and hence  $-K_{Y}\cdot \tilde{C}_{\lambda}=0$.

\end{Note}


\begin{center}
\begin{longtable}{|l|c|c|c|c|c|}
\hline
\multicolumn{6}{|l|}{\underline{\textbf{No. 14}}: $X_{12}\subset\mathbb{P}(1,1,1,4,6)$\hfill $A^3=1/2$}\\
\multicolumn{6}{|l|}{
\begin{minipage}[m]{.86\linewidth}
\vspace*{1.2mm} $w^2+t^3+wf_{6}(x,y,z,t)+f_{12}(x,y,z,t)$
\vspace*{1.2mm}
\end{minipage}
}\\
\hline \hline
\begin{minipage}[m]{.28\linewidth}
\begin{center}
Singularity
\end{center}
\end{minipage}&
\begin{minipage}[m]{.04\linewidth}
\begin{center}
$B^3$
\end{center}
\end{minipage}&
\begin{minipage}[m]{.11\linewidth}
\begin{center}
Linear

system
\end{center}
\end{minipage}&
\begin{minipage}[m]{.11\linewidth}
\begin{center}
Surface $T$
\end{center}
\end{minipage}&
\begin{minipage}[m]{.11\linewidth}
\begin{center}
\vspace*{1mm}
 \vorder
\vspace*{1mm}
\end{center}
\end{minipage}&
\begin{minipage}[m]{.18\linewidth}
\begin{center}
Condition
\end{center}
\end{minipage}\\
\hline
\begin{minipage}[m]{.28\linewidth}

$O_tO_w=1\times\frac{1}{2}(1_x,1_y,1_z)$ \boundary

\end{minipage}&
\begin{minipage}[m]{.04\linewidth}
\begin{center}
$0$
\end{center}
\end{minipage}&
\begin{minipage}[m]{.11\linewidth}
\begin{center}
$B$
\end{center}
\end{minipage}&
\begin{minipage}[m]{.11\linewidth}
\begin{center}
$y$
\end{center}
\end{minipage}
&
\begin{minipage}[m]{.11\linewidth}
\begin{center}
$y$
\end{center}
\end{minipage}&
\begin{minipage}[m]{.18\linewidth}
\begin{center}

\end{center}
\end{minipage}\\
\hline

\end{longtable}
\end{center}

\begin{Note}

\item
The curve defined by $x=y=0$ is irreducible because we have the
monomials $w^2$ and $t^3$ in the quasi-homogenous polynomial
defining $X_{12}$. Therefore, the $1$-cycle $\Gamma$ is
irreducible.
\end{Note}


\begin{center}
\begin{longtable}{|l|c|c|c|c|c|}
\hline
\multicolumn{6}{|l|}{\textbf{No. 15}: $X_{12}\subset\mathbb{P}(1,1,2,3,6)$\hfill $A^3=1/3$}\\
\multicolumn{6}{|l|}{
\begin{minipage}[m]{.86\linewidth}
\vspace*{1.2mm} $(w-\alpha_1z^3)(w-\alpha_2z^3)+t^2w
+wf_6(x,y,z,t)+tg_{9}(x,y,z)+g_{12}(x,y,z)$ \vspace*{1.2mm}
\vspace*{1.2mm}
\end{minipage}
}\\
\hline \hline
\begin{minipage}[m]{.28\linewidth}
\begin{center}
Singularity
\end{center}
\end{minipage}&
\begin{minipage}[m]{.04\linewidth}
\begin{center}
$B^3$
\end{center}
\end{minipage}&
\begin{minipage}[m]{.11\linewidth}
\begin{center}
Linear

system
\end{center}
\end{minipage}&
\begin{minipage}[m]{.11\linewidth}
\begin{center}
Surface $T$
\end{center}
\end{minipage}&
\begin{minipage}[m]{.11\linewidth}
\begin{center}
\vspace*{1mm}
 \vorder
\vspace*{1mm}
\end{center}
\end{minipage}&
\begin{minipage}[m]{.18\linewidth}
\begin{center}
Condition
\end{center}
\end{minipage}\\
\hline
\begin{minipage}[m]{.28\linewidth}

$O_tO_w=2\times\frac{1}{3}(1,1,2)$ \quadratic

\end{minipage}&
\multicolumn{4}{|l|}{\begin{minipage}[m]{.37\linewidth}
\begin{center}
$wt^2$
\end{center}
\end{minipage}}&
\begin{minipage}[m]{.18\linewidth}
\begin{center}

\end{center}
\end{minipage}\\
\hline
\begin{minipage}[m]{.28\linewidth}

$O_zO_w=2\times\frac{1}{2}(1_x,1_y,1_t)$ \boundary

\end{minipage}&
\begin{minipage}[m]{.04\linewidth}
\begin{center}
$-$
\end{center}
\end{minipage}&
\begin{minipage}[m]{.11\linewidth}
\begin{center}
$B$
\end{center}
\end{minipage}&
\begin{minipage}[m]{.11\linewidth}
\begin{center}
$y$
\end{center}
\end{minipage}
&
\begin{minipage}[m]{.11\linewidth}
\begin{center}
$y$
\end{center}
\end{minipage}&
\begin{minipage}[m]{.18\linewidth}
\begin{center}
 $\alpha_1\alpha_2\ne0$
\end{center}
\end{minipage}\\
\hline
\begin{minipage}[m]{.28\linewidth}

$O_zO_w=2\times\frac{1}{2}(1_x,1_y,1_t)$ \surface

\end{minipage}&
\begin{minipage}[m]{.04\linewidth}
\begin{center}
$-$
\end{center}
\end{minipage}&
\begin{minipage}[m]{.11\linewidth}
\begin{center}
$B$
\end{center}
\end{minipage}&
\begin{minipage}[m]{.11\linewidth}
\begin{center}
$x$, $y$
\end{center}
\end{minipage}
&
\begin{minipage}[m]{.11\linewidth}
\begin{center}
$x, y$
\end{center}
\end{minipage}&
\begin{minipage}[m]{.18\linewidth}
\begin{center}
$\alpha_1\alpha_2=0$
\end{center}
\end{minipage}\\

\hline
\end{longtable}
\end{center}

\begin{Note}

\item
To see how to deal with the singular
points of type $\frac{1}{3}(1,1,2)$ we have only to consider the
singular point $O_t$. The other point can be treated in the same
way after a suitable coordinate change.

\item The $1$-cycle $\Gamma$ for each singular point of type
$\frac{1}{2}(1,1,1)$ with $\alpha_1\alpha_2\ne 0$  is irreducible   since $(w-\alpha_1z^3)(w-\alpha_2z^3)+t^2w$ is irreducible.

\item For the singular points of type $\frac{1}{2}(1,1,1)$ with $\alpha_1\alpha_2 = 0$, we suppose that $\alpha_2=0$. Then $\alpha_1\ne 0$.
We take a general surface $H$ from the pencil $|-K_{X_{12}}|$.
The intersection of  its proper transform $T$ and the surface $S$
defines a divisor consisting of two irreducible curves on the
normal surface $T$. One is the proper transform  of the curve
$L_{zt}$. The other is the proper transform of the curve $C$
defined by $$x=y=w- \alpha_1 z^3+t^2=0.$$ From the intersection numbers
\[(\tilde{L}_{zt}+\tilde{C})\cdot\tilde{L}_{zt}=-K_Y\cdot\tilde{L}_{zt}=-\frac{1}{3},\ \ \ (\tilde{L}_{zt}+\tilde{C})^2=B^3=-\frac{1}{6}\]
on the surface $T$, we obtain
\[\tilde{L}_{zt}^2=-\frac{1}{3}-\tilde{L}_{zt}\cdot\tilde{C}, \ \ \ \tilde{C}^2=\frac{1}{6}-\tilde{L}_{zt}\cdot\tilde{C}\]
With these intersection numbers we see that the matrix
\[\left(\begin{array}{cc}
       \tilde{L}_{zt}^2&\tilde{L}_{zt}\cdot\tilde{C}\\
       \tilde{L}_{zt}\cdot\tilde{C}& \tilde{C}^2\\
\end{array}\right)= \left(\begin{array}{cc}
-\frac{1}{3}-\tilde{L}_{zt}\cdot\tilde{C} & \tilde{L}_{zt}\cdot\tilde{C}\\
     \tilde{L}_{zt}\cdot\tilde{C} & \frac{1}{6} -\tilde{L}_{zt}\cdot\tilde{C} \\
        \end{array}\right)
\]
is negative-definite since $\tilde{L}_{zt}\cdot\tilde{C}=1$.

\end{Note}

\begin{center}
\begin{longtable}{|l|c|c|c|c|c|}
\hline
\multicolumn{6}{|l|}{\textbf{No. 16}: $X_{12}\subset\mathbb{P}(1,1,2,4,5)$\hfill $A^3=3/10$}\\
\multicolumn{6}{|l|}{
\begin{minipage}[m]{.86\linewidth}
\vspace*{1.2mm}

$zw^2+(t-\alpha_1z^2)(t-\alpha_2z^2)(t-\alpha_3z^2)+wf_{7}(x,y,z,t)+f_{12}(x,y,z,t)$
\vspace*{1.2mm}
\end{minipage}
}\\
\hline \hline
\begin{minipage}[m]{.28\linewidth}
\begin{center}
Singularity
\end{center}
\end{minipage}&
\begin{minipage}[m]{.04\linewidth}
\begin{center}
$B^3$
\end{center}
\end{minipage}&
\begin{minipage}[m]{.11\linewidth}
\begin{center}
Linear

system
\end{center}
\end{minipage}&
\begin{minipage}[m]{.11\linewidth}
\begin{center}
Surface $T$
\end{center}
\end{minipage}&
\begin{minipage}[m]{.11\linewidth}
\begin{center}
\vspace*{1mm}
 \vorder
\vspace*{1mm}
\end{center}
\end{minipage}&
\begin{minipage}[m]{.18\linewidth}
\begin{center}
Condition
\end{center}
\end{minipage}\\
\hline
\begin{minipage}[m]{.28\linewidth}

$O_w=\frac{1}{5}(1,1,4)$ \quadratic

\end{minipage}&
\multicolumn{4}{|l|}{\begin{minipage}[m]{.37\linewidth}
\begin{center}
$zw^2$
\end{center}
\end{minipage}}&
\begin{minipage}[m]{.18\linewidth}
\begin{center}

\end{center}
\end{minipage}\\
\hline
\begin{minipage}[m]{.28\linewidth}

$O_zO_t=3\times\frac{1}{2}(1_x,1_y,1_w)$ \boundary

\end{minipage}&
\begin{minipage}[m]{.04\linewidth}
\begin{center}
$-$
\end{center}
\end{minipage}&
\begin{minipage}[m]{.11\linewidth}
\begin{center}
$B$
\end{center}
\end{minipage}&
\begin{minipage}[m]{.11\linewidth}
\begin{center}
$y$
\end{center}
\end{minipage}
&
\begin{minipage}[m]{.11\linewidth}
\begin{center}
$y$
\end{center}
\end{minipage}&
\begin{minipage}[m]{.18\linewidth}
\begin{center}

\end{center}
\end{minipage}\\

\hline
\end{longtable}
\end{center}

\begin{Note}

\item
The $1$-cycle $\Gamma$ for each singular point of type
$\frac{1}{2}(1,1,1)$ is irreducible  due to $zw^2$ and $t^3$.
\end{Note}



\begin{center}
\begin{longtable}{|l|c|c|c|c|c|}
\hline
\multicolumn{6}{|l|}{\textbf{No. 17}: $X_{12}\subset\mathbb{P}(1,1,3,4,4)$\hfill $A^3=1/4$}\\
\multicolumn{6}{|l|}{
\begin{minipage}[m]{.86\linewidth}
\vspace*{1.2mm}

$(t-\alpha_1w)(t-\alpha_2w)(t-\alpha_3w)+z^4+wf_{8}(x,y,z,t)+f_{12}(x,y,z,t)$
\vspace*{1.2mm}
\end{minipage}
}\\
\hline \hline
\begin{minipage}[m]{.28\linewidth}
\begin{center}
Singularity
\end{center}
\end{minipage}&
\begin{minipage}[m]{.04\linewidth}
\begin{center}
$B^3$
\end{center}
\end{minipage}&
\begin{minipage}[m]{.11\linewidth}
\begin{center}
Linear

system
\end{center}
\end{minipage}&
\begin{minipage}[m]{.11\linewidth}
\begin{center}
Surface $T$
\end{center}
\end{minipage}&
\begin{minipage}[m]{.11\linewidth}
\begin{center}
\vspace*{1mm}
 \vorder
\vspace*{1mm}
\end{center}
\end{minipage}&
\begin{minipage}[m]{.18\linewidth}
\begin{center}
Condition
\end{center}
\end{minipage}\\
\hline
\begin{minipage}[m]{.28\linewidth}

$O_tO_w=3\times\frac{1}{4}(1,1,3)$ \quadratic

\end{minipage}&
\multicolumn{4}{|l|}{\begin{minipage}[m]{.37\linewidth}
\begin{center}
$tw^2$
\end{center}
\end{minipage}}&
\begin{minipage}[m]{.18\linewidth}
\begin{center}

\end{center}
\end{minipage}\\

\hline
\end{longtable}
\end{center}

\begin{Note}

\item 
To see how to deal with the singular points of type
$\frac{1}{4}(1,1,3)$ we may assume that $\alpha_1=0$. We then
consider the singular point $O_w$. The other points can be treated
in the same way.
\end{Note}

\begin{center}
\begin{longtable}{|l|c|c|c|c|c|}
\hline
\multicolumn{6}{|l|}{\textbf{No. 18}: $X_{12}\subset\mathbb{P}(1,2,2,3,5)$\hfill $A^3=1/5$}\\
\multicolumn{6}{|l|}{
\begin{minipage}[m]{.86\linewidth}
\vspace*{1.2mm}
$yw^2+t^4+\prod_{i=1}^6(y-\alpha_iz)+wf_7(x,y,z,t)+f_{12}(x,y,z,t)$
\vspace*{1.2mm}
\end{minipage}
}\\
\hline \hline
\begin{minipage}[m]{.28\linewidth}
\begin{center}
Singularity
\end{center}
\end{minipage}&
\begin{minipage}[m]{.04\linewidth}
\begin{center}
$B^3$
\end{center}
\end{minipage}&
\begin{minipage}[m]{.11\linewidth}
\begin{center}
Linear

system
\end{center}
\end{minipage}&
\begin{minipage}[m]{.11\linewidth}
\begin{center}
Surface $T$
\end{center}
\end{minipage}&
\begin{minipage}[m]{.11\linewidth}
\begin{center}
\vspace*{1mm}
 \vorder
\vspace*{1mm}
\end{center}
\end{minipage}&
\begin{minipage}[m]{.18\linewidth}
\begin{center}
Condition
\end{center}
\end{minipage}\\
\hline
\begin{minipage}[m]{.28\linewidth}

$O_w=\frac{1}{5}(1,2,3)$ \quadratic

\end{minipage}&
\multicolumn{4}{|l|}{\begin{minipage}[m]{.37\linewidth}
\begin{center}
$yw^2$
\end{center}
\end{minipage}}&
\begin{minipage}[m]{.18\linewidth}
\begin{center}

\end{center}
\end{minipage}\\

\hline
\begin{minipage}[m]{.28\linewidth}

$O_yO_z=6\times\frac{1}{2}(1_x,1_t,1_w)$ \boundary
\end{minipage}&
\begin{minipage}[m]{.04\linewidth}
\begin{center}
$-$
\end{center}
\end{minipage}&
\begin{minipage}[m]{.11\linewidth}
\begin{center}
$2B$
\end{center}
\end{minipage}&
\begin{minipage}[m]{.11\linewidth}
\begin{center}
$y-\alpha_i z$
\end{center}
\end{minipage}
&
\begin{minipage}[m]{.11\linewidth}
\begin{center}
$zw^2$
\end{center}
\end{minipage}&
\begin{minipage}[m]{.18\linewidth}
\begin{center}

\end{center}
\end{minipage}\\

\hline
\end{longtable}
\end{center}

\begin{Note}

\item 

The $1$-cycle $\Gamma$ for each singular point of type
$\frac{1}{2}(1,1,1)$ is irreducible due to $yw^2$ and $t^4$.
\end{Note}


\begin{center}
\begin{longtable}{|l|c|c|c|c|c|}
\hline
\multicolumn{6}{|l|}{\underline{\textbf{No. 19}}: $X_{12}\subset\mathbb{P}(1,2,3,3,4)$\hfill $A^3=1/6$}\\
\multicolumn{6}{|l|}{
\begin{minipage}[m]{0.9\linewidth}
\vspace*{1.2mm} $(w-\alpha_1 y^2)(w-\alpha_2 y^2)(w-\alpha_3
y^2)+(z-\beta_1 t)(z-\beta_2 t)(z-\beta_3 t)(z-\beta_4
t)+w^2f_{4}(x,y,z,t)+wf_{8}(x,y,z,t)+f_{12}(x,y,z,t)$
\vspace*{1.2mm}
\end{minipage}
}\\
\hline \hline
\begin{minipage}[m]{.28\linewidth}
\begin{center}
Singularity
\end{center}
\end{minipage}&
\begin{minipage}[m]{.04\linewidth}
\begin{center}
$B^3$
\end{center}
\end{minipage}&
\begin{minipage}[m]{.11\linewidth}
\begin{center}
Linear

system
\end{center}
\end{minipage}&
\begin{minipage}[m]{.11\linewidth}
\begin{center}
Surface $T$
\end{center}
\end{minipage}&
\begin{minipage}[m]{.11\linewidth}
\begin{center}
\vspace*{1mm}
 \vorder
\vspace*{1mm}
\end{center}
\end{minipage}&
\begin{minipage}[m]{.18\linewidth}
\begin{center}
Condition
\end{center}
\end{minipage}\\
\hline
\begin{minipage}[m]{.28\linewidth}

$O_yO_w=3\times\frac{1}{2}(1_x,1_z,1_t)$ $\nef$

\end{minipage}&
\begin{minipage}[m]{.04\linewidth}
\begin{center}
$-$
\end{center}
\end{minipage}&
\begin{minipage}[m]{.11\linewidth}
\begin{center}
$3B+E$
\end{center}
\end{minipage}&
\begin{minipage}[m]{.11\linewidth}
\begin{center}
$xy$, $z$, $t$
\end{center}
\end{minipage}
&
\begin{minipage}[m]{.11\linewidth}
\begin{center}
$xy$, $z$, $t$
\end{center}
\end{minipage}&
\begin{minipage}[m]{.18\linewidth}
\begin{center}

\end{center}
\end{minipage}\\
\hline
\begin{minipage}[m]{.28\linewidth}

$O_zO_t=4\times\frac{1}{3}(1_x,2_y,1_w)$ \boundary

\end{minipage}&
\begin{minipage}[m]{.04\linewidth}
\begin{center}
$0$
\end{center}
\end{minipage}&
\begin{minipage}[m]{.11\linewidth}
\begin{center}
$2B$
\end{center}
\end{minipage}&
\begin{minipage}[m]{.11\linewidth}
\begin{center}
$y$
\end{center}
\end{minipage}
&
\begin{minipage}[m]{.11\linewidth}
\begin{center}
$y$
\end{center}
\end{minipage}&
\begin{minipage}[m]{.18\linewidth}
\begin{center}

\end{center}
\end{minipage}\\
\hline

\end{longtable}
\end{center}


\begin{Note}

\item 
The divisor $T$ for each singular point of type
$\frac{1}{2}(1,1,1)$ is nef since the linear system generated by
$xy$, $z$, $t$ has no base curve.

\item
The $1$-cycle $\Gamma$ for each singular point of type
$\frac{1}{3}(1,2,1)$ is irreducible since the curve cut by $x=y=0$
is irreducible.
\end{Note}





\begin{center}
\begin{longtable}{|l|c|c|c|c|c|}
\hline
\multicolumn{6}{|l|}{\textbf{No. 20}: $X_{13}\subset\mathbb{P}(1,1,3,4,5)$\hfill $A^3=13/60$}\\
\multicolumn{6}{|l|}{
\begin{minipage}[m]{.86\linewidth}
\vspace*{1.2mm}

$zw^2+t^2(a_1w+a_2yt)-z^3(b_1t+b_2yz+b_3xz)+wf_8(x,y,z,t)+f_{13}(x,y,z,t)$
\vspace*{1.2mm}
\end{minipage}
}\\
\hline \hline
\begin{minipage}[m]{.28\linewidth}
\begin{center}
Singularity
\end{center}
\end{minipage}&
\begin{minipage}[m]{.04\linewidth}
\begin{center}
$B^3$
\end{center}
\end{minipage}&
\begin{minipage}[m]{.11\linewidth}
\begin{center}
Linear

system
\end{center}
\end{minipage}&
\begin{minipage}[m]{.11\linewidth}
\begin{center}
Surface $T$
\end{center}
\end{minipage}&
\begin{minipage}[m]{.11\linewidth}
\begin{center}
\vspace*{1mm}
 \vorder
\vspace*{1mm}
\end{center}
\end{minipage}&
\begin{minipage}[m]{.18\linewidth}
\begin{center}
Condition
\end{center}
\end{minipage}\\
\hline
\begin{minipage}[m]{.28\linewidth}

$O_w=\frac{1}{5}(1,1,4)$ \quadratic

\end{minipage}&
\multicolumn{4}{|l|}{\begin{minipage}[m]{.37\linewidth}
\begin{center}
$zw^2$
\end{center}
\end{minipage}}&
\begin{minipage}[m]{.18\linewidth}
\begin{center}

\end{center}
\end{minipage}\\
\hline
\begin{minipage}[m]{.28\linewidth}

$O_t=\frac{1}{4}(1,1,3)$ \quadraticone

\end{minipage}&
\multicolumn{4}{|l|}{\begin{minipage}[m]{.37\linewidth}
\begin{center}
$wt^2$
\end{center}
\end{minipage}}&
\begin{minipage}[m]{.18\linewidth}
\begin{center}

\end{center}
\end{minipage}\\
\hline
\begin{minipage}[m]{.28\linewidth}

$O_z=\frac{1}{3}(1,1,2)$ \elliptictwo

\end{minipage}&\multicolumn{4}{|l|}{\begin{minipage}[m]{.37\linewidth}
\begin{center}
$zw^2-tz^3$
\end{center}
\end{minipage}}&
\begin{minipage}[m]{.18\linewidth}
\begin{center}
\end{center}
\end{minipage}\\

\hline
\end{longtable}
\end{center}



\begin{center}
\begin{longtable}{|l|c|c|c|c|c|}
\hline
\multicolumn{6}{|l|}{\underline{\textbf{No. 21}}: $X_{14}\subset\mathbb{P}(1,1,2,4,7)$\hfill $A^3=1/4$}\\
\multicolumn{6}{|l|}{
\begin{minipage}[m]{.86\linewidth}
\vspace*{1.2mm} $w^2+z(t-\alpha_1 z^2)(t-\alpha_2 z^2)(t-\alpha_3
z^2)+wf_7(x,y,z,t)+f_{14}(x,y,z,t)$ \vspace*{1.2mm}
\end{minipage}
}\\
\hline \hline
\begin{minipage}[m]{.28\linewidth}
\begin{center}
Singularity
\end{center}
\end{minipage}&
\begin{minipage}[m]{.04\linewidth}
\begin{center}
$B^3$
\end{center}
\end{minipage}&
\begin{minipage}[m]{.11\linewidth}
\begin{center}
Linear

system
\end{center}
\end{minipage}&
\begin{minipage}[m]{.11\linewidth}
\begin{center}
Surface $T$
\end{center}
\end{minipage}&
\begin{minipage}[m]{.11\linewidth}
\begin{center}
\vspace*{1mm}
 \vorder
\vspace*{1mm}
\end{center}
\end{minipage}&
\begin{minipage}[m]{.18\linewidth}
\begin{center}
Condition
\end{center}
\end{minipage}\\
\hline
\begin{minipage}[m]{.28\linewidth}

$O_t=\frac{1}{4}(1_x,1_y,3_w)$ $\positive$

\end{minipage}&
\begin{minipage}[m]{.04\linewidth}
\begin{center}
$+$
\end{center}
\end{minipage}&
\begin{minipage}[m]{.11\linewidth}
\begin{center}
$2B-E$
\end{center}
\end{minipage}&
\begin{minipage}[m]{.11\linewidth}
\begin{center}
$z$
\end{center}
\end{minipage}
&
\begin{minipage}[m]{.11\linewidth}
\begin{center}
$w^2$
\end{center}
\end{minipage}&
\begin{minipage}[m]{.18\linewidth}
\begin{center}

\end{center}
\end{minipage}\\
\hline
\begin{minipage}[m]{.28\linewidth}

$O_zO_t=3\times\frac{1}{2}(1_x,1_y,1_w)$ \boundary

\end{minipage}&
\begin{minipage}[m]{.04\linewidth}
\begin{center}
$-$
\end{center}
\end{minipage}&
\begin{minipage}[m]{.11\linewidth}
\begin{center}
$B$
\end{center}
\end{minipage}&
\begin{minipage}[m]{.11\linewidth}
\begin{center}
$y$
\end{center}
\end{minipage}
&
\begin{minipage}[m]{.11\linewidth}
\begin{center}
$y$
\end{center}
\end{minipage}&
\begin{minipage}[m]{.18\linewidth}
\begin{center}

\end{center}
\end{minipage}\\
\hline

\end{longtable}
\end{center}

\begin{Note}

\item 
The curve defined by $x=y=0$ is irreducible, and hence the
$1$-cycle $\Gamma$ for the singularities of type
$\frac{1}{2}(1,1,1)$ is irreducible.
\end{Note}


\begin{center}
\begin{longtable}{|l|c|c|c|c|c|}
\hline
\multicolumn{6}{|l|}{\underline{\textbf{No. 22}}: $X_{14}\subset\mathbb{P}(1,2,2,3,7)$\hfill $A^3=1/6$}\\
\multicolumn{6}{|l|}{
\begin{minipage}[m]{.86\linewidth}
\vspace*{1.2mm} $w^2+zt^4+h_{14}(y,z)+wf_{7}(x,y,z,t)+
t^3g_{5}(x,y,z)+t^2g_{8}(x,y,z)+tg_{11}(x,y,z)+g_{14}(x,y,z)$
\vspace*{1.2mm}
\end{minipage}
}\\
\hline \hline
\begin{minipage}[m]{.28\linewidth}
\begin{center}
Singularity
\end{center}
\end{minipage}&
\begin{minipage}[m]{.04\linewidth}
\begin{center}
$B^3$
\end{center}
\end{minipage}&
\begin{minipage}[m]{.11\linewidth}
\begin{center}
Linear

system
\end{center}
\end{minipage}&
\begin{minipage}[m]{.11\linewidth}
\begin{center}
Surface $T$
\end{center}
\end{minipage}&
\begin{minipage}[m]{.11\linewidth}
\begin{center}
\vspace*{1mm}
 \vorder
\vspace*{1mm}
\end{center}
\end{minipage}&
\begin{minipage}[m]{.18\linewidth}
\begin{center}
Condition
\end{center}
\end{minipage}\\
\hline
\begin{minipage}[m]{.28\linewidth}

$O_t=\frac{1}{3}(1_x,2_y,1_w)$ \boundary

\end{minipage}&
\begin{minipage}[m]{.04\linewidth}
\begin{center}
$0$
\end{center}
\end{minipage}&
\begin{minipage}[m]{.11\linewidth}
\begin{center}
$2B$
\end{center}
\end{minipage}&
\begin{minipage}[m]{.11\linewidth}
\begin{center}
$y$
\end{center}
\end{minipage}
&
\begin{minipage}[m]{.11\linewidth}
\begin{center}
$y$
\end{center}
\end{minipage}&
\begin{minipage}[m]{.18\linewidth}
\begin{center}

\end{center}
\end{minipage}\\
\hline
\begin{minipage}[m]{.28\linewidth}

$O_yO_z=7\times\frac{1}{2}(1_x,1_t,1_w)$ \boundary

\end{minipage}&
\begin{minipage}[m]{.04\linewidth}
\begin{center}
$-$
\end{center}
\end{minipage}&
\begin{minipage}[m]{.11\linewidth}
\begin{center}
$2B$
\end{center}
\end{minipage}&
\begin{minipage}[m]{.11\linewidth}
\begin{center}
$y-\alpha_iz$
\end{center}
\end{minipage}
&
\begin{minipage}[m]{.11\linewidth}
\begin{center}
$w^2$
\end{center}
\end{minipage}&
\begin{minipage}[m]{.18\linewidth}
\begin{center}

\end{center}
\end{minipage}\\
\hline

\end{longtable}
\end{center}

Note that the homogenous polynomial $h_{14}$ cannot be divisible
by $z$ since the hypersurface $X_{14}$ is quasi-smooth. Therefore,
we may write
\[h_{14}(y,z)=\prod_{i=1}^7(y-\alpha_i z). \]  

\begin{Note}

\item 
The curve defined by $x=y=0$ is irreducible
because we have the monomials $w^2$ and $zt^4$.  

\item 
The curves
defined by $x=y-\alpha_i z=0$ are also irreducible for
 the same reason. Therefore, the $1$-cycle $\Gamma$  for each singular point is irreducible.
\end{Note}

\begin{center}
\begin{longtable}{|l|c|c|c|c|c|}
\hline
\multicolumn{6}{|l|}{\textbf{No. 23}: $X_{14}\subset\mathbb{P}(1,2,3,4,5)$\hfill $A^3=7/60$}\\
\multicolumn{6}{|l|}{
\begin{minipage}[m]{.86\linewidth}
\vspace*{1.2mm} $(t+by^2)w^2+
y(t-\alpha_1y^2)(t-\alpha_2y^2)(t-\alpha_3y^2)+
z^3(a_1w+a_2yz)+cz^2t^2+wf_9(x,y,z,t)+f_{14}(x,y,z,t)$
\vspace*{1.2mm}
\end{minipage}
}\\
\hline \hline
\begin{minipage}[m]{.28\linewidth}
\begin{center}
Singularity
\end{center}
\end{minipage}&
\begin{minipage}[m]{.04\linewidth}
\begin{center}
$B^3$
\end{center}
\end{minipage}&
\begin{minipage}[m]{.11\linewidth}
\begin{center}
Linear

system
\end{center}
\end{minipage}&
\begin{minipage}[m]{.11\linewidth}
\begin{center}
Surface $T$
\end{center}
\end{minipage}&
\begin{minipage}[m]{.11\linewidth}
\begin{center}
\vspace*{1mm}
 \vorder
\vspace*{1mm}
\end{center}
\end{minipage}&
\begin{minipage}[m]{.18\linewidth}
\begin{center}
Condition
\end{center}
\end{minipage}\\
\hline
\begin{minipage}[m]{.28\linewidth}

$O_w=\frac{1}{5}(1,2,3)$ \quadratic

\end{minipage}&
\multicolumn{4}{|l|}{\begin{minipage}[m]{.37\linewidth}
\begin{center}
$tw^2$
\end{center}
\end{minipage}}&
\begin{minipage}[m]{.18\linewidth}
\begin{center}

\end{center}
\end{minipage}\\
\hline
\begin{minipage}[m]{.28\linewidth}

$O_t=\frac{1}{4}(1,3,1)$ \elliptic

\end{minipage}&
\multicolumn{4}{|l|}{\begin{minipage}[m]{.37\linewidth}
\begin{center}
$tw^2+yt^3$
\end{center}
\end{minipage}}&
\begin{minipage}[m]{.18\linewidth}
\begin{center}

\end{center}
\end{minipage}\\
\hline
\begin{minipage}[m]{.28\linewidth}

$O_z=\frac{1}{3}(1_x,2_y,1_t)$ \boundary

\end{minipage}&
\begin{minipage}[m]{.04\linewidth}
\begin{center}
$-$
\end{center}
\end{minipage}&
\begin{minipage}[m]{.11\linewidth}
\begin{center}
$2B$
\end{center}
\end{minipage}&
\begin{minipage}[m]{.11\linewidth}
\begin{center}
$y$
\end{center}
\end{minipage}
&
\begin{minipage}[m]{.11\linewidth}
\begin{center}
$y$
\end{center}
\end{minipage}&
\begin{minipage}[m]{.18\linewidth}
\begin{center}
$c\not =0$,  $a_1\not =0$

\end{center}
\end{minipage}\\

\hline
\begin{minipage}[m]{.28\linewidth}

$O_z=\frac{1}{3}(1_x,1_t, 2_w)$ $\surface$

\end{minipage}&
\begin{minipage}[m]{.04\linewidth}
\begin{center}
$-$
\end{center}
\end{minipage}&
\begin{minipage}[m]{.11\linewidth}
\begin{center}
$2B$
\end{center}
\end{minipage}&
\begin{minipage}[m]{.11\linewidth}
\begin{center}
$x^2$, $y$
\end{center}
\end{minipage}
&
\begin{minipage}[m]{.11\linewidth}
\begin{center}
$x^2, z^2t^2$
\end{center}
\end{minipage}&
\begin{minipage}[m]{.18\linewidth}
\begin{center}
$c\not =0$,  $a_1=0$
\end{center}
\end{minipage}\\
\hline
\begin{minipage}[m]{.28\linewidth}

$O_z=\frac{1}{3}(1_x,2_y,1_t)$ \family

\end{minipage}&
\begin{minipage}[m]{.04\linewidth}
\begin{center}
$-$
\end{center}
\end{minipage}&
\begin{minipage}[m]{.11\linewidth}
\begin{center}
$2B$
\end{center}
\end{minipage}&
\begin{minipage}[m]{.11\linewidth}
\begin{center}
$y$
\end{center}
\end{minipage}
&
\begin{minipage}[m]{.11\linewidth}
\begin{center}
$y$
\end{center}
\end{minipage}&
\begin{minipage}[m]{.18\linewidth}
\begin{center}
$c=0$, $a_1\not =0$
\end{center}
\end{minipage}\\
\hline
\begin{minipage}[m]{.28\linewidth}

$O_z=\frac{1}{3}(1_x,1_t,2_w)$ \ellipticfive

\end{minipage}&
\multicolumn{4}{|l|}{\begin{minipage}[m]{.37\linewidth}
\begin{center}

\end{center}
\end{minipage}}&
\begin{minipage}[m]{.18\linewidth}
\begin{center}
$c=0$,  $a_1=0$
\end{center}
\end{minipage}\\

\hline
\begin{minipage}[m]{.28\linewidth}

$O_yO_t=3\times\frac{1}{2}(1_x,1_z,1_w)$ \nef

\end{minipage}&
\begin{minipage}[m]{.04\linewidth}
\begin{center}
$-$
\end{center}
\end{minipage}&
\begin{minipage}[m]{.11\linewidth}
\begin{center}
$3B+E$
\end{center}
\end{minipage}&
\begin{minipage}[m]{.11\linewidth}
\begin{center}
$xy$, $z$
\end{center}
\end{minipage}
&
\begin{minipage}[m]{.11\linewidth}
\begin{center}
$xy$, $z$
\end{center}
\end{minipage}&
\begin{minipage}[m]{.18\linewidth}
\begin{center}
$b\ne 0$
\end{center}
\end{minipage}\\
\hline

\hline
\begin{minipage}[m]{.28\linewidth}

$O_yO_t=3\times\frac{1}{2}(1_x,1_z,1_w)$ \surface

\end{minipage}&
\begin{minipage}[m]{.04\linewidth}
\begin{center}
$-$
\end{center}
\end{minipage}&
\begin{minipage}[m]{.11\linewidth}
\begin{center}
$3B+E$
\end{center}
\end{minipage}&
\begin{minipage}[m]{.11\linewidth}
\begin{center}
$x^3$, $xy$, $z$
\end{center}
\end{minipage}
&
\begin{minipage}[m]{.11\linewidth}
\begin{center}
$xy$, $z$
\end{center}
\end{minipage}&
\begin{minipage}[m]{.18\linewidth}
\begin{center}
$b=0$
\end{center}
\end{minipage}\\
\hline

\end{longtable}
\end{center}

\begin{Note}

\item 
For the singular point $O_z$ with $c\not=0$ and $a_1\ne 0$ the
$1$-cycle $\Gamma$ is irreducible due to the monomials $tw^2$,
$z^3w$ and $z^2t^2$.

\item For the singular point $O_z$ with $c\not=0$ and $a_1= 0$ we may
assume that $a_2=1$ and $c=1$. We take a general surface $H$ from
the pencil $|-2K_{X_{14}}|$ and then let $T$ be the proper
transform of the surface.  The surface $H$ is normal. However, it
is not quasi-smooth at the points $O_z$ and $O_t$. The
intersection of $T$ with the surface $S$ defines a divisor
consisting of two irreducible curves on the normal surface $T$.
One is the proper transform of the curve $L_{zw}$ on $H$.
 The other is the proper transform of the curve $C$ defined by $$x=y=w^2+z^2t=0$$ in $\mathbb{P}(1,2,3,4,5)$.  From the
intersection numbers
\[(\tilde{L}_{zw}+\tilde{C})\cdot\tilde{L}_{zw}=-K_Y\cdot\tilde{L}_{zw}=-\frac{1}{10},\ \ \ (\tilde{L}_{zw}+\tilde{C})^2=2B^3=-\frac{1}{10}\]
on the surface $T$, we obtain
\[\tilde{L}_{zw}^2=-\frac{1}{10}-\tilde{L}_{zw}\cdot\tilde{C}, \ \ \ \tilde{C}^2=-\tilde{L}_{zw}\cdot\tilde{C}\]
With these intersection numbers we see that the matrix
\[\left(\begin{array}{cc}
       \tilde{L}_{zw}^2&\tilde{L}_{zw}\cdot\tilde{C}\\
       \tilde{L}_{zw}\cdot\tilde{C}& \tilde{C}^2\\
\end{array}\right)= \left(\begin{array}{cc}
-\frac{1}{10}-\tilde{L}_{zw}\cdot\tilde{C} & \tilde{L}_{zw}\cdot\tilde{C}\\
     \tilde{L}_{zw}\cdot\tilde{C} & -\tilde{L}_{zw}\cdot\tilde{C} \\
        \end{array}\right)
\]
is negative-definite since $\tilde{L}_{zw}\cdot\tilde{C}$ is
positive.

\item For the singular point $O_z$ with $c=0$ and $a_1\ne 0$ we may
assume that $a_1=1$ and $a_2=0$. Furthermore, we may also assume
that  $f_{14}$  does not contain the monomial $xz^3t$ by changing
the coordinate $w$ in a suitable way. We then consider the surface
$S_w$ cut by the equation $w=0$. Let $Z_{\lambda}$ be the curve on
the surface $S_w$ defined by
$$
\left\{%
\aligned
&w=0\\%
&y=\lambda x^2\\%
\endaligned\right.%
$$
for a sufficiently general complex number  $\lambda$. Then
$Z_{\lambda}=2L_{zt}+C_{\lambda }$, where $C_{\lambda}$ is an
irreducible and reduced curve. We have
$$
\left\{%
\aligned
&-K_{Y}\cdot (2\tilde{L}_{zt}+\tilde{C}_{\lambda})=10B^3=-\frac{1}{2},\\%
&-K_Y\cdot \tilde{L}_{zt}=-K_X\cdot L_{zt}-\frac{1}{3}E\cdot\tilde{L}_{zt}=-\frac{1}{4},\\%
\endaligned\right.%
$$
and hence  $-K_{Y}\cdot \tilde{C}_{\lambda}=0$.

\item For the singular point $O_z$ with $c=0$ and $a_1= 0$ we observe
that   $f_{14}$ must contain the monomial $xz^3t$ for $X_{14}$ to
be quasi-smooth (see right before
Theorem~\ref{theorem:amazing-23}). We may assume that $a_2=1$ and
that the coefficient of $xz^3t$ in $f_{14}$ is $1$.  Then
Theorem~\ref{theorem:amazing-23} untwists the singular point
$O_z$.

\end{Note}
 For the singular points of type $\frac{1}{2}(1,1,1)$ we may assume
that $\alpha_3=0$ and we have only to consider the singular  point
$O_y$. The other singular points can be treated in the same way
after suitable coordinate changes.

\begin{Note}

\item For the singular point $O_y$ with $b\ne 0$ consider the linear
system generated by $xy$ and $z$ on $X_{14}$. Its base curves are
defined by $x=z=0$ and $y=z=0$. The curve defined by $y=z=0$ does
not pass through the singular point $O_y$.  The curve   defined by
$x=z=0$   is irreducible. Indeed, the curve is  defined by
$$x=z=(t+by^2)w^2+yt(t-\alpha_1y^2)(t-\alpha_2y^2)=0.$$   Moreover,
its proper transform is equivalent to the $1$-cycle defined by
$(3B+E)\cdot B$. Therefore, the divisor $T$ is nef since
$(3B+E)^2\cdot B>0$.

\item For the singular point $O_y$ with $b=0$ we take a general member
$H$ in the linear system  generated by $x^3$, $xy$ and $z$. Note
that the defining equation of $X_{14}$ must contain either $y^3zw$
or $xy^4w$. The surface $H$ is a normal surface of degree $14$ in
$\mathbb{P}(1,2,4,5)$  that is smooth at the point
$x=t=w^2+\alpha_1\alpha_2y^5=0$. Let $T$ be the proper transform
of the surface $H$. The intersection of $T$ with the surface $S$
defines a divisor consisting of two irreducible curves on the
normal surface $T$. One is  the proper transform of the curve
$L_{yw}$
 and the other  is the proper
transform   of the curve $C$ defined by
$$x=z=w^2+y(t-\alpha_1y^2)(t-\alpha_2y^2)=0.$$  From the
intersection numbers
\[(\tilde{L}_{yw}+\tilde{C})\cdot\tilde{L}_{yw}=-K_Y\cdot \tilde{L}_{yw}=-\frac{2}{5}, \ \ \ (\tilde{L}_{yw}+\tilde{C})^2=B^2\cdot (3B+E)=-\frac{3}{20}\]
on the surface $T$, we obtain
\[\tilde{L}_{yw}^2=-\frac{2}{5}-\tilde{L}_{yw}\cdot\tilde{C}, \ \ \ \tilde{C}^2=\frac{1}{4}-\tilde{L}_{yw}\cdot\tilde{C}\]
With these intersection numbers we see that the matrix
\[\left(\begin{array}{cc}
       \tilde{L}_{yw}^2&\tilde{L}_{yw}\cdot\tilde{C}\\
       \tilde{L}_{yw}\cdot\tilde{C}& \tilde{C}^2\\
\end{array}\right)= \left(\begin{array}{cc}
-\frac{2}{5}-\tilde{L}_{yw}\cdot\tilde{C} & \tilde{L}_{yw}\cdot\tilde{C}\\
     \tilde{L}_{yw}\cdot\tilde{C} & \frac{1}{4}-\tilde{L}_{yw}\cdot\tilde{C} \\
        \end{array}\right)
\]
is negative-definite since the curves $L_{yw}$ and the curve $C$
intersect  at the smooth point of $H$ defined by
$x=z=t=w^2+\alpha_1\alpha_2y^5=0$.
\end{Note}


\begin{center}
\begin{longtable}{|l|c|c|c|c|c|}
\hline
\multicolumn{6}{|l|}{\textbf{No. 24}: $X_{15}\subset\mathbb{P}(1,1,2,5,7)$\hfill $A^3=3/14$}\\
\multicolumn{6}{|l|}{
\begin{minipage}[m]{.86\linewidth}
\vspace*{1.2mm}
$yw^2+t^3+z^4(a_1w+a_2zt+a_3xz^3+a_4yz^3)+wf_{8}(x,y,z,t)+f_{15}(x,y,z,t)$

\vspace*{1.2mm}
\end{minipage}
}\\
\hline \hline
\begin{minipage}[m]{.28\linewidth}
\begin{center}
Singularity
\end{center}
\end{minipage}&
\begin{minipage}[m]{.04\linewidth}
\begin{center}
$B^3$
\end{center}
\end{minipage}&
\begin{minipage}[m]{.11\linewidth}
\begin{center}
Linear

system
\end{center}
\end{minipage}&
\begin{minipage}[m]{.11\linewidth}
\begin{center}
Surface $T$
\end{center}
\end{minipage}&
\begin{minipage}[m]{.11\linewidth}
\begin{center}
\vspace*{1mm}
 \vorder
\vspace*{1mm}
\end{center}
\end{minipage}&
\begin{minipage}[m]{.18\linewidth}
\begin{center}
Condition
\end{center}
\end{minipage}\\
\hline
\begin{minipage}[m]{.28\linewidth}

$O_w=\frac{1}{7}(1,2,5)$ \quadratic

\end{minipage}&
\multicolumn{4}{|l|}{\begin{minipage}[m]{.37\linewidth}
\begin{center}
$yw^2$
\end{center}
\end{minipage}}&
\begin{minipage}[m]{.18\linewidth}
\begin{center}

\end{center}
\end{minipage}\\

\hline
\begin{minipage}[m]{.28\linewidth}

$O_z=\frac{1}{2}(1_x,1_y,1_t)$ \boundary

\end{minipage}&
\begin{minipage}[m]{.04\linewidth}
\begin{center}
$-$
\end{center}
\end{minipage}&
\begin{minipage}[m]{.11\linewidth}
\begin{center}
$B$
\end{center}
\end{minipage}&
\begin{minipage}[m]{.11\linewidth}
\begin{center}
$y$
\end{center}
\end{minipage}
&
\begin{minipage}[m]{.11\linewidth}
\begin{center}
$y$
\end{center}
\end{minipage}&
\begin{minipage}[m]{.18\linewidth}
\begin{center}
$a_1\ne 0$
\end{center}
\end{minipage}\\
\hline
\begin{minipage}[m]{.28\linewidth}

$O_z=\frac{1}{2}(1,1,1)$ \boundary

\end{minipage}&
\begin{minipage}[m]{.04\linewidth}
\begin{center}
$-$
\end{center}
\end{minipage}&
\begin{minipage}[m]{.11\linewidth}
\begin{center}
$B$
\end{center}
\end{minipage}&
\begin{minipage}[m]{.11\linewidth}
\begin{center}
$y$
\end{center}
\end{minipage}
&
\begin{minipage}[m]{.11\linewidth}
\begin{center}
$y$
\end{center}
\end{minipage}&
\begin{minipage}[m]{.18\linewidth}
\begin{center}
$a_1=0$, $a_2=0$
\end{center}
\end{minipage}\\
\hline
\begin{minipage}[m]{.28\linewidth}

$O_z=\frac{1}{2}(1_x,1_y,1_w)$ \surface

\end{minipage}&
\begin{minipage}[m]{.04\linewidth}
\begin{center}
$-$
\end{center}
\end{minipage}&
\begin{minipage}[m]{.11\linewidth}
\begin{center}
$B$
\end{center}
\end{minipage}&
\begin{minipage}[m]{.11\linewidth}
\begin{center}
$x$, $y$
\end{center}
\end{minipage}
&
\begin{minipage}[m]{.11\linewidth}
\begin{center}
$x, y$
\end{center}
\end{minipage}&
\begin{minipage}[m]{.18\linewidth}
\begin{center}
$a_1=0$, $a_2\ne 0$
\end{center}
\end{minipage}\\

\hline
\end{longtable}
\end{center}

 Since
$X_{15}$ is quasi-smooth, one of the constants $a_1$, $a_2$, $a_3$
must be non-zero.

\begin{Note}

\item 
The $1$-cycle $\Gamma$ for the singular point $O_z$ with $a_1\ne
0$ is irreducible since we have $t^3$ and $z^4w$.

\item The $1$-cycle $\Gamma$ for the singular point $O_z$ with
$a_1=a_2=0$ is also irreducible even though it is not reduced.

\item For the singular point $O_z$ with $a_1=0$ and $a_2\ne 0$ we may
assume that $a_2=1$ and $a_3=a_4=0$. Choose a general member $H$
in the linear system $|-K_{15}|$ and then take  the intersection
of  its proper transform $T$ with  $S$. This  gives us a divisor
consisting of two irreducible curves on the normal surface $T$.
One is the proper transform of the curve $L_{zw}$. The other is
the proper transform
 of the curve $C$ defined by $$x=y=t^2+z^5=0.$$ The curves $L_{zw}$ and $C$ intersect at
the point $O_w$. From the intersection numbers
\[(\tilde{L}_{zw}+\tilde{C})\cdot\tilde{L}_{zw}=-K_Y\cdot\tilde{L}_{zw}=-\frac{3}{7},\ \ \ (\tilde{L}_{zw}+\tilde{C})^2=B^3=-\frac{2}{7}\]
on the surface $T$, we obtain
\[\tilde{L}_{zw}^2=-\frac{3}{7}-\tilde{L}_{zw}\cdot\tilde{C}, \ \ \ \tilde{C}^2=\frac{1}{7}-\tilde{L}_{zw}\cdot\tilde{C}\]
With these intersection numbers we see that the matrix
\[\left(\begin{array}{cc}
       \tilde{L}_{zw}^2&\tilde{L}_{zw}\cdot\tilde{C}\\
       \tilde{L}_{zw}\cdot\tilde{C}& \tilde{C}^2\\
\end{array}\right)= \left(\begin{array}{cc}
-\frac{3}{7}-\tilde{L}_{zw}\cdot\tilde{C} & \tilde{L}_{zw}\cdot\tilde{C}\\
     \tilde{L}_{zw}\cdot\tilde{C} & \frac{1}{7} -\tilde{L}_{zw}\cdot\tilde{C} \\
        \end{array}\right)
\]
is negative-definite since
$\tilde{L}_{zw}\cdot\tilde{C}=\frac{5}{7}$.
\end{Note}


\begin{center}
\begin{longtable}{|l|c|c|c|c|c|}
\hline
\multicolumn{6}{|l|}{\textbf{No. 25}: $X_{15}\subset\mathbb{P}(1,1,3,4,7)$\hfill $A^3=5/28$}\\
\multicolumn{6}{|l|}{
\begin{minipage}[m]{.86\linewidth}
\vspace*{1.2mm}

$yw^2+t^2(a_1w+a_2zt)+z^5+wf_8(x,y,z,t)+f_{15}(x,y,z,t)$
\vspace*{1.2mm}
\end{minipage}
}\\
\hline \hline
\begin{minipage}[m]{.28\linewidth}
\begin{center}
Singularity
\end{center}
\end{minipage}&
\begin{minipage}[m]{.04\linewidth}
\begin{center}
$B^3$
\end{center}
\end{minipage}&
\begin{minipage}[m]{.11\linewidth}
\begin{center}
Linear

system
\end{center}
\end{minipage}&
\begin{minipage}[m]{.11\linewidth}
\begin{center}
Surface $T$
\end{center}
\end{minipage}&
\begin{minipage}[m]{.11\linewidth}
\begin{center}
\vspace*{1mm}
 \vorder
\vspace*{1mm}
\end{center}
\end{minipage}&
\begin{minipage}[m]{.18\linewidth}
\begin{center}
Condition
\end{center}
\end{minipage}\\
\hline
\begin{minipage}[m]{.28\linewidth}

$O_w=\frac{1}{7}(1,3,4)$ \quadratic

\end{minipage}&
\multicolumn{4}{|l|}{\begin{minipage}[m]{.37\linewidth}
\begin{center}
$yw^2$
\end{center}
\end{minipage}}&
\begin{minipage}[m]{.18\linewidth}
\begin{center}

\end{center}
\end{minipage}\\

\hline
\begin{minipage}[m]{.28\linewidth}

$O_t=\frac{1}{4}(1,1,3)$ \quadraticone

\end{minipage}&
\multicolumn{4}{|l|}{\begin{minipage}[m]{.37\linewidth}
\begin{center}
$wt^2$
\end{center}
\end{minipage}}&
\begin{minipage}[m]{.18\linewidth}
\begin{center}

\end{center}
\end{minipage}\\
\hline
\end{longtable}
\end{center}



\begin{center}
\begin{longtable}{|l|c|c|c|c|c|}
\hline
\multicolumn{6}{|l|}{\textbf{No. 26}: $X_{15}\subset\mathbb{P}(1,1,3,5,6)$\hfill $A^3=1/6$}\\
\multicolumn{6}{|l|}{
\begin{minipage}[m]{.86\linewidth}
\vspace*{1.2mm}

$zw^2+t^3+z^5+wf_9(x,y,z,t)+f_{15}(x,y,z,t)$ \vspace*{1.2mm}
\end{minipage}
}\\

\hline \hline
\begin{minipage}[m]{.28\linewidth}
\begin{center}
Singularity
\end{center}
\end{minipage}&
\begin{minipage}[m]{.04\linewidth}
\begin{center}
$B^3$
\end{center}
\end{minipage}&
\begin{minipage}[m]{.11\linewidth}
\begin{center}
Linear

system
\end{center}
\end{minipage}&
\begin{minipage}[m]{.11\linewidth}
\begin{center}
Surface $T$
\end{center}
\end{minipage}&
\begin{minipage}[m]{.11\linewidth}
\begin{center}
\vspace*{1mm}
 \vorder
\vspace*{1mm}
\end{center}
\end{minipage}&
\begin{minipage}[m]{.18\linewidth}
\begin{center}
Condition
\end{center}
\end{minipage}\\
\hline
\begin{minipage}[m]{.28\linewidth}

$O_w=\frac{1}{6}(1,1,5)$ \quadratic

\end{minipage}&
\multicolumn{4}{|l|}{\begin{minipage}[m]{.37\linewidth}
\begin{center}
$zw^2$
\end{center}
\end{minipage}}&
\begin{minipage}[m]{.18\linewidth}
\begin{center}

\end{center}
\end{minipage}\\

\hline
\begin{minipage}[m]{.28\linewidth}

$O_zO_w=2\times\frac{1}{3}(1_x,1_y,2_t)$ \boundary

\end{minipage}&
\begin{minipage}[m]{.04\linewidth}
\begin{center}
$0$
\end{center}
\end{minipage}&
\begin{minipage}[m]{.11\linewidth}
\begin{center}
$B$
\end{center}
\end{minipage}&
\begin{minipage}[m]{.11\linewidth}
\begin{center}
$y$
\end{center}
\end{minipage}
&
\begin{minipage}[m]{.11\linewidth}
\begin{center}
$y$
\end{center}
\end{minipage}&
\begin{minipage}[m]{.18\linewidth}
\begin{center}

\end{center}
\end{minipage}\\

\hline
\end{longtable}
\end{center}

\begin{Note}

\item 
The $1$-cycle $\Gamma$ for each singular point of type
$\frac{1}{3}(1,1,2)$ is irreducible since we have the monomials
$zw^2$ and $t^3$.
\end{Note}

\begin{center}
\begin{longtable}{|l|c|c|c|c|c|}
\hline
\multicolumn{6}{|l|}{\textbf{No. 27}: $X_{15}\subset\mathbb{P}(1,2,3,5,5)$\hfill $A^3=1/10$}\\
\multicolumn{6}{|l|}{
\begin{minipage}[m]{.86\linewidth}
\vspace*{1.2mm}
$(w-\alpha_1t)(w-\alpha_2t)(w-\alpha_3t)+y^5(a_1w+a_2yz+a_3xy^2)+w^2f_{5}(x,y,z,t)+
wf_{10}(x,y,z,t)+f_{15}(x,y,z,t)$ \vspace*{1.2mm}
\end{minipage}
}\\
\hline \hline
\begin{minipage}[m]{.28\linewidth}
\begin{center}
Singularity
\end{center}
\end{minipage}&
\begin{minipage}[m]{.04\linewidth}
\begin{center}
$B^3$
\end{center}
\end{minipage}&
\begin{minipage}[m]{.11\linewidth}
\begin{center}
Linear

system
\end{center}
\end{minipage}&
\begin{minipage}[m]{.11\linewidth}
\begin{center}
Surface $T$
\end{center}
\end{minipage}&
\begin{minipage}[m]{.11\linewidth}
\begin{center}
\vspace*{1mm}
 \vorder
\vspace*{1mm}
\end{center}
\end{minipage}&
\begin{minipage}[m]{.18\linewidth}
\begin{center}
Condition
\end{center}
\end{minipage}\\
\hline
\begin{minipage}[m]{.28\linewidth}

$O_tO_w=3\times\frac{1}{5}(1,2,3)$ \quadratic

\end{minipage}&
\multicolumn{4}{|l|}{\begin{minipage}[m]{.37\linewidth}
\begin{center}
$wt^2$
\end{center}
\end{minipage}}&
\begin{minipage}[m]{.18\linewidth}
\begin{center}

\end{center}
\end{minipage}\\

\hline
\begin{minipage}[m]{.28\linewidth}

$O_y=\frac{1}{2}(1,1,1)$ \nef

\end{minipage}&
\begin{minipage}[m]{.04\linewidth}
\begin{center}
$-$
\end{center}
\end{minipage}&
\begin{minipage}[m]{.11\linewidth}
\begin{center}
$5B+2E$
\end{center}
\end{minipage}&
\begin{minipage}[m]{.11\linewidth}
\begin{center}
$t$
\end{center}
\end{minipage}
&
\begin{minipage}[m]{.11\linewidth}
\begin{center}
$t$
\end{center}
\end{minipage}&
\begin{minipage}[m]{.18\linewidth}
\begin{center}
\end{center}
\end{minipage}\\
\hline
\end{longtable}
\end{center}
 We
may assume that $\alpha_3=0$, i.e., the hypersurface $X_{15}$ has
a singular point of type $\frac{1}{5}(1,2,3)$ at the point $O_t$.

\begin{Note}

\item 
To see how to treat the singular points of type
$\frac{1}{5}(1,2,3)$ we have only to consider  the singular point
$O_t$. The others can be dealt with in the same way.

\item
For the singular point $O_y$ we consider  the linear system
$|-5K_{X_{15}}|$. Every member in the linear system passes through
the point $O_y$. It has no base curve. Since the proper transform
of a general member in $|-5K_{X_{15}}|$ belongs to the linear
system $|5B+2E|$,  the divisor $T$ is nef.
\end{Note}


\begin{center}
\begin{longtable}{|l|c|c|c|c|c|}
\hline
\multicolumn{6}{|l|}{\underline{\textbf{No. 28}}: $X_{15}\subset\mathbb{P}(1,3,3,4,5)$\hfill $A^3=1/12$}\\
\multicolumn{6}{|l|}{
\begin{minipage}[m]{.86\linewidth}
\vspace*{1.2mm}
$w^3+zt^3+h_{15}(y,z)+w^2f_{5}(x,y,z,t)+wf_{10}(x,y,z,t)+
t^2g_{7}(x,y,z)+tg_{11}(x,y,z)+g_{15}(x,y,z)$ \vspace*{1.2mm}
\end{minipage}
}\\
\hline \hline
\begin{minipage}[m]{.28\linewidth}
\begin{center}
Singularity
\end{center}
\end{minipage}&
\begin{minipage}[m]{.04\linewidth}
\begin{center}
$B^3$
\end{center}
\end{minipage}&
\begin{minipage}[m]{.11\linewidth}
\begin{center}
Linear

system
\end{center}
\end{minipage}&
\begin{minipage}[m]{.11\linewidth}
\begin{center}
Surface $T$
\end{center}
\end{minipage}&
\begin{minipage}[m]{.11\linewidth}
\begin{center}
\vspace*{1mm}
 \vorder
\vspace*{1mm}
\end{center}
\end{minipage}&
\begin{minipage}[m]{.18\linewidth}
\begin{center}
Condition
\end{center}
\end{minipage}\\
\hline
\begin{minipage}[m]{.28\linewidth}

$O_t=1\times\frac{1}{4}(1_x,3_y,1_w)$ \boundary

\end{minipage}&
\begin{minipage}[m]{.04\linewidth}
\begin{center}
$0$
\end{center}
\end{minipage}&
\begin{minipage}[m]{.11\linewidth}
\begin{center}
$3B$
\end{center}
\end{minipage}&
\begin{minipage}[m]{.11\linewidth}
\begin{center}
$y$
\end{center}
\end{minipage}
&
\begin{minipage}[m]{.11\linewidth}
\begin{center}
$y$
\end{center}
\end{minipage}&
\begin{minipage}[m]{.18\linewidth}
\begin{center}

\end{center}
\end{minipage}\\
\hline
\begin{minipage}[m]{.28\linewidth}

$O_yO_z=5\times\frac{1}{3}(1_x,1_t,2_w)$ \boundary

\end{minipage}&
\begin{minipage}[m]{.04\linewidth}
\begin{center}
$-$
\end{center}
\end{minipage}&
\begin{minipage}[m]{.11\linewidth}
\begin{center}
$3B$
\end{center}
\end{minipage}&
\begin{minipage}[m]{.11\linewidth}
\begin{center}
$y-\alpha_i z$
\end{center}
\end{minipage}
&
\begin{minipage}[m]{.11\linewidth}
\begin{center}
$t^3z$
\end{center}
\end{minipage}&
\begin{minipage}[m]{.18\linewidth}
\begin{center}

\end{center}
\end{minipage}\\
\hline

\end{longtable}
\end{center}

Note that the homogenous polynomial $h_{15}$ cannot be divisible
by $z$ since the hypersurface $X_{15}$ is quasi-smooth. Therefore,
we may write
\[h_{15}(y,z)=\prod_{i=1}^5(y-\alpha_i z). \]  

\begin{Note}

\item 
The curve defined by $x=y=0$ is irreducible
because we have the monomials $w^3$ and $zt^3$.
 \item
 The curves defined by $x=y-\alpha_i z=0$ are also irreducible for
 the same reason.  Therefore, the $1$-cycle $\Gamma$  for each singular point is irreducible.

\end{Note}





\begin{center}
\begin{longtable}{|l|c|c|c|c|c|}
\hline
\multicolumn{6}{|l|}{\underline{\textbf{No. 29}}: $X_{16}\subset\mathbb{P}(1,1,2,5,8)$\hfill $A^3=1/5$}\\
\multicolumn{6}{|l|}{
\begin{minipage}[m]{.86\linewidth}
\vspace*{1.2mm} $(w-\alpha_1 z^4)(w-\alpha_2
z^4)+yt^3+az^3t^2+wf_{8}(x,y,z,t)+
t^2g_{6}(x,y,z)+tg_{11}(x,y,z)+g_{16}(x,y,z)$ \vspace*{1.2mm}
\end{minipage}
}\\
\hline \hline
\begin{minipage}[m]{.28\linewidth}
\begin{center}
Singularity
\end{center}
\end{minipage}&
\begin{minipage}[m]{.04\linewidth}
\begin{center}
$B^3$
\end{center}
\end{minipage}&
\begin{minipage}[m]{.11\linewidth}
\begin{center}
Linear

system
\end{center}
\end{minipage}&
\begin{minipage}[m]{.11\linewidth}
\begin{center}
Surface $T$
\end{center}
\end{minipage}&
\begin{minipage}[m]{.11\linewidth}
\begin{center}
\vspace*{1mm}
 \vorder
\vspace*{1mm}
\end{center}
\end{minipage}&
\begin{minipage}[m]{.18\linewidth}
\begin{center}
Condition
\end{center}
\end{minipage}\\
\hline
\begin{minipage}[m]{.28\linewidth}

$O_t=\frac{1}{5}(1_x,2_z,3_w)$ $\positive$

\end{minipage}&
\begin{minipage}[m]{.04\linewidth}
\begin{center}
$+$
\end{center}
\end{minipage}&
\begin{minipage}[m]{.11\linewidth}
\begin{center}
$B-E$
\end{center}
\end{minipage}&
\begin{minipage}[m]{.11\linewidth}
\begin{center}
$y$
\end{center}
\end{minipage}
&
\begin{minipage}[m]{.11\linewidth}
\begin{center}
$w^2$
\end{center}
\end{minipage}&
\begin{minipage}[m]{.18\linewidth}
\begin{center}

\end{center}
\end{minipage}\\
\hline
\begin{minipage}[m]{.28\linewidth}
$O_zO_w=2\times\frac{1}{2}(1_x,1_y,1_t)$ \boundary

\end{minipage}&
\begin{minipage}[m]{.04\linewidth}
\begin{center}
$-$
\end{center}
\end{minipage}&
\begin{minipage}[m]{.11\linewidth}
\begin{center}
$B$
\end{center}
\end{minipage}&
\begin{minipage}[m]{.11\linewidth}
\begin{center}
$y$
\end{center}
\end{minipage}
&
\begin{minipage}[m]{.11\linewidth}
\begin{center}
$y$
\end{center}
\end{minipage}&
\begin{minipage}[m]{.18\linewidth}
\begin{center}
$a\ne 0$
\end{center}
\end{minipage}\\
\hline
\begin{minipage}[m]{.28\linewidth}
$O_zO_w=2\times\frac{1}{2}(1_x,1_y,1_t)$ $\surface$

\end{minipage}&
\begin{minipage}[m]{.04\linewidth}
\begin{center}
$-$
\end{center}
\end{minipage}&
\begin{minipage}[m]{.11\linewidth}
\begin{center}
$B$
\end{center}
\end{minipage}&
\begin{minipage}[m]{.11\linewidth}
\begin{center}
$x$, $y$
\end{center}
\end{minipage}
&
\begin{minipage}[m]{.11\linewidth}
\begin{center}
$x$, $y$
\end{center}
\end{minipage}&
\begin{minipage}[m]{.18\linewidth}
\begin{center}
$a=0$
\end{center}
\end{minipage}\\
\hline

\end{longtable}
\end{center}

\begin{Note}

\item 
If the constant $a$ is non-zero, then the $1$-cycle $\Gamma$ for
each singular point of type $\frac{1}{2}(1,1,1)$ is irreducible.

\item Suppose that $a=0$. We have only to consider one of the singular
points of type $\frac{1}{2}(1,1,1)$. The other singular point can
be excluded in the same way. Moreover, we may assume that the
singular point is located at the point $O_z$, i.e., $\alpha_1=0$,
by a suitable coordinate change.

We take a general surface $H$ from the pencil $|-K_{X_{16}}|$. It
is a  K3 surface only with du Val singularities. Let $T$ be the
proper transform of the surface. The intersection of $T$ with the
surface $S$ gives us a divisor consisting of two irreducible
curves on the normal surface $T$. One is the proper transform
$\tilde{L}_{zt}$. The other is the proper transform $\tilde{C}$ of
the curve $C$ defined by  $$x=y=w-\alpha_2z^4=0.$$

From the intersection numbers
\[(\tilde{L}_{zt}+\tilde{C})\cdot\tilde{L}_{zt}=-K_Y\cdot \tilde{L}_{zt}=-\frac{2}{5}, \ \ \ (\tilde{L}_{zt}+\tilde{C})^2=B^3=-\frac{3}{10}\]
on the surface $T$, we obtain
\[\tilde{L}_{zt}^2=-\frac{2}{5}-\tilde{L}_{zt}\cdot\tilde{C}, \ \ \ \tilde{C}^2=\frac{1}{10}-\tilde{L}_{zt}\cdot\tilde{C}\]
With these intersection numbers we see that the matrix
\[\left(\begin{array}{cc}
       \tilde{L}_{zt}^2&\tilde{L}_{zt}\cdot\tilde{C}\\
       \tilde{L}_{zt}\cdot\tilde{C}& \tilde{C}^2\\
\end{array}\right)= \left(\begin{array}{cc}
-\frac{2}{5}-\tilde{L}_{zt}\cdot\tilde{C} & \tilde{L}_{zt}\cdot\tilde{C}\\
     \tilde{L}_{zt}\cdot\tilde{C} & \frac{1}{10}-\tilde{L}_{zt}\cdot\tilde{C} \\
        \end{array}\right)
\]
is negative-definite since
$\tilde{L}_{yw}\cdot\tilde{C}=\frac{4}{5}$.
\end{Note}





\begin{center}
\begin{longtable}{|l|c|c|c|c|c|}
\hline
\multicolumn{6}{|l|}{\textbf{No. 30}: $X_{16}\subset\mathbb{P}(1,1,3,4,8)$\hfill $A^3=1/6$}\\
\multicolumn{6}{|l|}{
\begin{minipage}[m]{.86\linewidth}
\vspace*{1.2mm}

$(w-\alpha_1 t^2)(w-\alpha_2
t^2)+z^4(a_1t+a_2yz)+wf_{8}(x,y,z,t)+f_{16}(x,y,z,t)$
\vspace*{1.2mm}
\end{minipage}
}\\
\hline \hline
\begin{minipage}[m]{.28\linewidth}
\begin{center}
Singularity
\end{center}
\end{minipage}&
\begin{minipage}[m]{.04\linewidth}
\begin{center}
$B^3$
\end{center}
\end{minipage}&
\begin{minipage}[m]{.11\linewidth}
\begin{center}
Linear

system
\end{center}
\end{minipage}&
\begin{minipage}[m]{.11\linewidth}
\begin{center}
Surface $T$
\end{center}
\end{minipage}&
\begin{minipage}[m]{.11\linewidth}
\begin{center}
\vspace*{1mm}
 \vorder
\vspace*{1mm}
\end{center}
\end{minipage}&
\begin{minipage}[m]{.18\linewidth}
\begin{center}
Condition
\end{center}
\end{minipage}\\
\hline
\begin{minipage}[m]{.28\linewidth}

$O_tO_w=2\times\frac{1}{4}(1,1,3)$ \quadratic

\end{minipage}&
\multicolumn{4}{|l|}{\begin{minipage}[m]{.37\linewidth}
\begin{center}
$wt^2$
\end{center}
\end{minipage}}&
\begin{minipage}[m]{.18\linewidth}
\begin{center}

\end{center}
\end{minipage}\\

\hline
\begin{minipage}[m]{.28\linewidth}

$O_z=\frac{1}{3}(1_x,1_y,2_w)$ \boundary

\end{minipage}&
\begin{minipage}[m]{.04\linewidth}
\begin{center}
$0$
\end{center}
\end{minipage}&
\begin{minipage}[m]{.11\linewidth}
\begin{center}
$B$
\end{center}
\end{minipage}&
\begin{minipage}[m]{.11\linewidth}
\begin{center}
$y$
\end{center}
\end{minipage}
&
\begin{minipage}[m]{.11\linewidth}
\begin{center}
$y$
\end{center}
\end{minipage}&
\begin{minipage}[m]{.18\linewidth}
\begin{center}
$a_1\ne 0$
\end{center}
\end{minipage}\\
\hline
\begin{minipage}[m]{.28\linewidth}

$O_z=\frac{1}{3}(1_x,1_t,2_w)$ \boundary

\end{minipage}&
\begin{minipage}[m]{.04\linewidth}
\begin{center}
$0$
\end{center}
\end{minipage}&
\begin{minipage}[m]{.11\linewidth}
\begin{center}
$B-E$
\end{center}
\end{minipage}&
\begin{minipage}[m]{.11\linewidth}
\begin{center}
$y$
\end{center}
\end{minipage}
&
\begin{minipage}[m]{.11\linewidth}
\begin{center}
$w^2$
\end{center}
\end{minipage}&
\begin{minipage}[m]{.18\linewidth}
\begin{center}
$a_1=0$
\end{center}
\end{minipage}\\
\hline
\end{longtable}
\end{center}

\begin{Note}

\item 
We may assume that $\alpha_1=0$. To see how to treat the singular
points of type $\frac{1}{4}(1,1,3)$, we have only to consider the
singular point $O_t$. The other point can be treated  in the same
way.

\item The $1$-cycle $\Gamma$ for the singular point $O_z$ with $a_1\ne
0$ is irreducible due to $w^2$ and $z^4t$.

\item  The $1$-cycle $\Gamma$ for the singular point $O_z$ with $a_1= 0$ consists of  the proper
transforms of the curves defined by $$x=y=w-\alpha_1 t^2=0$$ and
$$x=y=w-\alpha_2 t^2=0.$$ These two irreducible components are
symmetric with respect to the biregular involution of $X_{16}$.
Consequently, the components of $\Gamma$ are numerically
equivalent to each other.
\end{Note}


\begin{center}
\begin{longtable}{|l|c|c|c|c|c|}
\hline
\multicolumn{6}{|l|}{\textbf{No. 31}: $X_{16}\subset\mathbb{P}(1,1,4,5,6)$\hfill $A^3=2/15$}\\
\multicolumn{6}{|l|}{
\begin{minipage}[m]{.86\linewidth}
\vspace*{1.2mm}

$zw^2+t^2(a_1w+a_2yt)+z^4+wf_{10}(x,y,z,t)+f_{16}(x,y,z,t)$
\vspace*{1.2mm}
\end{minipage}
}\\
\hline \hline
\begin{minipage}[m]{.28\linewidth}
\begin{center}
Singularity
\end{center}
\end{minipage}&
\begin{minipage}[m]{.04\linewidth}
\begin{center}
$B^3$
\end{center}
\end{minipage}&
\begin{minipage}[m]{.11\linewidth}
\begin{center}
Linear

system
\end{center}
\end{minipage}&
\begin{minipage}[m]{.11\linewidth}
\begin{center}
Surface $T$
\end{center}
\end{minipage}&
\begin{minipage}[m]{.11\linewidth}
\begin{center}
\vspace*{1mm}
 \vorder
\vspace*{1mm}
\end{center}
\end{minipage}&
\begin{minipage}[m]{.18\linewidth}
\begin{center}
Condition
\end{center}
\end{minipage}\\
\hline
\begin{minipage}[m]{.28\linewidth}

$O_w=\frac{1}{6}(1,1,5)$ \quadratic

\end{minipage}&
\multicolumn{4}{|l|}{\begin{minipage}[m]{.37\linewidth}
\begin{center}
$zw^2$
\end{center}
\end{minipage}}&
\begin{minipage}[m]{.18\linewidth}
\begin{center}

\end{center}
\end{minipage}\\

\hline
\begin{minipage}[m]{.28\linewidth}

$O_t=\frac{1}{5}(1,1,4)$ \quadraticone

\end{minipage}&
\multicolumn{4}{|l|}{\begin{minipage}[m]{.37\linewidth}
\begin{center}
$wt^2$
\end{center}
\end{minipage}}&
\begin{minipage}[m]{.18\linewidth}
\begin{center}

\end{center}
\end{minipage}\\

\hline
\begin{minipage}[m]{.28\linewidth}

$O_zO_w=1\times\frac{1}{2}(1_x,1_y,1_t)$ \boundary

\end{minipage}&
\begin{minipage}[m]{.04\linewidth}
\begin{center}
$-$
\end{center}
\end{minipage}&
\begin{minipage}[m]{.11\linewidth}
\begin{center}
$B$
\end{center}
\end{minipage}&
\begin{minipage}[m]{.11\linewidth}
\begin{center}
$y$
\end{center}
\end{minipage}
&
\begin{minipage}[m]{.11\linewidth}
\begin{center}
$y$
\end{center}
\end{minipage}&
\begin{minipage}[m]{.18\linewidth}
\begin{center}
$a_1\ne0$
\end{center}
\end{minipage}\\
\hline
\begin{minipage}[m]{.28\linewidth}

$O_zO_w=1\times\frac{1}{2}(1_x,1_y,1_t)$ $\surface$

\end{minipage}&
\begin{minipage}[m]{.04\linewidth}
\begin{center}
$-$
\end{center}
\end{minipage}&
\begin{minipage}[m]{.11\linewidth}
\begin{center}
$B$
\end{center}
\end{minipage}&
\begin{minipage}[m]{.11\linewidth}
\begin{center}
$x$, $y$
\end{center}
\end{minipage}
&
\begin{minipage}[m]{.11\linewidth}
\begin{center}
$x, y$
\end{center}
\end{minipage}&
\begin{minipage}[m]{.18\linewidth}
\begin{center}
$a_1=0$
\end{center}
\end{minipage}\\

\hline
\end{longtable}
\end{center}

\begin{Note}

\item 
If $a_1\ne0$, the $1$-cycle $\Gamma$ for the singular point of
type $\frac{1}{2}(1,1,1)$  is irreducible due to the monomials
$z^4$ and $t^2w$.

\item Suppose  $a_1= 0$.  Choose a general member $H$ in the linear
system $|-K_X|$. Then it is a normal K3 surface of degree $16$ in
$\mathbb{P}(1,4,5,6)$. Let $T$ be the proper transform of the
surface $H$. The intersection of $T$ with the surface $S$ defines
a divisor consisting of
 two irreducible curves $\tilde{L}_{tw}$ and $\tilde{C}$ on the normal surface $T$.
The curve $\tilde{C}$ is the proper transform of the curve $C$
defined by $$x=y=w^2+z^3=0.$$ On the surface $T$, we have
\[\tilde{L}_{tw}\cdot\tilde{C}=L_{tw}\cdot C=\frac{2}{5}. \] From the
intersections
\[(\tilde{L}_{tw}+\tilde{C})\cdot\tilde{L}_{tw}=-K_Y\cdot \tilde{L}_{tw} =\frac{1}{30}, \ \ \ (\tilde{L}_{tw}+\tilde{C})^2=B^3=-\frac{11}{30}\]
on the surface $T$, we obtain
\[\tilde{L}_{tw}^2=-\frac{11}{30}, \ \ \ \tilde{C}^2=-\frac{4}{5}.\]
The intersection  matrix
\[\left(\begin{array}{cc}
      \tilde{L}_{tw}^2 &\tilde{L}_{tw}\cdot \tilde{C} \\
       \tilde{L}_{tw}\cdot \tilde{C}& \tilde{C}^2\\
\end{array}\right)= \left(\begin{array}{cc}
-\frac{11}{30} &\frac{2}{5} \\
        \frac{2}{5}& -\frac{4}{5} \\
        \end{array}\right)
\]
is negative-definite.
\end{Note}



\begin{center}
\begin{longtable}{|l|c|c|c|c|c|}
\hline
\multicolumn{6}{|l|}{\textbf{No. 32}: $X_{16}\subset\mathbb{P}(1,2,3,4,7)$\hfill $A^3=2/21$}\\
\multicolumn{6}{|l|}{
\begin{minipage}[m]{.86\linewidth}
\vspace*{1.2mm} $yw^2+\prod^{4}_{i=1}(t-\alpha_i y^2)
+z^3(a_1w+a_2tz+a_3xz^2)+wf_9(x,y,z,t)+f_{16}(x,y,z,t)$

\vspace*{1.2mm}
\end{minipage}
}\\
\hline \hline
\begin{minipage}[m]{.28\linewidth}
\begin{center}
Singularity
\end{center}
\end{minipage}&
\begin{minipage}[m]{.04\linewidth}
\begin{center}
$B^3$
\end{center}
\end{minipage}&
\begin{minipage}[m]{.11\linewidth}
\begin{center}
Linear

system
\end{center}
\end{minipage}&
\begin{minipage}[m]{.11\linewidth}
\begin{center}
Surface $T$
\end{center}
\end{minipage}&
\begin{minipage}[m]{.11\linewidth}
\begin{center}
\vspace*{1mm}
 \vorder
\vspace*{1mm}
\end{center}
\end{minipage}&
\begin{minipage}[m]{.18\linewidth}
\begin{center}
Condition
\end{center}
\end{minipage}\\
\hline
\begin{minipage}[m]{.28\linewidth}

$O_w=\frac{1}{7}(1,3,4)$ \quadratic

\end{minipage}&
\multicolumn{4}{|l|}{\begin{minipage}[m]{.37\linewidth}
\begin{center}
$yw^2$
\end{center}
\end{minipage}}&
\begin{minipage}[m]{.18\linewidth}
\begin{center}
\end{center}
\end{minipage}\\
\hline
\begin{minipage}[m]{.28\linewidth}

$O_z=\frac{1}{3}(1_x,2_y,1_t)$ \boundary

\end{minipage}&
\begin{minipage}[m]{.04\linewidth}
\begin{center}
$-$
\end{center}
\end{minipage}&
\begin{minipage}[m]{.11\linewidth}
\begin{center}
$2B$
\end{center}
\end{minipage}&
\begin{minipage}[m]{.11\linewidth}
\begin{center}
$y$
\end{center}
\end{minipage}
&
\begin{minipage}[m]{.11\linewidth}
\begin{center}
$y$
\end{center}
\end{minipage}&
\begin{minipage}[m]{.18\linewidth}
\begin{center}
$a_1\ne 0$
\end{center}
\end{minipage}\\
\hline
\begin{minipage}[m]{.28\linewidth}

$O_z=\frac{1}{3}(1_x,2_y,1_w)$ $\family$

\end{minipage}&
\begin{minipage}[m]{.04\linewidth}
\begin{center}
$-$
\end{center}
\end{minipage}&
\begin{minipage}[m]{.11\linewidth}
\begin{center}
$2B$
\end{center}
\end{minipage}&
\begin{minipage}[m]{.11\linewidth}
\begin{center}
$x^2$, $y$
\end{center}
\end{minipage}
&
\begin{minipage}[m]{.11\linewidth}
\begin{center}
$x^2, y$
\end{center}
\end{minipage}&
\begin{minipage}[m]{.18\linewidth}
\begin{center}
$a_1=0$, $a_2\ne 0$
\end{center}
\end{minipage}\\
\hline
\begin{minipage}[m]{.28\linewidth}

$O_z=\frac{1}{3}(2_y,1_t,1_w)$ \boundary

\end{minipage}&
\begin{minipage}[m]{.04\linewidth}
\begin{center}
$-$
\end{center}
\end{minipage}&
\begin{minipage}[m]{.11\linewidth}
\begin{center}
$2B$
\end{center}
\end{minipage}&
\begin{minipage}[m]{.11\linewidth}
\begin{center}
$y$
\end{center}
\end{minipage}
&
\begin{minipage}[m]{.11\linewidth}
\begin{center}
$y$
\end{center}
\end{minipage}&
\begin{minipage}[m]{.18\linewidth}
\begin{center}
$a_1=a_2=0$
\end{center}
\end{minipage}\\

\hline
\begin{minipage}[m]{.28\linewidth}
 $O_yO_t=4\times\frac{1}{2}(1_x,1_z,1_w)$ $\nef$

\end{minipage}&
\begin{minipage}[m]{.04\linewidth}
\begin{center}
$-$
\end{center}
\end{minipage}&
\begin{minipage}[m]{.11\linewidth}
\begin{center}
$3B+E$
\end{center}
\end{minipage}&
\begin{minipage}[m]{.11\linewidth}
\begin{center}
$xy$, $z$
\end{center}
\end{minipage}
&
\begin{minipage}[m]{.11\linewidth}
\begin{center}
$xy$, $z$
\end{center}
\end{minipage}&
\begin{minipage}[m]{.18\linewidth}
\begin{center}

\end{center}
\end{minipage}\\

\hline
\end{longtable}
\end{center}

\begin{Note}

\item 
The $1$-cycle $\Gamma$ for the singular point $O_z$ with $a_1\ne
0$ is irreducible due to $t^4$ and $z^3w$.

\item For the singular point $O_z$ with $a_1=0$ and $a_2\ne 0$ we may
assume that $a_3=0$. The curve $L_{zw}$ is contained in $X_{16}$
because $a_1=0$. Let $Z_{\lambda, \mu}$ be the curve  on $X_{16}$
cut out by
$$
\left\{%
\aligned
&y=\lambda x^2\\%
&t=\mu x^4,\\%
\endaligned\right.%
$$
for some sufficiently general complex numbers $\lambda$ and $\mu$.
Then $Z_{\lambda, \mu}=2L_{zw}+C_{\lambda, \mu}$, where
$C_{\lambda, \mu}$ is an irreducible and reduced curve. We have
$$
\left\{%
\aligned
&-K_{Y}\cdot (2\tilde{L}_{zw}+\tilde{C}_{\lambda, \mu})=8B^3=-\frac{4}{7},\\%
&-K_Y\cdot \tilde{L}_{zw}=-K_X\cdot L_{zw} -\frac{1}{3}E\cdot \tilde{L}_{zw}=-\frac{2}{7},\\%
\endaligned\right.%
$$
and hence  $-K_{Y}\cdot \tilde{C}_{\lambda, \mu}=0$.

\item The $1$-cycle $\Gamma$ for the singular point $O_z$ with $a_1=a_2
=0$ is irreducible even though it is non-reduced.

\item For the singular points of type $\frac{1}{2}(1,1,1)$,  consider
the linear system generated by $xy$ and $z$. Its base curves are
defined by $x=z=0$ and $y=z=0$.  The curve defined by $y=z=0$ does
not pass through any singular point of type $\frac{1}{2}(1,1,1)$.
The curve  defined by $x=z=0$  is irreducible because of the
monomial $yw^2$ and $t^4$. Since its proper transform is the
$1$-cycle defined by  $(3B+E)\cdot B$ and $(3B+E)^2\cdot B>0$, the
divisor $T$ is nef.
\end{Note}


\begin{center}
\begin{longtable}{|l|c|c|c|c|c|}
\hline
\multicolumn{6}{|l|}{\textbf{No. 33}: $X_{17}\subset\mathbb{P}(1,2,3,5,7)$\hfill $A^3=17/210$}\\
\multicolumn{6}{|l|}{
\begin{minipage}[m]{.86\linewidth}
\vspace*{1.2mm} $(dx^3+exy+z)w^2+
t^2(a_1w+a_2yt)+z^4(b_1t+b_2yz)+y^5(c_1w+c_2yt+c_3y^2z+c_4y^3x)+wf_{10}(x,y,z,t)+f_{17}(x,y,z,t)$
\vspace*{1.2mm}
\end{minipage}
}\\

\hline \hline
\begin{minipage}[m]{.28\linewidth}
\begin{center}
Singularity
\end{center}
\end{minipage}&
\begin{minipage}[m]{.04\linewidth}
\begin{center}
$B^3$
\end{center}
\end{minipage}&
\begin{minipage}[m]{.11\linewidth}
\begin{center}
Linear

system
\end{center}
\end{minipage}&
\begin{minipage}[m]{.11\linewidth}
\begin{center}
Surface $T$
\end{center}
\end{minipage}&
\begin{minipage}[m]{.11\linewidth}
\begin{center}
\vspace*{1mm}
 \vorder
\vspace*{1mm}
\end{center}
\end{minipage}&
\begin{minipage}[m]{.18\linewidth}
\begin{center}
Condition
\end{center}
\end{minipage}\\
\hline
\begin{minipage}[m]{.28\linewidth}

$O_w=\frac{1}{7}(1,2,5)$ \quadratic

\end{minipage}&
\multicolumn{4}{|l|}{\begin{minipage}[m]{.37\linewidth}
\begin{center}
$zw^2$
\end{center}
\end{minipage}}&
\begin{minipage}[m]{.18\linewidth}
\begin{center}

\end{center}
\end{minipage}\\
\hline
\begin{minipage}[m]{.28\linewidth}

$O_t=\frac{1}{5}(1,2,3)$ \quadraticone

\end{minipage}&
\multicolumn{4}{|l|}{\begin{minipage}[m]{.37\linewidth}
\begin{center}
$wt^2$
\end{center}
\end{minipage}}&
\begin{minipage}[m]{.18\linewidth}
\begin{center}

\end{center}
\end{minipage}\\
\hline
\begin{minipage}[m]{.28\linewidth}

$O_z=\frac{1}{3}(1_x,2_y,1_w)$ \boundary

\end{minipage}&
\begin{minipage}[m]{.04\linewidth}
\begin{center}
$-$
\end{center}
\end{minipage}&
\begin{minipage}[m]{.11\linewidth}
\begin{center}
$2B$
\end{center}
\end{minipage}&
\begin{minipage}[m]{.11\linewidth}
\begin{center}
$y$
\end{center}
\end{minipage}
&
\begin{minipage}[m]{.11\linewidth}
\begin{center}
$y$
\end{center}
\end{minipage}&
\begin{minipage}[m]{.18\linewidth}
\begin{center}
$a_1\ne 0$, $b_1\ne 0$
\end{center}
\end{minipage}\\
\hline
\begin{minipage}[m]{.28\linewidth}

$O_z=\frac{1}{3}(1_x,2_y,1_w)$ $\surface$

\end{minipage}&
\begin{minipage}[m]{.04\linewidth}
\begin{center}
$-$
\end{center}
\end{minipage}&
\begin{minipage}[m]{.11\linewidth}
\begin{center}
$2B$
\end{center}
\end{minipage}&
\begin{minipage}[m]{.11\linewidth}
\begin{center}
$x^2$, $y$
\end{center}
\end{minipage}
&
\begin{minipage}[m]{.11\linewidth}
\begin{center}
$x^2, y$
\end{center}
\end{minipage}&
\begin{minipage}[m]{.18\linewidth}
\begin{center}
$a_1=0$, $b_1\ne 0$
\end{center}
\end{minipage}\\

\hline
\begin{minipage}[m]{.28\linewidth}

$O_z=\frac{1}{3}(1_x,2_t,1_w)$ $\surface$

\end{minipage}&
\begin{minipage}[m]{.04\linewidth}
\begin{center}
$-$
\end{center}
\end{minipage}&
\begin{minipage}[m]{.11\linewidth}
\begin{center}
$2B$
\end{center}
\end{minipage}&
\begin{minipage}[m]{.11\linewidth}
\begin{center}
$x^2$, $y$
\end{center}
\end{minipage}
&
\begin{minipage}[m]{.11\linewidth}
\begin{center}
$x^2, zw^2$
\end{center}
\end{minipage}&
\begin{minipage}[m]{.18\linewidth}
\begin{center}
$b_1=0$
\end{center}
\end{minipage}\\

\hline
\begin{minipage}[m]{.28\linewidth}

$O_y=\frac{1}{2}(1_x,1_z,1_t)$ $\nef$

\end{minipage}&
\begin{minipage}[m]{.04\linewidth}
\begin{center}
$-$
\end{center}
\end{minipage}&
\begin{minipage}[m]{.11\linewidth}
\begin{center}
$5B+2E$
\end{center}
\end{minipage}&
\begin{minipage}[m]{.11\linewidth}
\begin{center}
$t$
\end{center}
\end{minipage}
&
\begin{minipage}[m]{.11\linewidth}
\begin{center}
$t$
\end{center}
\end{minipage}&
\begin{minipage}[m]{.18\linewidth}
\begin{center}
$c_1\ne 0$
\end{center}
\end{minipage}\\
\hline
\begin{minipage}[m]{.28\linewidth}

$O_y=\frac{1}{2}(1_x,1_z,1_w)$ $\surface$

\end{minipage}&
\begin{minipage}[m]{.04\linewidth}
\begin{center}
$-$
\end{center}
\end{minipage}&
\begin{minipage}[m]{.11\linewidth}
\begin{center}
$5B+E$
\end{center}
\end{minipage}&
\begin{minipage}[m]{.11\linewidth}
\begin{center}
$x^5$, $t$
\end{center}
\end{minipage}
&
\begin{minipage}[m]{.11\linewidth}
\begin{center}
$x^5$, $zw^2$
\end{center}
\end{minipage}&
\begin{minipage}[m]{.18\linewidth}
\begin{center}
$c_1=0$, $c_2\ne 0$
\end{center}
\end{minipage}\\
\hline
\begin{minipage}[m]{.28\linewidth}

$O_y=\frac{1}{2}(1_x,1_t,1_w)$ $\surface$

\end{minipage}&
\begin{minipage}[m]{.04\linewidth}
\begin{center}
$-$
\end{center}
\end{minipage}&
\begin{minipage}[m]{.11\linewidth}
\begin{center}
$3B$
\end{center}
\end{minipage}&
\begin{minipage}[m]{.11\linewidth}
\begin{center}
$x^3$, $z$
\end{center}
\end{minipage}
&
\begin{minipage}[m]{.11\linewidth}
\begin{center}
$x^3, yt^3, t^2w$
\end{center}
\end{minipage}&
\begin{minipage}[m]{.18\linewidth}
\begin{center}
$c_1=c_2=0$\\$c_3\ne 0$
\end{center}
\end{minipage}\\

\hline
\begin{minipage}[m]{.28\linewidth}

$O_y=\frac{1}{2}(1_z,1_t,1_w)$ $\nef$

\end{minipage}&
\begin{minipage}[m]{.04\linewidth}
\begin{center}
$-$
\end{center}
\end{minipage}&
\begin{minipage}[m]{.11\linewidth}
\begin{center}
$7B+3E$
\end{center}
\end{minipage}&
\begin{minipage}[m]{.11\linewidth}
\begin{center}
$w$
\end{center}
\end{minipage}
&
\begin{minipage}[m]{.11\linewidth}
\begin{center}
$w$
\end{center}
\end{minipage}&
\begin{minipage}[m]{.18\linewidth}
\begin{center}
$c_1=c_2=c_3=0$
\end{center}
\end{minipage}\\

\hline
\end{longtable}
\end{center}

 \begin{Note}

\item 
The $1$-cycle $\Gamma$ for the singular point $O_z$ with $a_1\ne0$
and $b_1\ne 0$ is irreducible since we have the monomials $t^2w$,
$z^4t$ and $zw^2$.

\end{Note}
 For the singular point $O_z$ with $a_1b_1=0$  choose a general
member $H$ in the linear system $|-2K_{X_{17}}|$. Then it is a
normal surface of degree $17$ in $\mathbb{P}(1,3,5,7)$. Let $T$ be
the proper transform of the divisor $H$. The curve $\tilde{D}$ on
$T$ cut out by the surface $S$ is the proper transform of the
curve cut by the equations $x=y=0$.

\begin{Note}

\item Suppose that $b_1\ne 0$  and $a_1=0$. Then $a_2\ne 0$. The curve
$\tilde{D}$ then consists of two irreducible curves
$\tilde{L}_{tw}$ and $\tilde{C}_1$.  The curve $\tilde{C}_1$ is
the proper transform of the curve $C_1$ defined by
$$x=y=w^2+b_1z^3t=0.$$ Note that the curve $L_{tw}$ and $C_1$
intersect at the point $O_t$. The surface $H$ is not quasi-smooth
at the point $O_t$. We also consider the divisor $D_z$ on $H$ cut
by the equation $z=0$. We easily see that  $D_z=2L_{tw}+R$, where
$R$ is a curve whose support does not contain $L_{tw}$. The curves
$R$ and $L_{tw}$ intersect at the point $O_w$. The surface $H$ is
quasi-smooth at the point $O_w$. Then we have $\tilde{L}_{tw}\cdot
\tilde{R}=\frac{3}{7}$. From the intersection
\[(2\tilde{L}_{tw}+\tilde{R})\cdot\tilde{L}_{tw}=3A\cdot \tilde{L}_{tw}=\frac{3}{35}\]
we obtain $\tilde{L}_{tw}^2=-\frac{6}{35}$. From the intersections
\[(\tilde{L}_{tw}+\tilde{C}_1)\cdot\tilde{L}_{tw}=-K_Y\cdot\tilde{L}_{tw}=\frac{1}{35}, \ \ \ (\tilde{L}_{tw}+\tilde{C}_1)^2=2B^3=-\frac{6}{35}\]
on the surface $T$, we obtain
\[\tilde{L}_{tw}^2=-\frac{6}{35},\ \ \ \tilde{L}_{tw}\cdot\tilde{C}_1=\frac{1}{5},  \ \ \ \tilde{C}_1^2=-\frac{2}{5}.\]
The intersection  matrix
\[\left(\begin{array}{cc}
       \tilde{L}_{tw}^2 &\tilde{L}_{tw}\cdot \tilde{C}_1\\
        \tilde{L}_{tw}\cdot \tilde{C}_1& \tilde{C}_1^2\\
\end{array}\right)= \left(\begin{array}{cc}
-\frac{6}{35} &\frac{1}{5} \\
        \frac{1}{5}& -\frac{2}{5} \\
        \end{array}\right)
\]
is negative-definite.

\item Suppose that $b_1=0$ and $a_1\ne 0$. Then $b_2\ne 0$. The curve
$\tilde{D}$ consists of two irreducible curves $\tilde{L}_{zt}$
and $\tilde{C}_2$. The curve $\tilde{C}_2$ is the proper transform
of the curve $C_2$ defined by $x=y=zw+a_1t^2=0$. From the
intersections
\[(\tilde{L}_{zt}+\tilde{C}_2)\cdot\tilde{L}_{zt}=-K_Y\cdot\tilde{L}_{zt}=-\frac{1}{10}, \ \ \ (\tilde{L}_{zt}+\tilde{C}_2)^2=2B^3=-\frac{6}{35}\]
on the surface $T$, we obtain
\[\tilde{L}_{zt}^2=-\frac{1}{10}-\tilde{L}_{zt}\cdot \tilde{C}_2, \ \ \
\tilde{C}_2^2=-\frac{1}{14}-\tilde{L}_{zt}\cdot \tilde{C}_2.\]
With these intersection numbers we see that the matrix
\[\left(\begin{array}{cc}
       \tilde{L}_{zt}^2 &\tilde{L}_{zt}\cdot \tilde{C}_2 \\
        \tilde{L}_{zt}\cdot \tilde{C}_2& \tilde{C}_2^2\\
\end{array}\right)= \left(\begin{array}{cc}
-\frac{1}{10}-\tilde{L}_{zt}\cdot \tilde{C}_2& \tilde{L}_{zt}\cdot\tilde{C}_2 \\
        \tilde{L}_{zt}\cdot \tilde{C}_2 & -\frac{1}{14}-\tilde{L}_{zt}\cdot \tilde{C}_2 \\
        \end{array}\right)
\]
is negative-definite since $\tilde{L}_{zt}\cdot\tilde{C}_2$ is
non-negative.

Suppose that $b_1=0$ and $a_1= 0$.  We then have $b_2\ne 0$ and
$a_2\ne 0$. Furthermore, the defining equation of $X_{17}$ must
contain $xz^2t$; otherwise $X_{17}$ would not be quasi-smooth at
the point $x=y=w=a_2t^3+b_2z^5=0$. Note that the presence of
$xz^2t$ implies the normality of the surfaces $H$ and $T$. The
curve $\tilde{D}$ consists of two irreducible curves
$\tilde{L}_{tw}$ and $\tilde{L}_{zt}$. Indeed,
$\tilde{D}=\tilde{L}_{tw}+2\tilde{L}_{zt}$. The curves $L_{tw}$
and $L_{zt}$ intersect at the point $O_t$. The surface $H$ is not
quasi-smooth at the point $O_t$. We consider the divisor $D_z$ on
$H$ cut by the equation $z=0$. We easily see that $D_z=2L_{tw}+R$,
where $R$ is a curve whose support does not contain $L_{tw}$. The
curves $R$ and $L_{tw}$ intersect at the point $O_w$. The surface
$H$ is quasi-smooth at the point $O_w$.
 Then we have
$\tilde{L}_{tw}\cdot \tilde{R}=\frac{3}{7}$. From the intersection
\[(2\tilde{L}_{tw}+\tilde{R})\cdot\tilde{L}_{tw}=3A\cdot \tilde{L}_{tw}=\frac{3}{35}\]
we obtain $\tilde{L}_{tw}^2=-\frac{6}{35}$. From the intersections
\[(\tilde{L}_{tw}+2\tilde{L}_{zt})\cdot\tilde{L}_{tw}=-K_Y\cdot\tilde{L}_{tw}=\frac{1}{35}, \ \ \ (\tilde{L}_{tw}+2\tilde{L}_{zt})^2=2B^3=-\frac{6}{35}\]
on the surface $T$, we obtain
\[\tilde{L}_{tw}^2=-\frac{6}{35}, \ \ \ \tilde{L}_{tw}\cdot\tilde{L}_{zt}=\frac{1}{10}, \ \ \ \tilde{L}_{zt}^2=-\frac{1}{10}.\]
Therefore, the curves $\tilde{L}_{tw}$ and $\tilde{L}_{zt}$ form a
negative-definite divisor on $T$.

 \item For the singular point $O_y$ with $c_1\ne 0$ we consider  the
linear system $|-5K_{X_{17}}|$. Every member in the linear system
passes through the point $O_y$. The base locus of $|-5K_{X_{17}}|$
is the union of the loci defined by $x=t=y=0$ and $x=t=z=0$.  It
is a $0$-dimensional locus. Since the proper transform of a
general member in $|-5K_{X_{15}}|$ belongs to the linear system
$|5B+2E|$, the divisor $T$ is nef.

 \item For the singular point $O_y$ with  $c_1=0$ and $c_2\ne 0$ we may
assume that $c_2=1$ and $c_3=c_4=0$ by a coordinate change. Choose
a general member $H$ in the linear system generated by $x^5$ and
$t$. Then it is a normal surface of degree $17$ in
$\mathbb{P}(1,2,3,7)$. Let $T$  be the proper transform of the
surface $H$. The intersection of $T$ with the surface $S$ gives us
a divisor consisting of
 two curves $\tilde{L}_{yw}$ and $\tilde{C}$. The curve $\tilde{C}$ is the proper
transform of the curve $C$ defined by
$$x=t=w^2+b_2yz^4+awy^2z+by^4z^2=0,$$ where $a$ and $b$ are
constants.

Suppose that $b_2\ne 0$. Then the curve $\tilde{C}$ is
irreducible. From the intersection numbers
\[(\tilde{L}_{yw}+\tilde{C})\cdot\tilde{L}_{yw}=-K_Y\cdot \tilde{L}_{yw}=-\frac{3}{7}\]
\[(\tilde{L}_{yw}+\tilde{C})^2=B^2\cdot (5B+E)=-\frac{23}{21}\]
on the surface $T$, we obtain
\[\tilde{L}_{yw}^2=-\frac{3}{7}-\tilde{L}_{yw}\cdot\tilde{C}, \ \ \ \tilde{C}^2=-\frac{2}{3}-\tilde{L}_{yw}\cdot\tilde{C}\]
With these intersection numbers we see that the matrix
\[\left(\begin{array}{cc}
       \tilde{L}_{yw}^2&\tilde{L}_{yw}\cdot\tilde{C}\\
       \tilde{L}_{yw}\cdot\tilde{C}& \tilde{C}^2\\
\end{array}\right)= \left(\begin{array}{cc}
-\frac{3}{7}-\tilde{L}_{yw}\cdot\tilde{C} & \tilde{L}_{yw}\cdot\tilde{C}\\
     \tilde{L}_{yw}\cdot\tilde{C} & -\frac{2}{3}-\tilde{L}_{yw}\cdot\tilde{C} \\
        \end{array}\right)
\]
is negative-definite since $\tilde{L}_{yw}\cdot\tilde{C}$ is
non-negative.

Suppose that $b_2= 0$. The curve $C$ then consists of two
irreducible curves $C_1$ and $C_2$ defined by $$x=t=w-\alpha_1
y^2z=0$$ and $$x=t=w-\alpha_2 y^2z=0,$$ respectively. Therefore, the
curve $\tilde{C}$ consists of their proper transforms
$\tilde{C}_1$ and $\tilde{C}_2$. From the intersections
\[(\tilde{L}_{yw}+\tilde{C}_1+\tilde{C}_2)\cdot\tilde{L}_{yw}=-K_Y\cdot\tilde{L}_{yw}=-\frac{3}{7}, \]
\[(\tilde{L}_{yw}+\tilde{C}_1+\tilde{C}_2)\cdot\tilde{C}_{1}=-K_Y\cdot\tilde{C}_{1}=-\frac{1}{3}, \ \ \ (\tilde{L}_{yw}+\tilde{C}_1+\tilde{C}_2)\cdot\tilde{C}_{2}=-K_Y\cdot\tilde{C}_{2}=-\frac{1}{3} \]
on the surface $T$, we obtain the intersection matrix of the
curves $\tilde{L}_{yw}$,  $\tilde{C}_{1 }$ and $\tilde{C}_2$
\[ \left(\begin{array}{ccc}
- \frac{3}{7}-\tilde{L}_{yw}\cdot \tilde{C}_{1}- \tilde{L}_{yw}\cdot \tilde{C}_{2}&\tilde{L}_{yw}\cdot \tilde{C}_{1}& \tilde{L}_{yw}\cdot \tilde{C}_{2}\\
        \tilde{L}_{yw}\cdot \tilde{C}_{1}& -\frac{1}{3}- \tilde{L}_{yw}\cdot \tilde{C}_{1}-\tilde{C}_{1}\cdot \tilde{C}_2&\tilde{C}_{1}\cdot \tilde{C}_2\\
         \tilde{L}_{yw}\cdot \tilde{C}_{2}& \tilde{C}_{1}\cdot\tilde{C}_2&-\frac{1}{3}- \tilde{L}_{zw}\cdot \tilde{C}_{2}- \tilde{C}_{1}\cdot\tilde{C}_2\\
        \end{array}\right).
\]
It is easy to check that it is negative-definite since
$\tilde{L}_{yw}\cdot \tilde{C}_{1}$,
$\tilde{L}_{yw}\cdot\tilde{C}_2$  and $\tilde{C}_{1}\cdot
\tilde{C}_2$ are non-negative.

\item For the singular point $O_y$ with  $c_1=c_2=0$ and $c_3\ne 0$ we
may assume that  $c_3=1$ and $c_4=0$ by a coordinate change. Note
that in such a case, we must have the monomial $xyw^2$, i.e.,
$e\ne 0$: otherwise the hypersurface $X_{17}$ is not quasi-smooth
at the point defined by $x=z=t=w^2+y^7=0$.

Choose a general member $H$ in the linear system generated by
$x^3$ and $z$. Then it is a normal surface of degree $17$ in
$\mathbb{P}(1,2,5,7)$. Let $D$ be the curve on $H$ cut out by the
equation $x=0$. Let $T$ be the proper transform of the surface
$H$. Then $T$ is normal and the curve $\tilde{D}$ is cut out by
the surface $S$.

Suppose that  $a_1\ne 0$. We may then assume that $a_1=1$ and
$a_2=0$ by a coordinate change. The curve $\tilde{D}$ then
consists of two irreducible curves $\tilde{L}_{yw}$ and
$\tilde{L}_{yt}$.  From the intersection numbers
\[(2\tilde{L}_{yw}+\tilde{L}_{yt})\cdot\tilde{L}_{yw}=-K_Y\cdot \tilde{L}_{yw}=-\frac{3}{7}, \ \ \ (2\tilde{L}_{yw}+\tilde{L}_{yt})^2=3B^3=-\frac{44}{35}\]
on the surface $T$, we obtain
\[\tilde{L}_{yw}^2=-\frac{3}{14}-\frac{1}{2}\tilde{L}_{yw}\cdot\tilde{L}_{yt}, \ \ \ \tilde{L}_{yt}^2=-\frac{2}{5}-2\tilde{L}_{yw}\cdot\tilde{L}_{yt}\]
With these intersection numbers we see that the matrix
\[\left(\begin{array}{cc}
       \tilde{L}_{yw}^2&\tilde{L}_{yw}\cdot\tilde{L}_{yt}\\
       \tilde{L}_{yw}\cdot\tilde{L}_{yt}& \tilde{L}_{yt}^2\\
\end{array}\right)= \left(\begin{array}{cc}
-\frac{3}{14}-\frac{1}{2}\tilde{L}_{yw}\cdot\tilde{L}_{yt} & \tilde{L}_{yw}\cdot\tilde{L}_{yt}\\
     \tilde{L}_{yw}\cdot\tilde{L}_{yt} & -\frac{2}{5}-2\tilde{L}_{yw}\cdot\tilde{L}_{yt} \\
        \end{array}\right)
\]
is negative-definite since $\tilde{L}_{yw}\cdot\tilde{L}_{yt}$ is
non-negative.

Suppose that  $a_1=0$. By changing the coordinate $y$, we may
assume that the defining equation of $X_{17}$ does not contain the
monomial $x^2t^3$. The curve $D$ consists of two irreducible
curves $L_{yw}$ and $L_{tw}$. In fact, we have
$\tilde{D}=3\tilde{L}_{yw}+\tilde{L}_{tw}$. Since the curve
$L_{yw}$ passes through the point $O_y$ but  the curve $L_{tw}$
does not, we have
$$
L_{yw}\cdot L_{tw}=\tilde{L}_{yw}\cdot\tilde{L}_{tw}, \ \ \
L_{tw}^2=\tilde{L}_{tw}^2.
$$
 We also
have
$$
(3\tilde{L}_{yw}+\tilde{L}_{tw})\cdot \tilde{L}_{yw}=-K_Y\cdot
\tilde{L}_{yw}=-\frac{3}{7}, \ \ \ (3L_{yw}+L_{tw})\cdot
L_{tw}=-K_{X_{17}}\cdot L_{tw}=\frac{1}{35}.
$$
To compute $L_{yw}\cdot L_{tw}$, we consider the divisor $D_y$ on
$H$ given by the equation $y=0$. Since the defining equation of
$X_{17}$ does not contain the monomial $x^2t^3$, we have
$D_y=3L_{tw}+R$, where $R$ is a curve whose support does not
contain the curve $L_{tw}$. Note that $R$ meets $L_{tw}$ only at
the point $O_{t}$. Moreover, we can easily see that $L_{tw}\cdot
R=\frac{2}{5}$ since $H$ is quasi-smooth at the point $O_t$. Then
the intersection
$$
(3L_{tw}+R)\cdot L_{tw}=-2K_{X_{17}}\cdot L_{tw}=\frac{2}{35}%
$$
 implies that $L_{tw}^2=-\frac{4}{35}$. This gives a negative-definite matrix
\[\left(\begin{array}{cc}
       \tilde{L}_{yw}^2&\tilde{L}_{yw}\cdot\tilde{L}_{tw}\\
       \tilde{L}_{yw}\cdot\tilde{L}_{tw}& \tilde{L}_{tw}^2\\
\end{array}\right)= \left(\begin{array}{cc}
-\frac{10}{63}& \frac{1}{21}\\
   \frac{1}{21}& -\frac{4}{35}\\
        \end{array}\right).
\]

\item For the singular point $O_y$ with $c_1=c_2=c_3=0$, we consider
linear system $|-7K_{X_{17}}|$. Every member in the linear system
passes through the point $O_y$. The proper transform of a general
member in $|-7K_{X_{17}}|$ belongs to the linear system $|7B+3E|$.
The base locus of the linear system $|-7K_{X_{17}}|$ possibly
contains only the curve $L_{yz}$ and the curve $L_{zt}$. If they
are contained in $X_{17}$, we see  $$(7B+3E)\cdot
\tilde{L}_{yz}=-7K_{X_{17}}\cdot L_{yz}-\frac{1}{2}E\cdot
\tilde{L}_{yz}=\frac{2}{3}, \ \ \ (7B+3E)\cdot
\tilde{L}_{zt}=-7K_{X_{17}}\cdot L_{zt}=\frac{7}{15}.$$ Therefore,
$T$ is nef.

\end{Note}




\begin{center}
\begin{longtable}{|l|c|c|c|c|c|}
\hline
\multicolumn{6}{|l|}{\underline{\textbf{No. 34}}: $X_{18}\subset\mathbb{P}(1,1,2,6,9)$\hfill $A^3=1/6$}\\
\multicolumn{6}{|l|}{
\begin{minipage}[m]{.86\linewidth}
\vspace*{1.2mm} $w^2+t^3+wf_{9}(x,y,z,t)+f_{18}(x,y,z,t)$
\vspace*{1.2mm}
\end{minipage}
}\\
\hline \hline
\begin{minipage}[m]{.28\linewidth}
\begin{center}
Singularity
\end{center}
\end{minipage}&
\begin{minipage}[m]{.04\linewidth}
\begin{center}
$B^3$
\end{center}
\end{minipage}&
\begin{minipage}[m]{.11\linewidth}
\begin{center}
Linear

system
\end{center}
\end{minipage}&
\begin{minipage}[m]{.11\linewidth}
\begin{center}
Surface $T$
\end{center}
\end{minipage}&
\begin{minipage}[m]{.11\linewidth}
\begin{center}
\vspace*{1mm}
 \vorder
\vspace*{1mm}
\end{center}
\end{minipage}&
\begin{minipage}[m]{.18\linewidth}
\begin{center}
Condition
\end{center}
\end{minipage}\\
\hline
\begin{minipage}[m]{.28\linewidth}

$O_tO_w=1\times\frac{1}{3}(1_x,1_y,2_z)$ \boundary

\end{minipage}&
\begin{minipage}[m]{.04\linewidth}
\begin{center}
$0$
\end{center}
\end{minipage}&
\begin{minipage}[m]{.11\linewidth}
\begin{center}
$B$
\end{center}
\end{minipage}&
\begin{minipage}[m]{.11\linewidth}
\begin{center}
$y$
\end{center}
\end{minipage}
&
\begin{minipage}[m]{.11\linewidth}
\begin{center}
$y$
\end{center}
\end{minipage}&
\begin{minipage}[m]{.18\linewidth}
\begin{center}

\end{center}
\end{minipage}\\
\hline
\begin{minipage}[m]{.28\linewidth}

$O_zO_t=3\times\frac{1}{2}(1_x,1_y,1_w)$ \boundary

\end{minipage}&
\begin{minipage}[m]{.04\linewidth}
\begin{center}
$-$
\end{center}
\end{minipage}&
\begin{minipage}[m]{.11\linewidth}
\begin{center}
$B$
\end{center}
\end{minipage}&
\begin{minipage}[m]{.11\linewidth}
\begin{center}
$y$
\end{center}
\end{minipage}
&
\begin{minipage}[m]{.11\linewidth}
\begin{center}
$y$
\end{center}
\end{minipage}&
\begin{minipage}[m]{.18\linewidth}
\begin{center}

\end{center}
\end{minipage}\\
\hline

\end{longtable}
\end{center}

 \begin{Note}

\item 
The curve defined by $x=y=0$ is always  irreducible
since we have
 the monomials $w^2$ and $t^3$.  Therefore, the $1$-cycle $\Gamma$  for each singular point is irreducible.
\end{Note}


\begin{center}
\begin{longtable}{|l|c|c|c|c|c|}
\hline
\multicolumn{6}{|l|}{\underline{\textbf{No. 35}}: $X_{18}\subset\mathbb{P}(1,1,3,5,9)$\hfill $A^3=2/15$}\\
\multicolumn{6}{|l|}{
\begin{minipage}[m]{.86\linewidth}
\vspace*{1.2mm} $w^2+zt^3+wf_{9}(x,y,z,t)+f_{18}(x,y,z,t)$
\vspace*{1.2mm}
\end{minipage}
}\\
\hline \hline
\begin{minipage}[m]{.28\linewidth}
\begin{center}
Singularity
\end{center}
\end{minipage}&
\begin{minipage}[m]{.04\linewidth}
\begin{center}
$B^3$
\end{center}
\end{minipage}&
\begin{minipage}[m]{.11\linewidth}
\begin{center}
Linear

system
\end{center}
\end{minipage}&
\begin{minipage}[m]{.11\linewidth}
\begin{center}
Surface $T$
\end{center}
\end{minipage}&
\begin{minipage}[m]{.11\linewidth}
\begin{center}
\vspace*{1mm}
 \vorder
\vspace*{1mm}
\end{center}
\end{minipage}&
\begin{minipage}[m]{.18\linewidth}
\begin{center}
Condition
\end{center}
\end{minipage}\\
\hline
\begin{minipage}[m]{.28\linewidth}

$O_t=\frac{1}{5}(1_x,1_y,4_w)$ $\positive$

\end{minipage}&
\begin{minipage}[m]{.04\linewidth}
\begin{center}
$+$
\end{center}
\end{minipage}&
\begin{minipage}[m]{.11\linewidth}
\begin{center}
$3B-E$
\end{center}
\end{minipage}&
\begin{minipage}[m]{.11\linewidth}
\begin{center}
$z$
\end{center}
\end{minipage}
&
\begin{minipage}[m]{.11\linewidth}
\begin{center}
$w^2$
\end{center}
\end{minipage}&
\begin{minipage}[m]{.18\linewidth}
\begin{center}

\end{center}
\end{minipage}\\
\hline
\begin{minipage}[m]{.28\linewidth}

$O_zO_w=2\times\frac{1}{2}(1_x,1_y,1_t)$ \boundary

\end{minipage}&
\begin{minipage}[m]{.04\linewidth}
\begin{center}
$-$
\end{center}
\end{minipage}&
\begin{minipage}[m]{.11\linewidth}
\begin{center}
$B$
\end{center}
\end{minipage}&
\begin{minipage}[m]{.11\linewidth}
\begin{center}
$y$
\end{center}
\end{minipage}
&
\begin{minipage}[m]{.11\linewidth}
\begin{center}
$y$
\end{center}
\end{minipage}&
\begin{minipage}[m]{.18\linewidth}
\begin{center}

\end{center}
\end{minipage}\\
\hline

\end{longtable}
\end{center}

\begin{Note}

\item 
The $1$-cycle $\Gamma$ for the singular points of type
$\frac{1}{2}(1,1,1)$ is irreducible since we have  $w^2$ and
$zt^3$.
\end{Note}




\begin{center}
\begin{longtable}{|l|c|c|c|c|c|}
\hline
\multicolumn{6}{|l|}{\textbf{No. 36}: $X_{18}\subset\mathbb{P}(1,1,4,6,7)$\hfill $A^3=3/28$}\\
\multicolumn{6}{|l|}{
\begin{minipage}[m]{.86\linewidth}
\vspace*{1.2mm}

$zw^2+t^3-z^3t+wf_{11}(x,y,z,t)+f_{18}(x,y,z,t)$ \vspace*{1.2mm}
\end{minipage}
}\\
\hline \hline
\begin{minipage}[m]{.28\linewidth}
\begin{center}
Singularity
\end{center}
\end{minipage}&
\begin{minipage}[m]{.04\linewidth}
\begin{center}
$B^3$
\end{center}
\end{minipage}&
\begin{minipage}[m]{.11\linewidth}
\begin{center}
Linear

system
\end{center}
\end{minipage}&
\begin{minipage}[m]{.11\linewidth}
\begin{center}
Surface $T$
\end{center}
\end{minipage}&
\begin{minipage}[m]{.11\linewidth}
\begin{center}
\vspace*{1mm}
 \vorder
\vspace*{1mm}
\end{center}
\end{minipage}&
\begin{minipage}[m]{.18\linewidth}
\begin{center}
Condition
\end{center}
\end{minipage}\\
\hline
\begin{minipage}[m]{.28\linewidth}

$O_w=\frac{1}{7}(1,1,6)$ \quadratic

\end{minipage}&
\multicolumn{4}{|l|}{\begin{minipage}[m]{.37\linewidth}
\begin{center}
$zw^2$
\end{center}
\end{minipage}}&
\begin{minipage}[m]{.18\linewidth}
\begin{center}

\end{center}
\end{minipage}\\
\hline
\begin{minipage}[m]{.28\linewidth}

$O_z=\frac{1}{4}(1,1,3)$ \ellipticone

\end{minipage}&\multicolumn{4}{|l|}{\begin{minipage}[m]{.37\linewidth}
\begin{center}
$zw^2-z^3t$
\end{center}
\end{minipage}}&
\begin{minipage}[m]{.18\linewidth}
\begin{center}
\end{center}
\end{minipage}\\
\hline
\begin{minipage}[m]{.28\linewidth}

$O_zO_t=1\times\frac{1}{2}(1_x,1_y,1_w)$ \boundary

\end{minipage}&
\begin{minipage}[m]{.04\linewidth}
\begin{center}
$-$
\end{center}
\end{minipage}&
\begin{minipage}[m]{.11\linewidth}
\begin{center}
$B$
\end{center}
\end{minipage}&
\begin{minipage}[m]{.11\linewidth}
\begin{center}
$y$
\end{center}
\end{minipage}
&
\begin{minipage}[m]{.11\linewidth}
\begin{center}
$y$
\end{center}
\end{minipage}&
\begin{minipage}[m]{.18\linewidth}
\begin{center}

\end{center}
\end{minipage}\\

\hline
\end{longtable}
\end{center}

\begin{Note}

\item 
The $1$-cycle $\Gamma$ for the singular point of type
$\frac{1}{2}(1,1,1)$ is irreducible because of the monomials
$zw^2$ and $t^3$.
\end{Note}

\begin{center}
\begin{longtable}{|l|c|c|c|c|c|}
\hline
\multicolumn{6}{|l|}{\underline{\textbf{No. 37}}: $X_{18}\subset\mathbb{P}(1,2,3,4,9)$\hfill $A^3=1/12$}\\
\multicolumn{6}{|l|}{
\begin{minipage}[m]{.86\linewidth}
\vspace*{1.2mm}
$(w-\beta_1z^3)(w-\beta_2z^3)+y\prod_{i=1}^{4}(t-\alpha_i
y^2)+at^3z^2+wf_{9}(x,y,z,t)+
t^3g_{6}(x,y,z)+t^2g_{10}(x,y,z)+tg_{14}(x,y,z)+g_{18}(x,y,z)$
\vspace*{1.2mm}
\end{minipage}
}\\

\hline \hline
\begin{minipage}[m]{.28\linewidth}
\begin{center}
Singularity
\end{center}
\end{minipage}&
\begin{minipage}[m]{.04\linewidth}
\begin{center}
$B^3$
\end{center}
\end{minipage}&
\begin{minipage}[m]{.11\linewidth}
\begin{center}
Linear

system
\end{center}
\end{minipage}&
\begin{minipage}[m]{.11\linewidth}
\begin{center}
Surface $T$
\end{center}
\end{minipage}&
\begin{minipage}[m]{.11\linewidth}
\begin{center}
\vspace*{1mm}
 \vorder
\vspace*{1mm}
\end{center}
\end{minipage}&
\begin{minipage}[m]{.18\linewidth}
\begin{center}
Condition
\end{center}
\end{minipage}\\
\hline
\begin{minipage}[m]{.28\linewidth}

$O_t=\frac{1}{4}(1_x,3_z,1_w)$ \boundary

\end{minipage}&
\begin{minipage}[m]{.04\linewidth}
\begin{center}
$0$
\end{center}
\end{minipage}&
\begin{minipage}[m]{.11\linewidth}
\begin{center}
$2B$
\end{center}
\end{minipage}&
\begin{minipage}[m]{.11\linewidth}
\begin{center}
$y$
\end{center}
\end{minipage}
&
\begin{minipage}[m]{.11\linewidth}
\begin{center}
$w^2$
\end{center}
\end{minipage}&
\begin{minipage}[m]{.18\linewidth}
\begin{center}

\end{center}
\end{minipage}\\
\hline
\begin{minipage}[m]{.28\linewidth}
 $O_zO_w=2\times\frac{1}{3}(1_x,2_y,1_t)$ \boundary

\end{minipage}&
\begin{minipage}[m]{.04\linewidth}
\begin{center}
$-$
\end{center}
\end{minipage}&
\begin{minipage}[m]{.11\linewidth}
\begin{center}
$2B$
\end{center}
\end{minipage}&
\begin{minipage}[m]{.11\linewidth}
\begin{center}
$y$
\end{center}
\end{minipage}
&
\begin{minipage}[m]{.11\linewidth}
\begin{center}
$y$
\end{center}
\end{minipage}&
\begin{minipage}[m]{.18\linewidth}
\begin{center}
$a\ne0$
\end{center}
\end{minipage}\\
\hline
\begin{minipage}[m]{.28\linewidth}
$O_zO_w=2\times\frac{1}{3}(1_x,2_y,1_t)$ $\surface$

\end{minipage}&
\begin{minipage}[m]{.04\linewidth}
\begin{center}
$-$
\end{center}
\end{minipage}&
\begin{minipage}[m]{.11\linewidth}
\begin{center}
$2B$
\end{center}
\end{minipage}&
\begin{minipage}[m]{.11\linewidth}
\begin{center}
$x^2$, $y$
\end{center}
\end{minipage}
&
\begin{minipage}[m]{.11\linewidth}
\begin{center}
$x^2$, $y$
\end{center}
\end{minipage}&
\begin{minipage}[m]{.18\linewidth}
\begin{center}

$a=0$

\end{center}
\end{minipage}\\
\hline
\begin{minipage}[m]{.28\linewidth}

$O_yO_t=4\times\frac{1}{2}(1_x,1_z,1_w)$ \boundary

\end{minipage}&
\begin{minipage}[m]{.04\linewidth}
\begin{center}
$-$
\end{center}
\end{minipage}&
\begin{minipage}[m]{.11\linewidth}
\begin{center}
$4B+E$
\end{center}
\end{minipage}&
\begin{minipage}[m]{.11\linewidth}
\begin{center}
$t-\alpha_i y^2$
\end{center}
\end{minipage}
&
\begin{minipage}[m]{.11\linewidth}
\begin{center}
$w^2$
\end{center}
\end{minipage}&
\begin{minipage}[m]{.18\linewidth}
\begin{center}

\end{center}
\end{minipage}\\
\hline

\end{longtable}
\end{center}

\begin{Note}

\item 
For the singular point $O_t$, the $1$-cycle $\Gamma$ can be
reducible. In case, we see that $\Gamma$ consists of  the proper
transforms of the curves defined by $x=y=w-\beta_1z^{3}=0$ and
$x=y=w-\beta_2z^{3}=0$. These two irreducible components are
symmetric with respect to the biregular involution of $X_{18}$. In
addition, the point $O_t$ is the intersection point of these two
curves. Consequently, the components of $\Gamma$ are numerically
equivalent to each other.

\item For each singular point of type $\frac{1}{3}(1,2,1)$, the
$1$-cycle $\Gamma$ is irreducible if the constant $a$ is not zero.

\item Suppose that  the constant $a$ is zero. We have only to consider
one of the singular points of type $\frac{1}{3}(1,2,1)$. The other
singular point can be excluded in the same way. We put $\beta_1=0$
and consider the singular point $O_z$. We may also assume that the
defining equation of $X_{18}$ contains neither $xz^3t^2$ nor
$x^2z^4t$ by changing the coordinate $w$.

We take a general surface $H$ from the pencil $|-2K_{X_{18}}|$ and
then let $T$ be the proper transform of the surface. Note that the
surface $H$ is normal. However, it is not quasi-smooth at the
point $O_t$. The intersection of $T$ with the surface $S$ gives us
a divisor consisting of two irreducible curves on the normal
surface $T$. They are the proper transforms $\tilde{L}_{zt}$ and
$\tilde{C}$   of the curve  $L_{zt}$ and the curve $C$ defined by
$$x=y=w-\beta_2z^3=0,$$ respectively. From the intersection numbers
\[(\tilde{L}_{zt}+\tilde{C})\cdot\tilde{L}_{zt}=-K_Y\cdot\tilde{L}_{zt}=-\frac{1}{4}, \ \ \ (\tilde{L}_{zt}+\tilde{C})^2=2B^3=-\frac{1}{6}\]
on the surface $T$, we obtain
\[\tilde{L}_{zt}^2=-\frac{1}{4}-\tilde{L}_{zt}\cdot\tilde{C}, \ \ \ \tilde{C}^2=\frac{1}{12}-\tilde{L}_{zt}\cdot\tilde{C}.\]

To compute the intersection number $\tilde{L}_{zt}\cdot\tilde{C}$,
we consider the divisor $D_w$ on $H$ cut by the equation $w=0$. We
easily see that $D_w=2L_{zt}+R$, where $R$ is a curve whose
support does not contain $L_{zt}$. The curve $R$ and $L_{zt}$
intersects at the point $O_z$. Let $\tilde{R}$ be the proper
transform of $R$. Then we have $\tilde{L}_{zt}\cdot \tilde{R}=0$
since they are disconnected on $T$. From the intersection
\[(2\tilde{L}_{zt}+\tilde{R})\cdot\tilde{L}_{zt}=(9B+E)\cdot \tilde{L}_{zt}=-\frac{5}{4}\]
we obtain $\tilde{L}_{zt}^2=-\frac{5}{8}$.  Therefore,
$\tilde{L}_{zt}\cdot\tilde{C}=\frac{3}{8}$ and
$\tilde{C}^2=-\frac{7}{24}$. With these intersection numbers we
see that the matrix
\[\left(\begin{array}{cc}
       \tilde{L}_{zt}^2&\tilde{L}_{zt}\cdot\tilde{C}\\
       \tilde{L}_{zt}\cdot\tilde{C}& \tilde{C}^2\\
\end{array}\right)= \left(\begin{array}{cc}
-\frac{5}{8}& \frac{3}{8}\\
  \frac{3}{8} & -\frac{7}{24} \\
        \end{array}\right)
\]
is negative-definite.

\item For each singular point of type $\frac{1}{2}(1,1,1)$, the
$1$-cycle $\Gamma$ may be reducible. In case, it consists of  the
proper transforms of the curves defined by $$x=t-\alpha_i
y^2=w+by^3z+cz^{3}=0$$ and $$x=t-\alpha_i y^2=w+dy^3z+ez^{3}=0,$$
where $b$, $c$, $d$, $e$ are constants. These two irreducible
components are also symmetric with respect to the biregular
involution of $X_{18}$. In addition, the singular point  is the
intersection point of these two curves. Therefore, the components
of $\Gamma$ are numerically equivalent.

\end{Note}



\begin{center}
\begin{longtable}{|l|c|c|c|c|c|}
\hline
\multicolumn{6}{|l|}{\textbf{No. 38}: $X_{18}\subset\mathbb{P}(1,2,3,5,8)$\hfill $A^3=3/40$}\\
\multicolumn{6}{|l|}{
\begin{minipage}[m]{.86\linewidth}
\vspace*{1.2mm}

$yw^2+t^2(a_1w+a_2zt)+z^6+y^9+wf_{10}(x,y,z,t)+f_{18}(x,y,z,t)$
\vspace*{1.2mm}
\end{minipage}
}\\
\hline \hline
\begin{minipage}[m]{.28\linewidth}
\begin{center}
Singularity
\end{center}
\end{minipage}&
\begin{minipage}[m]{.04\linewidth}
\begin{center}
$B^3$
\end{center}
\end{minipage}&
\begin{minipage}[m]{.11\linewidth}
\begin{center}
Linear

system
\end{center}
\end{minipage}&
\begin{minipage}[m]{.11\linewidth}
\begin{center}
Surface $T$
\end{center}
\end{minipage}&
\begin{minipage}[m]{.11\linewidth}
\begin{center}
\vspace*{1mm}
 \vorder
\vspace*{1mm}
\end{center}
\end{minipage}&
\begin{minipage}[m]{.18\linewidth}
\begin{center}
Condition
\end{center}
\end{minipage}\\
\hline
\begin{minipage}[m]{.28\linewidth}

$O_w=\frac{1}{8}(1,3,5)$ \quadratic

\end{minipage}&
\multicolumn{4}{|l|}{\begin{minipage}[m]{.37\linewidth}
\begin{center}
$yw^2$
\end{center}
\end{minipage}}&
\begin{minipage}[m]{.18\linewidth}
\begin{center}

\end{center}
\end{minipage}\\
\hline
\begin{minipage}[m]{.28\linewidth}

$O_t=\frac{1}{5}(1,2,3)$ \quadraticone

\end{minipage}&
\multicolumn{4}{|l|}{\begin{minipage}[m]{.37\linewidth}
\begin{center}
$wt^2$
\end{center}
\end{minipage}}&
\begin{minipage}[m]{.18\linewidth}
\begin{center}

\end{center}
\end{minipage}\\
\hline
\begin{minipage}[m]{.28\linewidth}
 $O_yO_w=2\times\frac{1}{2}(1_x,1_z,1_t)$ $\nef$

\end{minipage}&
\begin{minipage}[m]{.04\linewidth}
\begin{center}
$-$
\end{center}
\end{minipage}&
\begin{minipage}[m]{.11\linewidth}
\begin{center}
$5B+2E$
\end{center}
\end{minipage}&
\begin{minipage}[m]{.11\linewidth}
\begin{center}
$t$
\end{center}
\end{minipage}
&
\begin{minipage}[m]{.11\linewidth}
\begin{center}
$t$
\end{center}
\end{minipage}&
\begin{minipage}[m]{.18\linewidth}
\begin{center}

\end{center}
\end{minipage}\\

\hline
\end{longtable}
\end{center}

\begin{Note}

\item 
For the singular points  of type $\frac{1}{2}(1,1,1)$ we consider
the linear system $|-5K_{X_{18}}|$.  Every member of the linear
system passes through  the singular points  of type
$\frac{1}{2}(1,1,1)$ and the base locus of the linear system
contains no curves. Since the proper transform of a general member
in $|-5K_{X_{18}}|$ belongs to the linear system $|5B+2E|$, the
divisor $T$ is nef.
\end{Note}


\begin{center}
\begin{longtable}{|l|c|c|c|c|c|}
\hline
\multicolumn{6}{|l|}{\underline{\textbf{No. 39}}: $X_{18}\subset\mathbb{P}(1,3,4,5,6)$\hfill $A^3=1/20$}\\
\multicolumn{6}{|l|}{
\begin{minipage}[m]{.86\linewidth}
\vspace*{1.2mm}
$(w-\alpha_1y^2)(w-\alpha_2y^2)(w-\alpha_3y^2)+yt^3+z^3w+at^2z^2+by^2z^3+w^2f_{6}(x,y,z,t)+wf_{12}(x,y,z,t)+t^2g_{8}(x,y,z)+tg_{13}(x,y,z)+g_{18}(x,y,z)$
\vspace*{1.2mm}
\end{minipage}
}\\
\hline \hline
\begin{minipage}[m]{.28\linewidth}
\begin{center}
Singularity
\end{center}
\end{minipage}&
\begin{minipage}[m]{.04\linewidth}
\begin{center}
$B^3$
\end{center}
\end{minipage}&
\begin{minipage}[m]{.11\linewidth}
\begin{center}
Linear

system
\end{center}
\end{minipage}&
\begin{minipage}[m]{.11\linewidth}
\begin{center}
Surface $T$
\end{center}
\end{minipage}&
\begin{minipage}[m]{.11\linewidth}
\begin{center}
\vspace*{1mm}
 \vorder
\vspace*{1mm}
\end{center}
\end{minipage}&
\begin{minipage}[m]{.18\linewidth}
\begin{center}
Condition
\end{center}
\end{minipage}\\
\hline
\begin{minipage}[m]{.28\linewidth}

$O_t=\frac{1}{5}(1_x,4_z,1_w)$ \boundary

\end{minipage}&
\begin{minipage}[m]{.04\linewidth}
\begin{center}
$0$
\end{center}
\end{minipage}&
\begin{minipage}[m]{.11\linewidth}
\begin{center}
$3B$
\end{center}
\end{minipage}&
\begin{minipage}[m]{.11\linewidth}
\begin{center}
$y$
\end{center}
\end{minipage}
&
\begin{minipage}[m]{.11\linewidth}
\begin{center}
$w^3$
\end{center}
\end{minipage}&
\begin{minipage}[m]{.18\linewidth}
\begin{center}

\end{center}
\end{minipage}\\
\hline
\begin{minipage}[m]{.28\linewidth}

$O_z=\frac{1}{4}(1_x,3_y,1_t)$ \boundary

\end{minipage}&
\begin{minipage}[m]{.04\linewidth}
\begin{center}
$-$
\end{center}
\end{minipage}&
\begin{minipage}[m]{.11\linewidth}
\begin{center}
$3B$
\end{center}
\end{minipage}&
\begin{minipage}[m]{.11\linewidth}
\begin{center}
$y$
\end{center}
\end{minipage}
&
\begin{minipage}[m]{.11\linewidth}
\begin{center}
$y$
\end{center}
\end{minipage}&
\begin{minipage}[m]{.18\linewidth}
\begin{center}
$a\ne0$
\end{center}
\end{minipage}\\
\hline
\begin{minipage}[m]{.28\linewidth}

$O_z=\frac{1}{4}(1_x,3_y,1_t)$ $\surface$

\end{minipage}&
\begin{minipage}[m]{.04\linewidth}
\begin{center}
$-$
\end{center}
\end{minipage}&
\begin{minipage}[m]{.11\linewidth}
\begin{center}
$3B$
\end{center}
\end{minipage}&
\begin{minipage}[m]{.11\linewidth}
\begin{center}
$x^3,y$
\end{center}
\end{minipage}
&
\begin{minipage}[m]{.11\linewidth}
\begin{center}
$x^3$, $y$
\end{center}
\end{minipage}&
\begin{minipage}[m]{.18\linewidth}
\begin{center}
$a=0$
\end{center}
\end{minipage}\\
\hline
\begin{minipage}[m]{.28\linewidth}
$O_zO_w=1\times\frac{1}{2}(1_x,1_y,1_t)$ $\nef$

\end{minipage}&
\begin{minipage}[m]{.04\linewidth}
\begin{center}
$-$
\end{center}
\end{minipage}&
\begin{minipage}[m]{.11\linewidth}
\begin{center}
$5B+2E$
\end{center}
\end{minipage}&
\begin{minipage}[m]{.11\linewidth}
\begin{center}
$t$
\end{center}
\end{minipage}
&
\begin{minipage}[m]{.11\linewidth}
\begin{center}
$t$
\end{center}
\end{minipage}&
\begin{minipage}[m]{.18\linewidth}
\begin{center}

\end{center}
\end{minipage}\\
\hline
\begin{minipage}[m]{.28\linewidth}
 $O_yO_w=3\times\frac{1}{3}(1_x,1_z,2_t)$ $\nef$

\end{minipage}&
\begin{minipage}[m]{.04\linewidth}
\begin{center}
$-$
\end{center}
\end{minipage}&
\begin{minipage}[m]{.11\linewidth}
\begin{center}
$5B+E$
\end{center}
\end{minipage}&
\begin{minipage}[m]{.11\linewidth}
\begin{center}
$t$
\end{center}
\end{minipage}
&
\begin{minipage}[m]{.11\linewidth}
\begin{center}
$t$
\end{center}
\end{minipage}&
\begin{minipage}[m]{.18\linewidth}
\begin{center}
$b\ne 0$
\end{center}
\end{minipage}\\
\hline
\begin{minipage}[m]{.28\linewidth}
 $O_yO_w=3\times\frac{1}{3}(1_x,1_z,2_t)$  \surface

\end{minipage}&
\begin{minipage}[m]{.04\linewidth}
\begin{center}
$-$
\end{center}
\end{minipage}&
\begin{minipage}[m]{.11\linewidth}
\begin{center}
$5B+E$
\end{center}
\end{minipage}&
\begin{minipage}[m]{.11\linewidth}
\begin{center}
$x^5$, $xz$, $t$
\end{center}
\end{minipage}
&
\begin{minipage}[m]{.11\linewidth}
\begin{center}
$t$
\end{center}
\end{minipage}&
\begin{minipage}[m]{.18\linewidth}
\begin{center}
$b=0$
\end{center}
\end{minipage}\\
\hline

\end{longtable}
\end{center}

\begin{Note}

\item 
For the singular point $O_t$, the $1$-cycle $\Gamma$ may be
reducible. However, in case, it consists of two irreducible
components. One is the proper transform $\tilde{L}_{zt}$ of the
curve $L_{zt}$  and the other is the proper transform $\tilde{C}$
of the curve defined by $$x=y=w^2+z^3=0.$$ We can easily check that
\[E\cdot \tilde{C}=2E\cdot \tilde{L}_{zt}=\frac{1}{2},\ \ \ B\cdot \tilde{C}=2B\cdot \tilde{L}_{zt}=0.\]
Therefore, the irreducible curves $\tilde{L}_{zt}$ and $\tilde{C}$
are numerically proportional on $Y$.

\item The $1$-cycle $\Gamma$ for the singular point $O_z$ with $a\ne 0$
is irreducible due to $w^3$ and $t^2z^2$.

\item For the singular point $O_z$ with $a=0$ we may assume that the
defining equation of $X_{18}$ contains neither $xz^3t$ nor
$x^2z^4$ by changing the coordinate $w$.

We take a general surface $H$ from the pencil $|-3K_{X_{18}}|$ and
then let $T$ be the proper transform of the surface. Note that the
surface $H$ is normal. However, it is not quasi-smooth at the
point $O_t$. The intersection of $T$ with the surface $S$ gives us
a divisor consisting of two irreducible curves
 $\tilde{L}_{zt}$ and $\tilde{C}$ on the normal surface $T$.
 The curve $\tilde{C}$ is the proper transform of the
curve $C$ defined by $$x=y=w^2+z^3=0.$$ From the intersection
numbers
\[(\tilde{L}_{zt}+\tilde{C})\cdot\tilde{L}_{zt}=-K_Y\cdot \tilde{L}_{zt}=-\frac{1}{5}, \ \ \ (\tilde{L}_{zt}+\tilde{C})^2=3B^3=-\frac{1}{10}\]
on the surface $T$, we obtain
\[\tilde{L}_{zt}^2=-\frac{1}{5}-\tilde{L}_{zt}\cdot\tilde{C}, \ \ \ \tilde{C}^2=\frac{1}{10}-\tilde{L}_{zt}\cdot\tilde{C}.\]
Therefore,
\[(3\tilde{L}_{zt}+\tilde{R})\cdot\tilde{L}_{zt}=-6K_Y\cdot \tilde{L}_{zt}=-\frac{6}{5}\]
we obtain $\tilde{L}_{zt}^2=-\frac{2}{5}$.  With these
intersection numbers we see
\[\left(\begin{array}{cc}
       \tilde{L}_{zt}^2&\tilde{L}_{zt}\cdot\tilde{C}\\
       \tilde{L}_{zt}\cdot\tilde{C}& \tilde{C}^2\\
\end{array}\right)= \left(\begin{array}{cc}
-\frac{1}{5}-\tilde{L}_{zt}\cdot\tilde{C} & \tilde{L}_{zt}\cdot\tilde{C}\\
     \tilde{L}_{zt}\cdot\tilde{C} & \frac{1}{10}-\tilde{L}_{zt}\cdot\tilde{C} \\
        \end{array}\right).
\]
To compute the intersection number $\tilde{L}_{zt}\cdot\tilde{C}$,
we take the divisor $D_w$ on $H$ cut by the equation $w=0$. This
divisor can be written as $D_w=3L_{zt}+R$, where $R$ is a curve
whose support does not contain $L_{zt}$. The curve $R$ and
$L_{zt}$ intersects at the point $O_z$. Let $\tilde{R}$ be the
proper transform of $R$.  We have $\tilde{L}_{zt}\cdot
\tilde{R}=0$ since they are disconnected on $T$. From the
intersection
\[(3\tilde{L}_{zt}+\tilde{R})\cdot\tilde{L}_{zt}=-6K_Y\cdot \tilde{L}_{zt}=-\frac{6}{5}\]
we obtain $\tilde{L}_{zt}^2=-\frac{2}{5}$.  Therefore,
$\tilde{L}_{zt}\cdot\tilde{C}=\frac{1}{5}$. This shows that the
matrix is negative-definite.

\item For the singular point of type $\frac{1}{2}(1,1,1)$  we consider
the linear system generated by $x^{15}$, $y^5$ and $t^3$ on the
hypersurface $X_{18}$. Its base locus is cut out by $x=y=t=0$.
Since we have the monomial $z^3w$, the base locus does not contain
curves. Therefore the proper transform of a general member in the
linear system is nef by Lemma~\ref{lemma:nefness} and it belongs
to $|15B+6E|$. Consequently, the surface $T$ is nef since
$3T\sim_{\mathbb{Q}} 15B+6E$.

\end{Note}

 For the singular points of type $\frac{1}{3}(1,1,2)$ we may assume
that $\alpha_1=0$ and consider the singular point $O_y$. The other
points can be dealt with in the same way.
Since $\alpha_1=0$,  the defining equation of $X_{18}$ does not
contain the monomial $y^6$. We may also assume that it does not
contain the monomials $x^6y^4$, $x^3y^5$, $x^2y^4z$ and $xy^4t$ by
changing the coordinate $w$.

\begin{Note}

\item 
 For the singular point $O_y$ with $b\ne 0$ we consider the linear
system generated by $x^{30}$, $t^5$ and $w^6$ on the hypersurface
$X_{18}$. Its base locus is cut out by $x=t=w=0$. Since $b\ne 0$,
its base locus does not contain curves, and hence the proper
transform of a general member in the linear system is nef by
Lemma~\ref{lemma:nefness}. It belongs to $|30B+6E|$. Consequently,
the surface $T$ is nef since $6T\sim_{\mathbb{Q}} 30B+6E$.

\item For the singular point $O_y$ with $b=0$  we take a general surface
$H$ from the linear system generated by $x^5$, $xz$ and $t$. Then
$H$ is normal. Moreover, the surface $H$ is smooth at the point
$x=t=w=z^3+\alpha_1\alpha_2y^4=0$. Indeed, the defining equation
of $X_{18}$ must contain at least one of the monomials $xz^2y^3$,
 $ty^3z$; otherwise $X_{18}$ would be singular at
the point $x=t=w=z^3+\alpha_1\alpha_2y^4=0$. Plugging in
$t=\lambda xz+\mu x^5$ with general complex numbers  $\lambda$ and
$\mu$
 into the defining equation of $X_{18}$, we obtain
the defining equation of $H$ in $\mathbb{P}(1,3,4,6)$. It must
contain the monomial $xz^2y^3$. Therefore, the surface $H$ is
smooth at the point $x=t=w=z^3+\alpha_1\alpha_2y^4=0$

 Let
$T$ be the proper transform of the surface $H$. The intersection
of $T$ with the surface $S$ defines a divisor consisting of two
irreducible curves
 $\tilde{L}_{yz}$ and $\tilde{C}$ on the normal surface $T$.
 The curve $\tilde{C}$ is the proper transform of the
curve $C$ defined by
$$
x=t=z^3+(w-\alpha_2 y^2)(w-\alpha_3 y^2)=0.
$$
The curves $L_{yz}$ and $C$ intersect at the point defined by
$x=t=w=z^3+\alpha_1\alpha_2y^4=0$. From the intersection numbers
\[(\tilde{L}_{yz}+\tilde{C})\cdot\tilde{L}_{yz}=-K_Y\cdot \tilde{L}_{yz}=-\frac{1}{4}, \ \ \ (\tilde{L}_{yz}+\tilde{C})^2=B^2\cdot(5B+E)=-\frac{1}{12}\]
on the surface $T$, we obtain
\[\tilde{L}_{yz}^2=-\frac{1}{4}-\tilde{L}_{yz}\cdot\tilde{C}, \ \ \ \tilde{C}^2=\frac{1}{6}-\tilde{L}_{yz}\cdot\tilde{C}\]
With these intersection numbers we see that the matrix
\[\left(\begin{array}{cc}
       \tilde{L}_{yz}^2&\tilde{L}_{yz}\cdot\tilde{C}\\
       \tilde{L}_{yz}\cdot\tilde{C}& \tilde{C}^2\\
\end{array}\right)= \left(\begin{array}{cc}
-\frac{1}{4}-\tilde{L}_{yz}\cdot\tilde{C} & \tilde{L}_{yz}\cdot\tilde{C}\\
     \tilde{L}_{yz}\cdot\tilde{C} & \frac{1}{6}-\tilde{L}_{yz}\cdot\tilde{C} \\
        \end{array}\right)
\]
is negative-definite since $\tilde{L}_{yz}$ and $\tilde{C}$
intersect at a smooth point of the surface $T$.
\end{Note}





\begin{center}
\begin{longtable}{|l|c|c|c|c|c|}
\hline
\multicolumn{6}{|l|}{\textbf{No. 40}: $X_{19}\subset\mathbb{P}(1,3,4,5,7)$\hfill $A^3=19/420$}\\
\multicolumn{6}{|l|}{
\begin{minipage}[m]{.86\linewidth}
\vspace*{1.2mm}
$tw^2-zt^3+z^3(a_1w+a_2yz)+y^4(b_1w+b_2yz+b_3y^2x)+ay^2z^2t+by^3t^2+wf_{12}(x,y,z,t)+f_{19}(x,y,z,t)$
\vspace*{1.2mm}
\end{minipage}
}\\
\hline \hline
\begin{minipage}[m]{.28\linewidth}
\begin{center}
Singularity
\end{center}
\end{minipage}&
\begin{minipage}[m]{.04\linewidth}
\begin{center}
$B^3$
\end{center}
\end{minipage}&
\begin{minipage}[m]{.11\linewidth}
\begin{center}
Linear

system
\end{center}
\end{minipage}&
\begin{minipage}[m]{.11\linewidth}
\begin{center}
Surface $T$
\end{center}
\end{minipage}&
\begin{minipage}[m]{.11\linewidth}
\begin{center}
\vspace*{1mm}
 \vorder
\vspace*{1mm}
\end{center}
\end{minipage}&
\begin{minipage}[m]{.18\linewidth}
\begin{center}
Condition
\end{center}
\end{minipage}\\
\hline
\begin{minipage}[m]{.28\linewidth}

$O_w=\frac{1}{7}(1,3,4)$ \quadratic

\end{minipage}&
\multicolumn{4}{|l|}{\begin{minipage}[m]{.37\linewidth}
\begin{center}
$tw^2$
\end{center}
\end{minipage}}&
\begin{minipage}[m]{.18\linewidth}
\begin{center}

\end{center}
\end{minipage}\\
\hline
\begin{minipage}[m]{.28\linewidth}

$O_t=\frac{1}{5}(1,3,2)$ \elliptic

\end{minipage}&\multicolumn{4}{|l|}{\begin{minipage}[m]{.37\linewidth}
\begin{center}
$tw^2-zt^3$
\end{center}
\end{minipage}}&
\begin{minipage}[m]{.18\linewidth}
\begin{center}

\end{center}
\end{minipage}\\
\hline

\begin{minipage}[m]{.28\linewidth}

$O_z=\frac{1}{4}(1_x,3_y,1_t)$ \boundary

\end{minipage}&
\begin{minipage}[m]{.04\linewidth}
\begin{center}
$-$
\end{center}
\end{minipage}&
\begin{minipage}[m]{.11\linewidth}
\begin{center}
$3B$
\end{center}
\end{minipage}&
\begin{minipage}[m]{.11\linewidth}
\begin{center}
$y$
\end{center}
\end{minipage}
&
\begin{minipage}[m]{.11\linewidth}
\begin{center}
$zt^3$
\end{center}
\end{minipage}&
\begin{minipage}[m]{.18\linewidth}
\begin{center}
$a_1\ne 0$
\end{center}
\end{minipage}\\
\hline
\begin{minipage}[m]{.28\linewidth}

$O_z=\frac{1}{4}(1_x,1_t,3_w)$ $\surface$

\end{minipage}&
\begin{minipage}[m]{.04\linewidth}
\begin{center}
$-$
\end{center}
\end{minipage}&
\begin{minipage}[m]{.11\linewidth}
\begin{center}
$3B$
\end{center}
\end{minipage}&
\begin{minipage}[m]{.11\linewidth}
\begin{center}
$x^3$, $z$
\end{center}
\end{minipage}
&
\begin{minipage}[m]{.11\linewidth}
\begin{center}
$x^3$, $zt^3$
\end{center}
\end{minipage}&
\begin{minipage}[m]{.18\linewidth}
\begin{center}
$a_1=0$
\end{center}
\end{minipage}\\
\hline
\begin{minipage}[m]{.28\linewidth}

$O_y=\frac{1}{3}(1_x,1_z,2_t)$ \boundary

\end{minipage}&
\begin{minipage}[m]{.04\linewidth}
\begin{center}
$-$
\end{center}
\end{minipage}&
\begin{minipage}[m]{.11\linewidth}
\begin{center}
$7B+E$
\end{center}
\end{minipage}&
\begin{minipage}[m]{.11\linewidth}
\begin{center}
$w$
\end{center}
\end{minipage}
&
\begin{minipage}[m]{.11\linewidth}
\begin{center}
$y^3t^2$
\end{center}
\end{minipage}&
\begin{minipage}[m]{.18\linewidth}
\begin{center}
$b_1\ne 0$, $b\ne 0$ \\$a_2\ne 0$
\end{center}
\end{minipage}\\\hline
\begin{minipage}[m]{.28\linewidth}

$O_y=\frac{1}{3}(1_x,1_z,2_t)$ \surface

\end{minipage}&
\begin{minipage}[m]{.04\linewidth}
\begin{center}
$-$
\end{center}
\end{minipage}&
\begin{minipage}[m]{.11\linewidth}
\begin{center}
$7B+E$
\end{center}
\end{minipage}&
\begin{minipage}[m]{.11\linewidth}
\begin{center}
$x^7$, $w$
\end{center}
\end{minipage}
&
\begin{minipage}[m]{.11\linewidth}
\begin{center}
$y^3t^2$
\end{center}
\end{minipage}&
\begin{minipage}[m]{.18\linewidth}
\begin{center}
$b_1\ne 0$, $b\ne 0$ \\$a_2= 0$
\end{center}
\end{minipage}\\

\hline
\begin{minipage}[m]{.28\linewidth}

$O_y=\frac{1}{3}(1_x,1_z,2_t)$ \surface

\end{minipage}&
\begin{minipage}[m]{.04\linewidth}
\begin{center}
$-$
\end{center}
\end{minipage}&
\begin{minipage}[m]{.11\linewidth}
\begin{center}
$7B+ E$\\

\end{center}
\end{minipage}&
\begin{minipage}[m]{.11\linewidth}
\begin{center}
$x^7$, $w$
\end{center}
\end{minipage}
&
\begin{minipage}[m]{.11\linewidth}
\begin{center}
$z^4y$
\end{center}
\end{minipage}&
\begin{minipage}[m]{.18\linewidth}
\begin{center}
$b_1\ne 0$, $b= 0$\\
$a_2\ne 0$
\end{center}
\end{minipage}\\
\hline
\begin{minipage}[m]{.28\linewidth}

$O_y=\frac{1}{3}(1_x,1_z,2_t)$  \surface

\end{minipage}&
\begin{minipage}[m]{.04\linewidth}
\begin{center}
$-$
\end{center}
\end{minipage}&
\begin{minipage}[m]{.11\linewidth}
\begin{center}
$7B$\\

\end{center}
\end{minipage}&
\begin{minipage}[m]{.11\linewidth}
\begin{center}
$x^7$, $w$
\end{center}
\end{minipage}
&
\begin{minipage}[m]{.11\linewidth}
\begin{center}
$x^7$, $zt^3$
\end{center}
\end{minipage}&
\begin{minipage}[m]{.18\linewidth}
\begin{center}
$b_1\ne 0$, $b= 0$\\
$a_2= 0$
\end{center}
\end{minipage}\\

\hline
\begin{minipage}[m]{.28\linewidth}

$O_y=\frac{1}{3}(1_x,2_t,1_w)$ \surface

\end{minipage}&
\begin{minipage}[m]{.04\linewidth}
\begin{center}
$-$
\end{center}
\end{minipage}&
\begin{minipage}[m]{.11\linewidth}
\begin{center}
$4B$
\end{center}
\end{minipage}&
\begin{minipage}[m]{.11\linewidth}
\begin{center}
$x^4$, $z$
\end{center}
\end{minipage}
&
\begin{minipage}[m]{.11\linewidth}
\begin{center}
$x^4, tw^2$
\end{center}
\end{minipage}&
\begin{minipage}[m]{.18\linewidth}
\begin{center}
$b_1=0$, $b_2\not =0$
\end{center}
\end{minipage}\\
\hline
\begin{minipage}[m]{.28\linewidth}

$O_y=\frac{1}{3}(1_z,2_t,1_w)$ \nef

\end{minipage}&
\begin{minipage}[m]{.04\linewidth}
\begin{center}
$-$
\end{center}
\end{minipage}&
\begin{minipage}[m]{.11\linewidth}
\begin{center}
$7B+2E$
\end{center}
\end{minipage}&
\begin{minipage}[m]{.11\linewidth}
\begin{center}
$w$
\end{center}
\end{minipage}
&
\begin{minipage}[m]{.11\linewidth}
\begin{center}
$w$
\end{center}
\end{minipage}&
\begin{minipage}[m]{.18\linewidth}
\begin{center}
$b_1=0$, $b_2 =0$
\end{center}
\end{minipage}\\

\hline

\end{longtable}
\end{center}

\begin{Note}

\item 
The $1$-cycle $\Gamma$ for the singular point $O_z$ with $a_1\ne
0$ is irreducible due to the monomials $tw^2$, $zt^3$ and $z^3w$.

\item For the singular point $O_z$ with $a_1= 0$ we choose a general
member $H$ in the linear system $|-3K_{X_{19}}|$. Then it is a
normal surface of degree $19$ in $\mathbb{P}(1,4,5,7)$. Let $T$ be
the proper transform of the divisor $H$. The intersection of $T$
with the surface $S$ defines a divisor consisting of
 two irreducible curves $\tilde{L}_{zw}$ and $\tilde{C}$. The curve $\tilde{C}$ is the proper transform of the
curve $C$ defined by $$x=y=w^2-zt^2=0.$$ From the intersection
\[(\tilde{L}_{zw}+\tilde{C})\cdot\tilde{L}_{zw}=-K_Y\cdot\tilde{L}_{zw}=-\frac{1}{21}, \ \ \ (\tilde{L}_{zw}+\tilde{C})^2=3B^3=-\frac{4}{35}\]
on the surface $T$, we obtain
\[\tilde{L}_{zw}^2=-\frac{1}{21}-\tilde{L}_{zw}\cdot\tilde{C}, \ \ \ \tilde{C}^2=-\frac{1}{15}-\tilde{L}_{zw}\cdot\tilde{C}.\]
With these intersection numbers we see that the matrix
\[\left(\begin{array}{cc}
       \tilde{L}_{zw}^2 &\tilde{L}_{zw}\cdot \tilde{C} \\
        \tilde{L}_{zw}\cdot \tilde{C}& \tilde{C}^2\\
\end{array}\right)= \left(\begin{array}{cc}
-\frac{1}{21}-\tilde{L}_{zw}\cdot\tilde{C} &\tilde{L}_{zw}\cdot\tilde{C} \\
        \tilde{L}_{zw}\cdot\tilde{C}& -\frac{1}{15}-\tilde{L}_{zw}\cdot\tilde{C} \\
        \end{array}\right)
\]
is negative-definite since $\tilde{L}_{zw}\cdot\tilde{C}$ is
non-negative number.

Consider  the singular point $O_y$ with  $b_1\ne 0$. In this case,
we may assume that $b_1=1$ and  $b_2=b_3=0$ by a suitable
coordinate change.

\item For the singular point $O_y$ with  $b_1\ne 0$, $b\ne 0$ and
$a_2\ne 0$ the $1$-cycle $\Gamma$ is irreducible because of the
monomials $zt^3$, $yz^4$ and $y^3t^2$.

\item For the singular point $O_y$ with  $b_1\ne 0$, $b\ne 0$ and
$a_2=0$ we may assume that the monomial $xy^2z^3$ does not appear
in $f_{19}$ by changing the coordinate $w$ in a suitable way. We
must then have $a\ne 0$; otherwise the hypersurface would  not be
quasi-smooth at the point defined by $$x=t=w=a_1z^3+y^4=0.$$

Take a general member $H$ in the linear system generated by $x^7$
and $w$.  Then it is a normal surface of degree $19$ in
$\mathbb{P}(1,3, 4,5)$. Let $T$  be the proper transform of the
divisor $H$. The intersection of $T$ with the surface $S$ gives us
a divisor consisting of
 two irreducible curves $\tilde{L}_{yz}$ and $\tilde{C}_1$. The curve $\tilde{C}_1$ is the proper transform of
the curve $C_1$ defined by $x=w=-zt^2+ay^2z^2+by^3t=0$. From the
intersection
\[(\tilde{L}_{yz}+\tilde{C}_1)\cdot\tilde{L}_{yz}=-K_Y\cdot\tilde{L}_{yz}=-\frac{1}{4}, \ \ \ (\tilde{L}_{yz}+\tilde{C}_1)^2=B^2\cdot(7B+E)=-\frac{7}{20}\]
on the surface $T$, we obtain
\[\tilde{L}_{yz}^2=-\frac{1}{4}-\tilde{L}_{yz}\cdot\tilde{C}_1, \ \ \ \tilde{C}_1^2=-\frac{1}{10}-\tilde{L}_{yz}\cdot\tilde{C}_1.\]
With these intersection numbers we see that the matrix
\[\left(\begin{array}{cc}
       \tilde{L}_{yz}^2 &\tilde{L}_{yz}\cdot \tilde{C}_1 \\
        \tilde{L}_{yz}\cdot \tilde{C}_1& \tilde{C}_1^2\\
\end{array}\right)= \left(\begin{array}{cc}
-\frac{1}{4}-\tilde{L}_{yz}\cdot\tilde{C}_1 &\tilde{L}_{yz}\cdot\tilde{C}_1 \\
        \tilde{L}_{yz}\cdot\tilde{C}_1& -\frac{1}{10}-\tilde{L}_{yz}\cdot\tilde{C}_1 \\
        \end{array}\right)
\]
is negative-definite since $\tilde{L}_{yz}\cdot\tilde{C}_1$ is
non-negative number.

 \item For the singular point $O_y$ with  $b_1\ne 0$,  $b= 0$ and $a_2\ne 0$,  we do the same as in the case where $b_1\ne0$, $b\ne0$ and $a_2=0$. The difference is that the intersection of $T$ with the surface $S$ gives us a divisor
consisting of
 two irreducible curves $\tilde{L}_{yt}$ and $\tilde{C}_2$. The curve $\tilde{C}_2$ is the proper transform of
the curve $C_2$ defined by $x=w=-t^3+a_2yz^3+ay^2zt=0$. From the
intersections
\[(\tilde{L}_{yt}+\tilde{C}_2)\cdot\tilde{L}_{yt}=-K_Y\cdot\tilde{L}_{yt}=-\frac{1}{10}, \ \ \ (\tilde{L}_{yt}+\tilde{C}_2)^2=B^2\cdot(7B+E)=-\frac{7}{20}\]
on the surface $T$, we obtain
\[\tilde{L}_{yt}^2=-\frac{1}{10}-\tilde{L}_{yt}\cdot\tilde{C}_2, \ \ \ \tilde{C}_2^2=-\frac{1}{4}-\tilde{L}_{yt}\cdot\tilde{C}_2.\]
This shows  $\tilde{L}_{yt}$ and $\tilde{C}_2$ forms a
negative-definite divisor on $T$.

\item  For the singular point $O_y$ with  $b_1\ne 0$,  $b= 0$ and $a_2= 0$  we may assume that the monomial $xy^2z^3$ does not appear in $f_{19}$ by changing the coordinate $w$ in a suitable way. We must then have $a\ne 0$; otherwise the hypersurface would  not be quasi-smooth at the point defined by $x=t=w=a_1z^3+y^4=0$.

We do the same as the previous case. In this case, we obtain a
divisor consisting of
 three irreducible curves $\tilde{L}_{yz}$,  $\tilde{L}_{yt }$ and $\tilde{C}_3$. The curve $\tilde{C}_3$ is the proper
transform of the curve $C_3$ defined by $x=w=-t^2+ay^2z=0$. From
the intersections
\[\begin{split} & (\tilde{L}_{yz}+\tilde{L}_{yt}+\tilde{C}_3)\cdot\tilde{L}_{yz}=-K_Y\cdot\tilde{L}_{yz}=-\frac{1}{4},\\
& (\tilde{L}_{yz}+\tilde{L}_{yt}+\tilde{C}_3)\cdot\tilde{L}_{yt}=-K_Y\cdot\tilde{L}_{yt}=-\frac{1}{10},\\
&
(\tilde{L}_{yz}+\tilde{L}_{yt}+\tilde{C}_3)\cdot\tilde{C}_{3}=-K_Y\cdot\tilde{C}_{3}=-\frac{1}{6}.\\
\end{split}\] on the surface $T$, we obtain the intersection
matrix of the curves $\tilde{L}_{yz}$,  $\tilde{L}_{yt }$ and
$\tilde{C}_3$
\[ \left(\begin{array}{ccc}
- \frac{1}{4}-\tilde{L}_{yz}\cdot \tilde{L}_{yt}- \tilde{L}_{yz}\cdot \tilde{C}_{3}&\tilde{L}_{yz}\cdot \tilde{L}_{yt}& \tilde{L}_{yz}\cdot \tilde{C}_{3}\\
        \tilde{L}_{yz}\cdot \tilde{L}_{yt}& -\frac{1}{10}- \tilde{L}_{yz}\cdot \tilde{L}_{yt}-\tilde{L}_{yt}\cdot \tilde{C}_3&\tilde{L}_{yt}\cdot \tilde{C}_3\\
         \tilde{L}_{yz}\cdot \tilde{C}_{3}& \tilde{L}_{yt}\cdot\tilde{C}_3&-\frac{1}{6}- \tilde{L}_{yz}\cdot \tilde{C}_{3}- \tilde{L}_{yt}\cdot\tilde{C}_3\\
        \end{array}\right).
\]
It is easy to check that it is negative-definite since
$\tilde{L}_{yz}\cdot \tilde{L}_{yt}$,
$\tilde{L}_{yz}\cdot\tilde{C}_3$ and
$\tilde{L}_{yt}\cdot\tilde{C}_3$ are non-negative numbers.

\item For the singular point $O_y$ with  $b_1=0$ and $b_2\ne 0$  we may
put $b_2=1$ by a coordinate change. Choose a general member $H$ in
the pencil on $X_{19}$ generated by $x^4$ and $z$. Then it is a
normal surface of degree $19$ in $\mathbb{P}(1,3,5,7)$. Let $T$ be
the proper transform of the divisor $H$. The surface $S$ cuts out
the surface $T$ into a divisor $\tilde{D}$.

We suppose that $b\ne 0$. The divisor $\tilde{D}$ then consists of
two irreducible curves $\tilde{L}_{yw}$ and $\tilde{C}$. The curve
$\tilde{C}$ is the proper transform of the curve $C$ defined by
$$x=z=w^2+by^3t=0.$$ From the intersections
\[(\tilde{L}_{yw}+\tilde{C})\cdot\tilde{L}_{yw}=-K_Y\cdot\tilde{L}_{yw}=-\frac{2}{7}, \ \ \ (\tilde{L}_{yw}+\tilde{C})^2=4B^3=-\frac{17}{35}\]
on the surface $T$, we obtain
\[\tilde{L}_{yw}^2=-\frac{2}{7}-\tilde{L}_{yw}\cdot\tilde{C},\ \ \ \tilde{C}^2=-\frac{1}{5}-\tilde{L}_{yw}\cdot\tilde{C}.\]
With these intersection numbers we see that the matrix
\[\left(\begin{array}{cc}
       \tilde{L}_{yw}^2 &\tilde{L}_{yw}\cdot \tilde{C} \\
        \tilde{L}_{yw}\cdot \tilde{C}& \tilde{C}^2\\
\end{array}\right)= \left(\begin{array}{cc}
-\frac{2}{7}-\tilde{L}_{yw}\cdot\tilde{C} &\tilde{L}_{yw}\cdot\tilde{C} \\
       \tilde{L}_{yw}\cdot\tilde{C}& -\frac{1}{5}-\tilde{L}_{yw}\cdot\tilde{C} \\
        \end{array}\right)
\]
is negative-definite since $\tilde{L}_{yw}\cdot\tilde{C}$ is
non-negative number.

We now suppose that  $b= 0$. The divisor $\tilde{D}$ then consists
of two irreducible curves $\tilde{L}_{yw}$ and $\tilde{L}_{yt}$.
  From the intersection
\[(\tilde{L}_{yw}+2\tilde{L}_{yt})\cdot\tilde{L}_{yw}=-K_Y\cdot\tilde{L}_{yw}=-\frac{2}{7}, \ \ \  (\tilde{L}_{yw}+2\tilde{L}_{yt})^2=4B^3=-\frac{17}{35}\]
on the surface $T$, we obtain
\[\tilde{L}_{yw}^2=-\frac{2}{7}-2\tilde{L}_{yw}\cdot\tilde{L}_{yt}, \ \ \ \tilde{L}_{yt}^2=-\frac{1}{20}-\frac{1}{2}\tilde{L}_{yw}\cdot\tilde{L}_{yt}.\]
This again shows that  $\tilde{L}_{yw}$ and $\tilde{L}_{yt}$ form
a negative-definite divisor on $T$.

\item  For the singular point $O_y$ with  $b_1=b_2= 0$ we consider the linear
system generated by $z^{35}$, $t^{28}$ and $w^{20}$ on the
hypersurface $X_{19}$. Its base locus is cut out by $z=t=w=0$.
Since we have the monomial $xy^6$, the base locus does not contain
curves. Therefore the proper transform of a general member in the
linear system is nef by Lemma~\ref{lemma:nefness}. It belongs to
$|140B+40E|$. Consequently, the surface $T$ is nef since
$20T\sim_{\mathbb{Q}} 140B+40E$.
\end{Note}



\begin{center}
\begin{longtable}{|l|c|c|c|c|c|}
\hline
\multicolumn{6}{|l|}{\textbf{No. 41}: $X_{20}\subset\mathbb{P}(1,1,4,5,10)$\hfill $A^3=1/10$}\\
\multicolumn{6}{|l|}{
\begin{minipage}[m]{.86\linewidth}
\vspace*{1.2mm}

$(w-\alpha_1t^2)(w-\alpha_2
t^2)+z^5+wf_{10}(x,y,z,t)+f_{20}(x,y,z,t)$ \vspace*{1.2mm}
\end{minipage}
}\\
\hline \hline
\begin{minipage}[m]{.28\linewidth}
\begin{center}
Singularity
\end{center}
\end{minipage}&
\begin{minipage}[m]{.04\linewidth}
\begin{center}
$B^3$
\end{center}
\end{minipage}&
\begin{minipage}[m]{.11\linewidth}
\begin{center}
Linear

system
\end{center}
\end{minipage}&
\begin{minipage}[m]{.11\linewidth}
\begin{center}
Surface $T$
\end{center}
\end{minipage}&
\begin{minipage}[m]{.11\linewidth}
\begin{center}
\vspace*{1mm}
 \vorder
\vspace*{1mm}
\end{center}
\end{minipage}&
\begin{minipage}[m]{.18\linewidth}
\begin{center}
Condition
\end{center}
\end{minipage}\\
\hline
\begin{minipage}[m]{.28\linewidth}

$O_tO_w=2\times\frac{1}{5}(1,1,4)$ \quadratic

\end{minipage}&
\multicolumn{4}{|l|}{\begin{minipage}[m]{.37\linewidth}
\begin{center}
$wt^2$
\end{center}
\end{minipage}}&
\begin{minipage}[m]{.18\linewidth}
\begin{center}

\end{center}
\end{minipage}\\

\hline
\begin{minipage}[m]{.28\linewidth}

$O_zO_w=1\times\frac{1}{2}(1_x,1_y,1_t)$ \boundary

\end{minipage}&
\begin{minipage}[m]{.04\linewidth}
\begin{center}
$-$
\end{center}
\end{minipage}&
\begin{minipage}[m]{.11\linewidth}
\begin{center}
$B$
\end{center}
\end{minipage}&
\begin{minipage}[m]{.11\linewidth}
\begin{center}
$y$
\end{center}
\end{minipage}
&
\begin{minipage}[m]{.11\linewidth}
\begin{center}
$y$
\end{center}
\end{minipage}&
\begin{minipage}[m]{.18\linewidth}
\begin{center}

\end{center}
\end{minipage}\\

\hline
\end{longtable}
\end{center}

\begin{Note}

\item 
We may assume that $\alpha_1=0$. To see how to treat the singular
points of type $\frac{1}{5}(1,1,4)$, we have only to consider the
singular point $O_t$. The other point can be untwisted or excluded
in the same way.

\item The $1$-cycle $\Gamma$ for the singular point of type
$\frac{1}{2}(1,1,1)$  is irreducible because of the monomial $w^2$
and $z^5$.
\end{Note}


\begin{center}
\begin{longtable}{|l|c|c|c|c|c|}
\hline
\multicolumn{6}{|l|}{\textbf{No. 42}: $X_{20}\subset\mathbb{P}(1,2,3,5,10)$\hfill $A^3=1/15$}\\
\multicolumn{6}{|l|}{
\begin{minipage}[m]{.86\linewidth}
\vspace*{1.2mm} $(w-\alpha_1 y^5)(w-\alpha_2
y^5)+wt^2+z^5(a_1t+a_2yz)+wf_{10}(x,y,z,t)+f_{20}(x,y,z,t)$

\vspace*{1.2mm}
\end{minipage}
}\\
\hline \hline
\begin{minipage}[m]{.28\linewidth}
\begin{center}
Singularity
\end{center}
\end{minipage}&
\begin{minipage}[m]{.04\linewidth}
\begin{center}
$B^3$
\end{center}
\end{minipage}&
\begin{minipage}[m]{.11\linewidth}
\begin{center}
Linear

system
\end{center}
\end{minipage}&
\begin{minipage}[m]{.11\linewidth}
\begin{center}
Surface $T$
\end{center}
\end{minipage}&
\begin{minipage}[m]{.11\linewidth}
\begin{center}
\vspace*{1mm}
 \vorder
\vspace*{1mm}
\end{center}
\end{minipage}&
\begin{minipage}[m]{.18\linewidth}
\begin{center}
Condition
\end{center}
\end{minipage}\\
\hline
\begin{minipage}[m]{.28\linewidth}

$O_z=\frac{1}{3}(1_x,2_y,1_w)$ \boundary

\end{minipage}&
\begin{minipage}[m]{.04\linewidth}
\begin{center}
$-$
\end{center}
\end{minipage}&
\begin{minipage}[m]{.11\linewidth}
\begin{center}
$2B$
\end{center}
\end{minipage}&
\begin{minipage}[m]{.11\linewidth}
\begin{center}
$y$
\end{center}
\end{minipage}
&
\begin{minipage}[m]{.11\linewidth}
\begin{center}
$y$
\end{center}
\end{minipage}&
\begin{minipage}[m]{.18\linewidth}
\begin{center}
$a_1\ne 0$
\end{center}
\end{minipage}\\
\hline
\begin{minipage}[m]{.28\linewidth}

$O_z=\frac{1}{3}(1_x,2_t,1_w)$ \boundary

\end{minipage}&
\begin{minipage}[m]{.04\linewidth}
\begin{center}
$-$
\end{center}
\end{minipage}&
\begin{minipage}[m]{.11\linewidth}
\begin{center}
$2B$
\end{center}
\end{minipage}&
\begin{minipage}[m]{.11\linewidth}
\begin{center}
$y$
\end{center}
\end{minipage}
&
\begin{minipage}[m]{.11\linewidth}
\begin{center}
$w^2$
\end{center}
\end{minipage}&
\begin{minipage}[m]{.18\linewidth}
\begin{center}
$a_1=0$
\end{center}
\end{minipage}\\
\hline
\begin{minipage}[m]{.28\linewidth}

$O_tO_w=2\times\frac{1}{5}(1,2,3)$ \quadratic

\end{minipage}&
\multicolumn{4}{|l|}{\begin{minipage}[m]{.37\linewidth}
\begin{center}
$wt^2$
\end{center}
\end{minipage}}&
\begin{minipage}[m]{.18\linewidth}
\begin{center}

\end{center}
\end{minipage}\\
\hline
\begin{minipage}[m]{.28\linewidth}
 $O_yO_w=2\times\frac{1}{2}(1,1,1)$ $\nef$

\end{minipage}&
\begin{minipage}[m]{.04\linewidth}
\begin{center}
$-$
\end{center}
\end{minipage}&
\begin{minipage}[m]{.11\linewidth}
\begin{center}
$5B+2E$
\end{center}
\end{minipage}&
\begin{minipage}[m]{.11\linewidth}
\begin{center}
$t$
\end{center}
\end{minipage}
&
\begin{minipage}[m]{.11\linewidth}
\begin{center}
$t$
\end{center}
\end{minipage}&
\begin{minipage}[m]{.18\linewidth}
\begin{center}

\end{center}
\end{minipage}\\

\hline
\end{longtable}
\end{center}

\begin{Note}

\item 
The $1$-cycle $\Gamma$ for the singular point $O_z$ with $a_1\ne
0$ is irreducible due to $w^2$ and $z^5t$.

\item The $1$-cycle $\Gamma$ for the singular point $O_z$ with $a_1= 0$
  consists of  the proper
transforms of the curve $L_{zt}$ and the curve defined by
$$x=y=w+t^2=0.$$ These two irreducible components are symmetric with
respect to the biregular involution of $X_{20}$.  Consequently the
components of $\Gamma$ are numerically equivalent to each other.

\item By changing the coordinate $w$ we may assume that $t^4$ is not in
the polynomial $f_{20}$. To see how to untwist or exclude  the
singular points of type $\frac{1}{5}(1,2,3)$ we have only to
consider the singular point $O_t$. The other point can be treated
in the same way.

\item For the singular points of type $\frac{1}{2}(1,1,1)$, we consider
the linear system $|-5K_{X_{20}}|$.  Every member of the linear
system passes through the singular points of type
$\frac{1}{2}(1,1,1)$ and the base locus of the linear system
contains no curves. Since the proper transform of a general member
in $|-5K_{X_{20}}|$ belongs to the linear system $|5B+2E|$, the
divisor $T$ is nef.

\end{Note}


\begin{center}
\begin{longtable}{|l|c|c|c|c|c|}
\hline
\multicolumn{6}{|l|}{\textbf{No. 43}: $X_{20}\subset\mathbb{P}(1,2,4,5,9)$\hfill $A^3=1/18$}\\
\multicolumn{6}{|l|}{
\begin{minipage}[m]{.86\linewidth}
\vspace*{1.2mm}

$yw^2+t^4+\prod^{5}_{i=1}(z-\alpha_i
y^2)+wf_{11}(x,y,z,t)+f_{20}(x,y,z,t)$ \vspace*{1.2mm}
\end{minipage}
}\\
\hline \hline
\begin{minipage}[m]{.28\linewidth}
\begin{center}
Singularity
\end{center}
\end{minipage}&
\begin{minipage}[m]{.04\linewidth}
\begin{center}
$B^3$
\end{center}
\end{minipage}&
\begin{minipage}[m]{.11\linewidth}
\begin{center}
Linear

system
\end{center}
\end{minipage}&
\begin{minipage}[m]{.11\linewidth}
\begin{center}
Surface $T$
\end{center}
\end{minipage}&
\begin{minipage}[m]{.11\linewidth}
\begin{center}
\vspace*{1mm}
 \vorder
\vspace*{1mm}
\end{center}
\end{minipage}&
\begin{minipage}[m]{.18\linewidth}
\begin{center}
Condition
\end{center}
\end{minipage}\\
\hline
\begin{minipage}[m]{.28\linewidth}

$O_w=\frac{1}{9}(1,4,5)$ \quadratic

\end{minipage}&
\multicolumn{4}{|l|}{\begin{minipage}[m]{.37\linewidth}
\begin{center}
$yw^2$
\end{center}
\end{minipage}}&
\begin{minipage}[m]{.18\linewidth}
\begin{center}

\end{center}
\end{minipage}\\

\hline
\begin{minipage}[m]{.28\linewidth}

$O_yO_z=5\times\frac{1}{2}(1_x,1_t,1_w)$ \boundary

\end{minipage}&
\begin{minipage}[m]{.04\linewidth}
\begin{center}
$-$
\end{center}
\end{minipage}&
\begin{minipage}[m]{.11\linewidth}
\begin{center}
$4B+E$
\end{center}
\end{minipage}&
\begin{minipage}[m]{.11\linewidth}
\begin{center}
$z-\alpha_i y^2$
\end{center}
\end{minipage}
&
\begin{minipage}[m]{.11\linewidth}
\begin{center}
$yw^2$
\end{center}
\end{minipage}&
\begin{minipage}[m]{.18\linewidth}
\begin{center}

\end{center}
\end{minipage}\\

\hline
\end{longtable}
\end{center}

\begin{Note}

\item 
The $1$-cycles $\Gamma$ for the singular points of type
$\frac{1}{2}(1,1,1)$ are irreducible due to the monomials $yw^2$
and $t^4$.
\end{Note}


\begin{center}
\begin{longtable}{|l|c|c|c|c|c|}
\hline
\multicolumn{6}{|l|}{\textbf{No. 44}: $X_{20}\subset\mathbb{P}(1,2,5,6,7)$\hfill $A^3=1/21$}\\
\multicolumn{6}{|l|}{
\begin{minipage}[m]{.86\linewidth}
\vspace*{1.2mm}

$tw^2+y(t-\alpha_1 y^3)(t-\alpha_2 y^3)(t-\alpha_3
y^3)+z^4+wf_{14}(x,y,z,t)+f_{20}(x,y,z,t)$ \vspace*{1.2mm}
\end{minipage}
}\\
\hline \hline
\begin{minipage}[m]{.28\linewidth}
\begin{center}
Singularity
\end{center}
\end{minipage}&
\begin{minipage}[m]{.04\linewidth}
\begin{center}
$B^3$
\end{center}
\end{minipage}&
\begin{minipage}[m]{.11\linewidth}
\begin{center}
Linear

system
\end{center}
\end{minipage}&
\begin{minipage}[m]{.11\linewidth}
\begin{center}
Surface $T$
\end{center}
\end{minipage}&
\begin{minipage}[m]{.11\linewidth}
\begin{center}
\vspace*{1mm}
 \vorder
\vspace*{1mm}
\end{center}
\end{minipage}&
\begin{minipage}[m]{.18\linewidth}
\begin{center}
Condition
\end{center}
\end{minipage}\\
\hline
\begin{minipage}[m]{.28\linewidth}

$O_w=\frac{1}{7}(1,2,5)$ \quadratic

\end{minipage}&
\multicolumn{4}{|l|}{\begin{minipage}[m]{.37\linewidth}
\begin{center}
$tw^2$
\end{center}
\end{minipage}}&
\begin{minipage}[m]{.18\linewidth}
\begin{center}

\end{center}
\end{minipage}\\
\hline
\begin{minipage}[m]{.28\linewidth}

$O_t=\frac{1}{6}(1,5,1)$ \elliptic

\end{minipage}&\multicolumn{4}{|l|}{\begin{minipage}[m]{.37\linewidth}
\begin{center}
$tw^2-yt^3$
\end{center}
\end{minipage}}&
\begin{minipage}[m]{.18\linewidth}
\begin{center}
\end{center}
\end{minipage}\\
\hline\begin{minipage}[m]{.28\linewidth}
$O_yO_t=3\times\frac{1}{2}(1_x,1_z,1_w)$ \nef

\end{minipage}&
\begin{minipage}[m]{.04\linewidth}
\begin{center}
$-$
\end{center}
\end{minipage}&
\begin{minipage}[m]{.11\linewidth}
\begin{center}
$7B+3E$
\end{center}
\end{minipage}&
\begin{minipage}[m]{.11\linewidth}
\begin{center}
$w$
\end{center}
\end{minipage}
&
\begin{minipage}[m]{.11\linewidth}
\begin{center}
$w$
\end{center}
\end{minipage}&
\begin{minipage}[m]{.18\linewidth}
\begin{center}

\end{center}
\end{minipage}\\
\hline
\end{longtable}
\end{center}

\begin{Note}

\item 
Consider the linear system generated by $x^{35}$, $z^7$ and
$w^5$. Since the defining equation of $X_{20}$ contains $yt^3$,
the base locus of the linear system contains no curve. Therefore,
the proper transform of a general member in this linear system is
nef. Since it belongs to $|35B+15E|$,  the surface $T$ is nef.
\end{Note}



\begin{center}
\begin{longtable}{|l|c|c|c|c|c|}
\hline
\multicolumn{6}{|l|}{\textbf{No. 45}: $X_{20}\subset\mathbb{P}(1,3,4,5,8)$\hfill $A^3=1/24$}\\
\multicolumn{6}{|l|}{
\begin{minipage}[m]{.86\linewidth}
\vspace*{1.2mm}

$z(w-\alpha_1 z^2)(w-\alpha_2
z^2)+t^4+y^4(a_1w+a_2yt)+wf_{12}(x,y,z,t)+f_{20}(x,y,z,t)$
\vspace*{1.2mm}
\end{minipage}
}\\

\hline \hline
\begin{minipage}[m]{.28\linewidth}
\begin{center}
Singularity
\end{center}
\end{minipage}&
\begin{minipage}[m]{.04\linewidth}
\begin{center}
$B^3$
\end{center}
\end{minipage}&
\begin{minipage}[m]{.11\linewidth}
\begin{center}
Linear

system
\end{center}
\end{minipage}&
\begin{minipage}[m]{.11\linewidth}
\begin{center}
Surface $T$
\end{center}
\end{minipage}&
\begin{minipage}[m]{.11\linewidth}
\begin{center}
\vspace*{1mm}
 \vorder
\vspace*{1mm}
\end{center}
\end{minipage}&
\begin{minipage}[m]{.18\linewidth}
\begin{center}
Condition
\end{center}
\end{minipage}\\
\hline
\begin{minipage}[m]{.28\linewidth}

$O_w=\frac{1}{8}(1,3,5)$ \quadratic

\end{minipage}&
\multicolumn{4}{|l|}{\begin{minipage}[m]{.37\linewidth}
\begin{center}
$zw^2$
\end{center}
\end{minipage}}&
\begin{minipage}[m]{.18\linewidth}
\begin{center}

\end{center}
\end{minipage}\\

\hline
\begin{minipage}[m]{.28\linewidth}

$O_y=\frac{1}{3}(1_x,1_z,2_t)$ $\nef$

\end{minipage}&
\begin{minipage}[m]{.04\linewidth}
\begin{center}
$-$
\end{center}
\end{minipage}&
\begin{minipage}[m]{.11\linewidth}
\begin{center}
$4B+E$
\end{center}
\end{minipage}&
\begin{minipage}[m]{.11\linewidth}
\begin{center}
$z$
\end{center}
\end{minipage}
&
\begin{minipage}[m]{.11\linewidth}
\begin{center}
$z$
\end{center}
\end{minipage}&
\begin{minipage}[m]{.18\linewidth}
\begin{center}

\end{center}
\end{minipage}\\
\hline
\begin{minipage}[m]{.28\linewidth}

$O_zO_w=2\times\frac{1}{4}(1_x,3_y,1_t)$ \boundary

\end{minipage}&
\begin{minipage}[m]{.04\linewidth}
\begin{center}
$-$
\end{center}
\end{minipage}&
\begin{minipage}[m]{.11\linewidth}
\begin{center}
$3B$
\end{center}
\end{minipage}&
\begin{minipage}[m]{.11\linewidth}
\begin{center}
$y$
\end{center}
\end{minipage}
&
\begin{minipage}[m]{.11\linewidth}
\begin{center}
$y$
\end{center}
\end{minipage}&
\begin{minipage}[m]{.18\linewidth}
\begin{center}

\end{center}
\end{minipage}\\

\hline
\end{longtable}
\end{center}

\begin{Note}

\item 
For the singular point $O_y$
 we consider the linear
system generated by $x^{40}$, $z^{10}$, $t^8$ and $w^5$ on the
hypersurface $X_{20}$. Its base locus does not contain curves.
Therefore the proper transform of a general member in the linear
system is nef by Lemma~\ref{lemma:nefness}. The proper transform
belongs to $|40B+10E|$. Consequently, the surface $T$ is nef since
$10T\sim_{\mathbb{Q}} 40B+10E$.

\item The $1$-cycles $\Gamma$ for the singular points of type
$\frac{1}{4}(1,3,1)$ are irreducible due to the monomials $zw^2$
and $t^4$.
\end{Note}



\begin{center}
\begin{longtable}{|l|c|c|c|c|c|}
\hline
\multicolumn{6}{|l|}{\textbf{No. 46}: $X_{21}\subset\mathbb{P}(1,1,3,7,10)$\hfill $A^3=1/10$}\\
\multicolumn{6}{|l|}{
\begin{minipage}[m]{.86\linewidth}
\vspace*{1.2mm}

$yw^2+t^3+z^7+wf_{11}(x,y,z,t)+f_{21}(x,y,z,t)$ \vspace*{1.2mm}
\end{minipage}
}\\

\hline \hline
\begin{minipage}[m]{.28\linewidth}
\begin{center}
Singularity
\end{center}
\end{minipage}&
\begin{minipage}[m]{.04\linewidth}
\begin{center}
$B^3$
\end{center}
\end{minipage}&
\begin{minipage}[m]{.11\linewidth}
\begin{center}
Linear

system
\end{center}
\end{minipage}&
\begin{minipage}[m]{.11\linewidth}
\begin{center}
Surface $T$
\end{center}
\end{minipage}&
\begin{minipage}[m]{.11\linewidth}
\begin{center}
\vspace*{1mm}
 \vorder
\vspace*{1mm}
\end{center}
\end{minipage}&
\begin{minipage}[m]{.18\linewidth}
\begin{center}
Condition
\end{center}
\end{minipage}\\
\hline
\begin{minipage}[m]{.28\linewidth}

$O_w=\frac{1}{10}(1,3,7)$ \quadratic

\end{minipage}&
\multicolumn{4}{|l|}{\begin{minipage}[m]{.37\linewidth}
\begin{center}
$yw^2$
\end{center}
\end{minipage}}&
\begin{minipage}[m]{.18\linewidth}
\begin{center}

\end{center}
\end{minipage}\\

\hline

\end{longtable}
\end{center}



\begin{center}
\begin{longtable}{|l|c|c|c|c|c|}
\hline
\multicolumn{6}{|l|}{\textbf{No. 47}: $X_{21}\subset\mathbb{P}(1,1,5,7,8)$\hfill $A^3=3/40$}\\
\multicolumn{6}{|l|}{
\begin{minipage}[m]{.86\linewidth}
\vspace*{1.2mm} $zw^2+t^3+yz^4+wf_{13}(x,y,z,t)+f_{21}(x,y,z,t)$
\vspace*{1.2mm}
\end{minipage}
}\\
\hline \hline
\begin{minipage}[m]{.28\linewidth}
\begin{center}
Singularity
\end{center}
\end{minipage}&
\begin{minipage}[m]{.04\linewidth}
\begin{center}
$B^3$
\end{center}
\end{minipage}&
\begin{minipage}[m]{.11\linewidth}
\begin{center}
Linear

system
\end{center}
\end{minipage}&
\begin{minipage}[m]{.11\linewidth}
\begin{center}
Surface $T$
\end{center}
\end{minipage}&
\begin{minipage}[m]{.11\linewidth}
\begin{center}
\vspace*{1mm}
 \vorder
\vspace*{1mm}
\end{center}
\end{minipage}&
\begin{minipage}[m]{.18\linewidth}
\begin{center}
Condition
\end{center}
\end{minipage}\\
\hline
\begin{minipage}[m]{.28\linewidth}

$O_w=\frac{1}{8}(1,1,7)$ \quadratic

\end{minipage}&
\multicolumn{4}{|l|}{\begin{minipage}[m]{.37\linewidth}
\begin{center}
$zw^2$
\end{center}
\end{minipage}}&
\begin{minipage}[m]{.18\linewidth}
\begin{center}

\end{center}
\end{minipage}\\

\hline
\begin{minipage}[m]{.28\linewidth}

$O_z=\frac{1}{5}(1_x,2_t,3_w)$ $\family$

\end{minipage}&
\begin{minipage}[m]{.04\linewidth}
\begin{center}
$+$
\end{center}
\end{minipage}&
\begin{minipage}[m]{.11\linewidth}
\begin{center}
$B-E$
\end{center}
\end{minipage}&
\begin{minipage}[m]{.11\linewidth}
\begin{center}
$y$
\end{center}
\end{minipage}
&
\begin{minipage}[m]{.11\linewidth}
\begin{center}
$t^3$
\end{center}
\end{minipage}&
\begin{minipage}[m]{.18\linewidth}
\begin{center}

\end{center}
\end{minipage}\\

\hline
\end{longtable}
\end{center}

\begin{Note}

\item 
For the singular point $O_z$, let $C_{\lambda}$ be the curve on
the surface $S_y$ defined by
$$
\left\{%
\aligned
&y=0,\\%
&w=\lambda x^{8}\\%
\endaligned\right.%
$$
for a sufficiently general complex number  $\lambda$. We then have
$$
-K_{Y}\cdot \tilde{C}_{\lambda}=(B-E)\cdot(8B+E)\cdot B=0.\\%
$$
Consider the linear system generated by $x^{72}$, $y^9t^9$ and
$y^8w^8$. Its base curve is defined by $x=y=0$. It is an
irreducible curve because we have the monomials $zw^2$ and $t^3$.
The proper transform of a general member of the linear system is
equivalent to $72B$. The only curve that intersects the divisor
$B$ negatively is the proper transform of the  irreducible curve
defined by $x=y=0$. It  is  not on the surface $T$. Therefore, if
the curve $\tilde{C}_\lambda$ is reducible, each component of the
curve $\tilde{C}_\lambda$ intersects $B$ trivially.
\end{Note}



\begin{center}
\begin{longtable}{|l|c|c|c|c|c|}
\hline
\multicolumn{6}{|l|}{\textbf{No. 48}: $X_{21}\subset\mathbb{P}(1,2,3,7,9)$\hfill $A^3=1/18$}\\
\multicolumn{6}{|l|}{
\begin{minipage}[m]{.86\linewidth}
\vspace*{1.2mm}
$zw^2+t^3+z^7+y^6(a_1w+a_2yt+a_3y^3z+a_4xy^4)+wf_{12}(x,y,z,t)+f_{21}(x,y,z,t)$

\vspace*{1.2mm}
\end{minipage}
}\\

\hline \hline
\begin{minipage}[m]{.28\linewidth}
\begin{center}
Singularity
\end{center}
\end{minipage}&
\begin{minipage}[m]{.04\linewidth}
\begin{center}
$B^3$
\end{center}
\end{minipage}&
\begin{minipage}[m]{.11\linewidth}
\begin{center}
Linear

system
\end{center}
\end{minipage}&
\begin{minipage}[m]{.11\linewidth}
\begin{center}
Surface $T$
\end{center}
\end{minipage}&
\begin{minipage}[m]{.11\linewidth}
\begin{center}
\vspace*{1mm}
 \vorder
\vspace*{1mm}
\end{center}
\end{minipage}&
\begin{minipage}[m]{.18\linewidth}
\begin{center}
Condition
\end{center}
\end{minipage}\\
\hline
\begin{minipage}[m]{.28\linewidth}

$O_w=\frac{1}{9}(1,2,7)$ \quadratic

\end{minipage}&
\multicolumn{4}{|l|}{\begin{minipage}[m]{.37\linewidth}
\begin{center}
$zw^2$
\end{center}
\end{minipage}}&
\begin{minipage}[m]{.18\linewidth}
\begin{center}

\end{center}
\end{minipage}\\
\hline
\begin{minipage}[m]{.28\linewidth}

$O_y=\frac{1}{2}(1,1,1)$ \nef

\end{minipage}&
\begin{minipage}[m]{.04\linewidth}
\begin{center}
$-$
\end{center}
\end{minipage}&
\begin{minipage}[m]{.11\linewidth}
\begin{center}
$9B+4E$
\end{center}
\end{minipage}&
\begin{minipage}[m]{.11\linewidth}
\begin{center}
$w+yt$
\end{center}
\end{minipage}
&
\begin{minipage}[m]{.11\linewidth}
\begin{center}
$w$ or $yt$
\end{center}
\end{minipage}&
\begin{minipage}[m]{.18\linewidth}
\begin{center}

\end{center}
\end{minipage}\\
\hline

\begin{minipage}[m]{.28\linewidth}
$O_zO_w=2\times\frac{1}{3}(1_x,2_y,1_t)$ \boundary

\end{minipage}&
\begin{minipage}[m]{.04\linewidth}
\begin{center}
$-$
\end{center}
\end{minipage}&
\begin{minipage}[m]{.11\linewidth}
\begin{center}
$2B$
\end{center}
\end{minipage}&
\begin{minipage}[m]{.11\linewidth}
\begin{center}
$y$
\end{center}
\end{minipage}
&
\begin{minipage}[m]{.11\linewidth}
\begin{center}
$y$
\end{center}
\end{minipage}&
\begin{minipage}[m]{.18\linewidth}
\begin{center}

\end{center}
\end{minipage}\\

\hline
\end{longtable}
\end{center}

\begin{Note}

\item 
For the singular point $O_y$
 we consider the
linear system $|-9K_{X_{24}}|$. Every member of the linear system
passes through  the singular point $O_y$ and its base locus
contains  no curves. Since the proper transform of a general
member in $|-9K_{X_{24}}|$ belongs to the linear system $|9B+4E|$,
the divisor $T$ is nef.

\item The $1$-cycles $\Gamma$ for the singular points of type
$\frac{1}{3}(1,2,1)$ are irreducible because of the monomials
$zw^2$ and $t^3$.
\end{Note}


\begin{center}
\begin{longtable}{|l|c|c|c|c|c|}
\hline
\multicolumn{6}{|l|}{\underline{\textbf{No. 49}}: $X_{21}\subset\mathbb{P}(1,3,5,6,7)$\hfill $A^3=1/30$}\\
\multicolumn{6}{|l|}{
\begin{minipage}[m]{.86\linewidth}
\vspace*{1.2mm}
$w^3+yt^3+z^3(a_1t+a_2xz)+w^2f_{7}(x,y,z,t)+wf_{14}(x,y,z,t)
+f_{21}(x,y,z,t)$ \vspace*{1.2mm}
\end{minipage}
}\\
\hline \hline
\begin{minipage}[m]{.28\linewidth}
\begin{center}
Singularity
\end{center}
\end{minipage}&
\begin{minipage}[m]{.04\linewidth}
\begin{center}
$B^3$
\end{center}
\end{minipage}&
\begin{minipage}[m]{.11\linewidth}
\begin{center}
Linear

system
\end{center}
\end{minipage}&
\begin{minipage}[m]{.11\linewidth}
\begin{center}
Surface $T$
\end{center}
\end{minipage}&
\begin{minipage}[m]{.11\linewidth}
\begin{center}
\vspace*{1mm}
 \vorder
\vspace*{1mm}
\end{center}
\end{minipage}&
\begin{minipage}[m]{.18\linewidth}
\begin{center}
Condition
\end{center}
\end{minipage}\\
\hline
\begin{minipage}[m]{.28\linewidth}

$O_t=\frac{1}{6}(1_x,5_z,1_w)$ \boundary

\end{minipage}&
\begin{minipage}[m]{.04\linewidth}
\begin{center}
$0$
\end{center}
\end{minipage}&
\begin{minipage}[m]{.11\linewidth}
\begin{center}
$3B$
\end{center}
\end{minipage}&
\begin{minipage}[m]{.11\linewidth}
\begin{center}
$y$
\end{center}
\end{minipage}
&
\begin{minipage}[m]{.11\linewidth}
\begin{center}
$w^3$
\end{center}
\end{minipage}&
\begin{minipage}[m]{.18\linewidth}
\begin{center}

\end{center}
\end{minipage}\\
\hline
\begin{minipage}[m]{.28\linewidth}

$O_z=\frac{1}{5}(1_x,3_y,2_w)$ \boundary

\end{minipage}&
\begin{minipage}[m]{.04\linewidth}
\begin{center}
$0$
\end{center}
\end{minipage}&
\begin{minipage}[m]{.11\linewidth}
\begin{center}
$3B$
\end{center}
\end{minipage}&
\begin{minipage}[m]{.11\linewidth}
\begin{center}
$y$
\end{center}
\end{minipage}
&
\begin{minipage}[m]{.11\linewidth}
\begin{center}
$y$
\end{center}
\end{minipage}&
\begin{minipage}[m]{.18\linewidth}
\begin{center}
$a_1\ne 0$
\end{center}
\end{minipage}\\
\hline
\begin{minipage}[m]{.28\linewidth}
$O_z=\frac{1}{5}(3_y,1_t,2_w)$ \boundary

\end{minipage}&
\begin{minipage}[m]{.04\linewidth}
\begin{center}
$0$
\end{center}
\end{minipage}&
\begin{minipage}[m]{.11\linewidth}
\begin{center}
$3B$
\end{center}
\end{minipage}&
\begin{minipage}[m]{.11\linewidth}
\begin{center}
$y$
\end{center}
\end{minipage}
&
\begin{minipage}[m]{.11\linewidth}
\begin{center}
$y$
\end{center}
\end{minipage}&
\begin{minipage}[m]{.18\linewidth}
\begin{center}
$a_1=0$
\end{center}
\end{minipage}\\
\hline
\begin{minipage}[m]{.28\linewidth}
$O_yO_t=3\times\frac{1}{3}(1_x,2_z,1_w)$ $\nef$

\end{minipage}&
\begin{minipage}[m]{.04\linewidth}
\begin{center}
$-$
\end{center}
\end{minipage}&
\begin{minipage}[m]{.11\linewidth}
\begin{center}
$5B+E$
\end{center}
\end{minipage}&
\begin{minipage}[m]{.11\linewidth}
\begin{center}
$x^2y$, $z$
\end{center}
\end{minipage}
&
\begin{minipage}[m]{.11\linewidth}
\begin{center}
$x^2y$, $z$
\end{center}
\end{minipage}&
\begin{minipage}[m]{.18\linewidth}
\begin{center}

\end{center}
\end{minipage}\\
\hline

\end{longtable}
\end{center}

\begin{Note}

\item 
For the singular points of types $\frac{1}{6}(1,5,1)$ and
$\frac{1}{5}(1,3,2)$, the $1$-cycle $\Gamma$ is the proper
transform of the curve defined by $$x=y=w^3+a_1z^3t=0.$$ It is
irreducible even though it can be non-reduced.

\item For the singular points of type $\frac{1}{3}(1,2,1)$, consider the
linear system on $X_{21}$ generated by $x^2y$ and $z$. Its base
curves are defined by $x=z=0$ and $y=z=0$. The curve defined by
$x=z=0$ is irreducible  because of the monomials $w^3$ and $yt^3$.
For its proper transform $\tilde{C}$, we have
$T\cdot\tilde{C}=\frac{1}{6}$. Therefore, the divisor $T$ is nef
since the curve defined by $y=z=0$ does not pass through any
singular point of type $\frac{1}{3}(1,2,1)$.
\end{Note}



\begin{center}
\begin{longtable}{|l|c|c|c|c|c|}
\hline
\multicolumn{6}{|l|}{\underline{\textbf{No. 50}}: $X_{22}\subset\mathbb{P}(1,1,3,7,11)$\hfill $A^3=2/21$}\\
\multicolumn{6}{|l|}{
\begin{minipage}[m]{.86\linewidth}
\vspace*{1.2mm} $w^2+yt^3+
z^5(a_1t+a_2xz^2+a_3yz^2)+wf_{11}(x,y,z,t)+ f_{22}(x,y,z,t)$
\vspace*{1.2mm}
\end{minipage}
}\\
\hline \hline
\begin{minipage}[m]{.28\linewidth}
\begin{center}
Singularity
\end{center}
\end{minipage}&
\begin{minipage}[m]{.04\linewidth}
\begin{center}
$B^3$
\end{center}
\end{minipage}&
\begin{minipage}[m]{.11\linewidth}
\begin{center}
Linear

system
\end{center}
\end{minipage}&
\begin{minipage}[m]{.11\linewidth}
\begin{center}
Surface $T$
\end{center}
\end{minipage}&
\begin{minipage}[m]{.11\linewidth}
\begin{center}
\vspace*{1mm}
 \vorder
\vspace*{1mm}
\end{center}
\end{minipage}&
\begin{minipage}[m]{.18\linewidth}
\begin{center}
Condition
\end{center}
\end{minipage}\\
\hline
\begin{minipage}[m]{.28\linewidth}

$O_t=\frac{1}{7}(1_x,3_z,4_w)$ $\positive$

\end{minipage}&
\begin{minipage}[m]{.04\linewidth}
\begin{center}
$+$
\end{center}
\end{minipage}&
\begin{minipage}[m]{.11\linewidth}
\begin{center}
$B-E$
\end{center}
\end{minipage}&
\begin{minipage}[m]{.11\linewidth}
\begin{center}
$y$
\end{center}
\end{minipage}
&
\begin{minipage}[m]{.11\linewidth}
\begin{center}
$w^2$
\end{center}
\end{minipage}&
\begin{minipage}[m]{.18\linewidth}
\begin{center}

\end{center}
\end{minipage}\\
\hline
\begin{minipage}[m]{.28\linewidth}

$O_z=\frac{1}{3}(1_x,1_y,2_w)$ \boundary

\end{minipage}&
\begin{minipage}[m]{.04\linewidth}
\begin{center}
$-$
\end{center}
\end{minipage}&
\begin{minipage}[m]{.11\linewidth}
\begin{center}
$B$
\end{center}
\end{minipage}&
\begin{minipage}[m]{.11\linewidth}
\begin{center}
$y$
\end{center}
\end{minipage}
&
\begin{minipage}[m]{.11\linewidth}
\begin{center}
$y$
\end{center}
\end{minipage}&
\begin{minipage}[m]{.18\linewidth}
\begin{center}
$a_1\ne 0$
\end{center}
\end{minipage}\\
\hline
\begin{minipage}[m]{.28\linewidth}

$O_z=\frac{1}{3}(1_y,1_t,2_w)$ \boundary

\end{minipage}&
\begin{minipage}[m]{.04\linewidth}
\begin{center}
$-$
\end{center}
\end{minipage}&
\begin{minipage}[m]{.11\linewidth}
\begin{center}
$B$
\end{center}
\end{minipage}&
\begin{minipage}[m]{.11\linewidth}
\begin{center}
$y$
\end{center}
\end{minipage}
&
\begin{minipage}[m]{.11\linewidth}
\begin{center}
$y$
\end{center}
\end{minipage}&
\begin{minipage}[m]{.18\linewidth}
\begin{center}
$a_1=0$
\end{center}
\end{minipage}\\

\hline

\end{longtable}
\end{center}

\begin{Note}

\item 
If $a_1=0$, then $a_2\ne 0$: otherwise the hypersurface $X_{22}$
would be singular at the point defined by $x=y=w=0$ and
$t^3+a_3z^7=0$.

\item If $a_1\ne 0$, the  $1$-cycle $\Gamma$ for the singular point
$O_z$ is irreducible because of the monomials $w^2$ and $z^5t$. If
$a_1= 0$, the $1$-cycle $\Gamma$ for the singular point $O_z$ is
still irreducible even though it is not reduced.
\end{Note}

\begin{center}
\begin{longtable}{|l|c|c|c|c|c|}
\hline
\multicolumn{6}{|l|}{\underline{\textbf{No. 51}}: $X_{22}\subset\mathbb{P}(1,1,4,6,11)$\hfill $A^3=1/12$}\\
\multicolumn{6}{|l|}{
\begin{minipage}[m]{.86\linewidth}
\vspace*{1.2mm} $w^2+zt^3+z^4t+wf_{11}(x,y,z,t)+ f_{22}(x,y,z,t)$
\vspace*{1.2mm}
\end{minipage}
}\\
\hline \hline
\begin{minipage}[m]{.28\linewidth}
\begin{center}
Singularity
\end{center}
\end{minipage}&
\begin{minipage}[m]{.04\linewidth}
\begin{center}
$B^3$
\end{center}
\end{minipage}&
\begin{minipage}[m]{.11\linewidth}
\begin{center}
Linear

system
\end{center}
\end{minipage}&
\begin{minipage}[m]{.11\linewidth}
\begin{center}
Surface $T$
\end{center}
\end{minipage}&
\begin{minipage}[m]{.11\linewidth}
\begin{center}
\vspace*{1mm}
 \vorder
\vspace*{1mm}
\end{center}
\end{minipage}&
\begin{minipage}[m]{.18\linewidth}
\begin{center}
Condition
\end{center}
\end{minipage}\\
\hline
\begin{minipage}[m]{.28\linewidth}

$O_t=\frac{1}{6}(1_x,1_y,5_w)$ $\positive$

\end{minipage}&
\begin{minipage}[m]{.04\linewidth}
\begin{center}
$+$
\end{center}
\end{minipage}&
\begin{minipage}[m]{.11\linewidth}
\begin{center}
$4B-E$
\end{center}
\end{minipage}&
\begin{minipage}[m]{.11\linewidth}
\begin{center}
$z$
\end{center}
\end{minipage}
&
\begin{minipage}[m]{.11\linewidth}
\begin{center}
$w^2$
\end{center}
\end{minipage}&
\begin{minipage}[m]{.18\linewidth}
\begin{center}

\end{center}
\end{minipage}\\
\hline
\begin{minipage}[m]{.28\linewidth}

$O_z=\frac{1}{4}(1_x,1_y,3_w)$ $\nef$

\end{minipage}&
\begin{minipage}[m]{.04\linewidth}
\begin{center}
$0$
\end{center}
\end{minipage}&
\begin{minipage}[m]{.11\linewidth}
\begin{center}
$B$
\end{center}
\end{minipage}&
\begin{minipage}[m]{.11\linewidth}
\begin{center}
$x$, $y$
\end{center}
\end{minipage}
&
\begin{minipage}[m]{.11\linewidth}
\begin{center}
$x$, $y$
\end{center}
\end{minipage}&
\begin{minipage}[m]{.18\linewidth}
\begin{center}

\end{center}
\end{minipage}\\
\hline
\begin{minipage}[m]{.28\linewidth}

$O_zO_t=1\times\frac{1}{2}(1_x,1_y,1_w)$ \boundary

\end{minipage}&
\begin{minipage}[m]{.04\linewidth}
\begin{center}
$-$
\end{center}
\end{minipage}&
\begin{minipage}[m]{.11\linewidth}
\begin{center}
$B$
\end{center}
\end{minipage}&
\begin{minipage}[m]{.11\linewidth}
\begin{center}
$y$
\end{center}
\end{minipage}
&
\begin{minipage}[m]{.11\linewidth}
\begin{center}
$y$
\end{center}
\end{minipage}&
\begin{minipage}[m]{.18\linewidth}
\begin{center}

\end{center}
\end{minipage}\\
\hline

\end{longtable}
\end{center}

\begin{Note}

\item 
For the singular point $O_z$, we can easily see that the surface
$T$ is nef since the base locus of the linear system
$|-K_{X_{22}}|$ is the irreducible curve cut by $x=y=0$ and
$B^3=0$.

\item For the singular point of type $\frac{1}{2}(1,1,1)$, the
intersection $\Gamma$ is irreducible since we have the monomials
$w^2$, $z^4t$, and $zt^3$.
\end{Note}



\begin{center}
\begin{longtable}{|l|c|c|c|c|c|}
\hline
\multicolumn{6}{|l|}{\underline{\textbf{No. 52}}: $X_{22}\subset\mathbb{P}(1,2,4,5,11)$\hfill $A^3=1/20$}\\
\multicolumn{6}{|l|}{
\begin{minipage}[m]{.86\linewidth}
\vspace*{1.2mm} $w^2+yt^4+y\prod^{5}_{i=1}(z-\alpha_i
y^2)+wf_{11}(x,y,z,t)+ f_{22}(x,y,z,t)$ \vspace*{1.2mm}
\end{minipage}
}\\
\hline \hline
\begin{minipage}[m]{.28\linewidth}
\begin{center}
Singularity
\end{center}
\end{minipage}&
\begin{minipage}[m]{.04\linewidth}
\begin{center}
$B^3$
\end{center}
\end{minipage}&
\begin{minipage}[m]{.11\linewidth}
\begin{center}
Linear

system
\end{center}
\end{minipage}&
\begin{minipage}[m]{.11\linewidth}
\begin{center}
Surface $T$
\end{center}
\end{minipage}&
\begin{minipage}[m]{.11\linewidth}
\begin{center}
\vspace*{1mm}
 \vorder
\vspace*{1mm}
\end{center}
\end{minipage}&
\begin{minipage}[m]{.18\linewidth}
\begin{center}
Condition
\end{center}
\end{minipage}\\
\hline
\begin{minipage}[m]{.28\linewidth}

$O_t=\frac{1}{5}(1_x,4_z,1_w)$ \boundary

\end{minipage}&
\begin{minipage}[m]{.04\linewidth}
\begin{center}
$0$
\end{center}
\end{minipage}&
\begin{minipage}[m]{.11\linewidth}
\begin{center}
$4B$
\end{center}
\end{minipage}&
\begin{minipage}[m]{.11\linewidth}
\begin{center}
$z$
\end{center}
\end{minipage}
&
\begin{minipage}[m]{.11\linewidth}
\begin{center}
$z$
\end{center}
\end{minipage}&
\begin{minipage}[m]{.18\linewidth}
\begin{center}

\end{center}
\end{minipage}\\
\hline
\begin{minipage}[m]{.28\linewidth}

$O_z=\frac{1}{4}(1_x,1_t,3_w)$ $\nef$

\end{minipage}&
\begin{minipage}[m]{.04\linewidth}
\begin{center}
$-$
\end{center}
\end{minipage}&
\begin{minipage}[m]{.11\linewidth}
\begin{center}
$5B+E$
\end{center}
\end{minipage}&
\begin{minipage}[m]{.11\linewidth}
\begin{center}
$xz$, $t$
\end{center}
\end{minipage}
&
\begin{minipage}[m]{.11\linewidth}
\begin{center}
$xz$, $t$
\end{center}
\end{minipage}&
\begin{minipage}[m]{.18\linewidth}
\begin{center}

\end{center}
\end{minipage}\\
\hline
\begin{minipage}[m]{.28\linewidth}

$O_yO_z=5\times\frac{1}{2}(1_x,1_t,1_w)$ \boundary

\end{minipage}&
\begin{minipage}[m]{.04\linewidth}
\begin{center}
$-$
\end{center}
\end{minipage}&
\begin{minipage}[m]{.11\linewidth}
\begin{center}
$4B+E$
\end{center}
\end{minipage}&
\begin{minipage}[m]{.11\linewidth}
\begin{center}
$z-\alpha_i y^2$
\end{center}
\end{minipage}
&
\begin{minipage}[m]{.11\linewidth}
\begin{center}
$w^2$
\end{center}
\end{minipage}&
\begin{minipage}[m]{.18\linewidth}
\begin{center}

\end{center}
\end{minipage}\\
\hline

\end{longtable}
\end{center}

\begin{Note}

\item 
For the singular points of types $\frac{1}{5}(1,4,1)$ and
$\frac{1}{2}(1,1,1)$, the $1$-cycle $\Gamma$ is always irreducible
because the defining polynomial of $X_{22}$ contains the monomials
$w^2$ and $yt^4$.

\item For the singular point $O_z$, consider the linear system on
$X_{22}$ generated by $xz$ and $t$. Its base curves are defined by
$x=t=0$ and $z=t=0$. The curve defined by $x=t=0$ is irreducible
because of the monomials $w^2$ and $yz^5$.   Its proper transform
intersects the divisor $T$ positively. Since the curve defined by
$z=t=0$ does not pass through the singular point $O_z$, its proper
transform also intersects $T$ positively. Therefore, the divisor
$T$ is nef.
\end{Note}

\begin{center}
\begin{longtable}{|l|c|c|c|c|c|}
\hline
\multicolumn{6}{|l|}{\underline{\textbf{No. 53}}: $X_{24}\subset\mathbb{P}(1,1,3,8,12)$\hfill $A^3=1/12$}\\
\multicolumn{6}{|l|}{
\begin{minipage}[m]{.86\linewidth}
\vspace*{1.2mm} $w^2+t^3+z^8+wf_{12}(x,y,z,t)+f_{24}(x,y,z,t)$
\vspace*{1.2mm}
\end{minipage}
}\\
\hline \hline
\begin{minipage}[m]{.28\linewidth}
\begin{center}
Singularity
\end{center}
\end{minipage}&
\begin{minipage}[m]{.04\linewidth}
\begin{center}
$B^3$
\end{center}
\end{minipage}&
\begin{minipage}[m]{.11\linewidth}
\begin{center}
Linear

system
\end{center}
\end{minipage}&
\begin{minipage}[m]{.11\linewidth}
\begin{center}
Surface $T$
\end{center}
\end{minipage}&
\begin{minipage}[m]{.11\linewidth}
\begin{center}
\vspace*{1mm}
 \vorder
\vspace*{1mm}
\end{center}
\end{minipage}&
\begin{minipage}[m]{.18\linewidth}
\begin{center}
Condition
\end{center}
\end{minipage}\\
\hline
\begin{minipage}[m]{.28\linewidth}

$O_tO_w=1\times\frac{1}{4}(1_x,1_y,3_z)$ \boundary

\end{minipage}&
\begin{minipage}[m]{.04\linewidth}
\begin{center}
$0$
\end{center}
\end{minipage}&
\begin{minipage}[m]{.11\linewidth}
\begin{center}
$B$
\end{center}
\end{minipage}&
\begin{minipage}[m]{.11\linewidth}
\begin{center}
$y$
\end{center}
\end{minipage}
&
\begin{minipage}[m]{.11\linewidth}
\begin{center}
$y$
\end{center}
\end{minipage}&
\begin{minipage}[m]{.18\linewidth}
\begin{center}

\end{center}
\end{minipage}\\
\hline
\begin{minipage}[m]{.28\linewidth}

$O_zO_w=2\times\frac{1}{3}(1_x,1_y,2_t)$ \boundary

\end{minipage}&
\begin{minipage}[m]{.04\linewidth}
\begin{center}
$-$
\end{center}
\end{minipage}&
\begin{minipage}[m]{.11\linewidth}
\begin{center}
$B$
\end{center}
\end{minipage}&
\begin{minipage}[m]{.11\linewidth}
\begin{center}
$y$
\end{center}
\end{minipage}
&
\begin{minipage}[m]{.11\linewidth}
\begin{center}
$y$
\end{center}
\end{minipage}&
\begin{minipage}[m]{.18\linewidth}
\begin{center}

\end{center}
\end{minipage}\\
\hline

\end{longtable}
\end{center}

\begin{Note}

\item 
The $1$-cycle  $\Gamma$ for each singular point is irreducible
because of  the monomials $w^2$ and $t^3$.
\end{Note}


\begin{center}
\begin{longtable}{|l|c|c|c|c|c|}
\hline
\multicolumn{6}{|l|}{\textbf{No. 54}: $X_{24}\subset\mathbb{P}(1,1,6,8,9)$\hfill $A^3=1/18$}\\
\multicolumn{6}{|l|}{
\begin{minipage}[m]{.86\linewidth}
\vspace*{1.2mm} $zw^2+t^3+z^4+wf_{15}(x,y,z,t)+f_{24}(x,y,z,t)$

\vspace*{1.2mm}
\end{minipage}
}\\

\hline \hline
\begin{minipage}[m]{.28\linewidth}
\begin{center}
Singularity
\end{center}
\end{minipage}&
\begin{minipage}[m]{.04\linewidth}
\begin{center}
$B^3$
\end{center}
\end{minipage}&
\begin{minipage}[m]{.11\linewidth}
\begin{center}
Linear

system
\end{center}
\end{minipage}&
\begin{minipage}[m]{.11\linewidth}
\begin{center}
Surface $T$
\end{center}
\end{minipage}&
\begin{minipage}[m]{.11\linewidth}
\begin{center}
\vspace*{1mm}
 \vorder
\vspace*{1mm}
\end{center}
\end{minipage}&
\begin{minipage}[m]{.18\linewidth}
\begin{center}
Condition
\end{center}
\end{minipage}\\
\hline
\begin{minipage}[m]{.28\linewidth}

$O_w=\frac{1}{9}(1,1,8)$ \quadratic

\end{minipage}&
\multicolumn{4}{|l|}{\begin{minipage}[m]{.37\linewidth}
\begin{center}
$zw^2$
\end{center}
\end{minipage}}&
\begin{minipage}[m]{.18\linewidth}
\begin{center}

\end{center}
\end{minipage}\\
\hline
\begin{minipage}[m]{.28\linewidth}

$O_zO_w=1\times\frac{1}{3}(1_x,1_y,2_t)$ \boundary

\end{minipage}&
\begin{minipage}[m]{.04\linewidth}
\begin{center}
$-$
\end{center}
\end{minipage}&
\begin{minipage}[m]{.11\linewidth}
\begin{center}
$B$
\end{center}
\end{minipage}&
\begin{minipage}[m]{.11\linewidth}
\begin{center}
$y$
\end{center}
\end{minipage}
&
\begin{minipage}[m]{.11\linewidth}
\begin{center}
$y$
\end{center}
\end{minipage}&
\begin{minipage}[m]{.18\linewidth}
\begin{center}

\end{center}
\end{minipage}\\
\hline
\begin{minipage}[m]{.28\linewidth}
$O_zO_t=1\times\frac{1}{2}(1_x,1_y,1_w)$ \boundary

\end{minipage}&
\begin{minipage}[m]{.04\linewidth}
\begin{center}
$-$
\end{center}
\end{minipage}&
\begin{minipage}[m]{.11\linewidth}
\begin{center}
$B$
\end{center}
\end{minipage}&
\begin{minipage}[m]{.11\linewidth}
\begin{center}
$y$
\end{center}
\end{minipage}
&
\begin{minipage}[m]{.11\linewidth}
\begin{center}
$y$
\end{center}
\end{minipage}&
\begin{minipage}[m]{.18\linewidth}
\begin{center}

\end{center}
\end{minipage}\\

\hline
\end{longtable}
\end{center}

\begin{Note}

\item 
The $1$-cycles $\Gamma$ for the singular points of types
$\frac{1}{3}(1,1,2)$ and $\frac{1}{2}(1,1,1)$ are irreducible
since we have the monomials $z^4$ and $t^3$.
\end{Note}




\begin{center}
\begin{longtable}{|l|c|c|c|c|c|}
\hline
\multicolumn{6}{|l|}{\underline{\textbf{No. 55}}: $X_{24}\subset\mathbb{P}(1,2,3,7,12)$\hfill $A^3=1/21$}\\
\multicolumn{6}{|l|}{
\begin{minipage}[m]{.86\linewidth}
\vspace*{1.2mm} $(w-\alpha_1 y^6)(w-\alpha_2 y^6)
+zt^3+wf_{12}(x,y,z,t)+f_{24}(x,y,z,t)$ \vspace*{1.2mm}
\end{minipage}
}\\
\hline \hline
\begin{minipage}[m]{.28\linewidth}
\begin{center}
Singularity
\end{center}
\end{minipage}&
\begin{minipage}[m]{.04\linewidth}
\begin{center}
$B^3$
\end{center}
\end{minipage}&
\begin{minipage}[m]{.11\linewidth}
\begin{center}
Linear

system
\end{center}
\end{minipage}&
\begin{minipage}[m]{.11\linewidth}
\begin{center}
Surface $T$
\end{center}
\end{minipage}&
\begin{minipage}[m]{.11\linewidth}
\begin{center}
\vspace*{1mm}
 \vorder
\vspace*{1mm}
\end{center}
\end{minipage}&
\begin{minipage}[m]{.18\linewidth}
\begin{center}
Condition
\end{center}
\end{minipage}\\
\hline
\begin{minipage}[m]{.28\linewidth}

$O_t=\frac{1}{7}(1_x,2_y,5_w)$ $\positive$

\end{minipage}&
\begin{minipage}[m]{.04\linewidth}
\begin{center}
$+$
\end{center}
\end{minipage}&
\begin{minipage}[m]{.11\linewidth}
\begin{center}
$3B-E$
\end{center}
\end{minipage}&
\begin{minipage}[m]{.11\linewidth}
\begin{center}
$z$
\end{center}
\end{minipage}
&
\begin{minipage}[m]{.11\linewidth}
\begin{center}
$w^2$
\end{center}
\end{minipage}&
\begin{minipage}[m]{.18\linewidth}
\begin{center}

\end{center}
\end{minipage}\\

\hline
\begin{minipage}[m]{.28\linewidth}

$O_zO_w=2\times\frac{1}{3}(1_x,2_y,1_t)$ \boundary

\end{minipage}&
\begin{minipage}[m]{.04\linewidth}
\begin{center}
$-$
\end{center}
\end{minipage}&
\begin{minipage}[m]{.11\linewidth}
\begin{center}
$2B$
\end{center}
\end{minipage}&
\begin{minipage}[m]{.11\linewidth}
\begin{center}
$y$
\end{center}
\end{minipage}
&
\begin{minipage}[m]{.11\linewidth}
\begin{center}
$y$
\end{center}
\end{minipage}&
\begin{minipage}[m]{.18\linewidth}
\begin{center}

\end{center}
\end{minipage}\\
\hline
\begin{minipage}[m]{.28\linewidth}
$O_yO_w=2\times\frac{1}{2}(1_x,1_z,1_t)$ $\nef$

\end{minipage}&
\begin{minipage}[m]{.04\linewidth}
\begin{center}
$-$
\end{center}
\end{minipage}&
\begin{minipage}[m]{.11\linewidth}
\begin{center}
$7B+3E$
\end{center}
\end{minipage}&
\begin{minipage}[m]{.11\linewidth}
\begin{center}
$xy^3$, $y^2z$, $t$
\end{center}
\end{minipage}
&
\begin{minipage}[m]{.11\linewidth}
\begin{center}
$xy^3$, $y^2z$, $t$
\end{center}
\end{minipage}&
\begin{minipage}[m]{.18\linewidth}
\begin{center}

\end{center}
\end{minipage}\\
\hline

\end{longtable}
\end{center}

\begin{Note}

\item 
The $1$-cycle $\Gamma$ for each singular point of type
$\frac{1}{3}(1,2,1)$ is irreducible because of the monomials $w^2$
and $zt^3$.

\item For each singular point of type $\frac{1}{2}(1,1,1)$, the  divisor
$T$ is nef. Indeed, the base curve of the linear system on
$X_{24}$ generated by $xy^3$, $y^2z$ and $t$ is cut out by
$y=t=0$. It does not passes through any singular point of type
$\frac{1}{2}(1,1,1)$. Therefore, the surface $T$ must be nef.
\end{Note}


\begin{center}
\begin{longtable}{|l|c|c|c|c|c|}
\hline
    \multicolumn{6}{|l|}{\textbf{No. 56}: $X_{24}\subset\mathbb{P}(1,2,3,8,11)$\hfill $A^3=1/22$}\\
\multicolumn{6}{|l|}{
\begin{minipage}[m]{.86\linewidth}
\vspace*{1.2mm}

$yw^2+t^3+z^8+y^{12}+wf_{13}(x,y,z,t)+f_{24}(x,y,z,t)$
\vspace*{1.2mm}
\end{minipage}
}\\

\hline \hline
\begin{minipage}[m]{.28\linewidth}
\begin{center}
Singularity
\end{center}
\end{minipage}&
\begin{minipage}[m]{.04\linewidth}
\begin{center}
$B^3$
\end{center}
\end{minipage}&
\begin{minipage}[m]{.11\linewidth}
\begin{center}
Linear

system
\end{center}
\end{minipage}&
\begin{minipage}[m]{.11\linewidth}
\begin{center}
Surface $T$
\end{center}
\end{minipage}&
\begin{minipage}[m]{.11\linewidth}
\begin{center}
\vspace*{1mm}
 \vorder
\vspace*{1mm}
\end{center}
\end{minipage}&
\begin{minipage}[m]{.18\linewidth}
\begin{center}
Condition
\end{center}
\end{minipage}\\
\hline
\begin{minipage}[m]{.28\linewidth}

$O_w=\frac{1}{11}(1,3,8)$ \quadratic

\end{minipage}&
\multicolumn{4}{|l|}{\begin{minipage}[m]{.37\linewidth}
\begin{center}
$yw^2$
\end{center}
\end{minipage}}&
\begin{minipage}[m]{.18\linewidth}
\begin{center}
\end{center}
\end{minipage}\\

\hline
\begin{minipage}[m]{.28\linewidth}
$O_yO_t=3\times\frac{1}{2}(1_x,1_z,1_w)$ \boundary

\end{minipage}&
\begin{minipage}[m]{.04\linewidth}
\begin{center}
$-$
\end{center}
\end{minipage}&
\begin{minipage}[m]{.11\linewidth}
\begin{center}
$3B+E$
\end{center}
\end{minipage}&
\begin{minipage}[m]{.11\linewidth}
\begin{center}
$z$
\end{center}
\end{minipage}
&
\begin{minipage}[m]{.11\linewidth}
\begin{center}
$z$
\end{center}
\end{minipage}&
\begin{minipage}[m]{.18\linewidth}
\begin{center}

\end{center}
\end{minipage}\\

\hline
\end{longtable}
\end{center}

\begin{Note}

\item 
The $1$-cycle $\Gamma$ for each singular point of type
$\frac{1}{2}(1,1,1)$ is irreducible due to the monomials $yw^2$
and $t^3$.
\end{Note}


\begin{center}
\begin{longtable}{|l|c|c|c|c|c|}
\hline
\multicolumn{6}{|l|}{\underline{\textbf{No. 57}}: $X_{24}\subset\mathbb{P}(1,3,4,5,12)$\hfill $A^3=1/30$}\\
\multicolumn{6}{|l|}{
\begin{minipage}[m]{.86\linewidth}
\vspace*{1.2mm} $(w-\alpha_1 y^4)(w-\alpha_2
y^4)+zt^4+z^6+wf_{12}(x,y,z,t)+f_{24}(x,y,z,t)$ \vspace*{1.2mm}
\end{minipage}
}\\
\hline \hline
\begin{minipage}[m]{.28\linewidth}
\begin{center}
Singularity
\end{center}
\end{minipage}&
\begin{minipage}[m]{.04\linewidth}
\begin{center}
$B^3$
\end{center}
\end{minipage}&
\begin{minipage}[m]{.11\linewidth}
\begin{center}
Linear

system
\end{center}
\end{minipage}&
\begin{minipage}[m]{.11\linewidth}
\begin{center}
Surface $T$
\end{center}
\end{minipage}&
\begin{minipage}[m]{.11\linewidth}
\begin{center}
\vspace*{1mm}
 \vorder
\vspace*{1mm}
\end{center}
\end{minipage}&
\begin{minipage}[m]{.18\linewidth}
\begin{center}
Condition
\end{center}
\end{minipage}\\
\hline
\begin{minipage}[m]{.28\linewidth}

$O_t=\frac{1}{5}(1_x,3_y,2_w)$ \boundary

\end{minipage}&
\begin{minipage}[m]{.04\linewidth}
\begin{center}
$0$
\end{center}
\end{minipage}&
\begin{minipage}[m]{.11\linewidth}
\begin{center}
$3B$
\end{center}
\end{minipage}&
\begin{minipage}[m]{.11\linewidth}
\begin{center}
$y$
\end{center}
\end{minipage}
&
\begin{minipage}[m]{.11\linewidth}
\begin{center}
$y$
\end{center}
\end{minipage}&
\begin{minipage}[m]{.18\linewidth}
\begin{center}

\end{center}
\end{minipage}\\
\hline
\begin{minipage}[m]{.28\linewidth}

$O_zO_w=2\times\frac{1}{4}(1_x,3_y,1_t)$ \boundary

\end{minipage}&
\begin{minipage}[m]{.04\linewidth}
\begin{center}
$-$
\end{center}
\end{minipage}&
\begin{minipage}[m]{.11\linewidth}
\begin{center}
$3B$
\end{center}
\end{minipage}&
\begin{minipage}[m]{.11\linewidth}
\begin{center}
$y$
\end{center}
\end{minipage}
&
\begin{minipage}[m]{.11\linewidth}
\begin{center}
$y$
\end{center}
\end{minipage}&
\begin{minipage}[m]{.18\linewidth}
\begin{center}

\end{center}
\end{minipage}\\
\hline
\begin{minipage}[m]{.28\linewidth}
 $O_yO_w=2\times\frac{1}{3}(1_x,1_z,2_t)$ $\nef$

\end{minipage}&
\begin{minipage}[m]{.04\linewidth}
\begin{center}
$-$
\end{center}
\end{minipage}&
\begin{minipage}[m]{.11\linewidth}
\begin{center}
$4B+E$
\end{center}
\end{minipage}&
\begin{minipage}[m]{.11\linewidth}
\begin{center}
$z$
\end{center}
\end{minipage}
&
\begin{minipage}[m]{.11\linewidth}
\begin{center}
$z$
\end{center}
\end{minipage}&
\begin{minipage}[m]{.18\linewidth}
\begin{center}

\end{center}
\end{minipage}\\

\hline

\end{longtable}
\end{center}

\begin{Note}

\item 
The cycles $\Gamma$ for the singular points of types
$\frac{1}{5}(1,3,2)$ and $\frac{1}{4}(1,3,1)$ are irreducible
because of the monomials $w^2$ and $zt^4$.

\item For each singular point of type $\frac{1}{3}(1,1,2)$ we consider
the linear system generated by $x^{20}$, $z^5$ and $t^4$ on the
hypersurface $X_{24}$. Its base locus is cut out by $x=z=t=0$.
Since the defining equation of $X_{24}$ contains the monomial
$wy^4$, its base locus does not contain any curves. Therefore the
proper transform of a general member in the linear system is nef
by Lemma~\ref{lemma:nefness}. The proper transform  belongs to
$|20B+5E|$. Consequently, the surface $T$ is nef since
$5T\sim_{\mathbb{Q}} 20B+5E$.
\end{Note}


\begin{center}
\begin{longtable}{|l|c|c|c|c|c|}
\hline
\multicolumn{6}{|l|}{\textbf{No. 58}: $X_{24}\subset\mathbb{P}(1,3,4,7,10)$\hfill $A^3=1/35$}\\
\multicolumn{6}{|l|}{
\begin{minipage}[m]{.86\linewidth}
\vspace*{1.2mm}

$zw^2+t^2(a_1w+a_2yt)+z^6+y^8+wf_{14}(x,y,z,t)+f_{24}(x,y,z,t)$
\vspace*{1.2mm}
\end{minipage}
}\\

\hline \hline
\begin{minipage}[m]{.28\linewidth}
\begin{center}
Singularity
\end{center}
\end{minipage}&
\begin{minipage}[m]{.04\linewidth}
\begin{center}
$B^3$
\end{center}
\end{minipage}&
\begin{minipage}[m]{.11\linewidth}
\begin{center}
Linear

system
\end{center}
\end{minipage}&
\begin{minipage}[m]{.11\linewidth}
\begin{center}
Surface $T$
\end{center}
\end{minipage}&
\begin{minipage}[m]{.11\linewidth}
\begin{center}
\vspace*{1mm}
 \vorder
\vspace*{1mm}
\end{center}
\end{minipage}&
\begin{minipage}[m]{.18\linewidth}
\begin{center}
Condition
\end{center}
\end{minipage}\\
\hline
\begin{minipage}[m]{.28\linewidth}

$O_w=\frac{1}{10}(1,3,7)$ \quadratic

\end{minipage}&
\multicolumn{4}{|l|}{\begin{minipage}[m]{.37\linewidth}
\begin{center}
$zw^2$
\end{center}
\end{minipage}}&
\begin{minipage}[m]{.18\linewidth}
\begin{center}
\end{center}
\end{minipage}\\
\hline
\begin{minipage}[m]{.28\linewidth}

$O_t=\frac{1}{7}(1,3,4)$ \quadraticone

\end{minipage}&
\multicolumn{4}{|l|}{\begin{minipage}[m]{.37\linewidth}
\begin{center}
$t^2w$
\end{center}
\end{minipage}}&
\begin{minipage}[m]{.18\linewidth}
\begin{center}

\end{center}
\end{minipage}\\
\hline
\begin{minipage}[m]{.28\linewidth}
 $O_zO_w=1\times\frac{1}{2}(1_x,1_y,1_t)$ $\nef$

\end{minipage}&
\begin{minipage}[m]{.04\linewidth}
\begin{center}
$-$
\end{center}
\end{minipage}&
\begin{minipage}[m]{.11\linewidth}
\begin{center}
$7B+3E$
\end{center}
\end{minipage}&
\begin{minipage}[m]{.11\linewidth}
\begin{center}
$t$
\end{center}
\end{minipage}
&
\begin{minipage}[m]{.11\linewidth}
\begin{center}
$t$
\end{center}
\end{minipage}&
\begin{minipage}[m]{.18\linewidth}
\begin{center}

\end{center}
\end{minipage}\\

\hline
\end{longtable}
\end{center}

\begin{Note}

\item 
For the singular point of type $\frac{1}{2}(1,1,1)$,
 we consider the
linear system $|-7K_{X_{24}}|$. Every member of the linear system
passes through  the singular point of type $\frac{1}{2}(1,1,1)$
and its base locus contains  no curves. Since the proper transform
of a general member in $|-7K_{X_{24}}|$ belongs to the linear
system $|7B+3E|$, the divisor $T$ is nef.
\end{Note}




\begin{center}
\begin{longtable}{|l|c|c|c|c|c|}
\hline
\multicolumn{6}{|l|}{\underline{\textbf{No. 59}}: $X_{24}\subset\mathbb{P}(1,3,6,7,8)$\hfill $A^3=1/42$}\\
\multicolumn{6}{|l|}{
\begin{minipage}[m]{.86\linewidth}
\vspace*{1.2mm} $w^3+yt^3+\prod^4_{i=1}(z-\alpha_i
y^2)+w^2f_8(x,y,z,t)+wf_{16}(x,y,z,t)+f_{24}(x,y,z,t)$
\vspace*{1.2mm}
\end{minipage}
}\\
\hline \hline
\begin{minipage}[m]{.28\linewidth}
\begin{center}
Singularity
\end{center}
\end{minipage}&
\begin{minipage}[m]{.04\linewidth}
\begin{center}
$B^3$
\end{center}
\end{minipage}&
\begin{minipage}[m]{.11\linewidth}
\begin{center}
Linear

system
\end{center}
\end{minipage}&
\begin{minipage}[m]{.11\linewidth}
\begin{center}
Surface $T$
\end{center}
\end{minipage}&
\begin{minipage}[m]{.11\linewidth}
\begin{center}
\vspace*{1mm}
 \vorder
\vspace*{1mm}
\end{center}
\end{minipage}&
\begin{minipage}[m]{.18\linewidth}
\begin{center}
Condition
\end{center}
\end{minipage}\\
\hline
\begin{minipage}[m]{.28\linewidth}

$O_t=\frac{1}{7}(1_x,6_z,1_w)$ \boundary

\end{minipage}&
\begin{minipage}[m]{.04\linewidth}
\begin{center}
$0$
\end{center}
\end{minipage}&
\begin{minipage}[m]{.11\linewidth}
\begin{center}
$3B$
\end{center}
\end{minipage}&
\begin{minipage}[m]{.11\linewidth}
\begin{center}
$y$
\end{center}
\end{minipage}
&
\begin{minipage}[m]{.11\linewidth}
\begin{center}
$w^3$
\end{center}
\end{minipage}&
\begin{minipage}[m]{.18\linewidth}
\begin{center}

\end{center}
\end{minipage}\\
\hline
\begin{minipage}[m]{.28\linewidth}

$O_zO_w=1\times\frac{1}{2}(1_x,1_y,1_t)$ \boundary

\end{minipage}&
\begin{minipage}[m]{.04\linewidth}
\begin{center}
$-$
\end{center}
\end{minipage}&
\begin{minipage}[m]{.11\linewidth}
\begin{center}
$3B+E$
\end{center}
\end{minipage}&
\begin{minipage}[m]{.11\linewidth}
\begin{center}
$y$
\end{center}
\end{minipage}
&
\begin{minipage}[m]{.11\linewidth}
\begin{center}
$y$
\end{center}
\end{minipage}&
\begin{minipage}[m]{.18\linewidth}
\begin{center}

\end{center}
\end{minipage}\\
\hline
\begin{minipage}[m]{.28\linewidth}

$O_yO_z=4\times\frac{1}{3}(1_x,1_t,2_w)$ \boundary

\end{minipage}&
\begin{minipage}[m]{.04\linewidth}
\begin{center}
$-$
\end{center}
\end{minipage}&
\begin{minipage}[m]{.11\linewidth}
\begin{center}
$6B+E$
\end{center}
\end{minipage}&
\begin{minipage}[m]{.11\linewidth}
\begin{center}
$z-\alpha_i y^2$
\end{center}
\end{minipage}
&
\begin{minipage}[m]{.11\linewidth}
\begin{center}
$yt^3$
\end{center}
\end{minipage}&
\begin{minipage}[m]{.18\linewidth}
\begin{center}

\end{center}
\end{minipage}\\
\hline

\end{longtable}
\end{center}

\begin{Note}

\item 
The $1$-cycle $\Gamma$ for each singular point is irreducible due
to the monomials $w^3$, $z^4$ and $yt^3$.

\end{Note}



\begin{center}
\begin{longtable}{|l|c|c|c|c|c|}
\hline
\multicolumn{6}{|l|}{\textbf{No. 60}: $X_{24}\subset\mathbb{P}(1,4,5,6,9)$\hfill $A^3=1/45$}\\
\multicolumn{6}{|l|}{
\begin{minipage}[m]{.86\linewidth}
\vspace*{1.2mm}

$tw^2+(t^2-\alpha_1 y^3)(t^2-\alpha_2
y^3)+z^3(a_1w+a_2yz)+wf_{15}(x,y,z,t)+f_{24}(x,y,z,t)$
\vspace*{1.2mm}
\end{minipage}
}\\

\hline \hline
\begin{minipage}[m]{.28\linewidth}
\begin{center}
Singularity
\end{center}
\end{minipage}&
\begin{minipage}[m]{.04\linewidth}
\begin{center}
$B^3$
\end{center}
\end{minipage}&
\begin{minipage}[m]{.11\linewidth}
\begin{center}
Linear

system
\end{center}
\end{minipage}&
\begin{minipage}[m]{.11\linewidth}
\begin{center}
Surface $T$
\end{center}
\end{minipage}&
\begin{minipage}[m]{.11\linewidth}
\begin{center}
\vspace*{1mm}
 \vorder
\vspace*{1mm}
\end{center}
\end{minipage}&
\begin{minipage}[m]{.18\linewidth}
\begin{center}
Condition
\end{center}
\end{minipage}\\
\hline
\begin{minipage}[m]{.28\linewidth}

$O_w=\frac{1}{9}(1,4,5)$ \quadratic

\end{minipage}&
\multicolumn{4}{|l|}{\begin{minipage}[m]{.37\linewidth}
\begin{center}
$tw^2$
\end{center}
\end{minipage}}&
\begin{minipage}[m]{.18\linewidth}
\begin{center}

\end{center}
\end{minipage}\\
\hline
\begin{minipage}[m]{.28\linewidth}

$O_z=\frac{1}{5}(1_x,4_y,1_t)$ \boundary

\end{minipage}&
\begin{minipage}[m]{.04\linewidth}
\begin{center}
$-$
\end{center}
\end{minipage}&
\begin{minipage}[m]{.11\linewidth}
\begin{center}
$4B$
\end{center}
\end{minipage}&
\begin{minipage}[m]{.11\linewidth}
\begin{center}
$y$
\end{center}
\end{minipage}
&
\begin{minipage}[m]{.11\linewidth}
\begin{center}
$y$
\end{center}
\end{minipage}&
\begin{minipage}[m]{.18\linewidth}
\begin{center}
$a_1\ne 0$
\end{center}
\end{minipage}\\
\hline
\begin{minipage}[m]{.28\linewidth}

$O_z=\frac{1}{5}(1_x,1_t,4_w)$ \boundary

\end{minipage}&
\begin{minipage}[m]{.04\linewidth}
\begin{center}
$-$
\end{center}
\end{minipage}&
\begin{minipage}[m]{.11\linewidth}
\begin{center}
$4B$
\end{center}
\end{minipage}&
\begin{minipage}[m]{.11\linewidth}
\begin{center}
$y$
\end{center}
\end{minipage}
&
\begin{minipage}[m]{.11\linewidth}
\begin{center}
$t^4$
\end{center}
\end{minipage}&
\begin{minipage}[m]{.18\linewidth}
\begin{center}
$a_1=0$
\end{center}
\end{minipage}\\
\hline
\begin{minipage}[m]{.28\linewidth}
$O_tO_w=1\times\frac{1}{3}(1_x,1_y,2_z)$ \boundary

\end{minipage}&
\begin{minipage}[m]{.04\linewidth}
\begin{center}
$-$
\end{center}
\end{minipage}&
\begin{minipage}[m]{.11\linewidth}
\begin{center}
$5B+E$
\end{center}
\end{minipage}&
\begin{minipage}[m]{.11\linewidth}
\begin{center}
$z$
\end{center}
\end{minipage}
&
\begin{minipage}[m]{.11\linewidth}
\begin{center}
$z$
\end{center}
\end{minipage}&
\begin{minipage}[m]{.18\linewidth}
\begin{center}

\end{center}
\end{minipage}\\
\hline
\begin{minipage}[m]{.28\linewidth}
 $O_yO_t=2\times\frac{1}{2}(1_x,1_z,1_w)$ $\nef$

\end{minipage}&
\begin{minipage}[m]{.04\linewidth}
\begin{center}
$-$
\end{center}
\end{minipage}&
\begin{minipage}[m]{.11\linewidth}
\begin{center}
$5B+2E$
\end{center}
\end{minipage}&
\begin{minipage}[m]{.11\linewidth}
\begin{center}
$xy$, $z$
\end{center}
\end{minipage}
&
\begin{minipage}[m]{.11\linewidth}
\begin{center}
$xy$, $z$
\end{center}
\end{minipage}&
\begin{minipage}[m]{.18\linewidth}
\begin{center}

\end{center}
\end{minipage}\\
\hline
\end{longtable}
\end{center}

\begin{Note}

\item 
The $1$-cycle $\Gamma$ for the singular point $O_z$ with $a_1\ne
0$ is irreducible due to the monomials $t^4$ and $z^3w$.

\item The $1$-cycle $\Gamma$ for the singular point $O_z$ with $a_1= 0$
has two irreducible components. One is $\tilde{L}_{zw}$ and the
other is the proper transform $\tilde{C}$ of the curve defined by
$$x=y=w^2+t^3=0.$$
 Then
we see that
\[E\cdot \tilde{C}=3E\cdot \tilde{L}_{zw}, \ \ \ B\cdot \tilde{C}=3B\cdot \tilde{L}_{zw}.\]
Therefore these two components are numerically proportional on
$Y$.

\item The $1$-cycle $\Gamma$ for the singular point of type
$\frac{1}{3}(1,1,2)$ is irreducible since we have terms $tw^2$ and
$(t^2-\alpha_1 y^3)(t^2-\alpha_2 y^3)$. Note that the constants
$\alpha_i$'s cannot be zero.

\item For the singular points of type $\frac{1}{2}(1,1,1)$,
 we consider the
linear system generated by $xy$ and $z$ on $X_{24}$. Its base
curves are defined by $x=z=0$ and $y=z=0$.  The curve defined by
$y=z=0$ passes though no singular point of type $\frac{1}{2}(1,1,
1)$.  The curve   defined by $x=z=0$   is irreducible.  Moreover,
its proper transform is the $1$-cycle defined by  $(5B+2E)\cdot
B$. Consequently, the divisor $T$ is nef since  $(5B+2E)^2\cdot
B>0$.
\end{Note}


\begin{center}
\begin{longtable}{|l|c|c|c|c|c|}
\hline
\multicolumn{6}{|l|}{\textbf{No. 61}: $X_{25}\subset\mathbb{P}(1,4,5,7,9)$\hfill $A^3=5/252$}\\
\multicolumn{6}{|l|}{
\begin{minipage}[m]{.86\linewidth}
\vspace*{1.2mm}
$tw^2-yt^3+z^5+y^4(a_1w+a_2yz+a_3xy^2)+wf_{16}(x,y,z,t)+f_{25}(x,y,z,t)$
\vspace*{1.2mm}
\end{minipage}
}\\
\hline \hline
\begin{minipage}[m]{.28\linewidth}
\begin{center}
Singularity
\end{center}
\end{minipage}&
\begin{minipage}[m]{.04\linewidth}
\begin{center}
$B^3$
\end{center}
\end{minipage}&
\begin{minipage}[m]{.11\linewidth}
\begin{center}
Linear

system
\end{center}
\end{minipage}&
\begin{minipage}[m]{.11\linewidth}
\begin{center}
Surface $T$
\end{center}
\end{minipage}&
\begin{minipage}[m]{.11\linewidth}
\begin{center}
\vspace*{1mm}
 \vorder
\vspace*{1mm}
\end{center}
\end{minipage}&
\begin{minipage}[m]{.18\linewidth}
\begin{center}
Condition
\end{center}
\end{minipage}\\
\hline
\begin{minipage}[m]{.28\linewidth}

$O_w=\frac{1}{9}(1,4,5)$ \quadratic

\end{minipage}&
\multicolumn{4}{|l|}{\begin{minipage}[m]{.37\linewidth}
\begin{center}
$tw^2$
\end{center}
\end{minipage}}&
\begin{minipage}[m]{.18\linewidth}
\begin{center}

\end{center}
\end{minipage}\\

\hline
\begin{minipage}[m]{.28\linewidth}

$O_t=\frac{1}{7}(1,5,2)$ \elliptic

\end{minipage}&\multicolumn{4}{|l|}{\begin{minipage}[m]{.37\linewidth}
\begin{center}
$tw^2-yt^3$
\end{center}
\end{minipage}}&
\begin{minipage}[m]{.18\linewidth}
\begin{center}

\end{center}
\end{minipage}\\

\hline
\begin{minipage}[m]{.28\linewidth}

$O_y=\frac{1}{4}(1_x,1_z,3_t)$ \boundary

\end{minipage}&
\begin{minipage}[m]{.04\linewidth}
\begin{center}
$-$
\end{center}
\end{minipage}&
\begin{minipage}[m]{.11\linewidth}
\begin{center}
$9B+E$
\end{center}
\end{minipage}&
\begin{minipage}[m]{.11\linewidth}
\begin{center}
$w$
\end{center}
\end{minipage}
&
\begin{minipage}[m]{.11\linewidth}
\begin{center}
$z^5$
\end{center}
\end{minipage}&
\begin{minipage}[m]{.18\linewidth}
\begin{center}
$a_1\ne 0$
\end{center}
\end{minipage}\\
\hline
\begin{minipage}[m]{.28\linewidth}

$O_y=\frac{1}{4}(1_x,3_t,1_w)$ \surface

\end{minipage}&
\begin{minipage}[m]{.04\linewidth}
\begin{center}
$-$
\end{center}
\end{minipage}&
\begin{minipage}[m]{.11\linewidth}
\begin{center}
$5B$
\end{center}
\end{minipage}&
\begin{minipage}[m]{.11\linewidth}
\begin{center}
$x^5$, $z$
\end{center}
\end{minipage}
&
\begin{minipage}[m]{.11\linewidth}
\begin{center}
$x^5, tw^2$
\end{center}
\end{minipage}&
\begin{minipage}[m]{.18\linewidth}
\begin{center}
$a_1=0$, $a_2\ne 0$
\end{center}
\end{minipage}\\
\hline
\begin{minipage}[m]{.28\linewidth}

$O_y=\frac{1}{4}(1_z,3_t,1_w)$ \boundary

\end{minipage}&
\begin{minipage}[m]{.04\linewidth}
\begin{center}
$-$
\end{center}
\end{minipage}&
\begin{minipage}[m]{.11\linewidth}
\begin{center}
$7B+E$
\end{center}
\end{minipage}&
\begin{minipage}[m]{.11\linewidth}
\begin{center}
$t$
\end{center}
\end{minipage}
&
\begin{minipage}[m]{.11\linewidth}
\begin{center}
$t$
\end{center}
\end{minipage}&
\begin{minipage}[m]{.18\linewidth}
\begin{center}
$a_1=a_2=0$
\end{center}
\end{minipage}\\

\hline
\end{longtable}
\end{center}

\begin{Note}

\item 
If $a_1\ne 0$, the $1$-cycle $\Gamma$ for the singular point $O_y$
is irreducible due to the monomials $yt^3$ and $z^5$.

\item If $a_1=a_2=0$, the $1$-cycle $\Gamma$ for the singular point
$O_y$ is irreducible even though it is not reduced.

\item Now we suppose that $a_1=0$ and $a_2\ne 0$. Then we may assume
that $a_2=1$ and $a_3=0$. We take a surface $H$ cut by an equation
$z=\lambda x^5$ with  a general complex number $\lambda$  and then
let $T$ be the proper transform of the surface. The surface $H$ is
normal but it is not quasi-smooth at the point $O_y$.

 The intersection of $T$ with the surface $S$ gives us a divisor consisting of two irreducible curves on the normal surface $T$. One is the proper transform of the curve $L_{yw}$.
 The other is the proper transform  of the curve $C$ defined by $$x=z=w^2-yt^2=0.$$
From the intersection numbers
\[(\tilde{L}_{yw}+\tilde{C})\cdot\tilde{L}_{yw}=-K_Y\cdot \tilde{L}_{yw}=-\frac{2}{9}, \ \ \ (\tilde{L}_{yw}+\tilde{C})^2=5B^3=-\frac{20}{63}\]
on the surface $T$, we obtain
\[\tilde{L}_{yw}^2=-\frac{2}{9}-\tilde{L}_{yw}\cdot\tilde{C}, \ \ \ \tilde{C}^2=-\frac{2}{21}-\tilde{L}_{yw}\cdot\tilde{C}.\]
With these intersection numbers we see that the matrix
\[\left(\begin{array}{cc}
       \tilde{L}_{yw}^2&\tilde{L}_{yw}\cdot\tilde{C}\\
       \tilde{L}_{yw}\cdot\tilde{C}& \tilde{C}^2\\
\end{array}\right)= \left(\begin{array}{cc}
-\frac{2}{9}-\tilde{L}_{yw}\cdot\tilde{C} & \tilde{L}_{yw}\cdot\tilde{C}\\
     \tilde{L}_{yw}\cdot\tilde{C} & -\frac{2}{21}-\tilde{L}_{yw}\cdot\tilde{C} \\
        \end{array}\right)
\]
is negative-definite since $\tilde{L}_{zw}\cdot\tilde{C}$ is
positive.
\end{Note}




\begin{center}
\begin{longtable}{|l|c|c|c|c|c|}
\hline
\multicolumn{6}{|l|}{\underline{\textbf{No. 62}}: $X_{26}\subset\mathbb{P}(1,1,5,7,13)$\hfill $A^3=2/35$}\\
\multicolumn{6}{|l|}{
\begin{minipage}[m]{.86\linewidth}
\vspace*{1.2mm} $w^2+zt^3+yz^5+wf_{13}(x,y,z,t)+f_{26}(x,y,z,t)$
\vspace*{1.2mm}
\end{minipage}
}\\
\hline \hline
\begin{minipage}[m]{.28\linewidth}
\begin{center}
Singularity
\end{center}
\end{minipage}&
\begin{minipage}[m]{.04\linewidth}
\begin{center}
$B^3$
\end{center}
\end{minipage}&
\begin{minipage}[m]{.11\linewidth}
\begin{center}
Linear

system
\end{center}
\end{minipage}&
\begin{minipage}[m]{.11\linewidth}
\begin{center}
Surface $T$
\end{center}
\end{minipage}&
\begin{minipage}[m]{.11\linewidth}
\begin{center}
\vspace*{1mm}
 \vorder
\vspace*{1mm}
\end{center}
\end{minipage}&
\begin{minipage}[m]{.18\linewidth}
\begin{center}
Condition
\end{center}
\end{minipage}\\
\hline
\begin{minipage}[m]{.28\linewidth}

$O_t=\frac{1}{7}(1_x,1_y,6_w)$ $\positive$

\end{minipage}&
\begin{minipage}[m]{.04\linewidth}
\begin{center}
$+$
\end{center}
\end{minipage}&
\begin{minipage}[m]{.11\linewidth}
\begin{center}
$5B-E$
\end{center}
\end{minipage}&
\begin{minipage}[m]{.11\linewidth}
\begin{center}
$z$
\end{center}
\end{minipage}
&
\begin{minipage}[m]{.11\linewidth}
\begin{center}
$w^2$
\end{center}
\end{minipage}&
\begin{minipage}[m]{.18\linewidth}
\begin{center}

\end{center}
\end{minipage}\\
\hline
\begin{minipage}[m]{.28\linewidth}

$O_z=\frac{1}{5}(1_x,2_t,3_w)$ $\family$

\end{minipage}&
\begin{minipage}[m]{.04\linewidth}
\begin{center}
$+$
\end{center}
\end{minipage}&
\begin{minipage}[m]{.11\linewidth}
\begin{center}
$B-E$
\end{center}
\end{minipage}&
\begin{minipage}[m]{.11\linewidth}
\begin{center}
$y$
\end{center}
\end{minipage}
&
\begin{minipage}[m]{.11\linewidth}
\begin{center}
$w^2$
\end{center}
\end{minipage}&
\begin{minipage}[m]{.18\linewidth}
\begin{center}

\end{center}
\end{minipage}\\
\hline

\end{longtable}
\end{center}

\begin{Note}

\item 
For the singular point $O_z$, let $C_{\lambda}$ be the curve on
the surface $S_y$ defined by
$$
\left\{%
\aligned
&y=0,\\%
&t=\lambda x^{7}\\%
\endaligned\right.%
$$
for a sufficiently general complex number  $\lambda$.
 Then
$$
-K_{Y}\cdot \tilde{C}_{\lambda}=(B-E)(7B+E)B=0.\\%
$$
If the curve $\tilde{C}_{\lambda}$ is reducible, it consists of
two irreducible components that are numerically equivalent to each
other since the two components of the curve $C_\lambda$ are
symmetric with respect to the biregular quadratic involution of
$X_{26}$. Then each component of $\tilde{C}_{\lambda}$ intersects
$-K_Y$ trivially.
\end{Note}


\begin{center}
\begin{longtable}{|l|c|c|c|c|c|}
\hline
\multicolumn{6}{|l|}{\underline{\textbf{No. 63}}: $X_{26}\subset\mathbb{P}(1,2,3,8,13)$\hfill $A^3=1/24$}\\
\multicolumn{6}{|l|}{
\begin{minipage}[m]{.86\linewidth}
\vspace*{1.2mm} $w^2
+y(t-\alpha_1y^4)(t-\alpha_2y^4)(t-\alpha_3y^4)+z^6(a_1t+a_2yz^2)+
wf_{13}(x,y,z,t)+f_{26}(x,y,z,t)$ \vspace*{1.2mm}
\end{minipage}
}\\
\hline \hline
\begin{minipage}[m]{.28\linewidth}
\begin{center}
Singularity
\end{center}
\end{minipage}&
\begin{minipage}[m]{.04\linewidth}
\begin{center}
$B^3$
\end{center}
\end{minipage}&
\begin{minipage}[m]{.11\linewidth}
\begin{center}
Linear

system
\end{center}
\end{minipage}&
\begin{minipage}[m]{.11\linewidth}
\begin{center}
Surface $T$
\end{center}
\end{minipage}&
\begin{minipage}[m]{.11\linewidth}
\begin{center}
\vspace*{1mm}
 \vorder
\vspace*{1mm}
\end{center}
\end{minipage}&
\begin{minipage}[m]{.18\linewidth}
\begin{center}
Condition
\end{center}
\end{minipage}\\
\hline
\begin{minipage}[m]{.28\linewidth}

$O_t=\frac{1}{8}(1_x,3_z,5_w)$ $\positive$

\end{minipage}&
\begin{minipage}[m]{.04\linewidth}
\begin{center}
$+$
\end{center}
\end{minipage}&
\begin{minipage}[m]{.11\linewidth}
\begin{center}
$2B-E$
\end{center}
\end{minipage}&
\begin{minipage}[m]{.11\linewidth}
\begin{center}
$y$
\end{center}
\end{minipage}
&
\begin{minipage}[m]{.11\linewidth}
\begin{center}
$w^2$
\end{center}
\end{minipage}&
\begin{minipage}[m]{.18\linewidth}
\begin{center}

\end{center}
\end{minipage}\\
\hline

\begin{minipage}[m]{.28\linewidth}

$O_z=\frac{1}{3}(1_x,2_y,1_w)$ \boundary

\end{minipage}&
\begin{minipage}[m]{.04\linewidth}
\begin{center}
$-$
\end{center}
\end{minipage}&
\begin{minipage}[m]{.11\linewidth}
\begin{center}
$2B$
\end{center}
\end{minipage}&
\begin{minipage}[m]{.11\linewidth}
\begin{center}
$y$
\end{center}
\end{minipage}
&
\begin{minipage}[m]{.11\linewidth}
\begin{center}
$y$
\end{center}
\end{minipage}&
\begin{minipage}[m]{.18\linewidth}
\begin{center}
$a_1\ne 0$
\end{center}
\end{minipage}\\
\hline

\begin{minipage}[m]{.28\linewidth}

$O_z=\frac{1}{3}(1_x,2_t,1_w)$ \boundary

\end{minipage}&
\begin{minipage}[m]{.04\linewidth}
\begin{center}
$-$
\end{center}
\end{minipage}&
\begin{minipage}[m]{.11\linewidth}
\begin{center}
$2B$
\end{center}
\end{minipage}&
\begin{minipage}[m]{.11\linewidth}
\begin{center}
$y$
\end{center}
\end{minipage}
&
\begin{minipage}[m]{.11\linewidth}
\begin{center}
$w^2$
\end{center}
\end{minipage}&
\begin{minipage}[m]{.18\linewidth}
\begin{center}
$a_1=0$
\end{center}
\end{minipage}\\
\hline
\begin{minipage}[m]{.28\linewidth}

$O_yO_t=3\times\frac{1}{2}(1_x,1_z,1_w)$ \boundary

\end{minipage}&
\begin{minipage}[m]{.04\linewidth}
\begin{center}
$-$
\end{center}
\end{minipage}&
\begin{minipage}[m]{.11\linewidth}
\begin{center}
$3B+E$
\end{center}
\end{minipage}&
\begin{minipage}[m]{.11\linewidth}
\begin{center}
$z$
\end{center}
\end{minipage}
&
\begin{minipage}[m]{.11\linewidth}
\begin{center}
$z$
\end{center}
\end{minipage}&
\begin{minipage}[m]{.18\linewidth}
\begin{center}

\end{center}
\end{minipage}\\
\hline

\end{longtable}
\end{center}

\begin{Note}

\item 
For each of the singular points of types $\frac{1}{3}(1,2,1)$ and
$\frac{1}{2}(1,1,1)$, the $1$-cycle $\Gamma$ is always irreducible
because of the monomials $w^2$, $yt^3$ and $z^6t$ even though it
is possibly non-reduced.
\end{Note}



\begin{center}
\begin{longtable}{|l|c|c|c|c|c|}
\hline
\multicolumn{6}{|l|}{\underline{\textbf{No. 64}}: $X_{26}\subset\mathbb{P}(1,2,5,6,13)$\hfill $A^3=1/30$}\\
\multicolumn{6}{|l|}{
\begin{minipage}[m]{.86\linewidth}
\vspace*{1.2mm} $w^{2}+y\prod^4_{i=1}(t-\alpha_i
y^3)+z^{4}(a_1t+a_2x)+ wf_{13}(x,y,z,t)+f_{26}(x,y,z,t)$
\vspace*{1.2mm}
\end{minipage}
}\\
\hline \hline
\begin{minipage}[m]{.28\linewidth}
\begin{center}
Singularity
\end{center}
\end{minipage}&
\begin{minipage}[m]{.04\linewidth}
\begin{center}
$B^3$
\end{center}
\end{minipage}&
\begin{minipage}[m]{.11\linewidth}
\begin{center}
Linear

system
\end{center}
\end{minipage}&
\begin{minipage}[m]{.11\linewidth}
\begin{center}
Surface $T$
\end{center}
\end{minipage}&
\begin{minipage}[m]{.11\linewidth}
\begin{center}
\vspace*{1mm}
 \vorder
\vspace*{1mm}
\end{center}
\end{minipage}&
\begin{minipage}[m]{.18\linewidth}
\begin{center}
Condition
\end{center}
\end{minipage}\\
\hline
\begin{minipage}[m]{.28\linewidth}

$O_t=\frac{1}{6}(1_x,5_z,1_w)$ \boundary

\end{minipage}&
\begin{minipage}[m]{.04\linewidth}
\begin{center}
$0$
\end{center}
\end{minipage}&
\begin{minipage}[m]{.11\linewidth}
\begin{center}
$2B$
\end{center}
\end{minipage}&
\begin{minipage}[m]{.11\linewidth}
\begin{center}
$y$
\end{center}
\end{minipage}
&
\begin{minipage}[m]{.11\linewidth}
\begin{center}
$w^2$
\end{center}
\end{minipage}&
\begin{minipage}[m]{.18\linewidth}
\begin{center}

\end{center}
\end{minipage}\\
\hline
\begin{minipage}[m]{.28\linewidth}

$O_z=\frac{1}{5}(1_x,2_y,3_w)$ \boundary

\end{minipage}&
\begin{minipage}[m]{.04\linewidth}
\begin{center}
$0$
\end{center}
\end{minipage}&
\begin{minipage}[m]{.11\linewidth}
\begin{center}
$2B$
\end{center}
\end{minipage}&
\begin{minipage}[m]{.11\linewidth}
\begin{center}
$y$
\end{center}
\end{minipage}
&
\begin{minipage}[m]{.11\linewidth}
\begin{center}
$y$
\end{center}
\end{minipage}&
\begin{minipage}[m]{.18\linewidth}
\begin{center}
$a_1\ne 0$
\end{center}
\end{minipage}\\
\hline
\begin{minipage}[m]{.28\linewidth}

$O_z=\frac{1}{5}(2_y,1_t, 3_w)$ \boundary

\end{minipage}&
\begin{minipage}[m]{.04\linewidth}
\begin{center}
$0$
\end{center}
\end{minipage}&
\begin{minipage}[m]{.11\linewidth}
\begin{center}
$2B$
\end{center}
\end{minipage}&
\begin{minipage}[m]{.11\linewidth}
\begin{center}
$y$
\end{center}
\end{minipage}
&
\begin{minipage}[m]{.11\linewidth}
\begin{center}
$y$
\end{center}
\end{minipage}&
\begin{minipage}[m]{.18\linewidth}
\begin{center}
$a_1=0$
\end{center}
\end{minipage}\\
\hline
\begin{minipage}[m]{.28\linewidth}

$O_yO_t=4\times\frac{1}{2}(1_x,1_z,1_w)$ \boundary

\end{minipage}&
\begin{minipage}[m]{.04\linewidth}
\begin{center}
$-$
\end{center}
\end{minipage}&
\begin{minipage}[m]{.11\linewidth}
\begin{center}
$6B+2E$
\end{center}
\end{minipage}&
\begin{minipage}[m]{.11\linewidth}
\begin{center}
$t-\alpha_i y^3$
\end{center}
\end{minipage}
&
\begin{minipage}[m]{.11\linewidth}
\begin{center}
$w^2$
\end{center}
\end{minipage}&
\begin{minipage}[m]{.18\linewidth}
\begin{center}

\end{center}
\end{minipage}\\
\hline

\end{longtable}
\end{center}

\begin{Note}

\item 
For each of the singular points the $1$-cycle $\Gamma$ is
irreducible due to the monomials $w^2$ and $z^{4}t$. In
particular, if $a_1=0$, then the $1$-cycle $\Gamma$ is $2L_{zt}$.
\end{Note}



\begin{center}
\begin{longtable}{|l|c|c|c|c|c|}
\hline
\multicolumn{6}{|l|}{\textbf{No. 65}: $X_{27}\subset\mathbb{P}(1,2,5,9,11)$\hfill $A^3=3/110$}\\
\multicolumn{6}{|l|}{
\begin{minipage}[m]{.86\linewidth}
\vspace*{1.2mm}

$zw^2+t^3+yz^5+y^8(a_1w+a_2yt+a_3y^3z+a_4xy^5)+wf_{16}(x,y,z,t)+f_{27}(x,y,z,t)$
\vspace*{1.2mm}
\end{minipage}
}\\

\hline \hline
\begin{minipage}[m]{.28\linewidth}
\begin{center}
Singularity
\end{center}
\end{minipage}&
\begin{minipage}[m]{.04\linewidth}
\begin{center}
$B^3$
\end{center}
\end{minipage}&
\begin{minipage}[m]{.11\linewidth}
\begin{center}
Linear

system
\end{center}
\end{minipage}&
\begin{minipage}[m]{.11\linewidth}
\begin{center}
Surface $T$
\end{center}
\end{minipage}&
\begin{minipage}[m]{.11\linewidth}
\begin{center}
\vspace*{1mm}
 \vorder
\vspace*{1mm}
\end{center}
\end{minipage}&
\begin{minipage}[m]{.18\linewidth}
\begin{center}
Condition
\end{center}
\end{minipage}\\
\hline
\begin{minipage}[m]{.28\linewidth}

$O_w=\frac{1}{11}(1,2,9)$ \quadratic

\end{minipage}&
\multicolumn{4}{|l|}{\begin{minipage}[m]{.37\linewidth}
\begin{center}
$zw^2$
\end{center}
\end{minipage}}&
\begin{minipage}[m]{.18\linewidth}
\begin{center}

\end{center}
\end{minipage}\\

\hline
\begin{minipage}[m]{.28\linewidth}

$O_z=\frac{1}{5}(1_x,4_t,1_w)$ \boundary

\end{minipage}&
\begin{minipage}[m]{.04\linewidth}
\begin{center}
$-$
\end{center}
\end{minipage}&
\begin{minipage}[m]{.11\linewidth}
\begin{center}
$2B$
\end{center}
\end{minipage}&
\begin{minipage}[m]{.11\linewidth}
\begin{center}
$y$
\end{center}
\end{minipage}
&
\begin{minipage}[m]{.11\linewidth}
\begin{center}
$zw^2$
\end{center}
\end{minipage}&
\begin{minipage}[m]{.18\linewidth}
\begin{center}

\end{center}
\end{minipage}\\
\hline
\begin{minipage}[m]{.28\linewidth}

$O_y=\frac{1}{2}(1,1,1)$ \nef

\end{minipage}&
\begin{minipage}[m]{.04\linewidth}
\begin{center}
$-$
\end{center}
\end{minipage}&
\begin{minipage}[m]{.11\linewidth}
\begin{center}
$11B+5E$
\end{center}
\end{minipage}&
\begin{minipage}[m]{.11\linewidth}
\begin{center}
$w+xy^5$
\end{center}
\end{minipage}
&
\begin{minipage}[m]{.11\linewidth}
\begin{center}
$xy^5$ or $w$
\end{center}
\end{minipage}&
\begin{minipage}[m]{.18\linewidth}
\begin{center}

\end{center}
\end{minipage}\\

\hline
\end{longtable}
\end{center}

\begin{Note}

\item 
The $1$-cycle  $\Gamma$ for the singular point $O_z$ is
irreducible due to the monomials $zw^2$ and $t^3$

\item For the singular point $O_y$, we consider the linear system
$|-11K_{X_{27}}|$.  Note that  every member of the linear system
passes through the point $O_y$ and the base locus of the linear
system contains no curves. Since the proper transform of a general
member in $|-11K_{X_{27}}|$ belongs to the linear system
$|11B+5E|$, the divisor $T$ is nef.
\end{Note}




\begin{center}
\begin{longtable}{|l|c|c|c|c|c|}
\hline
\multicolumn{6}{|l|}{\underline{\textbf{No. 66}}: $X_{27}\subset\mathbb{P}(1,5,6,7,9)$\hfill $A^3=1/70$}\\
\multicolumn{6}{|l|}{
\begin{minipage}[m]{.86\linewidth}
\vspace*{1.2mm}
$w^3+zt^3+z^3w+y^4t+ay^3z^2+w^2f_{9}(x,y,z,t)+wf_{18}(x,y,z,t)+f_{27}(x,y,z,t)$
\vspace*{1.2mm}

\end{minipage}
}\\
\hline \hline
\begin{minipage}[m]{.28\linewidth}
\begin{center}
Singularity
\end{center}
\end{minipage}&
\begin{minipage}[m]{.04\linewidth}
\begin{center}
$B^3$
\end{center}
\end{minipage}&
\begin{minipage}[m]{.11\linewidth}
\begin{center}
Linear

system
\end{center}
\end{minipage}&
\begin{minipage}[m]{.11\linewidth}
\begin{center}
Surface $T$
\end{center}
\end{minipage}&
\begin{minipage}[m]{.11\linewidth}
\begin{center}
\vspace*{1mm}
 \vorder
\vspace*{1mm}
\end{center}
\end{minipage}&
\begin{minipage}[m]{.18\linewidth}
\begin{center}
Condition
\end{center}
\end{minipage}\\
\hline
\begin{minipage}[m]{.28\linewidth}

$O_t=\frac{1}{7}(1_x,5_y,2_w)$ \boundary

\end{minipage}&
\begin{minipage}[m]{.04\linewidth}
\begin{center}
$0$
\end{center}
\end{minipage}&
\begin{minipage}[m]{.11\linewidth}
\begin{center}
$5B$
\end{center}
\end{minipage}&
\begin{minipage}[m]{.11\linewidth}
\begin{center}
$y$
\end{center}
\end{minipage}
&
\begin{minipage}[m]{.11\linewidth}
\begin{center}
$y$
\end{center}
\end{minipage}&
\begin{minipage}[m]{.18\linewidth}
\begin{center}

\end{center}
\end{minipage}\\
\hline
\begin{minipage}[m]{.28\linewidth}

$O_z=\frac{1}{6}(1_x,5_y,1_t)$ \boundary

\end{minipage}&
\begin{minipage}[m]{.04\linewidth}
\begin{center}
$-$
\end{center}
\end{minipage}&
\begin{minipage}[m]{.11\linewidth}
\begin{center}
$5B$
\end{center}
\end{minipage}&
\begin{minipage}[m]{.11\linewidth}
\begin{center}
$y$
\end{center}
\end{minipage}
&
\begin{minipage}[m]{.11\linewidth}
\begin{center}
$y$
\end{center}
\end{minipage}&
\begin{minipage}[m]{.18\linewidth}
\begin{center}

\end{center}
\end{minipage}\\

\hline
\begin{minipage}[m]{.28\linewidth}

$O_y=\frac{1}{5}(1_x,1_z,4_w)$ $\nef$

\end{minipage}&
\begin{minipage}[m]{.04\linewidth}
\begin{center}
$-$
\end{center}
\end{minipage}&
\begin{minipage}[m]{.11\linewidth}
\begin{center}
$7B+E$
\end{center}
\end{minipage}&
\begin{minipage}[m]{.11\linewidth}
\begin{center}
$t$
\end{center}
\end{minipage}
&
\begin{minipage}[m]{.11\linewidth}
\begin{center}
$y^3z^2$
\end{center}
\end{minipage}&
\begin{minipage}[m]{.18\linewidth}
\begin{center}
$a\ne0$
\end{center}
\end{minipage}\\
\hline
\begin{minipage}[m]{.28\linewidth}

$O_y=\frac{1}{5}(1_x,1_z,4_w)$ $\surface$

\end{minipage}&
\begin{minipage}[m]{.04\linewidth}
\begin{center}
$-$
\end{center}
\end{minipage}&
\begin{minipage}[m]{.11\linewidth}
\begin{center}
$7B$
\end{center}
\end{minipage}&
\begin{minipage}[m]{.11\linewidth}
\begin{center}
$x^7$, $t$
\end{center}
\end{minipage}
&
\begin{minipage}[m]{.11\linewidth}
\begin{center}
$x^7, wz^3$
\end{center}
\end{minipage}&
\begin{minipage}[m]{.18\linewidth}
\begin{center}
$a=0$
\end{center}
\end{minipage}\\
\hline
\begin{minipage}[m]{.28\linewidth}

$O_zO_w=1\times\frac{1}{3}(1_x,2_y,1_t)$ \boundary

\end{minipage}&
\begin{minipage}[m]{.04\linewidth}
\begin{center}
$-$
\end{center}
\end{minipage}&
\begin{minipage}[m]{.11\linewidth}
\begin{center}
$5B+E$
\end{center}
\end{minipage}&
\begin{minipage}[m]{.11\linewidth}
\begin{center}
$y$
\end{center}
\end{minipage}
&
\begin{minipage}[m]{.11\linewidth}
\begin{center}
$y$
\end{center}
\end{minipage}&
\begin{minipage}[m]{.18\linewidth}
\begin{center}

\end{center}
\end{minipage}\\
\hline

\end{longtable}
\end{center}

We may assume that the polynomial $f_{27}$ contains neither
$xy^4z$ nor $x^2y^5$ by changing the coordinate $t$ in an
appropriate  way.

\begin{Note}

\item 
 For the singular points except the point $O_y$, the
$1$-cycles $\Gamma$ are always irreducible because of the
monomials $w^3$, $zt^3$ and $z^3w$.

\item For the singular point $O_y$ with $a\ne 0$, we consider the linear
system generated by $x^2y$, $xz$ and $t$. Its base curve $C$ is
cut out by $x=t=0$. It is irreducible because of the monomials
$w^3$ and $y^3z^2$. Since we have $T\cdot\tilde{C}=(7B+E)^2\cdot
B=\frac{1}{2}$, the divisor $T$ is nef.

\item For the singular point $O_y$ with $a= 0$, we take a general member
$H$ in the linear system generated by $x^7$ and $t$. Then it is a
normal surface of degree $27$ in $\mathbb{P}(1,5,6,9)$. Let $T$
be the proper transform of the surface $H$. The intersection of
$T$ with the surface $S$ gives us a divisor consisting of two
irreducible curves
 $\tilde{L}_{yz}$ and $\tilde{C}$ on the normal surface $T$.
  The curve $\tilde{C}$ is the proper transform of the
curve $C$ defined by $$x=t=w^2+z^3=0.$$ From the intersection
numbers
\[(\tilde{L}_{yz}+\tilde{C})\cdot\tilde{L}_{yz}=-K_Y\cdot \tilde{L}_{yz}=-\frac{1}{6},
\ \ \ (\tilde{L}_{yz}+\tilde{C})^2=7B^3=-\frac{1}{4}\] on the
surface $T$, we obtain
\[\tilde{L}_{yz}^2=-\frac{1}{6}-\tilde{L}_{yz}\cdot\tilde{C}, \ \ \ \tilde{C}^2=-\frac{1}{12}-\tilde{L}_{yz}\cdot\tilde{C}.\]
With these intersection numbers we see that the matrix
\[\left(\begin{array}{cc}
       \tilde{L}_{yz}^2&\tilde{L}_{yz}\cdot\tilde{C}\\
       \tilde{L}_{yz}\cdot\tilde{C}& \tilde{C}^2\\
\end{array}\right)= \left(\begin{array}{cc}
-\frac{1}{6}-\tilde{L}_{yz}\cdot\tilde{C} & \tilde{L}_{yz}\cdot\tilde{C}\\
     \tilde{L}_{yz}\cdot\tilde{C} & -\frac{1}{12}-\tilde{L}_{yz}\cdot\tilde{C} \\
        \end{array}\right)
\]
is negative-definite since $\tilde{L}_{yz}\cdot\tilde{C}$ is
non-negative.

\end{Note}

\begin{center}
\begin{longtable}{|l|c|c|c|c|c|}
\hline
\multicolumn{6}{|l|}{\underline{\textbf{No. 67}}: $X_{28}\subset\mathbb{P}(1,1,4,9,14)$\hfill $A^3=1/18$}\\
\multicolumn{6}{|l|}{
\begin{minipage}[m]{0.86\linewidth}
\vspace*{1.2mm} $w^2 +yt^3+z^7+ wf_{14}(x,y,z,t)+f_{28}(x,y,z,t)$
\vspace*{1.2mm}
\end{minipage}
}\\
\hline \hline
\begin{minipage}[m]{.28\linewidth}
\begin{center}
Singularity
\end{center}
\end{minipage}&
\begin{minipage}[m]{.04\linewidth}
\begin{center}
$B^3$
\end{center}
\end{minipage}&
\begin{minipage}[m]{.11\linewidth}
\begin{center}
Linear

system
\end{center}
\end{minipage}&
\begin{minipage}[m]{.11\linewidth}
\begin{center}
Surface $T$
\end{center}
\end{minipage}&
\begin{minipage}[m]{.11\linewidth}
\begin{center}
\vspace*{1mm}
 \vorder
\vspace*{1mm}
\end{center}
\end{minipage}&
\begin{minipage}[m]{.18\linewidth}
\begin{center}
Condition
\end{center}
\end{minipage}\\
\hline
\begin{minipage}[m]{.28\linewidth}

$O_t=\frac{1}{9}(1_x,4_z,5_w)$ $\positive$

\end{minipage}&
\begin{minipage}[m]{.04\linewidth}
\begin{center}
$+$
\end{center}
\end{minipage}&
\begin{minipage}[m]{.11\linewidth}
\begin{center}
$B-E$
\end{center}
\end{minipage}&
\begin{minipage}[m]{.11\linewidth}
\begin{center}
$y$
\end{center}
\end{minipage}
&
\begin{minipage}[m]{.11\linewidth}
\begin{center}
$w^2$
\end{center}
\end{minipage}&
\begin{minipage}[m]{.18\linewidth}
\begin{center}

\end{center}
\end{minipage}\\
\hline
\begin{minipage}[m]{.28\linewidth}

$O_zO_w=1\times\frac{1}{2}(1_x,1_y,1_t)$ \boundary

\end{minipage}&
\begin{minipage}[m]{.04\linewidth}
\begin{center}
$-$
\end{center}
\end{minipage}&
\begin{minipage}[m]{.11\linewidth}
\begin{center}
$B$
\end{center}
\end{minipage}&
\begin{minipage}[m]{.11\linewidth}
\begin{center}
$y$
\end{center}
\end{minipage}
&
\begin{minipage}[m]{.11\linewidth}
\begin{center}
$y$
\end{center}
\end{minipage}&
\begin{minipage}[m]{.18\linewidth}
\begin{center}

\end{center}
\end{minipage}\\
\hline

\end{longtable}
\end{center}

\begin{Note}

\item 
The $1$-cycle  $\Gamma$ for the singular point of type
$\frac{1}{2}(1,1,1)$ is irreducible since we have the monomials
$w^2$ and $z^7$.
\end{Note}


\begin{center}
\begin{longtable}{|l|c|c|c|c|c|}
\hline
\multicolumn{6}{|l|}{\textbf{No. 68}: $X_{28}\subset\mathbb{P}(1,3,4,7,14)$\hfill $A^3=1/42$}\\
\multicolumn{6}{|l|}{
\begin{minipage}[m]{.86\linewidth}
\vspace*{1.2mm}

$(w-\alpha_1 t^2)(w-\alpha_2
t^2)+z^7+y^7(a_1t+a_2xy^2)+wf_{14}(x,y,z,t)+f_{28}(x,y,z,t)$
\vspace*{1.2mm}
\end{minipage}
}\\
\hline \hline
\begin{minipage}[m]{.28\linewidth}
\begin{center}
Singularity
\end{center}
\end{minipage}&
\begin{minipage}[m]{.04\linewidth}
\begin{center}
$B^3$
\end{center}
\end{minipage}&
\begin{minipage}[m]{.11\linewidth}
\begin{center}
Linear

system
\end{center}
\end{minipage}&
\begin{minipage}[m]{.11\linewidth}
\begin{center}
Surface $T$
\end{center}
\end{minipage}&
\begin{minipage}[m]{.11\linewidth}
\begin{center}
\vspace*{1mm}
 \vorder
\vspace*{1mm}
\end{center}
\end{minipage}&
\begin{minipage}[m]{.18\linewidth}
\begin{center}
Condition
\end{center}
\end{minipage}\\
\hline
\begin{minipage}[m]{.28\linewidth}

$O_y=\frac{1}{3}(1_x,1_z,2_w)$ \boundary

\end{minipage}&
\begin{minipage}[m]{.04\linewidth}
\begin{center}
$-$
\end{center}
\end{minipage}&
\begin{minipage}[m]{.11\linewidth}
\begin{center}
$7B+E$
\end{center}
\end{minipage}&
\begin{minipage}[m]{.11\linewidth}
\begin{center}
$t$
\end{center}
\end{minipage}
&
\begin{minipage}[m]{.11\linewidth}
\begin{center}
$w^2$
\end{center}
\end{minipage}&
\begin{minipage}[m]{.18\linewidth}
\begin{center}
$a_1\ne 0$
\end{center}
\end{minipage}\\
\hline
\begin{minipage}[m]{.28\linewidth}

$O_y=\frac{1}{3}(1_z,1_t,2_w)$ \boundary

\end{minipage}&
\begin{minipage}[m]{.04\linewidth}
\begin{center}
$-$
\end{center}
\end{minipage}&
\begin{minipage}[m]{.11\linewidth}
\begin{center}
$4B+E$
\end{center}
\end{minipage}&
\begin{minipage}[m]{.11\linewidth}
\begin{center}
$z$
\end{center}
\end{minipage}
&
\begin{minipage}[m]{.11\linewidth}
\begin{center}
$z$
\end{center}
\end{minipage}&
\begin{minipage}[m]{.18\linewidth}
\begin{center}
$a_1=0$
\end{center}
\end{minipage}\\
\hline
\begin{minipage}[m]{.28\linewidth}

$O_tO_w=2\times\frac{1}{7}(1,3,4)$ \quadratic

\end{minipage}&
\multicolumn{4}{|l|}{\begin{minipage}[m]{.37\linewidth}
\begin{center}
$wt^2$
\end{center}
\end{minipage}}&
\begin{minipage}[m]{.18\linewidth}
\begin{center}

\end{center}
\end{minipage}\\
\hline
\begin{minipage}[m]{.28\linewidth}
$O_zO_w=1\times\frac{1}{2}(1_x,1_y,1_t)$ \boundary

\end{minipage}&
\begin{minipage}[m]{.04\linewidth}
\begin{center}
$-$
\end{center}
\end{minipage}&
\begin{minipage}[m]{.11\linewidth}
\begin{center}
$3B+E$
\end{center}
\end{minipage}&
\begin{minipage}[m]{.11\linewidth}
\begin{center}
$y$
\end{center}
\end{minipage}
&
\begin{minipage}[m]{.11\linewidth}
\begin{center}
$y$
\end{center}
\end{minipage}&
\begin{minipage}[m]{.18\linewidth}
\begin{center}

\end{center}
\end{minipage}\\

\hline
\end{longtable}
\end{center}

\begin{Note}

\item 
For the singular point $O_y$ with $a_1\ne 0$ the $1$-cycle
$\Gamma$ is irreducible  because of the monomials $w^2$ and $z^7$.

\item The $1$-cycle $\Gamma$ for the singular point $O_y$ with $a_1=0$
consists of two irreducible curves. These are the proper
transforms of the curves  defined by $x=z=w-\alpha_i t^2=0$. Since
these two curves on $X_{28}$ are interchanged by the automorphism
defined by
$$[x,y,z,t,w]\mapsto[x,y,z,t,(\alpha_1+\alpha_2)t^2-f_{14}-w],$$
their proper transforms are numerically equivalent on $Y$.

\item To see how to deal with the singular points of type
$\frac{1}{7}(1,3,4)$ we may assume that $\alpha_1=0$ and we have
only to consider the singular point $O_t$. The other point can be
treated in the same way.

\item The $1$-cycle  $\Gamma$ for the singular point of type
$\frac{1}{2}(1,1,1)$ is irreducible due to the monomials $w^2$ and
$z^7$.
\end{Note}

\begin{center}
\begin{longtable}{|l|c|c|c|c|c|}
\hline
\multicolumn{6}{|l|}{\textbf{No. 69}: $X_{28}\subset\mathbb{P}(1,4,6,7,11)$\hfill $A^3=1/66$}\\
\multicolumn{6}{|l|}{
\begin{minipage}[m]{.86\linewidth}
\vspace*{1.2mm} $zw^2+t^4+y(z^2-\alpha_1 y^3)(z^2-\alpha_2
y^3)+wf_{17}(x,y,z,t)+f_{28}(x,y,z,t)$

\vspace*{1.2mm}
\end{minipage}
}\\
\hline \hline
\begin{minipage}[m]{.28\linewidth}
\begin{center}
Singularity
\end{center}
\end{minipage}&
\begin{minipage}[m]{.04\linewidth}
\begin{center}
$B^3$
\end{center}
\end{minipage}&
\begin{minipage}[m]{.11\linewidth}
\begin{center}
Linear

system
\end{center}
\end{minipage}&
\begin{minipage}[m]{.11\linewidth}
\begin{center}
Surface $T$
\end{center}
\end{minipage}&
\begin{minipage}[m]{.11\linewidth}
\begin{center}
\vspace*{1mm}
 \vorder
\vspace*{1mm}
\end{center}
\end{minipage}&
\begin{minipage}[m]{.18\linewidth}
\begin{center}
Condition
\end{center}
\end{minipage}\\
\hline
\begin{minipage}[m]{.28\linewidth}

$O_w=\frac{1}{11}(1,4,7)$ \quadratic

\end{minipage}&
\multicolumn{4}{|l|}{\begin{minipage}[m]{.37\linewidth}
\begin{center}
$zw^2$
\end{center}
\end{minipage}}&
\begin{minipage}[m]{.18\linewidth}
\begin{center}

\end{center}
\end{minipage}\\

\hline
\begin{minipage}[m]{.28\linewidth}

$O_z=\frac{1}{6}(1_x,1_t,5_w)$ \boundary

\end{minipage}&
\begin{minipage}[m]{.04\linewidth}
\begin{center}
$-$
\end{center}
\end{minipage}&
\begin{minipage}[m]{.11\linewidth}
\begin{center}
$4B$
\end{center}
\end{minipage}&
\begin{minipage}[m]{.11\linewidth}
\begin{center}
$y$
\end{center}
\end{minipage}
&
\begin{minipage}[m]{.11\linewidth}
\begin{center}
$t^4$
\end{center}
\end{minipage}&
\begin{minipage}[m]{.18\linewidth}
\begin{center}

\end{center}
\end{minipage}\\
\hline
\begin{minipage}[m]{.28\linewidth}
$O_yO_z=2\times\frac{1}{2}(1_x,1_t,1_w)$ $\nef$

\end{minipage}&
\begin{minipage}[m]{.04\linewidth}
\begin{center}
$-$
\end{center}
\end{minipage}&
\begin{minipage}[m]{.11\linewidth}
\begin{center}
$11B+5E$
\end{center}
\end{minipage}&
\begin{minipage}[m]{.11\linewidth}
\begin{center}
$xyz$, $yt$, $w$
\end{center}
\end{minipage}
&
\begin{minipage}[m]{.11\linewidth}
\begin{center}
$xyz$, $yt$, $w$
\end{center}
\end{minipage}&
\begin{minipage}[m]{.18\linewidth}
\begin{center}

\end{center}
\end{minipage}\\

\hline
\end{longtable}
\end{center}

\begin{Note}

\item 
The $1$-cycle  $\Gamma$ for the singular point $O_z$  is
irreducible due to the monomials  $zw^2$ and $t^4$.

\item For the singular points of type $\frac{1}{2}(1,1,1)$ consider the
linear system generated by $xyz$, $yt$ and $w$. Since the base
curves of the linear system  pass through no singular points of
type $\frac{1}{2}(1,1,1)$ the divisor $T$ is nef.
\end{Note}



\begin{center}
\begin{longtable}{|l|c|c|c|c|c|}
\hline
\multicolumn{6}{|l|}{\underline{\textbf{No. 70}}: $X_{30}\subset\mathbb{P}(1,1,4,10,15)$\hfill $A^3=1/20$}\\
\multicolumn{6}{|l|}{
\begin{minipage}[m]{.86\linewidth}
\vspace*{1.2mm} $w^2+t^3+z^5t+wf_{15}(x,y,z,t)+f_{30}(x,y,z,t)$
\vspace*{1.2mm}
\end{minipage}
}\\
\hline \hline
\begin{minipage}[m]{.28\linewidth}
\begin{center}
Singularity
\end{center}
\end{minipage}&
\begin{minipage}[m]{.04\linewidth}
\begin{center}
$B^3$
\end{center}
\end{minipage}&
\begin{minipage}[m]{.11\linewidth}
\begin{center}
Linear

system
\end{center}
\end{minipage}&
\begin{minipage}[m]{.11\linewidth}
\begin{center}
Surface $T$
\end{center}
\end{minipage}&
\begin{minipage}[m]{.11\linewidth}
\begin{center}
\vspace*{1mm}
 \vorder
\vspace*{1mm}
\end{center}
\end{minipage}&
\begin{minipage}[m]{.18\linewidth}
\begin{center}
Condition
\end{center}
\end{minipage}\\
\hline
\begin{minipage}[m]{.28\linewidth}

$O_z=\frac{1}{4}(1_x,1_y,3_w)$ \boundary

\end{minipage}&
\begin{minipage}[m]{.04\linewidth}
\begin{center}
$-$
\end{center}
\end{minipage}&
\begin{minipage}[m]{.11\linewidth}
\begin{center}
$B$
\end{center}
\end{minipage}&
\begin{minipage}[m]{.11\linewidth}
\begin{center}
$y$
\end{center}
\end{minipage}
&
\begin{minipage}[m]{.11\linewidth}
\begin{center}
$y$
\end{center}
\end{minipage}&
\begin{minipage}[m]{.18\linewidth}
\begin{center}

\end{center}
\end{minipage}\\
\hline
\begin{minipage}[m]{.28\linewidth}

$O_tO_w=1\times\frac{1}{5}(1_x,1_y,4_z)$ \boundary

\end{minipage}&
\begin{minipage}[m]{.04\linewidth}
\begin{center}
$0$
\end{center}
\end{minipage}&
\begin{minipage}[m]{.11\linewidth}
\begin{center}
$B$
\end{center}
\end{minipage}&
\begin{minipage}[m]{.11\linewidth}
\begin{center}
$y$
\end{center}
\end{minipage}
&
\begin{minipage}[m]{.11\linewidth}
\begin{center}
$y$
\end{center}
\end{minipage}&
\begin{minipage}[m]{.18\linewidth}
\begin{center}

\end{center}
\end{minipage}\\
\hline
\begin{minipage}[m]{.28\linewidth}

$O_zO_t=1\times\frac{1}{2}(1_x,1_y,1_w)$ \boundary

\end{minipage}&
\begin{minipage}[m]{.04\linewidth}
\begin{center}
$-$
\end{center}
\end{minipage}&
\begin{minipage}[m]{.11\linewidth}
\begin{center}
$B$
\end{center}
\end{minipage}&
\begin{minipage}[m]{.11\linewidth}
\begin{center}
$y$
\end{center}
\end{minipage}
&
\begin{minipage}[m]{.11\linewidth}
\begin{center}
$y$
\end{center}
\end{minipage}&
\begin{minipage}[m]{.18\linewidth}
\begin{center}

\end{center}
\end{minipage}\\

\hline

\end{longtable}
\end{center}

\begin{Note}

\item 
For each singular point the $1$-cycle $\Gamma$ is irreducible due
to the monomials $w^2$ and $t^3$.
\end{Note}


\begin{center}
\begin{longtable}{|l|c|c|c|c|c|}
\hline
\multicolumn{6}{|l|}{\underline{\textbf{No. 71}}: $X_{30}\subset\mathbb{P}(1,1,6,8,15)$\hfill $A^3=1/24$}\\
\multicolumn{6}{|l|}{
\begin{minipage}[m]{.86\linewidth}
\vspace*{1.2mm} $w^2+zt^3+wf_{15}(x,y,z,t)+ f_{30}(x,y,z,t)$
\vspace*{1.2mm}
\end{minipage}
}\\
\hline \hline
\begin{minipage}[m]{.28\linewidth}
\begin{center}
Singularity
\end{center}
\end{minipage}&
\begin{minipage}[m]{.04\linewidth}
\begin{center}
$B^3$
\end{center}
\end{minipage}&
\begin{minipage}[m]{.11\linewidth}
\begin{center}
Linear

system
\end{center}
\end{minipage}&
\begin{minipage}[m]{.11\linewidth}
\begin{center}
Surface $T$
\end{center}
\end{minipage}&
\begin{minipage}[m]{.11\linewidth}
\begin{center}
\vspace*{1mm}
 \vorder
\vspace*{1mm}
\end{center}
\end{minipage}&
\begin{minipage}[m]{.18\linewidth}
\begin{center}
Condition
\end{center}
\end{minipage}\\
\hline
\begin{minipage}[m]{.28\linewidth}

$O_t=\frac{1}{8}(1_x,1_y,7_w)$ $\positive$

\end{minipage}&
\begin{minipage}[m]{.04\linewidth}
\begin{center}
$+$
\end{center}
\end{minipage}&
\begin{minipage}[m]{.11\linewidth}
\begin{center}
$6B-E$
\end{center}
\end{minipage}&
\begin{minipage}[m]{.11\linewidth}
\begin{center}
$z$
\end{center}
\end{minipage}
&
\begin{minipage}[m]{.11\linewidth}
\begin{center}
$w^2$
\end{center}
\end{minipage}&
\begin{minipage}[m]{.18\linewidth}
\begin{center}

\end{center}
\end{minipage}\\
\hline
\begin{minipage}[m]{.28\linewidth}

$O_zO_w=1\times\frac{1}{3}(1_x,1_y,2_t)$ \boundary

\end{minipage}&
\begin{minipage}[m]{.04\linewidth}
\begin{center}
$-$
\end{center}
\end{minipage}&
\begin{minipage}[m]{.11\linewidth}
\begin{center}
$B$
\end{center}
\end{minipage}&
\begin{minipage}[m]{.11\linewidth}
\begin{center}
$y$
\end{center}
\end{minipage}
&
\begin{minipage}[m]{.11\linewidth}
\begin{center}
$y$
\end{center}
\end{minipage}&
\begin{minipage}[m]{.18\linewidth}
\begin{center}

\end{center}
\end{minipage}\\
\hline
\begin{minipage}[m]{.28\linewidth}

$O_zO_t=1\times\frac{1}{2}(1_x,1_y,1_w)$ \boundary

\end{minipage}&
\begin{minipage}[m]{.04\linewidth}
\begin{center}
$-$
\end{center}
\end{minipage}&
\begin{minipage}[m]{.11\linewidth}
\begin{center}
$B$
\end{center}
\end{minipage}&
\begin{minipage}[m]{.11\linewidth}
\begin{center}
$y$
\end{center}
\end{minipage}
&
\begin{minipage}[m]{.11\linewidth}
\begin{center}
$y$
\end{center}
\end{minipage}&
\begin{minipage}[m]{.18\linewidth}
\begin{center}

\end{center}
\end{minipage}\\
\hline

\end{longtable}
\end{center}

\begin{Note}

\item 
For the singular points of types $\frac{1}{2}(1,1,1)$ and
$\frac{1}{3}(1,1,2)$, the $1$-cycles  $\Gamma$ are irreducible
because of  $w^2$ and $t^3z$.
\end{Note}



\begin{center}
\begin{longtable}{|l|c|c|c|c|c|}
\hline
\multicolumn{6}{|l|}{\underline{\textbf{No. 72}}: $X_{30}\subset\mathbb{P}(1,2,3,10,15)$\hfill $A^3=1/30$}\\
\multicolumn{6}{|l|}{
\begin{minipage}[m]{.86\linewidth}
\vspace*{1.2mm}
$w^2+t^3+z^{10}+y^{15}+wf_{15}(x,y,z,t)+f_{30}(x,y,z,t)$
\vspace*{1.2mm}
\end{minipage}
}\\
\hline \hline
\begin{minipage}[m]{.28\linewidth}
\begin{center}
Singularity
\end{center}
\end{minipage}&
\begin{minipage}[m]{.04\linewidth}
\begin{center}
$B^3$
\end{center}
\end{minipage}&
\begin{minipage}[m]{.11\linewidth}
\begin{center}
Linear

system
\end{center}
\end{minipage}&
\begin{minipage}[m]{.11\linewidth}
\begin{center}
Surface $T$
\end{center}
\end{minipage}&
\begin{minipage}[m]{.11\linewidth}
\begin{center}
\vspace*{1mm}
 \vorder
\vspace*{1mm}
\end{center}
\end{minipage}&
\begin{minipage}[m]{.18\linewidth}
\begin{center}
Condition
\end{center}
\end{minipage}\\
\hline
\begin{minipage}[m]{.28\linewidth}

$O_tO_w=1\times\frac{1}{5}(1_x,2_y,3_z)$ \boundary

\end{minipage}&
\begin{minipage}[m]{.04\linewidth}
\begin{center}
$0$
\end{center}
\end{minipage}&
\begin{minipage}[m]{.11\linewidth}
\begin{center}
$2B$
\end{center}
\end{minipage}&
\begin{minipage}[m]{.11\linewidth}
\begin{center}
$y$
\end{center}
\end{minipage}
&
\begin{minipage}[m]{.11\linewidth}
\begin{center}
$y$
\end{center}
\end{minipage}&
\begin{minipage}[m]{.18\linewidth}
\begin{center}

\end{center}
\end{minipage}\\
\hline
\begin{minipage}[m]{.28\linewidth}

$O_zO_w=2\times\frac{1}{3}(1_x,2_y,1_t)$ \boundary

\end{minipage}&
\begin{minipage}[m]{.04\linewidth}
\begin{center}
$-$
\end{center}
\end{minipage}&
\begin{minipage}[m]{.11\linewidth}
\begin{center}
$2B$
\end{center}
\end{minipage}&
\begin{minipage}[m]{.11\linewidth}
\begin{center}
$y$
\end{center}
\end{minipage}
&
\begin{minipage}[m]{.11\linewidth}
\begin{center}
$y$
\end{center}
\end{minipage}&
\begin{minipage}[m]{.18\linewidth}
\begin{center}

\end{center}
\end{minipage}\\
\hline
\begin{minipage}[m]{.28\linewidth}

$O_yO_t=3\times\frac{1}{2}(1_x,1_z,1_w)$ \boundary

\end{minipage}&
\begin{minipage}[m]{.04\linewidth}
\begin{center}
$-$
\end{center}
\end{minipage}&
\begin{minipage}[m]{.11\linewidth}
\begin{center}
$3B+E$
\end{center}
\end{minipage}&
\begin{minipage}[m]{.11\linewidth}
\begin{center}
$z$
\end{center}
\end{minipage}
&
\begin{minipage}[m]{.11\linewidth}
\begin{center}
$z$
\end{center}
\end{minipage}&
\begin{minipage}[m]{.18\linewidth}
\begin{center}

\end{center}
\end{minipage}\\

\hline

\end{longtable}
\end{center}

\begin{Note}

\item 
For each singular point the $1$-cycle $\Gamma$ is irreducible due
to the monomials $w^2$ and $t^3$.
\end{Note}

\begin{center}
\begin{longtable}{|l|c|c|c|c|c|}
\hline
\multicolumn{6}{|l|}{\underline{\textbf{No. 73}}: $X_{30}\subset\mathbb{P}(1,2,6,7,15)$\hfill $A^3=1/42$}\\
\multicolumn{6}{|l|}{
\begin{minipage}[m]{.86\linewidth}
\vspace*{1.2mm} $w^2+yt^4+\prod_{i=1}^{5}(z-\alpha_iy^3)
+wf_{15}(x,y,z,t)+f_{30}(x,y,z,t)$ \vspace*{1.2mm}
\end{minipage}
}\\
\hline \hline
\begin{minipage}[m]{.28\linewidth}
\begin{center}
Singularity
\end{center}
\end{minipage}&
\begin{minipage}[m]{.04\linewidth}
\begin{center}
$B^3$
\end{center}
\end{minipage}&
\begin{minipage}[m]{.11\linewidth}
\begin{center}
Linear

system
\end{center}
\end{minipage}&
\begin{minipage}[m]{.11\linewidth}
\begin{center}
Surface $T$
\end{center}
\end{minipage}&
\begin{minipage}[m]{.11\linewidth}
\begin{center}
\vspace*{1mm}
 \vorder
\vspace*{1mm}
\end{center}
\end{minipage}&
\begin{minipage}[m]{.18\linewidth}
\begin{center}
Condition
\end{center}
\end{minipage}\\
\hline
\begin{minipage}[m]{.28\linewidth}

$O_t=\frac{1}{7}(1_x,6_z,1_w)$ \boundary

\end{minipage}&
\begin{minipage}[m]{.04\linewidth}
\begin{center}
$0$
\end{center}
\end{minipage}&
\begin{minipage}[m]{.11\linewidth}
\begin{center}
$2B$
\end{center}
\end{minipage}&
\begin{minipage}[m]{.11\linewidth}
\begin{center}
$y$
\end{center}
\end{minipage}
&
\begin{minipage}[m]{.11\linewidth}
\begin{center}
$w^2$
\end{center}
\end{minipage}&
\begin{minipage}[m]{.18\linewidth}
\begin{center}

\end{center}
\end{minipage}\\
\hline
\begin{minipage}[m]{.28\linewidth}

$O_zO_w=1\times\frac{1}{3}(1_x,2_y,1_t)$ \boundary

\end{minipage}&
\begin{minipage}[m]{.04\linewidth}
\begin{center}
$-$
\end{center}
\end{minipage}&
\begin{minipage}[m]{.11\linewidth}
\begin{center}
$2B$
\end{center}
\end{minipage}&
\begin{minipage}[m]{.11\linewidth}
\begin{center}
$y$
\end{center}
\end{minipage}
&
\begin{minipage}[m]{.11\linewidth}
\begin{center}
$y$
\end{center}
\end{minipage}&
\begin{minipage}[m]{.18\linewidth}
\begin{center}

\end{center}
\end{minipage}\\
\hline
\begin{minipage}[m]{.28\linewidth}

$O_yO_z=5\times\frac{1}{2}(1_x,1_t,1_w)$ \boundary

\end{minipage}&
\begin{minipage}[m]{.04\linewidth}
\begin{center}
$-$
\end{center}
\end{minipage}&
\begin{minipage}[m]{.11\linewidth}
\begin{center}
$6B+2E$
\end{center}
\end{minipage}&
\begin{minipage}[m]{.11\linewidth}
\begin{center}
$z-\alpha_i y^3$
\end{center}
\end{minipage}
&
\begin{minipage}[m]{.11\linewidth}
\begin{center}
$w^2$
\end{center}
\end{minipage}&
\begin{minipage}[m]{.18\linewidth}
\begin{center}

\end{center}
\end{minipage}\\
\hline

\end{longtable}
\end{center}

\begin{Note}

\item 
For each singular point the $1$-cycle $\Gamma$ is irreducible due
to the monomials  $w^2$, $z^5$ and $yt^4$.
\end{Note}



\begin{center}
\begin{longtable}{|l|c|c|c|c|c|}
\hline
\multicolumn{6}{|l|}{\textbf{No. 74}: $X_{30}\subset\mathbb{P}(1,3,4,10,13)$\hfill $A^3=1/52$}\\
\multicolumn{6}{|l|}{
\begin{minipage}[m]{.86\linewidth}
\vspace*{1.2mm}
$zw^2+t^3+z^5t+y^{10}+wf_{17}(x,y,z,t)+f_{30}(x,y,z,t)$

\vspace*{1.2mm}
\end{minipage}
}\\
\hline \hline
\begin{minipage}[m]{.28\linewidth}
\begin{center}
Singularity
\end{center}
\end{minipage}&
\begin{minipage}[m]{.04\linewidth}
\begin{center}
$B^3$
\end{center}
\end{minipage}&
\begin{minipage}[m]{.11\linewidth}
\begin{center}
Linear

system
\end{center}
\end{minipage}&
\begin{minipage}[m]{.11\linewidth}
\begin{center}
Surface $T$
\end{center}
\end{minipage}&
\begin{minipage}[m]{.11\linewidth}
\begin{center}
\vspace*{1mm}
 \vorder
\vspace*{1mm}
\end{center}
\end{minipage}&
\begin{minipage}[m]{.18\linewidth}
\begin{center}
Condition
\end{center}
\end{minipage}\\
\hline
\begin{minipage}[m]{.28\linewidth}

$O_w=\frac{1}{13}(1,3,7)$ \quadratic

\end{minipage}&
\multicolumn{4}{|l|}{\begin{minipage}[m]{.37\linewidth}
\begin{center}
$zw^2$
\end{center}
\end{minipage}}&
\begin{minipage}[m]{.18\linewidth}
\begin{center}

\end{center}
\end{minipage}\\
\hline
\begin{minipage}[m]{.28\linewidth}

$O_z=\frac{1}{4}(1_x,3_y,1_w)$ \boundary

\end{minipage}&
\begin{minipage}[m]{.04\linewidth}
\begin{center}
$-$
\end{center}
\end{minipage}&
\begin{minipage}[m]{.11\linewidth}
\begin{center}
$3B$
\end{center}
\end{minipage}&
\begin{minipage}[m]{.11\linewidth}
\begin{center}
$y$
\end{center}
\end{minipage}
&
\begin{minipage}[m]{.11\linewidth}
\begin{center}
$y$
\end{center}
\end{minipage}&
\begin{minipage}[m]{.18\linewidth}
\begin{center}

\end{center}
\end{minipage}\\
\hline
\begin{minipage}[m]{.28\linewidth}
$O_zO_t=1\times\frac{1}{2}(1_x,1_y,1_w)$ \boundary

\end{minipage}&
\begin{minipage}[m]{.04\linewidth}
\begin{center}
$-$
\end{center}
\end{minipage}&
\begin{minipage}[m]{.11\linewidth}
\begin{center}
$3B+E$
\end{center}
\end{minipage}&
\begin{minipage}[m]{.11\linewidth}
\begin{center}
$y$
\end{center}
\end{minipage}
&
\begin{minipage}[m]{.11\linewidth}
\begin{center}
$y$
\end{center}
\end{minipage}&
\begin{minipage}[m]{.18\linewidth}
\begin{center}

\end{center}
\end{minipage}\\

\hline
\end{longtable}
\end{center}

\begin{Note}

\item 
The $1$-cycles $\Gamma$ for the singular points of types
$\frac{1}{2}(1,1,1)$ and $\frac{1}{4}(1,3,1)$ are irreducible
because of the monomials $zw^2$, $t^3$ and $z^5t$.
\end{Note}



\begin{center}
\begin{longtable}{|l|c|c|c|c|c|}
\hline
\multicolumn{6}{|l|}{\underline{\textbf{No. 75}}: $X_{30}\subset\mathbb{P}(1,4,5,6,15)$\hfill $A^3=1/60$}\\
\multicolumn{6}{|l|}{
\begin{minipage}[m]{.86\linewidth}
\vspace*{1.2mm}
$w^2+t^5+z^6+y^6t+wf_{15}(x,y,z,t)+f_{30}(x,y,z,t)$
\vspace*{1.2mm}
\end{minipage}
}\\
\hline \hline
\begin{minipage}[m]{.28\linewidth}
\begin{center}
Singularity
\end{center}
\end{minipage}&
\begin{minipage}[m]{.04\linewidth}
\begin{center}
$B^3$
\end{center}
\end{minipage}&
\begin{minipage}[m]{.11\linewidth}
\begin{center}
Linear

system
\end{center}
\end{minipage}&
\begin{minipage}[m]{.11\linewidth}
\begin{center}
Surface $T$
\end{center}
\end{minipage}&
\begin{minipage}[m]{.11\linewidth}
\begin{center}
\vspace*{1mm}
 \vorder
\vspace*{1mm}
\end{center}
\end{minipage}&
\begin{minipage}[m]{.18\linewidth}
\begin{center}
Condition
\end{center}
\end{minipage}\\
\hline
\begin{minipage}[m]{.28\linewidth}

$O_y=\frac{1}{4}(1_x,1_z,3_w)$ $\nef$

\end{minipage}&
\begin{minipage}[m]{.04\linewidth}
\begin{center}
$-$
\end{center}
\end{minipage}&
\begin{minipage}[m]{.11\linewidth}
\begin{center}
$5B+E$
\end{center}
\end{minipage}&
\begin{minipage}[m]{.11\linewidth}
\begin{center}
$xy$, $z$
\end{center}
\end{minipage}
&
\begin{minipage}[m]{.11\linewidth}
\begin{center}
$xy$, $z$
\end{center}
\end{minipage}&
\begin{minipage}[m]{.18\linewidth}
\begin{center}

\end{center}
\end{minipage}\\
 \hline
\begin{minipage}[m]{.28\linewidth}

$O_tO_w=1\times\frac{1}{3}(1_x,1_y,2_z)$ \boundary

\end{minipage}&
\begin{minipage}[m]{.04\linewidth}
\begin{center}
$-$
\end{center}
\end{minipage}&
\begin{minipage}[m]{.11\linewidth}
\begin{center}
$5B+E$
\end{center}
\end{minipage}&
\begin{minipage}[m]{.11\linewidth}
\begin{center}
$z$
\end{center}
\end{minipage}
&
\begin{minipage}[m]{.11\linewidth}
\begin{center}
$z$
\end{center}
\end{minipage}&
\begin{minipage}[m]{.18\linewidth}
\begin{center}

\end{center}
\end{minipage}\\
\hline
\begin{minipage}[m]{.28\linewidth}

$O_zO_w=2\times\frac{1}{5}(1_x,4_y,1_t)$ \boundary

\end{minipage}&
\begin{minipage}[m]{.04\linewidth}
\begin{center}
$-$
\end{center}
\end{minipage}&
\begin{minipage}[m]{.11\linewidth}
\begin{center}
$4B$
\end{center}
\end{minipage}&
\begin{minipage}[m]{.11\linewidth}
\begin{center}
$y$
\end{center}
\end{minipage}
&
\begin{minipage}[m]{.11\linewidth}
\begin{center}
$y$
\end{center}
\end{minipage}&
\begin{minipage}[m]{.18\linewidth}
\begin{center}

\end{center}
\end{minipage}\\
\hline
\begin{minipage}[m]{.28\linewidth}

$O_yO_t=2\times\frac{1}{2}(1_x,1_z,1_w)$ \boundary

\end{minipage}&
\begin{minipage}[m]{.04\linewidth}
\begin{center}
$-$
\end{center}
\end{minipage}&
\begin{minipage}[m]{.11\linewidth}
\begin{center}
$5B+2E$
\end{center}
\end{minipage}&
\begin{minipage}[m]{.11\linewidth}
\begin{center}
$z$
\end{center}
\end{minipage}
&
\begin{minipage}[m]{.11\linewidth}
\begin{center}
$z$
\end{center}
\end{minipage}&
\begin{minipage}[m]{.18\linewidth}
\begin{center}

\end{center}
\end{minipage}\\
\hline

\end{longtable}
\end{center}

\begin{Note}

\item 
For the singular point $O_y$,  consider the linear system
generated by $xy$ and $z$. Its base curves are defined by $x=z=0$
and $y=z=0$.  The curve defined by $y=z=0$ does not pass through
the point $O_y$.  The curve  defined by $x=z=0$  is irreducible.
Moreover,  its proper transform is the $1$-cycle defined by
$(5B+E)\cdot B$. Consequently, the divisor $T$ is nef since
$(5B+E)^2\cdot B>0$.

\item For the other singular points we immediately see that the
$1$-cycles $\Gamma$ are  irreducible due to the monomials $w^2$
and $t^5$.

\end{Note}




\begin{center}
\begin{longtable}{|l|c|c|c|c|c|}
\hline
\multicolumn{6}{|l|}{\textbf{No. 76}: $X_{30}\subset\mathbb{P}(1,5,6,8,11)$\hfill $A^3=1/88$}\\
\multicolumn{6}{|l|}{
\begin{minipage}[m]{.86\linewidth}
\vspace*{1.2mm}
$tw^2+zt^3+z^5+y^{5}+wf_{19}(x,y,z,t)+f_{30}(x,y,z,t)$

\vspace*{1.2mm}
\end{minipage}
}\\

\hline \hline
\begin{minipage}[m]{.28\linewidth}
\begin{center}
Singularity
\end{center}
\end{minipage}&
\begin{minipage}[m]{.04\linewidth}
\begin{center}
$B^3$
\end{center}
\end{minipage}&
\begin{minipage}[m]{.11\linewidth}
\begin{center}
Linear

system
\end{center}
\end{minipage}&
\begin{minipage}[m]{.11\linewidth}
\begin{center}
Surface $T$
\end{center}
\end{minipage}&
\begin{minipage}[m]{.11\linewidth}
\begin{center}
\vspace*{1mm}
 \vorder
\vspace*{1mm}
\end{center}
\end{minipage}&
\begin{minipage}[m]{.18\linewidth}
\begin{center}
Condition
\end{center}
\end{minipage}\\
\hline
\begin{minipage}[m]{.28\linewidth}

$O_w=\frac{1}{11}(1,5,6)$ \quadratic

\end{minipage}&
\multicolumn{4}{|l|}{\begin{minipage}[m]{.37\linewidth}
\begin{center}
$tw^2$
\end{center}
\end{minipage}}&
\begin{minipage}[m]{.18\linewidth}
\begin{center}
\end{center}
\end{minipage}\\
\hline
\begin{minipage}[m]{.28\linewidth}

$O_t=\frac{1}{8}(1,5,3)$ \elliptic

\end{minipage}&\multicolumn{4}{|l|}{\begin{minipage}[m]{.37\linewidth}
\begin{center}
$tw^2-zt^3$
\end{center}
\end{minipage}}&
\begin{minipage}[m]{.18\linewidth}
\begin{center}

\end{center}
\end{minipage}\\

\hline
\begin{minipage}[m]{.28\linewidth}
$O_zO_t=1\times\frac{1}{2}(1_x,1_y,1_w)$ \boundary

\end{minipage}&
\begin{minipage}[m]{.04\linewidth}
\begin{center}
$-$
\end{center}
\end{minipage}&
\begin{minipage}[m]{.11\linewidth}
\begin{center}
$5B+2E$
\end{center}
\end{minipage}&
\begin{minipage}[m]{.11\linewidth}
\begin{center}
$y$
\end{center}
\end{minipage}
&
\begin{minipage}[m]{.11\linewidth}
\begin{center}
$y$
\end{center}
\end{minipage}&
\begin{minipage}[m]{.18\linewidth}
\begin{center}

\end{center}
\end{minipage}\\

\hline
\end{longtable}
\end{center}

\begin{Note}

\item 
The $1$-cycle $\Gamma$ for the singular point of type
$\frac{1}{2}(1,1,1)$ is irreducible because of the monomials
$tw^2$ and $z^5$.
\end{Note}



\begin{center}
\begin{longtable}{|l|c|c|c|c|c|}
\hline
\multicolumn{6}{|l|}{\underline{\textbf{No. 77}}: $X_{32}\subset\mathbb{P}(1,2,5,9,16)$\hfill $A^3=1/45$}\\
\multicolumn{6}{|l|}{
\begin{minipage}[m]{.86\linewidth}
\vspace*{1.2mm} $w^2+zt^3+yz^6+wf_{16}(x,y,z,t)+ f_{32}(x,y,z,t)$
\vspace*{1.2mm}
\end{minipage}
}\\
\hline \hline
\begin{minipage}[m]{.28\linewidth}
\begin{center}
Singularity
\end{center}
\end{minipage}&
\begin{minipage}[m]{.04\linewidth}
\begin{center}
$B^3$
\end{center}
\end{minipage}&
\begin{minipage}[m]{.11\linewidth}
\begin{center}
Linear

system
\end{center}
\end{minipage}&
\begin{minipage}[m]{.11\linewidth}
\begin{center}
Surface $T$
\end{center}
\end{minipage}&
\begin{minipage}[m]{.11\linewidth}
\begin{center}
\vspace*{1mm}
 \vorder
\vspace*{1mm}
\end{center}
\end{minipage}&
\begin{minipage}[m]{.18\linewidth}
\begin{center}
Condition
\end{center}
\end{minipage}\\
\hline
\begin{minipage}[m]{.28\linewidth}

$O_t=\frac{1}{9}(1_x,2_y,7_w)$ $\positive$

\end{minipage}&
\begin{minipage}[m]{.04\linewidth}
\begin{center}
$+$
\end{center}
\end{minipage}&
\begin{minipage}[m]{.11\linewidth}
\begin{center}
$5B-E$
\end{center}
\end{minipage}&
\begin{minipage}[m]{.11\linewidth}
\begin{center}
$z$
\end{center}
\end{minipage}
&
\begin{minipage}[m]{.11\linewidth}
\begin{center}
$w^2$
\end{center}
\end{minipage}&
\begin{minipage}[m]{.18\linewidth}
\begin{center}

\end{center}
\end{minipage}\\

\hline
\begin{minipage}[m]{.28\linewidth}

$O_z=\frac{1}{5}(1_x,4_t,1_w)$ \boundary

\end{minipage}&
\begin{minipage}[m]{.04\linewidth}
\begin{center}
$-$
\end{center}
\end{minipage}&
\begin{minipage}[m]{.11\linewidth}
\begin{center}
$2B$
\end{center}
\end{minipage}&
\begin{minipage}[m]{.11\linewidth}
\begin{center}
$y$
\end{center}
\end{minipage}
&
\begin{minipage}[m]{.11\linewidth}
\begin{center}
$w^2$
\end{center}
\end{minipage}&
\begin{minipage}[m]{.18\linewidth}
\begin{center}

\end{center}
\end{minipage}\\
\hline
\begin{minipage}[m]{.28\linewidth}
 $O_yO_w=2\times\frac{1}{2}(1_x,1_z,1_t)$ $\nef$

\end{minipage}&
\begin{minipage}[m]{.04\linewidth}
\begin{center}
$-$
\end{center}
\end{minipage}&
\begin{minipage}[m]{.11\linewidth}
\begin{center}
$9B+4E$
\end{center}
\end{minipage}&
\begin{minipage}[m]{.11\linewidth}
\begin{center}
$xy^4$, $y^2z$, $t$
\end{center}
\end{minipage}
&
\begin{minipage}[m]{.11\linewidth}
\begin{center}
$xy^4$, $y^2z$, $t$
\end{center}
\end{minipage}&
\begin{minipage}[m]{.18\linewidth}
\begin{center}

\end{center}
\end{minipage}\\

\hline

\end{longtable}
\end{center}

\begin{Note}

\item 
For the singular point $O_z$, the $1$-cycle $\Gamma$ is
irreducible due to the monomials $w^2$ and $zt^3$.

\item For the singular points of type $\frac{1}{2}(1,1,1)$,
 we consider the
linear system generated by $xy^4$, $y^2z$ and $t$ on $X_{32}$. Its
base curve is defined by $y=t=0$.  The curve defined by $y=t=0$
passes though no singular point of type $\frac{1}{2}(1,1,1)$.
Consequently, the divisor $T$ is nef.
\end{Note}



\begin{center}
\begin{longtable}{|l|c|c|c|c|c|}
\hline
\multicolumn{6}{|l|}{\underline{\textbf{No. 78}}: $X_{32}\subset\mathbb{P}(1,4,5,7,16)$\hfill $A^3=1/70$}\\
\multicolumn{6}{|l|}{
\begin{minipage}[m]{.86\linewidth}
\vspace*{1.2mm} $w^2+yt^4+z^5t+wf_{16}(x,y,z,t)+f_{32}(x,y,z,t)$
\vspace*{1.2mm}
\end{minipage}
}\\
\hline \hline
\begin{minipage}[m]{.28\linewidth}
\begin{center}
Singularity
\end{center}
\end{minipage}&
\begin{minipage}[m]{.04\linewidth}
\begin{center}
$B^3$
\end{center}
\end{minipage}&
\begin{minipage}[m]{.11\linewidth}
\begin{center}
Linear

system
\end{center}
\end{minipage}&
\begin{minipage}[m]{.11\linewidth}
\begin{center}
Surface $T$
\end{center}
\end{minipage}&
\begin{minipage}[m]{.11\linewidth}
\begin{center}
\vspace*{1mm}
 \vorder
\vspace*{1mm}
\end{center}
\end{minipage}&
\begin{minipage}[m]{.18\linewidth}
\begin{center}
Condition
\end{center}
\end{minipage}\\
\hline
\begin{minipage}[m]{.28\linewidth}

$O_t=\frac{1}{7}(1_x,5_z,2_w)$ \boundary

\end{minipage}&
\begin{minipage}[m]{.04\linewidth}
\begin{center}
$0$
\end{center}
\end{minipage}&
\begin{minipage}[m]{.11\linewidth}
\begin{center}
$4B$
\end{center}
\end{minipage}&
\begin{minipage}[m]{.11\linewidth}
\begin{center}
$y$
\end{center}
\end{minipage}
&
\begin{minipage}[m]{.11\linewidth}
\begin{center}
$w^2$
\end{center}
\end{minipage}&
\begin{minipage}[m]{.18\linewidth}
\begin{center}

\end{center}
\end{minipage}\\
\hline
\begin{minipage}[m]{.28\linewidth}

$O_z=\frac{1}{5}(1_x,4_y,1_w)$ \boundary

\end{minipage}&
\begin{minipage}[m]{.04\linewidth}
\begin{center}
$-$
\end{center}
\end{minipage}&
\begin{minipage}[m]{.11\linewidth}
\begin{center}
$4B$
\end{center}
\end{minipage}&
\begin{minipage}[m]{.11\linewidth}
\begin{center}
$y$
\end{center}
\end{minipage}
&
\begin{minipage}[m]{.11\linewidth}
\begin{center}
$y$
\end{center}
\end{minipage}&
\begin{minipage}[m]{.18\linewidth}
\begin{center}

\end{center}
\end{minipage}\\

\hline
\begin{minipage}[m]{.28\linewidth}

$O_yO_w=2\times\frac{1}{4}(1_x,1_z,3_t)$ $\nef$

\end{minipage}&
\begin{minipage}[m]{.04\linewidth}
\begin{center}
$-$
\end{center}
\end{minipage}&
\begin{minipage}[m]{.11\linewidth}
\begin{center}
$5B+E$
\end{center}
\end{minipage}&
\begin{minipage}[m]{.11\linewidth}
\begin{center}
$xy$, $z$
\end{center}
\end{minipage}
&
\begin{minipage}[m]{.11\linewidth}
\begin{center}
$xy$, $z$
\end{center}
\end{minipage}&
\begin{minipage}[m]{.18\linewidth}
\begin{center}

\end{center}
\end{minipage}\\

\hline

\end{longtable}
\end{center}

\begin{Note}

\item 
For the singular points other than those of type
$\frac{1}{4}(1,1,3)$, the $1$-cycles $\Gamma$ are always
irreducible due to the  monomials  $w^2$ and $z^5t$.

\item For the singular points of type $\frac{1}{4}(1,1,3)$,
 we consider the
linear system generated by $xy$ and $z$ on $X_{32}$. Its base
curves are defined by $x=z=0$ and $y=z=0$.  The curve defined by
$y=z=0$ passes though no singular point of type $\frac{1}{4}(1,1,
3)$.  The curve   defined by $x=z=0$   is irreducible because of
the monomials $w^2$ and $yt^4$.  Its proper transform is the
$1$-cycle defined by  $(5B+E)\cdot B$. Therefore, the divisor $T$
is nef since  $(5B+E)^2\cdot B>0$.
\end{Note}


\begin{center}
\begin{longtable}{|l|c|c|c|c|c|}
\hline
\multicolumn{6}{|l|}{\textbf{No. 79}: $X_{33}\subset\mathbb{P}(1,3,5,11,14)$\hfill $A^3=1/70$}\\
\multicolumn{6}{|l|}{
\begin{minipage}[m]{.86\linewidth}
\vspace*{1.2mm}

$zw^2+t^3+yz^6+y^{11}+wf_{19}(x,y,z,t)+f_{33}(x,y,z,t)$
\vspace*{1.2mm}
\end{minipage}
}\\

\hline \hline
\begin{minipage}[m]{.28\linewidth}
\begin{center}
Singularity
\end{center}
\end{minipage}&
\begin{minipage}[m]{.04\linewidth}
\begin{center}
$B^3$
\end{center}
\end{minipage}&
\begin{minipage}[m]{.11\linewidth}
\begin{center}
Linear

system
\end{center}
\end{minipage}&
\begin{minipage}[m]{.11\linewidth}
\begin{center}
Surface $T$
\end{center}
\end{minipage}&
\begin{minipage}[m]{.11\linewidth}
\begin{center}
\vspace*{1mm}
 \vorder
\vspace*{1mm}
\end{center}
\end{minipage}&
\begin{minipage}[m]{.18\linewidth}
\begin{center}
Condition
\end{center}
\end{minipage}\\
\hline
\begin{minipage}[m]{.28\linewidth}

$O_w=\frac{1}{14}(1,3,11)$ \quadratic

\end{minipage}&
\multicolumn{4}{|l|}{\begin{minipage}[m]{.37\linewidth}
\begin{center}
$zw^2$
\end{center}
\end{minipage}}&
\begin{minipage}[m]{.18\linewidth}
\begin{center}

\end{center}
\end{minipage}\\

\hline
\begin{minipage}[m]{.28\linewidth}

$O_z=\frac{1}{5}(1_x,1_t,4_w)$ \boundary

\end{minipage}&
\begin{minipage}[m]{.04\linewidth}
\begin{center}
$-$
\end{center}
\end{minipage}&
\begin{minipage}[m]{.11\linewidth}
\begin{center}
$3B$
\end{center}
\end{minipage}&
\begin{minipage}[m]{.11\linewidth}
\begin{center}
$y$
\end{center}
\end{minipage}
&
\begin{minipage}[m]{.11\linewidth}
\begin{center}
$t^3$
\end{center}
\end{minipage}&
\begin{minipage}[m]{.18\linewidth}
\begin{center}

\end{center}
\end{minipage}\\
\hline
\end{longtable}
\end{center}

\begin{Note}

\item 
The 1-cycle $\Gamma$ for the singular point $O_z$ is irreducible
because of the monomials $zw^2$ and $t^3$.
\end{Note}




\begin{center}
\begin{longtable}{|l|c|c|c|c|c|}
\hline
\multicolumn{6}{|l|}{\underline{\textbf{No. 80}}: $X_{34}\subset\mathbb{P}(1,3,4,10,17)$\hfill $A^3=1/60$}\\
\multicolumn{6}{|l|}{
\begin{minipage}[m]{.86\linewidth}
\vspace*{1.2mm} $w^2+zt^3+z^6t+y^8(a_1t+a_2y^2z+a_3xy^4)
+wf_{17}(x,y,z,t)+t^2g_{14}(x,y,z)+tg_{24}(x,y,z)+g_{34}(x,y,z)$

\vspace*{1.2mm}
\end{minipage}
}\\
\hline \hline
\begin{minipage}[m]{.28\linewidth}
\begin{center}
Singularity
\end{center}
\end{minipage}&
\begin{minipage}[m]{.04\linewidth}
\begin{center}
$B^3$
\end{center}
\end{minipage}&
\begin{minipage}[m]{.11\linewidth}
\begin{center}
Linear

system
\end{center}
\end{minipage}&
\begin{minipage}[m]{.11\linewidth}
\begin{center}
Surface $T$
\end{center}
\end{minipage}&
\begin{minipage}[m]{.11\linewidth}
\begin{center}
\vspace*{1mm}
 \vorder
\vspace*{1mm}
\end{center}
\end{minipage}&
\begin{minipage}[m]{.18\linewidth}
\begin{center}
Condition
\end{center}
\end{minipage}\\
\hline
\begin{minipage}[m]{.28\linewidth}

$O_t=\frac{1}{10}(1_x,3_y,7_w)$ $\positive$

\end{minipage}&
\begin{minipage}[m]{.04\linewidth}
\begin{center}
$+$
\end{center}
\end{minipage}&
\begin{minipage}[m]{.11\linewidth}
\begin{center}
$4B-E$
\end{center}
\end{minipage}&
\begin{minipage}[m]{.11\linewidth}
\begin{center}
$z$
\end{center}
\end{minipage}
&
\begin{minipage}[m]{.11\linewidth}
\begin{center}
$w^2$
\end{center}
\end{minipage}&
\begin{minipage}[m]{.18\linewidth}
\begin{center}

\end{center}
\end{minipage}\\

\hline
\begin{minipage}[m]{.28\linewidth}

$O_z=\frac{1}{4}(1_x,3_y,1_w)$ \boundary

\end{minipage}&
\begin{minipage}[m]{.04\linewidth}
\begin{center}
$-$
\end{center}
\end{minipage}&
\begin{minipage}[m]{.11\linewidth}
\begin{center}
$3B$
\end{center}
\end{minipage}&
\begin{minipage}[m]{.11\linewidth}
\begin{center}
$y$
\end{center}
\end{minipage}
&
\begin{minipage}[m]{.11\linewidth}
\begin{center}
$y$
\end{center}
\end{minipage}&
\begin{minipage}[m]{.18\linewidth}
\begin{center}

\end{center}
\end{minipage}\\
\hline
\begin{minipage}[m]{.28\linewidth}

$O_y=\frac{1}{3}(1_x,1_z,2_w)$ \nef

\end{minipage}&
\begin{minipage}[m]{.04\linewidth}
\begin{center}
$-$
\end{center}
\end{minipage}&
\begin{minipage}[m]{.11\linewidth}
\begin{center}
$4B+E$
\end{center}
\end{minipage}&
\begin{minipage}[m]{.11\linewidth}
\begin{center}
$z$
\end{center}
\end{minipage}
&
\begin{minipage}[m]{.11\linewidth}
\begin{center}
$z$
\end{center}
\end{minipage}&
\begin{minipage}[m]{.18\linewidth}
\begin{center}
$a_1\ne 0$
\end{center}
\end{minipage}\\
\hline
\begin{minipage}[m]{.28\linewidth}

$O_y=\frac{1}{3}(1_x,1_t,2_w)$ \boundary

\end{minipage}&
\begin{minipage}[m]{.04\linewidth}
\begin{center}
$-$
\end{center}
\end{minipage}&
\begin{minipage}[m]{.11\linewidth}
\begin{center}
$4B$
\end{center}
\end{minipage}&
\begin{minipage}[m]{.11\linewidth}
\begin{center}
$z$
\end{center}
\end{minipage}
&
\begin{minipage}[m]{.11\linewidth}
\begin{center}
$w^2$
\end{center}
\end{minipage}&
\begin{minipage}[m]{.18\linewidth}
\begin{center}
$a_1=0$, $a_2\ne 0$
\end{center}
\end{minipage}\\
\hline
\begin{minipage}[m]{.28\linewidth}

$O_y=\frac{1}{3}(1_z,1_t,2_w)$ \boundary

\end{minipage}&
\begin{minipage}[m]{.04\linewidth}
\begin{center}
$-$
\end{center}
\end{minipage}&
\begin{minipage}[m]{.11\linewidth}
\begin{center}
$4B+E$
\end{center}
\end{minipage}&
\begin{minipage}[m]{.11\linewidth}
\begin{center}
$z$
\end{center}
\end{minipage}
&
\begin{minipage}[m]{.11\linewidth}
\begin{center}
$z$
\end{center}
\end{minipage}&
\begin{minipage}[m]{.18\linewidth}
\begin{center}
$a_1=a_2=0$
\end{center}
\end{minipage}\\

\hline
\begin{minipage}[m]{.28\linewidth}

$O_zO_t=1\times\frac{1}{2}(1_x,1_y,1_w)$ \boundary

\end{minipage}&
\begin{minipage}[m]{.04\linewidth}
\begin{center}
$-$
\end{center}
\end{minipage}&
\begin{minipage}[m]{.11\linewidth}
\begin{center}
$3B+E$
\end{center}
\end{minipage}&
\begin{minipage}[m]{.11\linewidth}
\begin{center}
$y$
\end{center}
\end{minipage}
&
\begin{minipage}[m]{.11\linewidth}
\begin{center}
$y$
\end{center}
\end{minipage}&
\begin{minipage}[m]{.18\linewidth}
\begin{center}

\end{center}
\end{minipage}\\

\hline
\end{longtable}
\end{center}

\begin{Note}

\item 
For each of the  singular points to which the method $\boundary$
is applied, the $1$-cycle $\Gamma$ is always irreducible even
though it is possibly non-reduced.

\item For the singular point $O_y$ with $a_1\ne 0$,
 we consider the
linear system generated by $xy$ and $z$ on $X_{34}$. Its base
curves are defined by $x=z=0$ and $y=z=0$.  The curve defined by
$y=z=0$ does not pass through the point $O_y$.  The curve defined
by $x=z=0$   is irreducible because of the monomials $w^2$ and
$y^8t$.  Its proper transform is the $1$-cycle defined by
$(4B+E)\cdot B$, and hence it intersects $T$ positively.
Consequently, the divisor $T$ is nef.
\end{Note}


\begin{center}
\begin{longtable}{|l|c|c|c|c|c|}
\hline
\multicolumn{6}{|l|}{\underline{\textbf{No. 81}}: $X_{34}\subset\mathbb{P}(1,4,6,7,17)$ \hfill $A^3=1/84$}\\
\multicolumn{6}{|l|}{
\begin{minipage}[m]{.86\linewidth}
\vspace*{1.2mm}
$w^2+zt^4+yz^5+y^7z+wf_{17}(x,y,z,t)+f_{34}(x,y,z,t)$
\vspace*{1.2mm}
\end{minipage}
}\\
\hline \hline
\begin{minipage}[m]{.28\linewidth}
\begin{center}
Singularity
\end{center}
\end{minipage}&
\begin{minipage}[m]{.04\linewidth}
\begin{center}
$B^3$
\end{center}
\end{minipage}&
\begin{minipage}[m]{.11\linewidth}
\begin{center}
Linear

system
\end{center}
\end{minipage}&
\begin{minipage}[m]{.11\linewidth}
\begin{center}
Surface $T$
\end{center}
\end{minipage}&
\begin{minipage}[m]{.11\linewidth}
\begin{center}
\vspace*{1mm}
 \vorder
\vspace*{1mm}
\end{center}
\end{minipage}&
\begin{minipage}[m]{.18\linewidth}
\begin{center}
Condition
\end{center}
\end{minipage}\\
\hline
\begin{minipage}[m]{.28\linewidth}

$O_t=\frac{1}{7}(1_x,4_y,3_w)$ \boundary

\end{minipage}&
\begin{minipage}[m]{.04\linewidth}
\begin{center}
$0$
\end{center}
\end{minipage}&
\begin{minipage}[m]{.11\linewidth}
\begin{center}
$6B$
\end{center}
\end{minipage}&
\begin{minipage}[m]{.11\linewidth}
\begin{center}
$z$
\end{center}
\end{minipage}&
\begin{minipage}[m]{.11\linewidth}
\begin{center}
$w^2$
\end{center}
\end{minipage}
&
\begin{minipage}[m]{.18\linewidth}
\begin{center}

\end{center}
\end{minipage}\\
\hline
\begin{minipage}[m]{.28\linewidth}

$O_z=\frac{1}{6}(1_x,1_t,5_w)$ \boundary

\end{minipage}&
\begin{minipage}[m]{.04\linewidth}
\begin{center}
$-$
\end{center}
\end{minipage}&
\begin{minipage}[m]{.11\linewidth}
\begin{center}
$4B$
\end{center}
\end{minipage}&
\begin{minipage}[m]{.11\linewidth}
\begin{center}
$y$
\end{center}
\end{minipage}&
\begin{minipage}[m]{.11\linewidth}
\begin{center}
$zt^4$
\end{center}
\end{minipage}&
\begin{minipage}[m]{.18\linewidth}
\begin{center}

\end{center}
\end{minipage}\\
\hline
\begin{minipage}[m]{.28\linewidth}

$O_y=\frac{1}{4}(1_x,3_t,1_w)$ \boundary

\end{minipage}&
\begin{minipage}[m]{.04\linewidth}
\begin{center}
$-$
\end{center}
\end{minipage}&
\begin{minipage}[m]{.11\linewidth}
\begin{center}
$7B+E$
\end{center}
\end{minipage}&
\begin{minipage}[m]{.11\linewidth}
\begin{center}
$t$
\end{center}
\end{minipage}&
\begin{minipage}[m]{.11\linewidth}
\begin{center}
$t$
\end{center}
\end{minipage}&
\begin{minipage}[m]{.18\linewidth}
\begin{center}

\end{center}
\end{minipage}\\
\hline
\begin{minipage}[m]{.28\linewidth}
 $O_yO_z=2\times\frac{1}{2}(1_x,1_t,1_w)$ $\nef$

\end{minipage}&
\begin{minipage}[m]{.04\linewidth}
\begin{center}
$-$
\end{center}
\end{minipage}&
\begin{minipage}[m]{.11\linewidth}
\begin{center}
$7B+3E$
\end{center}
\end{minipage}&
\begin{minipage}[m]{.11\linewidth}
\begin{center}
$xz$, $t$
\end{center}
\end{minipage}&
\begin{minipage}[m]{.11\linewidth}
\begin{center}
$xz$, $t$
\end{center}
\end{minipage}&
\begin{minipage}[m]{.18\linewidth}
\begin{center}

\end{center}
\end{minipage}\\
\hline
\end{longtable}
\end{center}

\begin{Note}

\item 
The $1$-cycle $\Gamma$ for the singular point $O_t$ is irreducible
due to the monomials $w^2$ and $y^5t^2$ even though it can be
non-reduced.

\item For the singular point $O_z$, the $1$-cycle $\Gamma$ is
irreducible due to the monomials $w^2$ and $zt^4$.

\item For the singular point $O_y$, the $1$-cycle $\Gamma$ is
irreducible due to the monomials $w^2$, $yz^5$ and $y^7z$

\item For the singular points of type $\frac{1}{2}(1,1,1)$, we consider
the linear system generated by $xz$ and $t$ on $X_{34}$. Its base
curves are defined by $x=t=0$ and $z=t=0$.  The curve defined by
$z=t=0$ passes though no singular point of type $\frac{1}{2}(1,1,
1)$.  The curve   defined by $x=t=0$   is irreducible due to the
monomials $w^2$, $yz^5$ and $y^7z$.  Its proper transform is
equivalent to  the $1$-cycle defined by $(7B+3E)\cdot B$ that
intersects $T$ positively.  Therefore, the divisor $T$ is nef.
\end{Note}


\begin{center}
\begin{longtable}{|l|c|c|c|c|c|}
\hline
\multicolumn{6}{|l|}{\underline{\textbf{No. 82}}: $X_{36}\subset\mathbb{P}(1,1,5,12,18)$\hfill $A^3=1/30$}\\
\multicolumn{6}{|l|}{
\begin{minipage}[m]{.86\linewidth}
\vspace*{1.2mm} $w^2+t^3+yz^7+wf_{18}(x,y,z,t)+f_{36}(x,y,z,t)$
\vspace*{1.2mm}
\end{minipage}
}\\

\hline \hline
\begin{minipage}[m]{.28\linewidth}
\begin{center}
Singularity
\end{center}
\end{minipage}&
\begin{minipage}[m]{.04\linewidth}
\begin{center}
$B^3$
\end{center}
\end{minipage}&
\begin{minipage}[m]{.11\linewidth}
\begin{center}
Linear

system
\end{center}
\end{minipage}&
\begin{minipage}[m]{.11\linewidth}
\begin{center}
Surface $T$
\end{center}
\end{minipage}&
\begin{minipage}[m]{.11\linewidth}
\begin{center}
\vspace*{1mm}
 \vorder
\vspace*{1mm}
\end{center}
\end{minipage}&
\begin{minipage}[m]{.18\linewidth}
\begin{center}
Condition
\end{center}
\end{minipage}\\
\hline
\begin{minipage}[m]{.28\linewidth}

$O_z=\frac{1}{5}(1_x,2_t,3_w)$ $\family$

\end{minipage}&
\begin{minipage}[m]{.04\linewidth}
\begin{center}
$0$
\end{center}
\end{minipage}&
\begin{minipage}[m]{.11\linewidth}
\begin{center}
$B-E$
\end{center}
\end{minipage}&
\begin{minipage}[m]{.11\linewidth}
\begin{center}
$y$
\end{center}
\end{minipage}
&
\begin{minipage}[m]{.11\linewidth}
\begin{center}
$w^2$
\end{center}
\end{minipage}&
\begin{minipage}[m]{.18\linewidth}
\begin{center}

\end{center}
\end{minipage}\\
\hline
\begin{minipage}[m]{.28\linewidth}

$O_tO_w=1\times\frac{1}{6}(1_x,1_y,5_z)$ \boundary

\end{minipage}&
\begin{minipage}[m]{.04\linewidth}
\begin{center}
$0$
\end{center}
\end{minipage}&
\begin{minipage}[m]{.11\linewidth}
\begin{center}
$B$
\end{center}
\end{minipage}&
\begin{minipage}[m]{.11\linewidth}
\begin{center}
$y$
\end{center}
\end{minipage}
&
\begin{minipage}[m]{.11\linewidth}
\begin{center}
$y$
\end{center}
\end{minipage}&
\begin{minipage}[m]{.18\linewidth}
\begin{center}

\end{center}
\end{minipage}\\

\hline

\end{longtable}
\end{center}

\begin{Note}

\item 
For the singular point $O_z$, let $C_{\lambda}$ be the curve on
the surface $S_y$  cut by $t=\lambda x^{12}$ for a  general
complex number  $\lambda$. Then
$$
-K_{Y}\cdot \tilde{C}_{\lambda}=(B-E)(12B+2E)B=0.\\%
$$
If the curve $\tilde{C}_{\lambda}$ is reducible, it consists of
two irreducible components. Because these two components are
symmetric with respect to the biregular quadratic involution of
$X_{36}$, they must be numerically equivalent to each other.
Therefore, each component of $\tilde{C}_{\lambda}$ intersects
$-K_Y$ trivially.

\item For the singular point of type $\frac{1}{6}(1,1,5)$, the $1$-cycle
$\Gamma$ is irreducible due to $w^2$ and $t^3$.
\end{Note}



\begin{center}
\begin{longtable}{|l|c|c|c|c|c|}
\hline
\multicolumn{6}{|l|}{\underline{\textbf{No. 83}}: $X_{36}\subset\mathbb{P}(1,3,4,11,18)$\hfill $A^3=1/66$}\\
\multicolumn{6}{|l|}{
\begin{minipage}[m]{.86\linewidth}
\vspace*{1.2mm} $(w-\alpha_1 y^6)(w-\alpha_2
y^6)+yt^3+z^9+wf_{18}(x,y,z,t)+f_{36}(x,y,z,t)$ \vspace*{1.2mm}
\end{minipage}
}\\
\hline \hline
\begin{minipage}[m]{.28\linewidth}
\begin{center}
Singularity
\end{center}
\end{minipage}&
\begin{minipage}[m]{.04\linewidth}
\begin{center}
$B^3$
\end{center}
\end{minipage}&
\begin{minipage}[m]{.11\linewidth}
\begin{center}
Linear

system
\end{center}
\end{minipage}&
\begin{minipage}[m]{.11\linewidth}
\begin{center}
Surface $T$
\end{center}
\end{minipage}&
\begin{minipage}[m]{.11\linewidth}
\begin{center}
\vspace*{1mm}
 \vorder
\vspace*{1mm}
\end{center}
\end{minipage}&
\begin{minipage}[m]{.18\linewidth}
\begin{center}
Condition
\end{center}
\end{minipage}\\
\hline
\begin{minipage}[m]{.28\linewidth}

$O_t=\frac{1}{11}(1_x,4_z,7_w)$ $\positive$

\end{minipage}&
\begin{minipage}[m]{.04\linewidth}
\begin{center}
$+$
\end{center}
\end{minipage}&
\begin{minipage}[m]{.11\linewidth}
\begin{center}
$3B-E$
\end{center}
\end{minipage}&
\begin{minipage}[m]{.11\linewidth}
\begin{center}
$y$
\end{center}
\end{minipage}
&
\begin{minipage}[m]{.11\linewidth}
\begin{center}
$w^2$
\end{center}
\end{minipage}&
\begin{minipage}[m]{.18\linewidth}
\begin{center}

\end{center}
\end{minipage}\\

\hline
\begin{minipage}[m]{.28\linewidth}

$O_zO_w=1\times\frac{1}{2}(1_x,1_y,1_t)$ \boundary

\end{minipage}&
\begin{minipage}[m]{.04\linewidth}
\begin{center}
$-$
\end{center}
\end{minipage}&
\begin{minipage}[m]{.11\linewidth}
\begin{center}
$3B+E$
\end{center}
\end{minipage}&
\begin{minipage}[m]{.11\linewidth}
\begin{center}
$y$
\end{center}
\end{minipage}
&
\begin{minipage}[m]{.11\linewidth}
\begin{center}
$y$
\end{center}
\end{minipage}&
\begin{minipage}[m]{.18\linewidth}
\begin{center}

\end{center}
\end{minipage}\\
\hline
\begin{minipage}[m]{.28\linewidth}

$O_yO_w=2\times\frac{1}{3}(1_x,1_z,2_t)$ \boundary

\end{minipage}&
\begin{minipage}[m]{.04\linewidth}
\begin{center}
$-$
\end{center}
\end{minipage}&
\begin{minipage}[m]{.11\linewidth}
\begin{center}
$18B+4E$
\end{center}
\end{minipage}&
\begin{minipage}[m]{.11\linewidth}
\begin{center}
$w-\alpha_i y^6$
\end{center}
\end{minipage}
&
\begin{minipage}[m]{.11\linewidth}
\begin{center}
$yt^3$
\end{center}
\end{minipage}&
\begin{minipage}[m]{.18\linewidth}
\begin{center}

\end{center}
\end{minipage}\\

\hline

\end{longtable}
\end{center}

\begin{Note}

\item 
For each of the singular points corresponding to the method
$\boundary$, the $1$-cycle $\Gamma$ is irreducible since we have
the monomials $w^2$, $yt^3$, and $z^9$.
\end{Note}


\begin{center}
\begin{longtable}{|l|c|c|c|c|c|}
\hline
\multicolumn{6}{|l|}{\underline{\textbf{No. 84}}: $X_{36}\subset\mathbb{P}(1,7,8,9,12)$ \hfill $A^3=1/168$}\\
\multicolumn{6}{|l|}{
\begin{minipage}[m]{.86\linewidth}
\vspace*{1.2mm}
$w^3+t^4+z^3w+y^4(a_1z+a_2xy)+w^2f_{12}(x,y,z,t)+wf_{24}(x,y,z,t)+f_{36}(x,y,z,t)$
\vspace*{1.2mm}
\end{minipage}
}\\
\hline \hline
\begin{minipage}[m]{.28\linewidth}
\begin{center}
Singularity
\end{center}
\end{minipage}&
\begin{minipage}[m]{.04\linewidth}
\begin{center}
$B^3$
\end{center}
\end{minipage}&
\begin{minipage}[m]{.11\linewidth}
\begin{center}
Linear

system
\end{center}
\end{minipage}&
\begin{minipage}[m]{.11\linewidth}
\begin{center}
Surface $T$
\end{center}
\end{minipage}&
\begin{minipage}[m]{.11\linewidth}
\begin{center}
\vspace*{1mm}
 \vorder
\vspace*{1mm}
\end{center}
\end{minipage}&
\begin{minipage}[m]{.18\linewidth}
\begin{center}
Condition
\end{center}
\end{minipage}\\
\hline
\begin{minipage}[m]{.28\linewidth}

$O_z=\frac{1}{8}(1_x,7_y,1_t)$ \boundary

\end{minipage}&
\begin{minipage}[m]{.04\linewidth}
\begin{center}
$-$
\end{center}
\end{minipage}&
\begin{minipage}[m]{.11\linewidth}
\begin{center}
$7B$
\end{center}
\end{minipage}&
\begin{minipage}[m]{.11\linewidth}
\begin{center}
$y$
\end{center}
\end{minipage}&
\begin{minipage}[m]{.11\linewidth}
\begin{center}
$y$
\end{center}
\end{minipage}
&
\begin{minipage}[m]{.18\linewidth}
\begin{center}

\end{center}
\end{minipage}\\

\hline
\begin{minipage}[m]{.28\linewidth}

$O_y=\frac{1}{7}(1_x,2_t,5_w)$ \boundary

\end{minipage}&
\begin{minipage}[m]{.04\linewidth}
\begin{center}
$-$
\end{center}
\end{minipage}&
\begin{minipage}[m]{.11\linewidth}
\begin{center}
$8B$
\end{center}
\end{minipage}&
\begin{minipage}[m]{.11\linewidth}
\begin{center}
$z$
\end{center}
\end{minipage}&
\begin{minipage}[m]{.11\linewidth}
\begin{center}
$t^4$
\end{center}
\end{minipage}&
\begin{minipage}[m]{.18\linewidth}
\begin{center}
$a_1\ne 0$
\end{center}
\end{minipage}\\
\hline
\begin{minipage}[m]{.28\linewidth}

$O_y=\frac{1}{7}(1_z,2_t,5_w)$ \boundary

\end{minipage}&
\begin{minipage}[m]{.04\linewidth}
\begin{center}
$-$
\end{center}
\end{minipage}&
\begin{minipage}[m]{.11\linewidth}
\begin{center}
$12B+E$
\end{center}
\end{minipage}&
\begin{minipage}[m]{.11\linewidth}
\begin{center}
$w$
\end{center}
\end{minipage}&
\begin{minipage}[m]{.11\linewidth}
\begin{center}
$w$
\end{center}
\end{minipage}&
\begin{minipage}[m]{.18\linewidth}
\begin{center}
$a_1 =0 $
\end{center}
\end{minipage}\\
\hline
\begin{minipage}[m]{.28\linewidth}

$O_tO_w=1\times\frac{1}{3}(1_x,1_y,2_z)$ \boundary

\end{minipage}&
\begin{minipage}[m]{.04\linewidth}
\begin{center}
$-$
\end{center}
\end{minipage}&
\begin{minipage}[m]{.11\linewidth}
\begin{center}
$8B+2E$
\end{center}
\end{minipage}&
\begin{minipage}[m]{.11\linewidth}
\begin{center}
$z$
\end{center}
\end{minipage}&
\begin{minipage}[m]{.11\linewidth}
\begin{center}
$z$
\end{center}
\end{minipage}&
\begin{minipage}[m]{.18\linewidth}
\begin{center}

\end{center}
\end{minipage}\\
\hline
\begin{minipage}[m]{.28\linewidth}

$O_zO_w=1\times\frac{1}{4}(1_x,3_y,1_t)$ \boundary

\end{minipage}&
\begin{minipage}[m]{.04\linewidth}
\begin{center}
$-$
\end{center}
\end{minipage}&
\begin{minipage}[m]{.11\linewidth}
\begin{center}
$7B+E$
\end{center}
\end{minipage}&
\begin{minipage}[m]{.11\linewidth}
\begin{center}
$y$
\end{center}
\end{minipage}&
\begin{minipage}[m]{.11\linewidth}
\begin{center}
$y$
\end{center}
\end{minipage}&
\begin{minipage}[m]{.18\linewidth}
\begin{center}

\end{center}
\end{minipage}\\
\hline
\end{longtable}
\end{center}

\begin{Note}

\item 
For each singular point the $1$-cycle $\Gamma$ is always
irreducible because of the monomials $w^3$ and $t^4$. In
particular, the intersection $\Gamma$ for the singular point $O_y$
with $a_1=0$ is irreducible even though it is non-reduced.
\end{Note}

\begin{center}
\begin{longtable}{|l|c|c|c|c|c|}
\hline
\multicolumn{6}{|l|}{\underline{\textbf{No. 85}}: $X_{38}\subset\mathbb{P}(1,3,5,11,19)$ \hfill $A^3=2/165$}\\
\multicolumn{6}{|l|}{
\begin{minipage}[m]{.86\linewidth}
\vspace*{1.2mm}
$w^2+zt^3+yz^7+y^9(a_1t+a_2y^2z)+wf_{19}(x,y,z,t)+f_{38}(x,y,z,t)$
\vspace*{1.2mm}
\end{minipage}
}\\
\hline \hline
\begin{minipage}[m]{.28\linewidth}
\begin{center}
Singularity
\end{center}
\end{minipage}&
\begin{minipage}[m]{.04\linewidth}
\begin{center}
$B^3$
\end{center}
\end{minipage}&
\begin{minipage}[m]{.11\linewidth}
\begin{center}
Linear

system
\end{center}
\end{minipage}&
\begin{minipage}[m]{.11\linewidth}
\begin{center}
Surface $T$
\end{center}
\end{minipage}&
\begin{minipage}[m]{.11\linewidth}
\begin{center}
\vspace*{1mm}
 \vorder
\vspace*{1mm}
\end{center}
\end{minipage}&
\begin{minipage}[m]{.18\linewidth}
\begin{center}
Condition
\end{center}
\end{minipage}\\
\hline
\begin{minipage}[m]{.28\linewidth}

$O_t=\frac{1}{11}(1_x,3_y,8_w)$ $\positive$

\end{minipage}&
\begin{minipage}[m]{.04\linewidth}
\begin{center}
$+$
\end{center}
\end{minipage}&
\begin{minipage}[m]{.11\linewidth}
\begin{center}
$5B-E$
\end{center}
\end{minipage}&
\begin{minipage}[m]{.11\linewidth}
\begin{center}
$z$
\end{center}
\end{minipage}&
\begin{minipage}[m]{.11\linewidth}
\begin{center}
$w^2$
\end{center}
\end{minipage}
&
\begin{minipage}[m]{.18\linewidth}
\begin{center}

\end{center}
\end{minipage}\\
\hline
\begin{minipage}[m]{.28\linewidth}

$O_z=\frac{1}{5}(1_x,1_t,4_w)$ \boundary

\end{minipage}&
\begin{minipage}[m]{.04\linewidth}
\begin{center}
$-$
\end{center}
\end{minipage}&
\begin{minipage}[m]{.11\linewidth}
\begin{center}
$3B$
\end{center}
\end{minipage}&
\begin{minipage}[m]{.11\linewidth}
\begin{center}
$y$
\end{center}
\end{minipage}&
\begin{minipage}[m]{.11\linewidth}
\begin{center}
$zt^3$
\end{center}
\end{minipage}&
\begin{minipage}[m]{.18\linewidth}
\begin{center}

\end{center}
\end{minipage}\\
\hline
\begin{minipage}[m]{.28\linewidth}

$O_y=\frac{1}{3}(1_x,2_z,1_w)$ \boundary

\end{minipage}&
\begin{minipage}[m]{.04\linewidth}
\begin{center}
$-$
\end{center}
\end{minipage}&
\begin{minipage}[m]{.11\linewidth}
\begin{center}
$5B+E$
\end{center}
\end{minipage}&
\begin{minipage}[m]{.11\linewidth}
\begin{center}
$z$
\end{center}
\end{minipage}&
\begin{minipage}[m]{.11\linewidth}
\begin{center}
$z$
\end{center}
\end{minipage}&
\begin{minipage}[m]{.18\linewidth}
\begin{center}
$a_1\ne 0$
\end{center}
\end{minipage}\\
\hline
\begin{minipage}[m]{.28\linewidth}

$O_y=\frac{1}{3}(1_x,2_t,1_w)$ \boundary

\end{minipage}&
\begin{minipage}[m]{.04\linewidth}
\begin{center}
$-$
\end{center}
\end{minipage}&
\begin{minipage}[m]{.11\linewidth}
\begin{center}
$5B+E$
\end{center}
\end{minipage}&
\begin{minipage}[m]{.11\linewidth}
\begin{center}
$z$
\end{center}
\end{minipage}&
\begin{minipage}[m]{.11\linewidth}
\begin{center}
$w^2$
\end{center}
\end{minipage}&
\begin{minipage}[m]{.18\linewidth}
\begin{center}
$a_1=0$
\end{center}
\end{minipage}\\

\hline
\end{longtable}
\end{center}

\begin{Note}

\item 
For the singular point $O_z$, the $1$-cycle $\Gamma$ is
irreducible because of  the monomials $w^2$ and $zt^3$.

\item For the singular point $O_y$, the $1$-cycle $\Gamma$ is
irreducible because of the monomials $w^2$ and $y^9t$. Note that
in case when $a_1=0$ the $1$-cycle $\Gamma$ is still irreducible
but non-reduced.
\end{Note}


\begin{center}
\begin{longtable}{|l|c|c|c|c|c|}
\hline
\multicolumn{6}{|l|}{\underline{\textbf{No. 86}}: $X_{38}\subset\mathbb{P}(1,5,6,8,19)$\hfill $A^3=1/120$}\\
\multicolumn{6}{|l|}{
\begin{minipage}[m]{.86\linewidth}
\vspace*{1.2mm}
$w^2+zt^4+z^5t+y^6t+wf_{19}(x,y,z,t)+f_{38}(x,y,z,t)$
\vspace*{1.2mm}
\end{minipage}
}\\
\hline \hline
\begin{minipage}[m]{.28\linewidth}
\begin{center}
Singularity
\end{center}
\end{minipage}&
\begin{minipage}[m]{.04\linewidth}
\begin{center}
$B^3$
\end{center}
\end{minipage}&
\begin{minipage}[m]{.11\linewidth}
\begin{center}
Linear

system
\end{center}
\end{minipage}&
\begin{minipage}[m]{.11\linewidth}
\begin{center}
Surface $T$
\end{center}
\end{minipage}&
\begin{minipage}[m]{.11\linewidth}
\begin{center}
\vspace*{1mm}
 \vorder
\vspace*{1mm}
\end{center}
\end{minipage}&
\begin{minipage}[m]{.18\linewidth}
\begin{center}
Condition
\end{center}
\end{minipage}\\
\hline
\begin{minipage}[m]{.28\linewidth}

$O_t=\frac{1}{8}(1_x,5_y,3_w)$ \boundary

\end{minipage}&
\begin{minipage}[m]{.04\linewidth}
\begin{center}
$0$
\end{center}
\end{minipage}&
\begin{minipage}[m]{.11\linewidth}
\begin{center}
$5B$
\end{center}
\end{minipage}&
\begin{minipage}[m]{.11\linewidth}
\begin{center}
$y$
\end{center}
\end{minipage}
&
\begin{minipage}[m]{.11\linewidth}
\begin{center}
$y$
\end{center}
\end{minipage}&
\begin{minipage}[m]{.18\linewidth}
\begin{center}

\end{center}
\end{minipage}\\
\hline
\begin{minipage}[m]{.28\linewidth}

$O_z=\frac{1}{6}(1_x,5_y,1_w)$ \boundary

\end{minipage}&
\begin{minipage}[m]{.04\linewidth}
\begin{center}
$-$
\end{center}
\end{minipage}&
\begin{minipage}[m]{.11\linewidth}
\begin{center}
$5B$
\end{center}
\end{minipage}&
\begin{minipage}[m]{.11\linewidth}
\begin{center}
$y$
\end{center}
\end{minipage}
&
\begin{minipage}[m]{.11\linewidth}
\begin{center}
$y$
\end{center}
\end{minipage}&
\begin{minipage}[m]{.18\linewidth}
\begin{center}

\end{center}
\end{minipage}\\
\hline
\begin{minipage}[m]{.28\linewidth}

$O_y=\frac{1}{5}(1_x,1_z,4_w)$ $\nef$

\end{minipage}&
\begin{minipage}[m]{.04\linewidth}
\begin{center}
$-$
\end{center}
\end{minipage}&
\begin{minipage}[m]{.11\linewidth}
\begin{center}
$6B+E$
\end{center}
\end{minipage}&
\begin{minipage}[m]{.11\linewidth}
\begin{center}
$xy$, $z$
\end{center}
\end{minipage}
&
\begin{minipage}[m]{.11\linewidth}
\begin{center}
$xy$, $z$
\end{center}
\end{minipage}&
\begin{minipage}[m]{.18\linewidth}
\begin{center}

\end{center}
\end{minipage}\\

\hline
\begin{minipage}[m]{.28\linewidth}

$O_zO_t=1\times\frac{1}{2}(1_x,1_y,1_w)$ \boundary

\end{minipage}&
\begin{minipage}[m]{.04\linewidth}
\begin{center}
$-$
\end{center}
\end{minipage}&
\begin{minipage}[m]{.11\linewidth}
\begin{center}
$5B+2E$
\end{center}
\end{minipage}&
\begin{minipage}[m]{.11\linewidth}
\begin{center}
$y$
\end{center}
\end{minipage}
&
\begin{minipage}[m]{.11\linewidth}
\begin{center}
$y$
\end{center}
\end{minipage}&
\begin{minipage}[m]{.18\linewidth}
\begin{center}

\end{center}
\end{minipage}\\

\hline
\end{longtable}
\end{center}

\begin{Note}

\item 
 For the
singular points except the point $O_y$, the $1$-cycles $\Gamma$
are always irreducible because of the monomials $w^2$, $zt^4$ and
$z^5t$.

\item For the singularity $O_y$, we consider the linear system generated
by $xy$ and $z$ on $X_{38}$. Its base curves are defined by
$x=z=0$ and $y=z=0$. The curve defined by $y=z=0$ does not pass
though the singular point $O_y$.  The curve   defined by $x=z=0$
is irreducible because of the monomials $w^2$ and $y^6t$. The
proper transform is equivalent to the $1$-cycle defined by
$(6B+E)\cdot B$ and   $(6B+E)^2\cdot B>0$.  Therefore, the divisor
$T$ is nef.
\end{Note}


\begin{center}
\begin{longtable}{|l|c|c|c|c|c|}
\hline
\multicolumn{6}{|l|}{\underline{\textbf{No. 87}}: $X_{40}\subset\mathbb{P}(1,5,7,8,20)$ \hfill $A^3=1/140$}\\
\multicolumn{6}{|l|}{
\begin{minipage}[m]{.86\linewidth}
\vspace*{1.2mm} $(w-\alpha_1 y^4)(w-\alpha_2
y^4)+t^5+yz^5+wf_{20}(x,y,z,t)+f_{40}(x,y,z,t)$ \vspace*{1.2mm}
\end{minipage}
}\\

\hline \hline
\begin{minipage}[m]{.28\linewidth}
\begin{center}
Singularity
\end{center}
\end{minipage}&
\begin{minipage}[m]{.04\linewidth}
\begin{center}
$B^3$
\end{center}
\end{minipage}&
\begin{minipage}[m]{.11\linewidth}
\begin{center}
Linear

system
\end{center}
\end{minipage}&
\begin{minipage}[m]{.11\linewidth}
\begin{center}
Surface $T$
\end{center}
\end{minipage}&
\begin{minipage}[m]{.11\linewidth}
\begin{center}
\vspace*{1mm}
 \vorder
\vspace*{1mm}
\end{center}
\end{minipage}&
\begin{minipage}[m]{.18\linewidth}
\begin{center}
Condition
\end{center}
\end{minipage}\\
\hline
\begin{minipage}[m]{.28\linewidth}

$O_z=\frac{1}{7}(1_x,1_t,6_w)$ \boundary

\end{minipage}&
\begin{minipage}[m]{.04\linewidth}
\begin{center}
$-$
\end{center}
\end{minipage}&
\begin{minipage}[m]{.11\linewidth}
\begin{center}
$5B$
\end{center}
\end{minipage}&
\begin{minipage}[m]{.11\linewidth}
\begin{center}
$y$
\end{center}
\end{minipage}&
\begin{minipage}[m]{.11\linewidth}
\begin{center}
$t^5$
\end{center}
\end{minipage}
&
\begin{minipage}[m]{.18\linewidth}
\begin{center}

\end{center}
\end{minipage}\\
\hline
\begin{minipage}[m]{.28\linewidth}

$O_tO_w=1\times\frac{1}{4}(1_x,1_y,3_z)$ \boundary

\end{minipage}&
\begin{minipage}[m]{.04\linewidth}
\begin{center}
$-$
\end{center}
\end{minipage}&
\begin{minipage}[m]{.11\linewidth}
\begin{center}
$7B+E$
\end{center}
\end{minipage}&
\begin{minipage}[m]{.11\linewidth}
\begin{center}
$z$
\end{center}
\end{minipage}&
\begin{minipage}[m]{.11\linewidth}
\begin{center}
$z$
\end{center}
\end{minipage}&
\begin{minipage}[m]{.18\linewidth}
\begin{center}

\end{center}
\end{minipage}\\
\hline
\begin{minipage}[m]{.28\linewidth}

$O_yO_w=2\times\frac{1}{5}(1_x,2_z,3_t)$ \nef

\end{minipage}&
\begin{minipage}[m]{.04\linewidth}
\begin{center}
$-$
\end{center}
\end{minipage}&
\begin{minipage}[m]{.11\linewidth}
\begin{center}
\vspace*{1mm} $7B+ E$
\end{center}
\end{minipage}&
\begin{minipage}[m]{.11\linewidth}
\begin{center}
$x^2y$, $z$
\end{center}
\end{minipage}&
\begin{minipage}[m]{.11\linewidth}
\begin{center}
$x^2y$, $z$
\end{center}
\end{minipage}&
\begin{minipage}[m]{.18\linewidth}
\begin{center}

\end{center}
\end{minipage}\\

\hline
\end{longtable}
\end{center}

 \begin{Note}

\item 
 The irreducibility of the $1$-cycle $\Gamma$ can be
immediately checked for each singular point corresponding to the
method $\boundary$ since we have the monomials $w^2$ and $t^5$.

\item For the singular points of type $\frac{1}{5}(1,2,3)$, we consider
the linear system generated by $x^2y$ and $z$ on $X_{40}$. Its
base curves are defined by $x=z=0$ and $y=z=0$. The curve defined
by $y=z=0$ passes though no singular points of type
$\frac{1}{5}(1,2,3)$.  The curve   defined by $x=z=0$   is
irreducible because of the monomials $w^2$ and $t^5$.  Its proper
transform is equivalent to the $1$-cycle defined by $(7B+E)\cdot
B$. Consequently, the divisor $T$ is nef since $(7B+E)^2\cdot
B>0$.
\end{Note}


\begin{center}
\begin{longtable}{|l|c|c|c|c|c|}
\hline
\multicolumn{6}{|l|}{\underline{\textbf{No. 88}}: $X_{42}\subset\mathbb{P}(1,1,6,14,21)$\hfill $A^3=1/42$}\\
\multicolumn{6}{|l|}{
\begin{minipage}[m]{.86\linewidth}
\vspace*{1.2mm} $w^2+t^3+z^7+wf_{21}(x,y,z,t)+f_{42}(x,y,z,t)$
\vspace*{1.2mm}
\end{minipage}
}\\

\hline \hline
\begin{minipage}[m]{.28\linewidth}
\begin{center}
Singularity
\end{center}
\end{minipage}&
\begin{minipage}[m]{.04\linewidth}
\begin{center}
$B^3$
\end{center}
\end{minipage}&
\begin{minipage}[m]{.11\linewidth}
\begin{center}
Linear

system
\end{center}
\end{minipage}&
\begin{minipage}[m]{.11\linewidth}
\begin{center}
Surface $T$
\end{center}
\end{minipage}&
\begin{minipage}[m]{.11\linewidth}
\begin{center}
\vspace*{1mm}
 \vorder
\vspace*{1mm}
\end{center}
\end{minipage}&
\begin{minipage}[m]{.18\linewidth}
\begin{center}
Condition
\end{center}
\end{minipage}\\
\hline
\begin{minipage}[m]{.28\linewidth}

$O_tO_w=1\times\frac{1}{7}(1_x,1_y,6_z)$ \boundary

\end{minipage}&
\begin{minipage}[m]{.04\linewidth}
\begin{center}
$0$
\end{center}
\end{minipage}&
\begin{minipage}[m]{.11\linewidth}
\begin{center}
$B$
\end{center}
\end{minipage}&
\begin{minipage}[m]{.11\linewidth}
\begin{center}
$y$
\end{center}
\end{minipage}
&
\begin{minipage}[m]{.11\linewidth}
\begin{center}
$y$
\end{center}
\end{minipage}&
\begin{minipage}[m]{.18\linewidth}
\begin{center}

\end{center}
\end{minipage}\\
\hline
\begin{minipage}[m]{.28\linewidth}

$O_zO_w=1\times\frac{1}{3}(1_x,1_y,2_t)$ \boundary

\end{minipage}&
\begin{minipage}[m]{.04\linewidth}
\begin{center}
$-$
\end{center}
\end{minipage}&
\begin{minipage}[m]{.11\linewidth}
\begin{center}
$B$
\end{center}
\end{minipage}&
\begin{minipage}[m]{.11\linewidth}
\begin{center}
$y$
\end{center}
\end{minipage}
&
\begin{minipage}[m]{.11\linewidth}
\begin{center}
$y$
\end{center}
\end{minipage}&
\begin{minipage}[m]{.18\linewidth}
\begin{center}

\end{center}
\end{minipage}\\
\hline
\begin{minipage}[m]{.28\linewidth}

$O_zO_t=1\times\frac{1}{2}(1_x,1_y,1_w)$ \boundary

\end{minipage}&
\begin{minipage}[m]{.04\linewidth}
\begin{center}
$-$
\end{center}
\end{minipage}&
\begin{minipage}[m]{.11\linewidth}
\begin{center}
$B$
\end{center}
\end{minipage}&
\begin{minipage}[m]{.11\linewidth}
\begin{center}
$y$
\end{center}
\end{minipage}
&
\begin{minipage}[m]{.11\linewidth}
\begin{center}
$y$
\end{center}
\end{minipage}&
\begin{minipage}[m]{.18\linewidth}
\begin{center}

\end{center}
\end{minipage}\\

\hline
\end{longtable}
\end{center}

\begin{Note}

\item 
For each singular point the $1$-cycle $\Gamma$ is irreducible due
to the monomials $w^2$, $t^3$, and $z^{7}$.
\end{Note}



\begin{center}
\begin{longtable}{|l|c|c|c|c|c|}
\hline
\multicolumn{6}{|l|}{\underline{\textbf{No. 89}}: $X_{42}\subset\mathbb{P}(1,2,5,14,21)$ \hfill $A^3=1/70$}\\
\multicolumn{6}{|l|}{
\begin{minipage}[m]{.86\linewidth}
\vspace*{1.2mm}
$w^2+t^3+yz^8+y^{21}+wf_{21}(x,y,z,t)+f_{42}(x,y,z,t)$
\vspace*{1.2mm}
\end{minipage}
}\\

\hline \hline
\begin{minipage}[m]{.28\linewidth}
\begin{center}
Singularity
\end{center}
\end{minipage}&
\begin{minipage}[m]{.04\linewidth}
\begin{center}
$B^3$
\end{center}
\end{minipage}&
\begin{minipage}[m]{.11\linewidth}
\begin{center}
Linear

system
\end{center}
\end{minipage}&
\begin{minipage}[m]{.11\linewidth}
\begin{center}
Surface $T$
\end{center}
\end{minipage}&
\begin{minipage}[m]{.11\linewidth}
\begin{center}
\vspace*{1mm}
 \vorder
\vspace*{1mm}
\end{center}
\end{minipage}&
\begin{minipage}[m]{.18\linewidth}
\begin{center}
Condition
\end{center}
\end{minipage}\\
\hline
\begin{minipage}[m]{.28\linewidth}

$O_z=\frac{1}{5}(1_x,4_t,1_w)$ \boundary

\end{minipage}&
\begin{minipage}[m]{.04\linewidth}
\begin{center}
$-$
\end{center}
\end{minipage}&
\begin{minipage}[m]{.11\linewidth}
\begin{center}
$2B$
\end{center}
\end{minipage}&
\begin{minipage}[m]{.11\linewidth}
\begin{center}
$y$
\end{center}
\end{minipage}&
\begin{minipage}[m]{.11\linewidth}
\begin{center}
$w^2$
\end{center}
\end{minipage}
&
\begin{minipage}[m]{.18\linewidth}
\begin{center}

\end{center}
\end{minipage}\\
\hline
\begin{minipage}[m]{.28\linewidth}

$O_tO_w=1\times\frac{1}{7}(1_x,2_y,5_z)$ \boundary

\end{minipage}&
\begin{minipage}[m]{.04\linewidth}
\begin{center}
$0$
\end{center}
\end{minipage}&
\begin{minipage}[m]{.11\linewidth}
\begin{center}
$2B$
\end{center}
\end{minipage}&
\begin{minipage}[m]{.11\linewidth}
\begin{center}
$y$
\end{center}
\end{minipage}&
\begin{minipage}[m]{.11\linewidth}
\begin{center}
$y$
\end{center}
\end{minipage}&
\begin{minipage}[m]{.18\linewidth}
\begin{center}

\end{center}
\end{minipage}\\
\hline
\begin{minipage}[m]{.28\linewidth}

$O_yO_t=3\times\frac{1}{2}(1_x,1_z,1_w)$ \boundary

\end{minipage}&
\begin{minipage}[m]{.04\linewidth}
\begin{center}
$-$
\end{center}
\end{minipage}&
\begin{minipage}[m]{.11\linewidth}
\begin{center}
$5B+2E$
\end{center}
\end{minipage}&
\begin{minipage}[m]{.11\linewidth}
\begin{center}
$z$
\end{center}
\end{minipage}&
\begin{minipage}[m]{.11\linewidth}
\begin{center}
$z$
\end{center}
\end{minipage}&
\begin{minipage}[m]{.18\linewidth}
\begin{center}

\end{center}
\end{minipage}\\

\hline
\end{longtable}
\end{center}

\begin{Note}

\item 
For each singular point the $1$-cycle $\Gamma$ is irreducible
because of the monomials $w^2$ and $t^3$.
\end{Note}



\begin{center}
\begin{longtable}{|l|c|c|c|c|c|}
\hline
\multicolumn{6}{|l|}{\underline{\textbf{No. 90}}: $X_{42}\subset\mathbb{P}(1,3,4,14,21)$\hfill $A^3=1/84$}\\
\multicolumn{6}{|l|}{
\begin{minipage}[m]{.86\linewidth}
\vspace*{1.2mm} $(w-\alpha_1 y^7)(w-\alpha_2
y^7)+t^{3}+z^{7}t+wf_{21}(x,y,z,t)+f_{42}(x,y,z,t)$
\vspace*{1.2mm}
\end{minipage}
}\\

\hline \hline
\begin{minipage}[m]{.28\linewidth}
\begin{center}
Singularity
\end{center}
\end{minipage}&
\begin{minipage}[m]{.04\linewidth}
\begin{center}
$B^3$
\end{center}
\end{minipage}&
\begin{minipage}[m]{.11\linewidth}
\begin{center}
Linear

system
\end{center}
\end{minipage}&
\begin{minipage}[m]{.11\linewidth}
\begin{center}
Surface $T$
\end{center}
\end{minipage}&
\begin{minipage}[m]{.11\linewidth}
\begin{center}
\vspace*{1mm}
 \vorder
\vspace*{1mm}
\end{center}
\end{minipage}&
\begin{minipage}[m]{.18\linewidth}
\begin{center}
Condition
\end{center}
\end{minipage}\\
\hline
\begin{minipage}[m]{.28\linewidth}

$O_z=\frac{1}{4}(1_x,3_y,1_w)$ \boundary

\end{minipage}&
\begin{minipage}[m]{.04\linewidth}
\begin{center}
$-$
\end{center}
\end{minipage}&
\begin{minipage}[m]{.11\linewidth}
\begin{center}
$3B$
\end{center}
\end{minipage}&
\begin{minipage}[m]{.11\linewidth}
\begin{center}
$y$
\end{center}
\end{minipage}
&
\begin{minipage}[m]{.11\linewidth}
\begin{center}
$y$
\end{center}
\end{minipage}&
\begin{minipage}[m]{.18\linewidth}
\begin{center}

\end{center}
\end{minipage}\\
\hline
\begin{minipage}[m]{.28\linewidth}

$O_tO_w=1\times\frac{1}{7}(1_x,3_y,4_z)$ \boundary

\end{minipage}&
\begin{minipage}[m]{.04\linewidth}
\begin{center}
$0$
\end{center}
\end{minipage}&
\begin{minipage}[m]{.11\linewidth}
\begin{center}
$3B$
\end{center}
\end{minipage}&
\begin{minipage}[m]{.11\linewidth}
\begin{center}
$y$
\end{center}
\end{minipage}
&
\begin{minipage}[m]{.11\linewidth}
\begin{center}
$y$
\end{center}
\end{minipage}&
\begin{minipage}[m]{.18\linewidth}
\begin{center}

\end{center}
\end{minipage}\\
\hline
\begin{minipage}[m]{.28\linewidth}

$O_yO_w=2\times\frac{1}{3}(1_x,1_z,2_t)$ \nef

\end{minipage}&
\begin{minipage}[m]{.04\linewidth}
\begin{center}
$-$
\end{center}
\end{minipage}&
\begin{minipage}[m]{.11\linewidth}
\begin{center}
$4B+E$
\end{center}
\end{minipage}&
\begin{minipage}[m]{.11\linewidth}
\begin{center}
$xy$, $z$
\end{center}
\end{minipage}
&
\begin{minipage}[m]{.11\linewidth}
\begin{center}
$xy$, $z$
\end{center}
\end{minipage}&
\begin{minipage}[m]{.18\linewidth}
\begin{center}

\end{center}
\end{minipage}\\

\hline
\begin{minipage}[m]{.28\linewidth}

$O_zO_t=1\times\frac{1}{2}(1_x,1_y,1_w)$ \boundary

\end{minipage}&
\begin{minipage}[m]{.04\linewidth}
\begin{center}
$-$
\end{center}
\end{minipage}&
\begin{minipage}[m]{.11\linewidth}
\begin{center}
$3B+E$
\end{center}
\end{minipage}&
\begin{minipage}[m]{.11\linewidth}
\begin{center}
$y$
\end{center}
\end{minipage}
&
\begin{minipage}[m]{.11\linewidth}
\begin{center}
$y$
\end{center}
\end{minipage}&
\begin{minipage}[m]{.18\linewidth}
\begin{center}

\end{center}
\end{minipage}\\

\hline
\end{longtable}
\end{center}

\begin{Note}

\item 
For the singular points other than those of type
$\frac{1}{3}(1,1,2)$, the $1$-cycle $\Gamma$ is irreducible since
we have monomials $w^2$, $t^3$, and $z^7t$.

\item For the singular points of type $\frac{1}{3}(1,1,2)$, consider the
linear system generated by $xy$ and $z$ on $X_{42}$. Its base
curves are defined by $x=z=0$ and $y=z=0$. The curve defined by
$y=z=0$ passes through no singular points of type
$\frac{1}{3}(1,1,2)$.  The curve   defined by $x=z=0$   is
irreducible because of the monomials $w^2$ and $t^3$.  Its proper
transform is equivalent to the $1$-cycle defined by $(4B+E)\cdot
B$ and $(4B+E)^2\cdot B>0$. Therefore, the divisor $T$ is nef.
\end{Note}



\begin{center}
\begin{longtable}{|l|c|c|c|c|c|}
\hline
\multicolumn{6}{|l|}{\underline{\textbf{No. 91}}: $X_{44}\subset\mathbb{P}(1,4,5,13,22)$ \hfill $A^3=1/130$}\\
\multicolumn{6}{|l|}{
\begin{minipage}[m]{.86\linewidth}
\vspace*{1.2mm}
$w^2+zt^3+yz^8+y^{11}+wf_{22}(x,y,z,t)+f_{44}(x,y,z,t)$
\vspace*{1.2mm}
\end{minipage}
}\\
\hline \hline
\begin{minipage}[m]{.28\linewidth}
\begin{center}
Singularity
\end{center}
\end{minipage}&
\begin{minipage}[m]{.04\linewidth}
\begin{center}
$B^3$
\end{center}
\end{minipage}&
\begin{minipage}[m]{.11\linewidth}
\begin{center}
Linear

system
\end{center}
\end{minipage}&
\begin{minipage}[m]{.11\linewidth}
\begin{center}
Surface $T$
\end{center}
\end{minipage}&
\begin{minipage}[m]{.11\linewidth}
\begin{center}
\vspace*{1mm}
 \vorder
\vspace*{1mm}
\end{center}
\end{minipage}&
\begin{minipage}[m]{.18\linewidth}
\begin{center}
Condition
\end{center}
\end{minipage}\\
\hline
\begin{minipage}[m]{.28\linewidth}

$O_t=\frac{1}{13}(1_x,4_y,9_w)$ $\positive$

\end{minipage}&
\begin{minipage}[m]{.04\linewidth}
\begin{center}
$+$
\end{center}
\end{minipage}&
\begin{minipage}[m]{.11\linewidth}
\begin{center}
$5B-E$
\end{center}
\end{minipage}&
\begin{minipage}[m]{.11\linewidth}
\begin{center}
$z$
\end{center}
\end{minipage}&
\begin{minipage}[m]{.11\linewidth}
\begin{center}
$w^2$
\end{center}
\end{minipage}
&
\begin{minipage}[m]{.18\linewidth}
\begin{center}

\end{center}
\end{minipage}\\
\hline
\begin{minipage}[m]{.28\linewidth}

$O_z=\frac{1}{5}(1_x,3_t,2_w)$ \boundary

\end{minipage}&
\begin{minipage}[m]{.04\linewidth}
\begin{center}
$-$
\end{center}
\end{minipage}&
\begin{minipage}[m]{.11\linewidth}
\begin{center}
$4B$
\end{center}
\end{minipage}&
\begin{minipage}[m]{.11\linewidth}
\begin{center}
$y$
\end{center}
\end{minipage}&
\begin{minipage}[m]{.11\linewidth}
\begin{center}
$w^2$
\end{center}
\end{minipage}&
\begin{minipage}[m]{.18\linewidth}
\begin{center}

\end{center}
\end{minipage}\\
\hline
\begin{minipage}[m]{.28\linewidth}

$O_yO_w=1\times\frac{1}{2}(1_x,1_z,1_t)$ \boundary

\end{minipage}&
\begin{minipage}[m]{.04\linewidth}
\begin{center}
$-$
\end{center}
\end{minipage}&
\begin{minipage}[m]{.11\linewidth}
\begin{center}
$5B+2E$
\end{center}
\end{minipage}&
\begin{minipage}[m]{.11\linewidth}
\begin{center}
$z$
\end{center}
\end{minipage}&
\begin{minipage}[m]{.11\linewidth}
\begin{center}
$z$
\end{center}
\end{minipage}&
\begin{minipage}[m]{.18\linewidth}
\begin{center}

\end{center}
\end{minipage}\\

\hline
\end{longtable}
\end{center}

\begin{Note}

\item 
For each singular point the $1$-cycle $\Gamma$ is irreducible due
to the monomials $w^2$, $y^{11}$, and $zt^3$.
\end{Note}



\begin{center}
\begin{longtable}{|l|c|c|c|c|c|}
\hline
\multicolumn{6}{|l|}{\underline{\textbf{No. 92}}: $X_{48}\subset\mathbb{P}(1,3,5,16,24)$\hfill $A^3=1/120$}\\
\multicolumn{6}{|l|}{
\begin{minipage}[m]{.86\linewidth}
\vspace*{1.2mm}
$w^2+t^3+yz^9+y^{16}+wf_{24}(x,y,z,t)+f_{48}(x,y,z,t)$
\vspace*{1.2mm}
\end{minipage}
}\\

\hline \hline
\begin{minipage}[m]{.28\linewidth}
\begin{center}
Singularity
\end{center}
\end{minipage}&
\begin{minipage}[m]{.04\linewidth}
\begin{center}
$B^3$
\end{center}
\end{minipage}&
\begin{minipage}[m]{.11\linewidth}
\begin{center}
Linear

system
\end{center}
\end{minipage}&
\begin{minipage}[m]{.11\linewidth}
\begin{center}
Surface $T$
\end{center}
\end{minipage}&
\begin{minipage}[m]{.11\linewidth}
\begin{center}
\vspace*{1mm}
 \vorder
\vspace*{1mm}
\end{center}
\end{minipage}&
\begin{minipage}[m]{.18\linewidth}
\begin{center}
Condition
\end{center}
\end{minipage}\\
\hline
\begin{minipage}[m]{.28\linewidth}

$O_z=\frac{1}{5}(1_x,1_t,4_w)$ \boundary

\end{minipage}&
\begin{minipage}[m]{.04\linewidth}
\begin{center}
$-$
\end{center}
\end{minipage}&
\begin{minipage}[m]{.11\linewidth}
\begin{center}
$3B$
\end{center}
\end{minipage}&
\begin{minipage}[m]{.11\linewidth}
\begin{center}
$y$
\end{center}
\end{minipage}
&
\begin{minipage}[m]{.11\linewidth}
\begin{center}
$t^3$
\end{center}
\end{minipage}&
\begin{minipage}[m]{.18\linewidth}
\begin{center}

\end{center}
\end{minipage}\\
\hline
\begin{minipage}[m]{.28\linewidth}

$O_tO_w=1\times\frac{1}{8}(1_x,3_y,5_z)$ \boundary

\end{minipage}&
\begin{minipage}[m]{.04\linewidth}
\begin{center}
$0$
\end{center}
\end{minipage}&
\begin{minipage}[m]{.11\linewidth}
\begin{center}
$3B$
\end{center}
\end{minipage}&
\begin{minipage}[m]{.11\linewidth}
\begin{center}
$y$
\end{center}
\end{minipage}
&
\begin{minipage}[m]{.11\linewidth}
\begin{center}
$y$
\end{center}
\end{minipage}&
\begin{minipage}[m]{.18\linewidth}
\begin{center}

\end{center}
\end{minipage}\\
\hline
\begin{minipage}[m]{.28\linewidth}

$O_yO_w=2\times\frac{1}{3}(1_x,2_z,1_t)$ \boundary

\end{minipage}&
\begin{minipage}[m]{.04\linewidth}
\begin{center}
$-$
\end{center}
\end{minipage}&
\begin{minipage}[m]{.11\linewidth}
\begin{center}
$5B+E$
\end{center}
\end{minipage}&
\begin{minipage}[m]{.11\linewidth}
\begin{center}
$z$
\end{center}
\end{minipage}
&
\begin{minipage}[m]{.11\linewidth}
\begin{center}
$z$
\end{center}
\end{minipage}&
\begin{minipage}[m]{.18\linewidth}
\begin{center}

\end{center}
\end{minipage}\\

\hline
\end{longtable}
\end{center}

\begin{Note}

\item 
For each singular point the $1$-cycle $\Gamma$ is irreducible due
to the monomials $w^2$ and $t^3$.
\end{Note}

\begin{center}
\begin{longtable}{|l|c|c|c|c|c|}
\hline
\multicolumn{6}{|l|}{\underline{\textbf{No. 93}}: $X_{50}\subset\mathbb{P}(1,7,8,10,25)$ \hfill $A^3=1/280$}\\
\multicolumn{6}{|l|}{
\begin{minipage}[m]{.86\linewidth}
\vspace*{1.2mm}
$w^2+t^5+z^5t+y^6(a_1z+a_2xy)+wf_{25}(x,y,z,t)+f_{50}(x,y,z,t)$
\vspace*{1.2mm}
\end{minipage}
}\\

\hline \hline
\begin{minipage}[m]{.28\linewidth}
\begin{center}
Singularity
\end{center}
\end{minipage}&
\begin{minipage}[m]{.04\linewidth}
\begin{center}
$B^3$
\end{center}
\end{minipage}&
\begin{minipage}[m]{.11\linewidth}
\begin{center}
Linear

system
\end{center}
\end{minipage}&
\begin{minipage}[m]{.11\linewidth}
\begin{center}
Surface $T$
\end{center}
\end{minipage}&
\begin{minipage}[m]{.11\linewidth}
\begin{center}
\vspace*{1mm}
 \vorder
\vspace*{1mm}
\end{center}
\end{minipage}&
\begin{minipage}[m]{.18\linewidth}
\begin{center}
Condition
\end{center}
\end{minipage}\\
\hline
\begin{minipage}[m]{.28\linewidth}

$O_z=\frac{1}{8}(1_x,7_y,1_w)$ \boundary

\end{minipage}&
\begin{minipage}[m]{.04\linewidth}
\begin{center}
$-$
\end{center}
\end{minipage}&
\begin{minipage}[m]{.11\linewidth}
\begin{center}
$7B$
\end{center}
\end{minipage}&
\begin{minipage}[m]{.11\linewidth}
\begin{center}
$y$
\end{center}
\end{minipage}&
\begin{minipage}[m]{.11\linewidth}
\begin{center}
$y$
\end{center}
\end{minipage}
&
\begin{minipage}[m]{.18\linewidth}
\begin{center}

\end{center}
\end{minipage}\\
\hline
\begin{minipage}[m]{.28\linewidth}

$O_y=\frac{1}{7}(1_x,3_t,4_w)$ \boundary

\end{minipage}&
\begin{minipage}[m]{.04\linewidth}
\begin{center}
$-$
\end{center}
\end{minipage}&
\begin{minipage}[m]{.11\linewidth}
\begin{center}
$8B$
\end{center}
\end{minipage}&
\begin{minipage}[m]{.11\linewidth}
\begin{center}
$z$
\end{center}
\end{minipage}&
\begin{minipage}[m]{.11\linewidth}
\begin{center}
$w^2$
\end{center}
\end{minipage}&
\begin{minipage}[m]{.18\linewidth}
\begin{center}
$a_1\ne 0$
\end{center}
\end{minipage}\\
\hline
\begin{minipage}[m]{.28\linewidth}

$O_y=\frac{1}{7}(1_z,3_t,4_w)$ \boundary

\end{minipage}&
\begin{minipage}[m]{.04\linewidth}
\begin{center}
$-$
\end{center}
\end{minipage}&
\begin{minipage}[m]{.11\linewidth}
\begin{center}
$10B+E$
\end{center}
\end{minipage}&
\begin{minipage}[m]{.11\linewidth}
\begin{center}
$t$
\end{center}
\end{minipage}&
\begin{minipage}[m]{.11\linewidth}
\begin{center}
$t$
\end{center}
\end{minipage}&
\begin{minipage}[m]{.18\linewidth}
\begin{center}
$a_1=0$
\end{center}
\end{minipage}\\
\hline
\begin{minipage}[m]{.28\linewidth}

$O_tO_w=1\times\frac{1}{5}(1_x,2_y,3_z)$ \boundary

\end{minipage}&
\begin{minipage}[m]{.04\linewidth}
\begin{center}
$-$
\end{center}
\end{minipage}&
\begin{minipage}[m]{.11\linewidth}
\begin{center}
$8B+E$
\end{center}
\end{minipage}&
\begin{minipage}[m]{.11\linewidth}
\begin{center}
$z$
\end{center}
\end{minipage}&
\begin{minipage}[m]{.11\linewidth}
\begin{center}
$z$
\end{center}
\end{minipage}&
\begin{minipage}[m]{.18\linewidth}
\begin{center}

\end{center}
\end{minipage}\\
\hline
\begin{minipage}[m]{.28\linewidth}

$O_zO_t=1\times\frac{1}{2}(1_x,1_y,1_w)$ \boundary

\end{minipage}&
\begin{minipage}[m]{.04\linewidth}
\begin{center}
$-$
\end{center}
\end{minipage}&
\begin{minipage}[m]{.11\linewidth}
\begin{center}
$7B+3E$
\end{center}
\end{minipage}&
\begin{minipage}[m]{.11\linewidth}
\begin{center}
$y$
\end{center}
\end{minipage}&
\begin{minipage}[m]{.11\linewidth}
\begin{center}
$y$
\end{center}
\end{minipage}&
\begin{minipage}[m]{.18\linewidth}
\begin{center}

\end{center}
\end{minipage}\\

\hline
\end{longtable}
\end{center}

\begin{Note}

\item 
For each singular point the $1$-cycle $\Gamma$ is always
irreducible because of the monomials $w^2$ and $t^5$. In
particular, the $1$-cycle $\Gamma$ for the singular point $O_y$
with $a_1=0$ is irreducible even though it is non-reduced.
\end{Note}

\begin{center}
\begin{longtable}{|l|c|c|c|c|c|}
\hline
\multicolumn{6}{|l|}{\underline{\textbf{No. 94}}: $X_{54}\subset\mathbb{P}(1,4,5,18,27)$\hfill $A^3=1/180$}\\
\multicolumn{6}{|l|}{
\begin{minipage}[m]{.86\linewidth}
\vspace*{1.2mm}
$w^2+t^3+yz^{10}+y^{9}t+wf_{27}(x,y,z,t)+f_{54}(x,y,z,t)$
\vspace*{1.2mm}
\end{minipage}
}\\

\hline \hline
\begin{minipage}[m]{.28\linewidth}
\begin{center}
Singularity
\end{center}
\end{minipage}&
\begin{minipage}[m]{.04\linewidth}
\begin{center}
$B^3$
\end{center}
\end{minipage}&
\begin{minipage}[m]{.11\linewidth}
\begin{center}
Linear

system
\end{center}
\end{minipage}&
\begin{minipage}[m]{.11\linewidth}
\begin{center}
Surface $T$
\end{center}
\end{minipage}&
\begin{minipage}[m]{.11\linewidth}
\begin{center}
\vspace*{1mm}
 \vorder
\vspace*{1mm}
\end{center}
\end{minipage}&
\begin{minipage}[m]{.18\linewidth}
\begin{center}
Condition
\end{center}
\end{minipage}\\
\hline
\begin{minipage}[m]{.28\linewidth}

$O_z=\frac{1}{5}(1_x,3_t,2_w)$ \boundary

\end{minipage}&
\begin{minipage}[m]{.04\linewidth}
\begin{center}
$-$
\end{center}
\end{minipage}&
\begin{minipage}[m]{.11\linewidth}
\begin{center}
$4B$
\end{center}
\end{minipage}&
\begin{minipage}[m]{.11\linewidth}
\begin{center}
$y$
\end{center}
\end{minipage}
&
\begin{minipage}[m]{.11\linewidth}
\begin{center}
$w^2$
\end{center}
\end{minipage}&
\begin{minipage}[m]{.18\linewidth}
\begin{center}

\end{center}
\end{minipage}\\
\hline
\begin{minipage}[m]{.28\linewidth}

$O_y=\frac{1}{4}(1_x,1_z,3_w)$ \boundary

\end{minipage}&
\begin{minipage}[m]{.04\linewidth}
\begin{center}
$-$
\end{center}
\end{minipage}&
\begin{minipage}[m]{.11\linewidth}
\begin{center}
$18B+3E$
\end{center}
\end{minipage}&
\begin{minipage}[m]{.11\linewidth}
\begin{center}
$t$
\end{center}
\end{minipage}
&
\begin{minipage}[m]{.11\linewidth}
\begin{center}
$w^2$
\end{center}
\end{minipage}&
\begin{minipage}[m]{.18\linewidth}
\begin{center}

\end{center}
\end{minipage}\\
\hline
\begin{minipage}[m]{.28\linewidth}

$O_tO_w=1\times\frac{1}{9}(1_x,4_y,5_z)$ \boundary

\end{minipage}&
\begin{minipage}[m]{.04\linewidth}
\begin{center}
$0$
\end{center}
\end{minipage}&
\begin{minipage}[m]{.11\linewidth}
\begin{center}
$4B$
\end{center}
\end{minipage}&
\begin{minipage}[m]{.11\linewidth}
\begin{center}
$y$
\end{center}
\end{minipage}
&
\begin{minipage}[m]{.11\linewidth}
\begin{center}
$y$
\end{center}
\end{minipage}&
\begin{minipage}[m]{.18\linewidth}
\begin{center}

\end{center}
\end{minipage}\\

\hline
\begin{minipage}[m]{.28\linewidth}

$O_yO_t=1\times\frac{1}{2}(1_x,1_z,1_w)$ \boundary

\end{minipage}&
\begin{minipage}[m]{.04\linewidth}
\begin{center}
$-$
\end{center}
\end{minipage}&
\begin{minipage}[m]{.11\linewidth}
\begin{center}
$5B+2E$
\end{center}
\end{minipage}&
\begin{minipage}[m]{.11\linewidth}
\begin{center}
$z$
\end{center}
\end{minipage}
&
\begin{minipage}[m]{.11\linewidth}
\begin{center}
$z$
\end{center}
\end{minipage}&
\begin{minipage}[m]{.18\linewidth}
\begin{center}

\end{center}
\end{minipage}\\

\hline
\end{longtable}
\end{center}

\begin{Note}

\item 
For each singular point the $1$-cycle $\Gamma$ is irreducible due
to the monomials $w^2$, $t^3$ and $yz^{10}$.
\end{Note}

\begin{center}
\begin{longtable}{|l|c|c|c|c|c|}
\hline
\multicolumn{6}{|l|}{\underline{\textbf{No. 95}}: $X_{66}\subset\mathbb{P}(1,5,6,22,33)$\hfill $A^3=1/330$}\\
\multicolumn{6}{|l|}{
\begin{minipage}[m]{.86\linewidth}
\vspace*{1.2mm}
$w^2+t^3+z^{11}+y^{12}(a_1z+a_2xy)+wf_{33}(x,y,z,t)+f_{66}(x,y,z,t)$
\vspace*{1.2mm}
\end{minipage}
}\\

\hline \hline
\begin{minipage}[m]{.29\linewidth}
\begin{center}
Singularity
\end{center}
\end{minipage}&
\begin{minipage}[m]{.04\linewidth}
\begin{center}
$B^3$
\end{center}
\end{minipage}&
\begin{minipage}[m]{.11\linewidth}
\begin{center}
Linear

system
\end{center}
\end{minipage}&
\begin{minipage}[m]{.11\linewidth}
\begin{center}
Surface $T$
\end{center}
\end{minipage}&
\begin{minipage}[m]{.11\linewidth}
\begin{center}
\vspace*{1mm}
 \vorder
\vspace*{1mm}
\end{center}
\end{minipage}&
\begin{minipage}[m]{.17\linewidth}
\begin{center}
Condition
\end{center}
\end{minipage}\\
\hline
\begin{minipage}[m]{.29\linewidth}

$O_y=\frac{1}{5}(1_x,2_t,3_w)$ \boundary

\end{minipage}&
\begin{minipage}[m]{.04\linewidth}
\begin{center}
$-$
\end{center}
\end{minipage}&
\begin{minipage}[m]{.11\linewidth}
\begin{center}
$6B$
\end{center}
\end{minipage}&
\begin{minipage}[m]{.11\linewidth}
\begin{center}
$z$
\end{center}
\end{minipage}
&
\begin{minipage}[m]{.11\linewidth}
\begin{center}
$w^2$
\end{center}
\end{minipage}&
\begin{minipage}[m]{.17\linewidth}
\begin{center}
$a_1\ne 0$
\end{center}
\end{minipage}\\
\hline
\begin{minipage}[m]{.29\linewidth}

$O_y=\frac{1}{5}(1_z,2_t,3_w)$ \boundary

\end{minipage}&
\begin{minipage}[m]{.04\linewidth}
\begin{center}
$-$
\end{center}
\end{minipage}&
\begin{minipage}[m]{.11\linewidth}
\begin{center}
$6B+E$
\end{center}
\end{minipage}&
\begin{minipage}[m]{.11\linewidth}
\begin{center}
$z$
\end{center}
\end{minipage}
&
\begin{minipage}[m]{.11\linewidth}
\begin{center}
$z$
\end{center}
\end{minipage}&
\begin{minipage}[m]{.17\linewidth}
\begin{center}
$a_1=0$
\end{center}
\end{minipage}\\
\hline
\begin{minipage}[m]{.29\linewidth}

$O_tO_w=1\times\frac{1}{11}(1_x,5_y,6_z)$ \boundary

\end{minipage}&
\begin{minipage}[m]{.04\linewidth}
\begin{center}
$0$
\end{center}
\end{minipage}&
\begin{minipage}[m]{.11\linewidth}
\begin{center}
$5B$
\end{center}
\end{minipage}&
\begin{minipage}[m]{.11\linewidth}
\begin{center}
$y$
\end{center}
\end{minipage}
&
\begin{minipage}[m]{.11\linewidth}
\begin{center}
$y$
\end{center}
\end{minipage}&
\begin{minipage}[m]{.17\linewidth}
\begin{center}

\end{center}
\end{minipage}\\
\hline
\begin{minipage}[m]{.29\linewidth}

$O_zO_w=1\times\frac{1}{3}(1_x,2_y,1_t)$ \boundary

\end{minipage}&
\begin{minipage}[m]{.04\linewidth}
\begin{center}
$-$
\end{center}
\end{minipage}&
\begin{minipage}[m]{.11\linewidth}
\begin{center}
$5B+E$
\end{center}
\end{minipage}&
\begin{minipage}[m]{.11\linewidth}
\begin{center}
$y$
\end{center}
\end{minipage}
&
\begin{minipage}[m]{.11\linewidth}
\begin{center}
$y$
\end{center}
\end{minipage}&
\begin{minipage}[m]{.17\linewidth}
\begin{center}

\end{center}
\end{minipage}\\

\hline
\begin{minipage}[m]{.29\linewidth}

$O_zO_t=1\times\frac{1}{2}(1_x,1_y,1_w)$ \boundary

\end{minipage}&
\begin{minipage}[m]{.04\linewidth}
\begin{center}
$-$
\end{center}
\end{minipage}&
\begin{minipage}[m]{.11\linewidth}
\begin{center}
$5B+2E$
\end{center}
\end{minipage}&
\begin{minipage}[m]{.11\linewidth}
\begin{center}
$y$
\end{center}
\end{minipage}
&
\begin{minipage}[m]{.11\linewidth}
\begin{center}
$y$
\end{center}
\end{minipage}&
\begin{minipage}[m]{.17\linewidth}
\begin{center}

\end{center}
\end{minipage}\\

\hline
\end{longtable}
\end{center}

\begin{Note}

\item 
The $1$-cycle $\Gamma$ for each singular point is irreducible
because of  the monomials $w^2$ and $t^3$.
\end{Note}
\end{proof}

\newpage

\section{Epilogue}

\subsection*{Open problems}

Let $X$ be a quasi-smooth hypersurface of degrees $d$ with only
terminal singularities in weighted projective space $\mathbb{P}(1,
a_1, a_2, a_3, a_4)$, where $d=\sum  a_i$. In Main Theorem, we
prove that $X$ is birationally rigid. In particular, $X$ is
non-rational. Moreover, the proof also explicitly describes the
generators of the group of birational automorphisms
$\mathrm{Bir}(X)$ modulo subgroup of biregular automorphisms
$\mathrm{Aut}(X)$ (see Theorem~\ref{theorem:Bir}). Furthermore,
Theorem~\ref{theorem:auxiliary} says that
$\mathrm{Bir}(X)=\mathrm{Aut}(X)$ for those families in the list
of Fletcher and Reid with entry numbers No.~1, 3, 10, 11, 14, 19,
21, 22, 28, 29, 34, 35, 37,  39, 49, 50, 51, 52, 53, 55, 57, 59,
62, 63, 64, 66, 67, 70, 71, 72, 73, 75, 77, 78, 80, 81, 82, 83,
84, 85, 86, 87, 88,  89, 90, 91, 92, 93, 94 and  95. Of course,
some quasi-smooth threefolds in other families may also be
birationally super-rigid.

Explicit birational involutions play a key role in the proof of
Main Theorem. In many cases, they arise from generically
$2$-to-$1$ rational maps  of $X$ to suitable  $3$-dimensional
weighted projective spaces (\emph{quadratic involutions}). However, in some cases they arise from rational
maps of $X$ to suitable  $2$-dimensional weighted projective
spaces whose general fibers are birational to smooth elliptic
curves (\emph{elliptic involutions}). Moreover, we often use such  elliptic rational
fibrations in order  to exclude some singular points of $X$ as
centers of non-canonical singularities of any log pair
$\left(X,\frac{1}{n}\mathcal{M}\right)$, where $\mathcal{M}$ is a mobile
linear subsystem in $|-nK_{X}|$. The latter is done using
Corollary~\ref{corollary:Mori-cone} or Lemma~\ref{lemma:bad-link}.
A similar role in the proof of Main Theorem is played by so-called
\emph{Halphen pencils} on $X$, i.e., pencils whose general
members are  irreducible surfaces of Kodaira dimension zero.
Implicitly Halphen pencils appear almost every time when we apply
Lemmas~\ref{lemma:boundary} and \ref{lemma:negative-definite}.
This leads us to three problems that are closely related to Main
Theorem. They are
\begin{enumerate}
\item[(1)] to find relations between generators of the birational automorphism group
$\mathrm{Bir}(X)$;

\item[(2)] to describe birational transformations of $X$ into
elliptic fibrations;

\item[(3)] to classify Halphen pencils on $X$.
\end{enumerate}

While proving Main Theorem, we noticed many interesting Halphen
pencils on $X$  even though we did not mention them explicitly in
the proofs.  We also observed that their general members are K3
surfaces. This gives an evidence for

\begin{conjec}
\label{conjecture:Halphen} Every Halphen pencil on $X$ is a pencil
of K3 surfaces.
\end{conjec}

We do not know any deep reason why this conjecture should be true.
When $X$ is a general threefold in its family,
Conjecture~\ref{conjecture:Halphen} was proved in \cite{ChPa09}.

The original proof of Theorem~\ref{theorem:IM} given by Iskovskikh
and Manin in \cite{IsM71} holds in arbitrary characteristic. This
also follows from \cite{Pu98}. The short proof of
Theorem~\ref{theorem:IM} given by Corti in \cite{Co00} holds only
in characteristic zero. For some families in the list of Fletcher
and Reid, the proof of Main Theorem requires vanishing type
results and, thus, is valid only in characteristic zero. This
suggests  the birational rigidity problem of $X$ and problems (1),
(2) and (3)  over an algebraically closed field of positive
characteristic. For double covers of $\mathbb{P}^3$ ramified along
smooth sextic surfaces, this was done in \cite{ChPa07} and
\cite{ChPa11}, which revealed special phenomenon of small
characteristics (see \cite[Example~1.5]{ChPa07}).

\subsection*{General vs. special}

The first three problems  listed in the previous section are
solved in the case when $X$ is a general hypersurface in its
family. This is done in \cite{Ch00}, \cite{Ch07}, \cite{ChPa05}
and \cite{ChPa09}. In many cases, the same methods can be applied
regardless of the assumption that $X$ is general. For example, we
proved in \cite{ChPa05} that a general hypersurface in the
families No.~3, 60, 75, 83, 87, 93 cannot have a birational
transformation to an
 elliptic fibration. We are able to  prove that it is also
true for every quasi-smooth hypersurface in  the families No.~3,
75, 83, 87, 93, using the methods given in this paper. However, in
the family No.~60, it is no longer true for an arbitrary
quasi-smooth hypersurface.

\begin{exam}
\label{example:60} Let $X_{24}$ be a quasi-smooth hypersurface in the
family No.~$60$. Suppose, in addition, that $X_{24}$ contains the curve
$L_{zw}$. We may then assume that it is defined by the equation
\begin{equation*}
\begin{split}
 & w^2t+w(at^2x^3+tg_9(x, y, z)+g_{15}(x,y,z))+\\ & t^4+t^3h_{6}(x,y,z)+t^2h_{12}(x, y, z)+th_{18}(x, y, z)+h_{24}(x, y,z)=0\\
 \end{split}
 \end{equation*}
 in $\mathbb{P}(1,4,5,6,9)$. For the hypersurface $X_{24}$
to be quasi-smooth at the point $O_z$, the polynomial $h_{24}$
must  contain the monomial $yz^4$. For the hypersurface $X_{24}$ to
contain the curve $L_{zw}$, the polynomial $g_{15}$ does not
contain the monomial $z^3$.

Consider the projection $\pi\colon
X_{24}\dashrightarrow \mathbb{P}(1,4,6)$. Its general fiber is an
irreducible curve birational to an elliptic curve. To see this, on
the hypersurface $X_{24}$, consider the surface cut by $y=\lambda x^4$
and the surface cut by $t=\mu x^6$, where $\lambda$ and $\mu$ are
sufficiently general complex numbers. Then the intersection of
these two surfaces is the $1$-cycle $4L_{zw}+C_{\lambda, \mu}$,
where $C_{\lambda, \mu}$ is a curve defined by the
equation
 \[\lambda w^2x^2+w\left(\mu^2x^{11}g_0+ \mu x^2g_9(x,\lambda x^4,z)+\frac{g_{15}(x,\lambda
x^4,z)}{x^4}\right)+\]\[\mu^4x^{20}+\mu^3x^{14}h_{6}(x,\lambda
x^4,z)+\mu^2x^{8}h_{12}(x,\lambda x^4,z)+ \mu
x^{2}h_{18}(x,\lambda x^4,z)+\frac{h_{24}(x,\lambda
x^4,z)}{x^4}=0\] in $\mathbb{P}(1,5,9)$. Plugging $x=1$ into the
equation, we see that the curve $C_{\lambda, \mu}$ is birational
to  a double cover of $\mathbb{C}$ ramified at four distinct
points.
\end{exam}

In some of the 95 families of Reid and Fletcher, special quasi-smooth
hypersurfaces may have simpler geometry than their general
representatives.

\begin{exam}
\label{example:no-2} Let $X_5$ be a quasi-smooth hypersurface in
$\mathbb{P}(1,1,1,1,2)$ (the family No. 2). The hypersurface $X_5$ can be
given by
$$
tw^2+wf_3(x,y,z,t)+f_5(x,y,z,t)=0.
$$
 The natural projection
$X_5\dasharrow\mathbb{P}^3$ is a generically double cover.
Therefore, it induces a birational involution of $X_5$, denoted by
$\tau$. By Theorem~\ref{theorem:Bir}, the birational automorphism
group $\mathrm{Bir}(X)$ is generated by the biregular automorphism
group $\mathrm{Aut}(X)$ and the involution $\tau$. By Main
Theorem, the hypersurface $X_5$ is birationally rigid. Moreover,
if the hypersurface $X_5$ is general, then it is not birationally
super-rigid, i.e., $\mathrm{Bir}(X)\ne \mathrm{Aut}(X)$. However,
in a special case, the involution $\tau$ can be biregular, and
hence the hypersurface $X_5$ is birationally super-rigid. To be
precise, the involution $\tau$ is biregular if and only if the
coefficient polynomial $f_3$ of $w$ is a zero polynomial. Thus,
the hypersurface $X_5$ is birationally super-rigid if and only if
$f_3$ is a zero polynomial.
\end{exam}

However, this is not always the case, i.e., special quasi-smooth
hypersurfaces usually have more complicated geometry than their
general representatives. Here we provide three illustrating examples.

\begin{exam}
\label{example:quartic-threefold} Let $X_4$ be a smooth quartic
threefold in $\mathbb{P}^4$ (the family No.~1). From
Theorem~\ref{theorem:IM} we know that every smooth quartic
hypersurface in $\mathbb{P}^4$ admits no non-biregular
birational automorphisms. Moreover, it was proved in \cite{Ch00} that
every rational map $\rho\colon X_4\dasharrow\mathbb{P}^2$ whose
general fiber is birational to a smooth elliptic curve fits a
commutative diagram
$$
\xymatrix{
&X_4\ar@{-->}[dl]_{\rho}\ar@{-->}[dr]^{\pi}&\\%
\mathbb{P}^2\ar@{-->}[rr]_{\sigma}&&\mathbb{P}^2,}
$$ %
where $\pi$ is a linear projection from a line and $\sigma$ is a
birational map. Furthermore, it was proved in \cite{ChPa09} that
every Halphen pencil on $X_4$ is contained in $|-K_{X_{4}}|$
provided that $X_4$ satisfies some generality assumptions.
Earlier, Iskovskikh pointed out in \cite{Is01} that this is no
longer true for an \emph{arbitrary} smooth quartic hypersurface in
$\mathbb{P}^4$. Indeed, a special smooth quartic hypersurface in $\mathbb{P}^4$ may have a Halphen pencil contained in $|-2K_{X_4}|$.
The complete classification of Halphen pencils on
$X_4$ was obtained in \cite{ChKa10}.
\end{exam}

\begin{exam}[{For details see the proof of
Theorem~\ref{theorem:amazing-23}}] \label{example:no-23} Let
$X_{14}$ be a quasi-smooth hypersurface in $\mathbb{P}(1,2,3,4,5)$
(the family No. 23). If $X_{14}$ is a general such hypersurface, then
there exists an exact sequence of groups
$$
1\longrightarrow\Gamma_{X_{14}}\longrightarrow\mathrm{Bir}(X_{14})\longrightarrow\mathrm{Aut}(X_{14})\longrightarrow 1,%
$$
where $\Gamma_{X_{14}}$ is a free product of two birational
involutions constructed in
Section~\ref{subsection:hard-involutions}. This follows from
\cite[Lemma~4.2]{ChPa05} (cf. Theorem~\ref{theorem:Bir}).
Moreover, let $\rho\colon X_{14}\dasharrow\mathbb{P}^{2}$ be a
rational map whose general fiber is birational to a smooth
elliptic curve. If $X_{14}$ is general, then there exists a
commutative diagram
$$
\xymatrix{
&&X_{14}\ar@{-->}[dl]_{\phi}\ar@{-->}[dr]^{\rho}&&&\\%
&\mathbb{P}(1,2,3)\ar@{-->}[rr]_{\sigma}&&\mathbb{P}^{2}&}
$$
where $\phi$ is the natural projection and $\sigma$ is some
birational map. Suppose now that $X_{14}$ is defined by the equation
$$
(t+by^2)w^2+yt(t-\alpha_1y^2)(t-\alpha_2 y^2) +
z^4y+xtz^3+xf_{13}(x,y,z,t,w)+yg_{12}(y,z,t,w)=0.
$$
Then none of these assertions are true. Indeed, let
$\mathcal{H}$ be the linear subsystem of $|-5K_{X_{14}}|$
generated by $x^5$, $xy^2$, $x^3y$ and $yz+xt$. Let $\pi\colon
X_{14}\dasharrow \mathbb{P}(1,2,5)$ be the rational map induced by
$$[x:y:z:t:w]\mapsto[x:y:yz+xt].$$ Then $\pi$ is dominant and its
general fiber is birational to an elliptic curve . Let $f\colon
Y\to X_{14}$ be the weighted blow up at the point $O_z$ with weight $(1,1,2)$. Denote by $E$
its exceptional surface. Let $g\colon W\to Y$ be the weighted blow
up at the point over $O_w$ with weight $(1,2,3)$. Denote by $G$ be
its exceptional divisor. Denote by $\hat{L}_{zw}$, $\hat{L}_{zt}$
and $\hat{L}_{yw}$ the proper transforms of the curves $L_{zw}$,
$L_{zt}$ and $L_{yw}$ by the morphism $f\circ g$. Then the curves
$\hat{L}_{zw}$ and $\hat{L}_{zt}$ are the only curves that
intersect $-K_W$ negatively. Moreover, there is an anti-flip
$\chi\colon W\dasharrow U$ along the curves $\hat{L}_{zw}$ and
$\hat{L}_{zt}$ (see the proof of
Theorem~\ref{theorem:amazing-23}). Let $\check{E}$ and $\check{G}$
be the proper transforms on $U$ of the divisors $E$ and $G$,
respectively. For $m\gg 0$, the linear system $|-mK_U|$ is free
and gives an elliptic fibration $\eta\colon U\to
 \Sigma$, where $\Sigma$ is a normal surface. Furthermore, there  exist a commutative diagram
$$
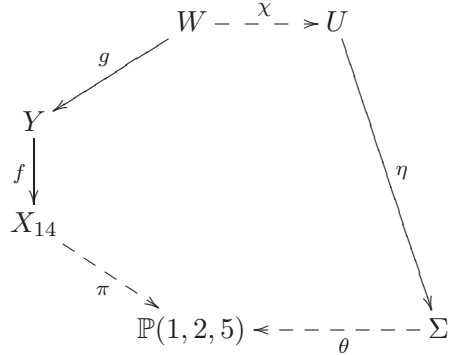
\xymatrix{
&W\ar@{->}[ld]_{g}\ar@{-->}[r]^{\chi}&U\ar@{->}[dddr]^{\eta}\\%
Y\ar@{->}[d]_{f}&&\\%
X_{14}\ar@{-->}[rd]_{\pi}&&\\
&\mathbb{P}(1,2,5)&&\Sigma\ar@{-->}[ll]^{\ \ \ \ \ \ \theta}}
$$ %
where $\theta$ is a birational map. The divisor $\check{G}$ is a
section of the elliptic fibration $\eta$ and $\check{E}$ is a
$2$-section of  $\eta$. Let $\tau_{U}$ be a birational involution
of the threefold $U$ that is induced by the reflection of the
general fiber of $\eta$ with respect to the section $\check{G}$.
The involution $\tau_U$ induces a birational involution of $X_{14}$.
This new involution is not biregular and not contained in the subgroup of the birational automorphism group $\mathrm{Bir}(X_{14})$ generated by two birational involutions constructed in
Section~\ref{subsection:hard-involutions}.
\end{exam}

\begin{exam}
\label{example:no-33} Let $X_{17}$ be a quasi-smooth hypersurface
in $\mathbb{P}(1,2,3,5,7)$ (the family No. 33). Then it can be given
by the quasi-homogenous polynomial equation
$$
(dx^3+exy+z)w^2+
t^2(a_1w+a_2yt)+z^4(b_1t+b_2yz)+$$
$$y^5(c_1w+c_2yt+c_3y^2z+c_4y^3x)+wf_{10}(x,y,z,t)+f_{17}(x,y,z,t)=0.
$$
 The
pencil $|-2K_{X_{17}}|$ is a Halphen pencil. Moreover, if the
defining equation of $X_{17}$ is sufficiently general, then this
is the only Halphen pencil on $X_{17}$ (see
\cite[Corollary~1.1]{ChPa09}). Suppose that $c_1=c_2=0$ and $c_3\ne
0$. Then we may assume that  $c_3=1$ and $c_4=0$ by a coordinate
change. Here we encounter an extra Halphen pencil. Indeed, the pencil on $X_{17}$ cut out
by $\lambda x^3+\mu z=0$, where $[\lambda:\mu]\in\mathbb{P}^1$, is a Halphen pencil contained in $|-3K_{X_{17}}|$ and  different from the Halphen pencil
$|-2K_{X_{17}}|$.
\end{exam}

\subsection*{Calabi problem}
 In many applications it
is useful to \emph{measure} how singular effective
$\mathbb{Q}$-divisors $D$ equivalent to $-K_X$  can be. A possible \emph{measurement} is
given by the so-called $\alpha$-invariant of the Fano hypersurface
$X$. It is defined by the number
$$
\alpha(X)=\mathrm{sup}\left\{\lambda\in\mathbb{Q}\ \left|%
\aligned
&\text{the log pair}\ \left(X, \lambda D\right)\ \text{is Kawamata log terminal}\\
&\text{for every effective $\mathbb{Q}$-divisor}\ D\sim_{\mathbb{Q}} -K_{X}.\\
\endaligned\right.\right\}.%
$$

If $X$ is a general hypersurface in its family, then $\alpha(X)=1$ by
\cite[Theorem~1.3]{Ch08} and \cite[Theorem~1.15]{Ch08b} except the
case when $X$ belongs to the families No. 1, 2, 3, 4 or 5. If $X$ is
a general quartic threefold in $\mathbb{P}^3$ (the family No. 1), we
have $\alpha(X)\ge\frac{7}{9}$ by \cite[Theorem~1.1.6]{ChPaWon}.
If $X$ is a double cover of $\mathbb{P}^3$ ramified along smooth
sextic surface (the family No. 3), then all possible values of
$\alpha(X)$ are found in \cite[Theorem~1.1.5]{ChPaWon}. For
general threefolds in the families No. 2, 4 and 5, the bound
$\alpha(X)>\frac{3}{4}$ proved in \cite{Ch08} and \cite{Ch09}. In
particular, we have

\begin{corol}
\label{corollary:alpha} Let $X$ be a quasi-smooth hypersurface of
degrees $d$ with only terminal singularities in the weighted
projective space $\mathbb{P}(1, a_1, a_2, a_3, a_4)$, where
$d=\sum  a_i$. Suppose that $X$ is a general hypersurface in this
family. Then $\alpha(X)>\frac{3}{4}$.
\end{corol}

Similarly, we can define the $\alpha$-invariant of any Fano variety
with at most Kawamata log terminal singularities. This invariant
has been studied intensively by many people who used different
notations for it. The notation $\alpha(X)$ is due to Tian who
defined the $\alpha$-invariant in a different way (see
\cite{Tian}). However, his definition coincides with the one we
just gave (see \cite[Theorem~A.3]{ChSh08c}).

Tian proved in \cite{Tian} that a smooth Fano variety of dimension
$n$ whose $\alpha$-invariant is greater than $\frac{n}{n+1}$
admits a K\"ahler--Einstein metric. This result was generalized
for Fano varieties with quotient singularities by Demailly and
Koll\'ar (see \cite[Criterion~6.4]{DeKo01}). Thus,
Corollary~\ref{corollary:alpha} implies

\begin{theor}
\label{theorem:KE-95} Let $X$ be a quasi-smooth hypersurface of
degrees $d$ with only terminal singularities in the weighted
projective space $\mathbb{P}(1, a_1, a_2, a_3, a_4)$, where
$d=\sum  a_i$. Suppose that $X$ is a general hypersurface in this
family. Then $X$ is admits an orbifold K\"ahler--Einstein metric.
\end{theor}

Recently, Chen, Donaldson and Sun and independently Tian proved
that a smooth Fano variety admits a K\"ahler--Einstein metric if and
only if it is $K$-stable (see \cite{Donaldson2013-0},
\cite{Donaldson2013-1}, \cite{Donaldson2013-2},
\cite{Donaldson2013-3} and \cite{Tian2013}). Earlier Odaka and
Okada proved that birationally super-rigid smooth Fano varieties with base-point-free anticanonical linear systems
must be slope-stable (see \cite{OdakaSano}). Furthermore, Odaka and Sano
proved that Fano varieties of dimension $n$ with at most log terminal singularities
whose $\alpha$-invariants are greater than
$\frac{n}{n+1}$ must be $K$-stable (see \cite{OdakaSano}). These results suggest that
\emph{every} quasi-smooth hypersurface in  the 95 families of Fletcher
and Reid should admit an orbifold K\"ahler--Einstein metric.

Using methods we developed in the proof of Main Theorem, it is
possible to explicitly describe all quasi-smooth hypersurfaces in
the 95 families of Fletcher and Reid whose $\alpha$-invariants are
greater than $\frac{3}{4}$. All of them must admit orbifold
K\"ahler--Einstein metrics by \cite[Criterion~6.4]{DeKo01}.

The $\alpha$-invariants can be applied to the non-rationality problem on products of Fano varieties.
In particular, we can apply
\cite[Theorem~6.5]{Ch08} to quasi-smooth hypersurfaces in the 95 families of Fletcher and Reid
 whose $\alpha$-invariants are $1$.

\subsection*{Arithmetics}

As it was pointed out by Pukhlikov and Tschinkel, the
problem (1) is closely related to the problem of potential density
of rational points on $X$ in the case when $X$ is defined over a
number field. For example, if $\mathrm{Bir}(X)$ is infinite, then
we are able to show that $X$ contains infinitely many rational surfaces.
It implies the potential density of rational points on $X$.

The papers \cite{BoTsch}, \cite{BoTsch2}, \cite{HarrisTschinkel}
use birational transformations into elliptic fibrations  in order to prove the
potential density on all smooth Fano threefolds possibly except
double covers of $\mathbb{P}^3$ ramified along smooth sextic
surfaces (the family No. 3 in the list of Fletcher and Reid).


If $X$ is defined over a number field, it seems likely that the
set of rational points on $X$ is potentially dense. For every
smooth quartic threefold in $\mathbb{P}^4$ (the family No.~1), this
was proved by Harris and Tschinkel in \cite{HarrisTschinkel}. For
general  Fano hypersurfaces  in the families No. 2, 4, 5, 6, 7, 9, 11,
12, 13, 15, 17, 19, 20, 23, 25, 27, 30, 31, 33, 36, 38, 40, 41,
42, 44, 58, 61, 68 and 76, this was proved in \cite{ChPa05} and
\cite{ChPa11}. Despite many attempts, this problem is still open
for double covers of $\mathbb{P}^3$ ramified along smooth sextic
surfaces.

The methods we use in the proof of Main Theorem can be applied to
prove the potential density of rational points on the quasi-smooth
hypersurfaces  in some families in the list of Fletcher and Reid.
In fact, for some families we can use our methods to prove the
density of rational points on $X$ (see \cite[Page~84 and
Section~5]{ChPa05}).

\subsection*{Fano threefold complete intersections}

In 2013 and 2014, after the present paper was announced on ArXiv,  new results on the birational rigidity of Fano threefold complete intersections were introduced (\cite{AZu14}, \cite{Od13}). Like the 95 families of Fano threefold hypersurfaces, it is well known that there are 85 families of Fano threefold complete intersections of codimension $2$ (\cite[Table~6]{IF00}). In addition, it is also known that there is only one family of Fano threefold complete intersections of codimension $3$, \emph{i.e.}, complete intersections of three quadrics in $\mathbb{P}^6$.  There is no Fano threefold complete intersections of codimensions $4$ and higher (\cite{3Ch11}). The lists of Fano threefold complete intersections in \cite[Tables ~5, 6, and 7]{IF00} are proved to be complete (\cite{3Ch11}).  In 1996,  a general member in the family of complete intersections of quadrics and cubics in $\mathbb{P}^5$ is proved to be birationally rigid (\cite{IskovskikhPukhlikov}).
In 2013, Odaka announced that general members in 19 families out of the 85 families of Fano threefold complete intersections of codimension $2$ are birationally rigid and that general members in the other 64 families are not birationally rigid (\cite{Od13}). After Odaka, a proof of the birational rigidity of quasi-smooth complete intersections in the 19 families (except the family of smooth complete intersections of quadrics and cubics in $\mathbb{P}^5$)  is announced by Ahmadinezhad and Zucconi (\cite{AZu14}).

\newpage

\emph{Ivan Cheltsov}
\medskip

School of Mathematics, The
University of Edinburgh

  Edinburgh, EH9 3JZ, UK

 \texttt{I.Cheltsov@ed.ac.uk}

 \bigskip
 \bigskip

 \emph{Jihun Park}
 \medskip

 Center for Geometry and Physics, Institute for Basic Science (IBS)

77 Cheongam-ro, Nam-gu, Pohang, Gyeongbuk, 790-784, Korea.
\medskip

 Department of
Mathematics, POSTECH

77 Cheongam-ro, Nam-gu, Pohang, Gyeongbuk, 790-784, Korea.

\texttt{wlog@postech.ac.kr}


\newpage

\begin{thebibliography}{99}

\bibitem{AZu14}
H.~Ahmadinezhad, F.~Zucconi,
\emph{Birational rigidity of Fano 3-folds and Mori dream spaces},
arXiv:1407.3724.

\bibitem{Beauville}
A.~Beauville, \emph{Non-rationality of the symmetric sextic Fano
threefold},  Geometry and arithmetic, EMS Ser. Congr. Rep., Eur.
Math. Soc., Z\"urich (2012), 57--60.

\bibitem{BCHM}
C.~Birkar, P.~Cascini, C.~Hacon, J.~McKernan, \emph{Existence of
minimal models for varieties of log general type},  J. Amer. Math.
Soc. \textbf{23} (2010), 405--468.

\bibitem{BoTsch}
F.~Bogomolov, Yu.~Tschinkel, \emph{On the density of rational
points on elliptic fibrations},  J. Reine Angew. Math.,
\textbf{511} (1999), 87--93.

\bibitem{BoTsch2}
F.~Bogomolov, Yu.~Tschinkel, \emph{Density of rational points on
elliptic K3 surfaces}, Asian J. Math. \textbf{4} (2000), 351--368.

\bibitem{Ch00}
I.~Cheltsov, \emph{Log models of birationally rigid varieties}, J.
Math. Sci. (New York) \textbf{102} (2000), 3843--3875.

\bibitem{Ch07}
I.~Cheltsov, \emph{Elliptic structures on weighted
three-dimensional Fano hypersurfaces}, Izv. Math., \textbf{71}
(2007) 765--862.

\bibitem{Ch08}
I.~Cheltsov, \emph{Fano varieties with many selfmaps}, Adv. Math.
\textbf{217} (2008), 97--124.

\bibitem{Ch08b} I.~Cheltsov, \emph{Log canonical
thresholds of Fano threefold hypersurfaces}, Izv. Math.,
\textbf{73} (2009) 727--795.

\bibitem{Ch09}
I.~Cheltsov, \emph{Extremal metrics on two Fano manifolds}, Sb.
Math. \textbf{200} (2009), 95--132.

\bibitem{ChKa10}
I.~Cheltsov, I.~Karzhemanov, \emph{Halphen pencils on quartic
threefolds}, Adv. Math. \textbf{223} (2010), 594--618.

\bibitem{ChPa05}
I.~Cheltsov, J.~Park, \emph{Weighted Fano threefold
hypersurfaces}, J. Reine Angew. Math., \textbf{600} (2006),
81--116.

\bibitem{ChPa07}
I.~Cheltsov, J.~Park, \emph{Two remarks on sextic double solids},
J. Number Theory \textbf{122} (2007), 1--12.

\bibitem{ChPa09}
I.~Cheltsov, J.~Park, \emph{Halphen Pencils on weighted Fano
threefold hypersurfaces}, Cent. Eur. J. of Math., \textbf{7}
(2009), 1--45.

\bibitem{ChPa11}
I.~Cheltsov, J.~Park, \emph{Sextic double solids}, Cohomological
and geometric approaches to rationality problems, Progr. Math.
\textbf{282}, Birkh\"auser Boston, Inc., Boston, MA (2010),
75--132.

\bibitem{ChPaWon}
I.~Cheltsov, J.~Park, J.~Won, \emph{Log canonical thresholds of
certain Fano hypersurfaces}, Math. Z. \textbf{276} (2014), no. 1-2, 51--79.

\bibitem{ChSh08c}
I.~Cheltsov, C.~Shramov, \emph{Log canonical thresholds of smooth
Fano threefolds} (with an appendix by Jean-Pierre Demailly),
Russian Math. Surveys \textbf{63} (2008), 73--180.


\bibitem{3Ch11}
J.-J.~Chen, J.~Chen and M.~Chen, \emph{On quasismooth weighted complete intersections},
J. Algebraic Geom. \textbf{20} (2011),  239--262.

\bibitem{Donaldson2013-0}
X.-X.~Chen, S.~Donaldson, S.~Sun, \emph{K\"ahler--Einstein metrics
and stability},
Int. Math. Res. Not. \textbf{2014}, no. 8, 2119--2125.

\bibitem{Donaldson2013-1}
X.-X.~Chen, S.~Donaldson, S.~Sun, \emph{K\"ahler--Einstein metrics
on Fano manifolds, I: approximation of metrics with cone
singularities}, preprint, arXiv:1211.4566 (2012).

\bibitem{Donaldson2013-2}
X.-X.~Chen, S.~Donaldson, S.~Sun, \emph{K\"ahler--Einstein metrics
on Fano manifolds, II: limits with cone angle less than $2\pi$}, to appear in J. Amer. Math. Soc..

\bibitem{Donaldson2013-3}
X.-X.~Chen, S.~Donaldson, S.~Sun, \emph{K\"ahler--Einstein metrics
on Fano manifolds, III: limits as cone angle approaches $2\pi$ and
completion of the main proof}, to appear in J. Amer. Math. Soc..

\bibitem{Co95}
A.~Corti, \emph{Factorizing birational maps of threefolds after
Sarkisov}, J. of Algebraic Geom. \textbf{4} (1995), 223--254.

\bibitem{Co00}
A.~Corti, \emph{Singularities of linear systems and 3-fold
birational geometry},
L.M.S. Lecture Note Series \textbf{281} (2000), 259--312.%

\bibitem{CPR}
A.~Corti, A.~Pukhlikov, M.~Reid, \emph{Fano 3-fold hypersurfaces},
L.M.S. Lecture Note Series \textbf{281} (2000), 175--258.%

\bibitem{DeKo01}
J.-P.\,Demailly, J.\,Koll\'ar, \emph{Semi-continuity of complex
singularity exponents and K\"ahler-Einstein metrics on Fano
orbifolds},  Ann. Sci. \'Ecole Norm. Sup. \textbf{34} (2001),
525--556.

\bibitem{dF}
T.~de~Fernex, \emph{Birationally rigid hypersurfaces}, Invent.
Math. \textbf{192} (2013), 533--566.

\bibitem{dFE}
T.~de~Fernex, L.~Ein, \emph{Resolution of indeterminacy of pairs}, Algebraic geometry, 165--177, de Gruyter, Berlin, 2002.


\bibitem{IF00}
A.~Fletcher, \emph{Working with weighted complete intersections},
L.M.S. Lecture Note Series \textbf{281} (2000), 101--173.




\bibitem{HarrisTschinkel}
J.~Harris, Yu.~Tschinkel, \emph{Rational points on quartics}, Duke
Math. J. \textbf{104} (2000), 477--500.

\bibitem{Is79}
V.~Iskovskikh, \emph{Birational automorphisms of three-dimensional
algebraic varieties}, Current problems in mathematics \textbf{12},
159--236, VINITI, Moscow, 1979.

\bibitem{Is01}
V.~Iskovskikh, \emph{Birational rigidity of Fano hypersurfaces in
the framework of Mori theory}, Russian Math. Surveys \textbf{56}
(2001) 207--291.

\bibitem{IsM71}
V.~Iskovskikh, Yu.~Manin, \emph{Three-dimensional quartics and
counterexamples to the L\"roth problem},  Mat. Sb. \textbf{86}
(1971), 140--166.

\bibitem{IskovskikhPukhlikov}
V.~Iskovskikh, A.~Pukhlikov, \emph{Birational automorphisms of
multidimensional algebraic manifolds}, J. Math. Sci. \textbf{82}
(1996) 3528--3613.

\bibitem{Ka88}
Y.~Kawamata, \emph{Crepant blowing-up of 3-dimensional canonical singularities and its application to degenerations of surfaces},
 Ann. of Math. (2) \textbf{127} (1988), no. 1, 93--163.
 
  
\bibitem{Ka96}
Y.~Kawamata,  \emph{Divisorial contractions to $3$-dimensional
terminal quotient singularities},
Higher-dimensional complex varieties (Trento, 1994), de Gruyter, Berlin (1996), 241--246.%

\bibitem{KMM94} S.~Keel, K.~Matsuki, J.~McKernan, \emph{Log abundance theorem for threefolds}, Duke Math. J. \textbf{75}
(1994), 99--119.


\bibitem{KoMo98}
J.~Koll\'ar, S.~Mori, \emph{Birational geometry of algebraic
varieties}, Cambridge University Press (1998).

\bibitem{OdakaOkada}
Yu.~Odaka, T.~Okada, \emph{Birational superrigidity and slope
stability of Fano manifolds}, Math.~Z. \textbf{275} (2013), no. 3-4, 1109--1119.

\bibitem{OdakaSano}
Yu.~Odaka, Yu. Sano, \emph{Alpha invariant and $K$-stability of
$\mathbb{Q}$-Fano varieties}, Adv. Math. \textbf{229} (2012),
2818--2834.


\bibitem{Od13}
T.~Okada, \emph{Birational Mori fiber structures of
$\mathbb{Q}$-Fano $3$-fold weighted complete intersections},  to
apear in Proc. Lond. Math. Soc..

\bibitem{Pro}
Yu.~Prokhorov, \emph{Lectures on complements on log surfaces},  MSJ Memoirs, \textbf{10}, Mathematical Society of Japan, Tokyo, 2001. viii+130 pp. 
 
\bibitem{Prokhorov}
Yu.~Prokhorov, \emph{Simple finite subgroups of the Cremona group
of rank $3$},  J. Algebraic Geom. \textbf{21} (2012), 563--600.

\bibitem{Pu98} A.~Pukhlikov, \emph{Birational automorphisms of Fano hypersurfaces}, Invent. Math. \textbf{134} (1998), no. 2,
401--426.

\bibitem{Rei79}
M.~Reid, \emph{Canonical $3$-folds}, Journ\'ees de G\'eometrie
Alg\'ebrique d'Angers, Juillet 1979/Algebraic Geometry, Angers,
1979, pp. 273--310, Sijthoff \& Noordhoff, Alphen aan den
Rijn--Germantown, Md., 1980.


\bibitem{Shioda}
T.~Shioda, \emph{On elliptic modular surfaces}, J. Math. Soc. Japan, \textbf{24} (1972), no.~1, 20--59.

\bibitem{Silverman}
J.~Silverman, \emph{The arithmetic of
elliptic curves}, Second edition. Graduate Texts in Mathematics
\textbf{106}. Springer, Dordrecht, 2009.

\bibitem{Sho93}
V.~Shokurov, \emph{3-fold log flips}, Russian Acad. Sci. Izv.
Math. \textbf{40 }(1993), no. 1, 105--203.

\bibitem{Tian}
G.~Tian, \emph{On K\"ahler--Einstein metrics on certain K\"ahler
manifolds with $c_{1}(M)>0$}, Invent. Math. \textbf{89} (1987),
225--246.

\bibitem{Tian2013}
G.~Tian, \emph{$K$-stability and K\"ahler--Einstein metrics},
preprint, arXiv:1211.4669 (2012).


\end{thebibliography}
\end{document}